\documentclass{article}
\usepackage[margin=1.25in]{geometry}

\usepackage[utf8]{inputenc}
\usepackage{graphics}
\usepackage{amsmath}
\usepackage{amssymb}
\usepackage{bbm}
\usepackage{mathtools}
\usepackage{amsfonts}
\usepackage{bbold}
\usepackage{stmaryrd}
\usepackage{relsize}
\usepackage[shortlabels]{enumitem}
\usepackage[hypertexnames=false]{hyperref}
\usepackage[dvipsnames]{xcolor}
\usepackage{verbatim}
\usepackage{caption, subcaption}
\usepackage{float}
\usepackage{overpic}

\usepackage{anyfontsize}

\usepackage{titlesec}
\titleformat*{\section}{\large\bfseries}
\titleformat*{\subsection}{\normalsize\bfseries}
\titleformat*{\subsubsection}{\normalsize\bfseries}
\titleformat*{\paragraph}{\normalsize\bfseries}
\titleformat*{\subparagraph}{\large\bfseries}

\usepackage{tikz-cd,mathtools}
\tikzset{mydescription/.style={anchor=center,fill=white}}

\usetikzlibrary{calc}
\tikzset{curve/.style={settings={#1},to path={(\tikztostart)
    .. controls ($(\tikztostart)!\pv{pos}!(\tikztotarget)!\pv{height}!270:(\tikztotarget)$)
    and ($(\tikztostart)!1-\pv{pos}!(\tikztotarget)!\pv{height}!270:(\tikztotarget)$)
    .. (\tikztotarget)\tikztonodes}},
    settings/.code={\tikzset{quiver/.cd,#1}
        \def\pv##1{\pgfkeysvalueof{/tikz/quiver/##1}}},
    quiver/.cd,pos/.initial=0.35,height/.initial=0}

\usepackage{amsthm}

\usepackage[
style=alphabetic,
sorting=nyt,
maxcitenames=99,
mincitenames=99,
minbibnames=99,
maxbibnames=99,
minalphanames=5,
maxalphanames=10,
backend=biber
]{biblatex}
\newtheorem{theorem}{Theorem}[section]

\newtheorem*{thm*}{Theorem}
\newtheorem{proposition}[theorem]{Proposition}
\newtheorem*{prop*}{Proposition}
\newtheorem{corollary}[theorem]{Corollary}
\newtheorem*{cor*}{Corollary}
\newtheorem{lemma}[theorem]{Lemma}
\newtheorem*{lemma*}{Lemma}

\theoremstyle{definition}
\newtheorem{definition}[theorem]{Definition}
\newtheorem*{defn*}{Definition}
\newtheorem{example}[theorem]{Example}
\newtheorem{observation}[theorem]{Observation}

\newtheorem{remark}[theorem]{Remark}
\newtheorem*{rem*}{Remark}

\newtheorem{notation}[theorem]{Notation}

\makeatletter
\newcommand\nparagraph{\@startsection{paragraph}{4}{\z@}%
                                    {3.25ex \@plus1ex \@minus.2ex}%
                                    {-1em}%
                                    {\normalfont\normalsize}}
\makeatother

\makeatletter
\newtheorem*{rep@theorem}{\rep@title}
\newcommand{\newreptheorem}[2]{%
\newenvironment{rep#1}[1]{%
 \def\rep@title{#2 \ref{##1}}%
 \begin{rep@theorem}}%
 {\end{rep@theorem}}}
\makeatother

\newreptheorem{theorem}{Theorem}
\newreptheorem{lemma}{Lemma}

\mathchardef\mhyphen="2D 
\mathtoolsset{showonlyrefs}   

\newcommand{\Aa}{\mathcal{A}}
\newcommand{\Ba}{\mathcal{B}}
\newcommand{\Ca}{\mathcal{C}}
\newcommand{\Da}{\mathcal{D}}
\newcommand{\Ea}{\mathcal{E}}
\newcommand{\Fa}{\mathcal{F}}
\newcommand{\Ga}{\mathcal{G}}

\newcommand{\Va}{\mathcal{V}}
\newcommand{\Ma}{\mathcal{M}}
\newcommand{\Na}{\mathcal{N}}
\newcommand{\Ia}{\mathcal{I}}
\newcommand{\Oa}{\mathcal{O}}
\newcommand{\Ra}{\mathcal{R}}
\newcommand{\Ua}{\mathcal{U}}

\newcommand{\Aaa}{\mathsf{A}}

\newcommand{\Caa}{\mathsf{C}}
\newcommand{\Daa}{\mathsf{D}}

\newcommand{\Eaa}{\mathsf{E}}

\newcommand{\Pois}{\mathsf{aP}}
\newcommand{\BD}{\mathsf{BD}}
\newcommand{\E}{\mathsf{E}}
\newcommand{\fE}{\mathsf{fE}}
\newcommand{\Cat}{\mathsf{Cat}}
\newcommand{\Pres}{\mathsf{Pres}}
\newcommand{\PCat}[1]{\mathsf{aP}_{#1}\mhyphen\mathsf{Cat}}
\newcommand{\BDCat}[1]{\mathsf{BD}_{#1}\mhyphen\mathsf{Cat}}
\newcommand{\VCat}{\mathcal{V}\mhyphen\mathsf{Cat}}
\newcommand{\VPres}{\mathcal{V}\mhyphen\mathsf{Pres}}
\newcommand{\VRex}{\mathcal{V}\mhyphen\mathsf{Rex}}
\newcommand{\VGph}{\mathcal{V}\mhyphen\mathsf{Gph}}
\newcommand{\lmod}{\mhyphen\textnormal{Mod}}
\newcommand{\Rep}{\operatorname{Rep}}
\newcommand{\Obs}{\operatorname{Obs}}
\newcommand{\cl}{\operatorname{cl}}
\newcommand{\loc}{\operatorname{loc}}
\newcommand{\Hom}{\operatorname{Hom}}
\newcommand{\Fun}{\operatorname{Fun}}
\newcommand{\End}{\operatorname{End}}

\newcommand{\Map}{\operatorname{Map}}
\newcommand{\colim}{\operatorname{colim}}
\newcommand{\Obj}{\operatorname{Ob}}
\newcommand{\id}{\operatorname{id}}
\newcommand{\op}{\operatorname{op}}
\newcommand{\ori}{\operatorname{or}}
\newcommand{\st}{\operatorname{st}}
\newcommand{\ad}{\operatorname{ad}}
\newcommand{\act}{\operatorname{act}}
\newcommand{\QCoh}{\operatorname{QCoh}}
\newcommand{\Man}{\mathsf{Man}}
\newcommand{\Disk}{\mathsf{Disk}}

\newcommand{\Comp}[1]{\widehat{#1}} 

\newcommand{\BalFunA}{\operatorname{Bal}_{\Aa}} 

\newcommand{\IEAsymbol}{\underline{\mathrm{A}}}
\newcommand{\IEAlong}[2]{\mathrm{\underline{End}}_{#1}{(#2)}}
\font\maljapanese=dmjhira at 2.5ex
\newcommand{\yo}{\textrm{\!\maljapanese\char"48}}

\newcommand{\R}{\mathbb{R}}
\newcommand{\C}{\mathbb{C}}
\newcommand{\Ce}{{\mathbb{C}_\varepsilon}}
\newcommand{\Ch}{{\mathbb{C}[[\hbar]]}}

\newcommand{\ootimesi}{
  \mathchoice
    {\mathbin{\raisebox{-1.1pt}{{\includegraphics[width=0.75em]{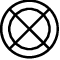}}}}}
    {\mathbin{\raisebox{-1.1pt}{{\includegraphics[width=0.75em]{pics/ootimesi}}}}}        
    {\mathbin{\raisebox{-1.1pt}{{\includegraphics[width=0.60em]{pics/ootimesi}}}}}
    {\mathbin{\raisebox{-1.1pt}{{\includegraphics[width=0.45em]{pics/ootimesi}}}}}}

\newcommand{\monprod}[1][{}]{\otimes^{#1}}
\newcommand{\catprod}{\ootimesi}
\newcommand{\relcatprod}[1]{\catprod_{\!#1}}
\newcommand{\coendtens}[1]{\relcatprod{{#1}}^{\int}}
\newcommand{\Tambaratens}[1]{
\relcatprod{{#1}}^{\!\textnormal{T}}}

\newcommand{\CompProd}{\widehat{\otimes}}

\newcommand{\SkC}[1]{\mathbf{Sk}_{\Ca}(#1)}
\newcommand{\SkCat}[2]{\mathbf{Sk}_{#1}(#2)}
\newcommand{\SkCM}{\mathbf{Sk}_{\Ca}(M)}
\newcommand{\hra}{\hookrightarrow}
\newcommand{\ra}{\rightarrow}
\newcommand{\trm}{\textrm}
\newcommand{\cat}[1]{\mathcal{#1}}

\definecolor{Blue} {rgb} {0.282352,0.239215,0.803921}
\definecolor{Green} {rgb} {0.133333,0.545098,0.133333}
\definecolor{Red}   {rgb} {0.803921,0.000000,0.000000}
\definecolor{Violet}{rgb} {0.580392,0.000000,0.827450}
\definecolor{darkspringgreen}{rgb}{0.09, 0.45, 0.27}
\definecolor{contrastblue} {RGB} {0, 167, 212}
\definecolor{contrastred} {RGB} {212, 0, 0}

\usepackage{scalerel}
\newcommand{\pics}[1]{\raisebox{-0.3\depth}{\scaleobj{2}{\scalerel*{\includegraphics{pics/#1}}{|}}}}

\newcommand{\Vect}{\operatorname{Vect}_{\mathbb{k}}}
\newcommand{\im}{\operatorname{Im}}
\newcommand{\iotaman}{\iota_{m' \triangleleft a,n}^{m', a \triangleright n} }
\newcommand{\iotamanInv}{\left(\iota_{m \triangleleft a,n'}^{m, a \triangleright n'} \right)^{-1}}
\newcommand{\Free}{\mathrm{Free}}
\newcommand{\fgt}{\mathrm{fgt}}
\newcommand{\infy}{\(\infty\)}

\begin{document}

	\vspace*{0.5cm}
	\begin{center}	\textbf{\large{Deformation Quantization via Categorical Factorization Homology}}\\	\vspace{1cm}	{\large  Eilind Karlsson,  Corina Keller, Lukas Müller and J\'{a}n Pulmann }\\ 	\vspace{5mm}

	\end{center}	\vspace{0.3cm}	
	\begin{abstract}\noindent 
			This paper develops an approach to categorical deformation quantization via factorization homology. We show that a  quantization of the local coefficients for factorization homology is equivalent to consistent quantizations of its value on manifolds. To formulate our results we introduce the concepts of shifted almost Poisson and $\BD$ categories. 

        Our main example is the character stack of flat principal bundles for a reductive algebraic group $G$, where we show that applying the general framework to the Drinfeld category reproduces deformations previously introduced by Li-Bland and \v{S}evera. As a direct consequence, we can conclude a precise relation between their quantization and those introduced by Alekseev, Grosse, and Schomerus. 

        To arrive at our results we compute factorization homology with values in a ribbon category enriched over complete $\C[[\hbar]]$-modules. More generally, we define enriched skein categories which compute factorization homology for ribbon categories enriched over a general closed symmetric monoidal category $\Va$. 
	\end{abstract}

\setcounter{tocdepth}{2}
\tableofcontents
\section{Introduction} 
This paper investigates and utilizes the interplay between factorization homology and categorical deformation quantization. An important inspiration for us is the construction of quantizations of character varieties through factorization homology for representation categories of quantum groups by Ben-Zvi, Brochier, and Jordan~\cite{BZBJIntegrating, BZBJ2}. Our results can be seen as a further development of the ideas presented in their papers considering formal deformations, for example those arising from Drinfeld associators or Drinfeld-Jimbo quantum groups. In addition, this work provides a systematic framework in which to interpret their results.   

Before discussing the content of this paper we describe the physical motivation for our work.
A classical field theory on a spacetime manifold $M$ assigns to every open region $\mathcal{U}\subset M$ a `space' 
of physical field configurations on \(\mathcal{U}\), denoted $\mathcal{F}(\mathcal{U})$. Usually, $\mathcal{F}(\mathcal{U})$ is the derived 
space of solutions to a set of non-linear partial differential equations. This assignment 
is local in $\mathcal{U}$, i.e.\ $\mathcal{F}$ forms a sheaf on $M$, or more generally on the 
category of all manifolds the theory is defined on.

Classical observables $\mathcal{O}_{\mathcal{F}(\mathcal{U})}$ localized on $\mathcal{U}$ are functions on $\mathcal{F}(\mathcal{U})$. Making this mathematically precise is a notoriously hard problem. In their two books~\cite{CG1, CG2} Costello and Gwilliam work out the details in classical perturbation theory, i.e.\ they study the formal neighborhood of a fixed classical solution. The locality of $\mathcal{F}$ implies that the functions on this formal neighborhood form a factorization algebra $\Obs^{\cl}$ with values in the category of chain complexes. 
The `space' $\mathcal{F}(\mathcal{U})$ is usually described as the critical values of an action functional, which equips the algebra of classical observables $\mathcal{O}_{\mathcal{F}(\mathcal{U})}$ with a local (shifted) Poisson structure, in the sense 
that $\Obs^{\cl}$ is a factorization algebra with values in $P_0$-algebras~\cite[Theorem~5.2.2]{CG2}. A perturbative quantization consists of a deformation of $\Obs^{\cl}$ to a factorization algebra with values in $BD_0$-algebras (for Beilinson-Drinfeld). The main result of~\cite{CG2} is a construction of such quantizations using Feynman graph techniques. 

The locality of a factorization algebra suggest a different approach to quantization. Namely, to construct compatible deformation quantizations of the values on disks and glue those back together. This approach is particularly promising for topological field theories, corresponding to locally constant factorization algebras. 
If we assume that the theory under consideration can be defined on all $n$-dimensional manifolds and is functorial with respect to embeddings, we can use factorization homology to implement this approach. More concretely, we consider theories where the classical observables give rise to a (symmetric monoidal) functor $\Obs^{\cl} \colon \Man_n\to \operatorname{P}_0\text{-Alg}$ from the topological category of manifolds and embeddings to the category of $P_0$-algebras. The local to global property for factorization algebras is equivalent to this being a homology theory in the sense of Ayala-Francis~\cite{AFfh,primer}. According to \cite[Theorem~3.24]{AFfh} every homology theory can be recovered via factorization homology from its restriction to $n$-dimensional disks $\Obs^{\cl}_{\loc} \colon \Disk_n\to \operatorname{P}_0\text{-Alg}$: 
\begin{align}
    \Obs^{\cl}(M) \cong \int_M \Obs^{\cl}_{\loc} \ . 
\end{align}
A global deformation quantization can now be constructed by first finding a local quantization $\Obs^{q}_{\loc} \colon \Disk_n\to \operatorname{BD}_0\text{-Alg}$ and then defining the global quantum observables as 
\begin{align}
    \Obs^{q}(M) \cong \int_M \Obs^{q}_{\loc} \ . 
\end{align}
For this construction to work one has to ensure that factorization homology commutes with taking classical limits. Symmetric monoidal functors $\Disk_n\to \operatorname{P}_0\text{-Alg}$ are equivalent to framed $\E_n$-algebras in the symmetric monoidal category $\operatorname{P}_0\text{-Alg}$. Ignoring the framed part for the moment, we can conclude from Poisson additivity~\cite{SafronovBracesPoissonAdditivity} that the local classical observables can be encoded into a single $\operatorname{P}_n$-algebra $\Obs^{\cl}_n$. It is natural to expect that the local quantum observables are a deformation quantization of this $\operatorname{P}_d$-algebra to a $\operatorname{BD}_d$-algebra. However, for this to be true a version of BD-additivity is needed, which to the best of our knowledge has not been established in the literature. By formality every $P_i$-algebra admits a deformation quantization to a $\operatorname{BD}_i$-algebra for $i\geq 2$ \cite{Kontsevich2003, CW2015, CPTVV}. 

Extending this picture beyond perturbation theory  leads to additional complications. One of the first problems one encounters is that the derived space of solutions $\mathcal{F}(\mathcal{U})$ is not determined by its algebra of (derived) functions. As a consequence $\mathcal{O}_\mathcal{F}(\mathcal{U})$ is in general not local. A simple example of this is gauge theory with finite gauge group $G$, also known as Dijkgraaf-Witten theory~\cite{DijkgraafWitten}. In this case, the space of fields is the stack $\operatorname{Bun}_G(\mathcal{U}) $ of principal $G$-bundles on $\mathcal{U}$. For $\mathcal{U}=\mathbb{D}^n$ an $n$-dimensional disk, this space is equivalent to $BG$, but the algebra of derived functions is $\mathcal{O}(BG)= C^\bullet(G,\C )\cong \C$. Hence, there are no non-trivial local functions. Moreover, the (co)sheaf condition fails for the algebras of functions. For instance, on the circle one finds $\mathcal{O}_{\operatorname{Bun}_G(\mathbb{S}^1)}= \operatorname{Cl}(G)$, the algebra of class functions on $G$. 
 
Categorified functions such as (nice) sheaves on $\mathcal{F}(\mathcal{U})$ and higher versions thereof usually have better locality properties.\footnote{This idea has also appeared in the context of AQFT~\cite{AQFTCat}.} In the previous example the categorical algebra of higher functions on $\mathbb{D}^n$ becomes $\Rep G $, the category of $G$-representations, which has better gluing behaviour.   
Hence, it seems fruitful to extend the local to global perspective on quantization from algebras of functions to categories of higher functions. 
\emph{The goal of this paper is to develop and apply such a categorical deformation theory via factorization homology in the non-derived setting.} The first question we have to address is what the categorifications of Poisson and BD algebras should be. The guess of algebras over the Poisson operad or BD operad in the bicategory of linear categories does not work (naively interpreted) because both operads are linear whereas the collection of functors between linear categories forms a (linear) category. 
Concretely, the problem is that trying to impose the graded Jacobi identity would in general involve taking the difference of linear functors, which is not defined.
We adapt a pragmatic approach to these definitions by defining \emph{almost Poisson structures} as first order deformations and \emph{BD structures} as formal deformations. With these definitions we are able to work out all the details for the program sketched above (including BD-additivity). 

Our main example is the moduli stack of flat principal $G$-bundles\footnote{To be precise, we always consider the character stack, which is analytically but not algebraically equivalent to the stack of flat bundles. We hope the reader will forgive us the slight inaccuracy.} on a Riemann surface for a reductive algebraic group $G$, which carries the Atiyah-Bott 0-shifted symplectic structure\footnote{The reason for this being 0-shifted instead of (-1)-shifted as in the work of Costello-Gwilliam is that the corresponding classical field theory is Chern-Simons theory on $\Sigma\times \R$. By Poisson additivity we can trade the $\R$-direction for a shifting from (-1)-shifted structures to 0-shifted ones. This explanation of Poisson structures on phase spaces from BV formalism and additivity appeared recently in \cite[Sec.~3.4.2]{GwilliamRejzner2017}.}. As in the case of Dijkgraaf-Witten theory, the classical local observables can be encoded by the category $\Rep G$ of $G$-representations.  
In~\cite{BZBJIntegrating, BZBJ2} Ben-Zvi, Brochier, and Jordan use factorization homology of the representation categories $\Rep^q(G)$ of quantum groups over a surface $\Sigma$ to construct functorial quantizations of this symplectic structure. An alternative deformation of $\Rep G$ is the Drinfeld category $U\mathfrak{g}\lmod^\Phi[[\hbar]]$. We show that factorization homology for the Drinfeld category recovers the quantization constructed by Li-Bland and \v{S}evera~\cite{LBSQ1}. There is an equivalence of ribbon categories between $U(\mathfrak{g})\lmod^\Phi[[\hbar]]$ and $\Rep_\hbar G$. Our formalism allows us to conclude directly from this equivalence that the corresponding quantizations are equivalent.        

To show that we recover known quantizations of the moduli spaces of flat bundles we have to explicitly compute factorization homology in the $(2,1)$-category of categories enriched over complete $\C[[\hbar]]$-modules. We do this by considering a $\C[[\hbar]]$-linear version of the usual skein category~\cite{Walker,Cooke19}. The proof that this computes factorization homology does not depend on the specifics of the category of complete $\C[[\hbar]]$-modules: We develop a version of skein theory for ribbon categories $\Ca$ enriched over an arbitrary closed symmetric monoidal category $\Va$ and show that the enriched skein categories $\SkC{\Sigma}$ associated to an oriented surface $\Sigma$ compute factorization homology.

The tools developed in this paper should also be useful to study Poisson structures and their deformation quantization in settings with defects or on orbifolds, by studying local quantizations and globalization via factorization homology. {Similarly, Batalin-Vilkovisky operators on moduli spaces of flat bundles were constructed in \cite{ANPS}. The construction is parallel to that of the (quasi)-Poisson structures mentioned above, with e.g. fusion \cite[Prop.~4]{ANPS} and skein-theoretic version \cite[Thm.~13]{ANPS}.  It is natural to ask whether there exists an odd extension of the current work, computing factorization homology of the category  $U\mathfrak{g}\lmod$ where $\mathfrak{g}$ has an \emph{odd} invariant inner product.} 

We now explain our main results in more detail. 
\subsection*{Outline and summary of results}  
The paper is split into two parts, where in the first part we extend skein theory to enriched ribbon categories. The second discusses the interaction of factorization homology and deformation quantization of categories.

\paragraph{Notation:}

Multiple types of tensor product appear in this work, sometimes simultaneously. Typically,  $\monprod[\Va]$ is the monoidal product of an enriching symmetric monoidal category $\Va$. $\Va$-enriched monoidal categories will usually be denoted by $\Aa$ or $\Ca$, and their tensor product is $\monprod[\Aa]$ resp. $\monprod[\Ca]$ or simply $\monprod{}$. Free cocompletions of these monoidal products are denoted with a hat $\Comp{\otimes}$.

One categorical level higher, the tensor product of two $\Va$-enriched SMCs will be denoted $\catprod$. A relative tensor product of $\Va$-enriched $\Aa$-module categories is written as $\catprod_\Aa$, or $\catprod_\Aa^{\textnormal{T}, \int}$, where $\textnormal{T}$ and $\int$ mark one of the (equivalent) models used to compute the relative tensor product. When considering presentable categories, their tensor and relative products use $\boxtimes$ instead of $\catprod$. Finally, still one categorical level up, the (fibered) product of 2-categories is denoted by $\Caa \times \Daa$.

\paragraph{Part I: Enriched skein theory.}
The first part is concerned with computing factorization homology in the symmetric monoidal 2-category $\VCat$ of $\Va$-enriched categories, where $\Va$ is a complete and cocomplete closed symmetric monoidal category. Section~\ref{Sec: FH} and Section~\ref{Sec: enriched categories} recall the necessary background on factorization homology and enriched category theory, respectively. Some of the more technical details are moved to Appendix~\ref{app:Background-VCat}. In particular, we prove there that the \((2,1)\)-category \(\VCat\) admits all bicolimits. 
Combining this with the fact that the tensor product of \(\VCat\) is a left adjoint we get that \(\VCat\) indeed is a suitable target for factorization homology (see Proposition \ref{prop:VCat-cocomplete-and-tensor-prod-commutes-with-colimits})

The content of Section~\ref{Sec: FH} and \ref{Sec: enriched categories} is mostly a recollection and combination of known results. One of the few new results in these sections concern concrete models for relative tensor products: For a monoidal $\Va$-category $\Aa$, a right $\Aa$-module category $\Ma$, and left $\Aa$-module category $\Na$, the relative tensor product \(\Ma \relcatprod{\Aa} \Na\) is defined as the bicolimit of the truncated simplicial diagram 
\begin{center}
\begin{tikzcd}[sep = small]
\Ma \catprod \Aa \catprod \Aa \catprod \Na \arrow[r,yshift=1ex] \arrow[r] \arrow[r,yshift=-1ex] & \Ma \catprod \Aa \catprod \Na \arrow[r,yshift=0.5ex] \arrow[r,yshift=-0.5ex] & \Ma \catprod \Na \ \ . 
\end{tikzcd}
\end{center}
We provide two different models which explicitly compute this bicolimit: 

\begin{itemize}
    \item The \emph{enriched Tambara relative tensor product} \(\Ma \Tambaratens{\Aa} \Na\), generalizing~\cite{Tambara, Cooke19}. This is a model defined explicitly in terms of generators and relations, which is convenient because it allows for very concrete manipulations. In Theorem~\ref{thm:eqTambaraBar} we prove that this computes the bicolimit for the relative tensor product. The drawback of this model is that it is hard to directly access the morphism spaces.    
    \item The \emph{coend relative tensor product} \(\Ma \coendtens{\Aa} \Na\) where $\Aa$ is assumed to be rigid\footnote{The same relative tensor product also appeared in \cite{AraujoGuuHudson2025}, which drew our attention to the fact that the rigidity assumption was missing in the first version of our paper.}. Its objects are pairs $(m,n)\in \Obj \Ma \times \Obj \Na $. The $\Va$-object of morphisms is given by the enriched coend (we refer to Section~\ref{Sec:coend_tensor_product} for details) 
    \begin{align}
    \mathcal{M}{\coendtens{\Aa}} \Na \left( m\otimes_{\Aa}n , m'\otimes_{\Aa}n' \right) \coloneqq \int^{a\in \Aa} \Ma(m,m'\triangleleft a ) \monprod[\Va] \Na(a\triangleright n,n') \in \Va \ \  .
    \end{align}
    In Proposition~\ref{prp:CoendTensRelTens} we show that this is an alternative model for the relative tensor product. The advantage of this model is that it comes with a direct description of morphism spaces, which will be very useful for results in the second part of the paper. 
\end{itemize}
In Section~\ref{Sec: Enriched skein} we define, for a $\Va$-ribbon category $\Ca$ and an oriented two dimensional manifold $\Sigma$, the \emph{enriched skein category} $\SkC{\Sigma}\in \VCat$. Roughly speaking, its objects are collections of points of $\Sigma$ labelled with the objects of $\Ca$. Intuitively, morphism spaces are ribbon diagrams in $\Sigma\times [0,1]$ where the ribbons are labelled with objects of $\Ca$ and coupons are labelled by morphisms in $\Ca$. 

\begin{figure}[H]
\centering
    \begin{overpic}[tics=10]{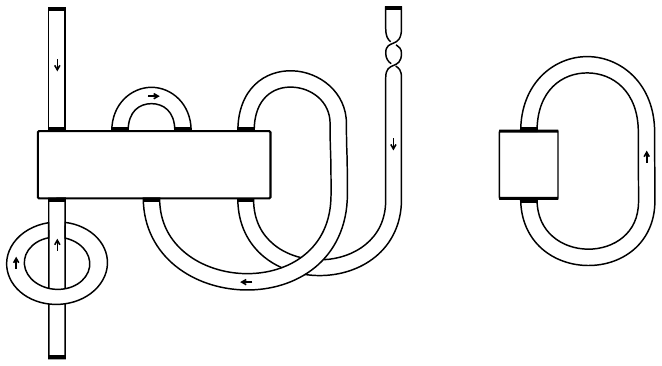}
         \put(6.7,26.5){$(\bullet$}
         \put(22.15,26.5){$\bullet)$}
          \put(36.4,26.5){$\bullet\phantom{)}$}

        \put(6.7, 32.5){$(\bullet$}
        \put(16, 32.5){$(\bullet$}
        \put(26.8, 32.5){$\bullet))$}
        \put(36.4, 32.5){$\bullet\phantom{(}$}

        \put(79.2, 32.5){$\bullet\phantom{(}$}
        \put(79.2,26.5){$\bullet\phantom{)}$}

        \put(10.5,5){$m_1$}
        \put(17,14.7){$m_2$}
        \put(10.5,45.5){$m_3$}
        \put(21.4,42.8){$m_4$}
        \put(43,45.5){$m_5$}
        \put(61.5,45.5){$m_6$}
        \put(87,47.5){$m_7$}
    \end{overpic}
	\caption{An example of a \(\Ca\)-colored ribbon q-graph \(\Gamma\); $m_i$ are objects of the ribbon category $\Ca$.}
	\label{fig:qRibbonGraphINTRO}
\end{figure}
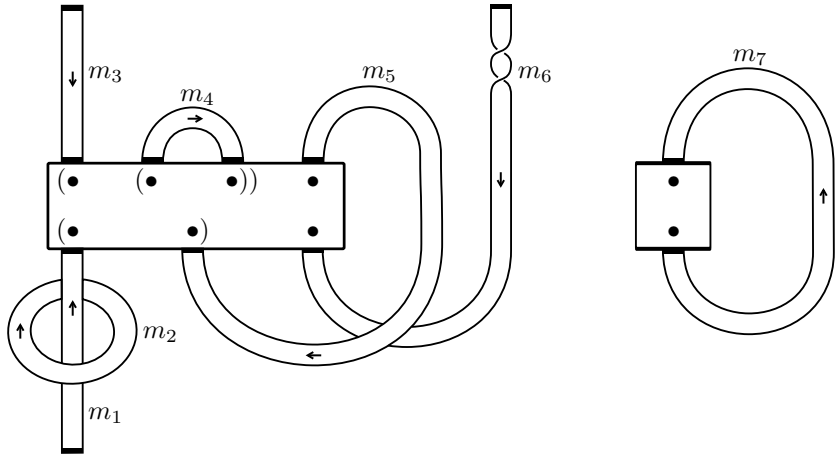

In a $\Va$-enriched category it only makes sense to talk about the $\Va$-object of morphism and not individual morphisms. Hence, the morphism objects in the enriched skein category are defined as appropriate colimits in $\Va$. In contrast to ordinary skein categories we do not assume that $\Ca$ is strict as a monoidal category. This requires us to additionally track the bracketing at every coupon. The reason for this is that one of the example we are interested in is the Drinfeld category $U(\mathfrak{g})\lmod^\Phi[[\hbar]]$ which is equipped with an interesting Drinfeld associator. 

The main result of the first part of our paper, which might be of independent interest, is:
 \begin{reptheorem}{Thm: FH=Sk}
Let $\Ca$ be a $\Va$-enriched ribbon category.
The enriched skein category computes factorization homology 
\begin{align}
    \SkC{\Sigma}\cong \int_\Sigma \Ca \in \Va\mhyphen \Cat \ \ .
\end{align}
 \end{reptheorem} 
Our proof is a generalization of Juliet Cooke's proof~\cite{Cooke19} of the same result in the $R$-linear case, which is based on an explicit verification that skein categories satisfy excision. The extension of her result to the enriched setting is Theorem~\ref{thm:Excision}. The general outline of the proof and many of its topological ingredients are exactly the same as in the $R$-linear case.    

\paragraph{Part II: Categorical deformation quantization.} 
The second part of this paper studies the connection between deformations of categories and factorization homology. The results of the first part are used for explicit computations in relation to the moduli stack of flat $G$-bundles on a Riemann surface.  

Traditional deformation quantization is concerned with deforming commutative algebras into non-commutative ones. We consider a $\C$-linear symmetric monoidal category to be the categorical analogue of a commutative algebra. Instead of only looking at deformations to monoidal categories, deformations to braided monoidal categories ($\E_2$-categories) or pointed categories ($\E_0$-categories) will also be useful. We call first order deformations, i.e. deformations to $\C_\varepsilon$-linear categories, \emph{almost Poisson} and formal deformations to $\C[[\hbar]]$-linear categories \emph{BD}.\footnote{Similar definitions of deformations of linear (compared to symmetric monoidal) $\infty$-categories have been introduced in~\cite{DAGXLurie} and used in~\cite{gepner2020integral}.}      
These can, loosely speaking, be defined as the following ``pullbacks'' of symmetric monoidal bicategories:
\begin{equation} \label{eq:DefnPCatBDnCatintro}
\begin{tikzcd}
	\PCat{i} & {\E_i(\Ce\mhyphen\mathsf{Cat})} && \BDCat{i} & {\E_i(\mathbb C[[\hbar]]\mhyphen \Cat}) \\
	{\E_\infty(\C \mhyphen\mathsf{Cat})} & {\E_i(\mathbb C\mhyphen\mathsf{Cat})} && {\E_\infty(\C \mhyphen \mathsf{Cat})} & {\E_i(\mathbb C \mhyphen\mathsf{Cat})}
	\arrow[hook, from=2-1, to=2-2]
	\arrow["{\otimes_{\Ce} \C}", from=1-2, to=2-2]
	\arrow[from=1-1, to=1-2]
	\arrow[from=1-1, to=2-1]
    \arrow[rd, "{\mbox{\fontsize{13}{15}\selectfont\( \lrcorner \)}}" description, very near start, phantom, start anchor={[xshift=-1.3ex, yshift=0.1ex]}, from=1-1, to=2-2]
	\arrow["{\otimes_{\Ch}\C}", from=1-5, to=2-5]
	\arrow[hook, from=2-4, to=2-5]
	\arrow[from=1-4, to=2-4]
	\arrow[from=1-4, to=1-5]
    \arrow[rd, "{\mbox{\fontsize{13}{15}\selectfont\( \lrcorner \)}}" description, very near start, phantom, start anchor={[xshift=-1.3ex, yshift=0.1ex]}, from=1-4, to=2-5]
\end{tikzcd}.
\end{equation}
A technical problem with this definition is that, to the best of our knowledge, there is no well-developed theory of pullbacks for symmetric monoidal bicategories. We circumvent this problem by explicitly defining what these pullbacks should be. Some of the more involved details related to this definition are spelled out in Appendix~\ref{appendix:pullback}. We do not restrict to `flat' deformations, since the property of being flat is not compatible with factorization homology, as we show explicitly in Example~\ref{Ex: torsion}.\footnote{This was pointed out to us by David Jordan.}  

Equipped with these definitions we prove a version of Dunn's additivity. 
\begin{reptheorem}{prop:additivity}
For $n = 0,1,2, \infty$ the following holds:
\begin{itemize}
    \item $\E_n(\Pois_0\mhyphen\Cat) \cong \Pois_n\mhyphen\Cat$
    \item $\E_n(\BD_0\mhyphen\Cat) \cong \BD_n\mhyphen\Cat$
\end{itemize}
\end{reptheorem} 
Giving the relevant definitions, providing examples, and proving this theorem is the content of Section~\ref{sect:DefQuant-Categories}. 

In Section~\ref{sect:Comp-FH-Quantization} we investigate the relation between these definitions and factorization homology. Firstly, we show that $\PCat{0}$ and $\BDCat{0}$ are tensor sifted cocomplete as (2,1)-categories (Proposition~\ref{lma:targetFH}), allowing us to compute factorization homology in those. Furthermore, we show that the value of factorization homology can be computed in $\C_\varepsilon\text{-}\Cat$ and $\C_\hbar\text{-}\Cat$, respectively. This allows us to use the results of the first part for concrete computations.    

We prove that deforming `local' functors (homology theories) from (oriented) two-dimensional manifolds to $\PCat{0}$ is equivalent to deforming the corresponding framed $\operatorname{aP}_2$-categories. More precisely we show 
\begin{reptheorem}{Thm: quant FH}
  The following diagram commutes and the horizontal functors are equivalences
\begin{equation}
\begin{tikzcd}
 \BD_2^\mathsf{f}\mhyphen\Cat \ar[r, "\sim"] \ar[d, "\lim_{h\rightarrow 0}",swap] & \fE_2(\BD_0\mhyphen\Cat) \ar[r, "\int"] \ar[d, "\lim_{h\rightarrow 0}"] & \Fun^{\otimes l.g.}(\Man_2,\BD_0\mhyphen\Cat) \ar[d, "\lim_{h\rightarrow 0}"] \\ 
  \Pois_2^\mathsf{f}\mhyphen\Cat \ar[r, "\sim"]  & \fE_2(\Pois_0\mhyphen\Cat) \ar[r, "\int"] & \Fun^{\otimes l.g.}(\Man_2,\Pois_0\mhyphen\Cat) 
\end{tikzcd} \ \ . 
\end{equation}
\end{reptheorem}

For a surface $\Sigma$ with boundary and $k$ marked points on its boundary $V\subset \partial \Sigma $, the embedding of disks near those points equips $\int_\Sigma \Ca$ with the structure of a module category over $\Ca^k$. Thus for every object ${o}\in \int_\Sigma \Ca $, one can define the \emph{internal endomorphism algebra} $\IEAlong{\Ca}{o}\in \widehat{\Ca}^{\,\boxtimes k} $ (see Definition \ref{defn:InternalEndAlg} and \cite[Def.~3.5]{safronovQMM}). Here $\widehat{\Ca}$ is the free cocompletion of $\Ca$ and $\boxtimes$ denotes the Kelly tensor product of locally presentable $\Va$-categories (see Section \ref{sect:the-2cat-VPres}). In Section~\ref{ssec:IntEndAlg} we show that categorical deformations also lead to deformations of internal endomorphism algebras.

The embedding of the empty set into $\Sigma$ induces a canonical pointing of $\int_\Sigma \Ca$ and we denote its internal endomorphism algebra by $\IEAsymbol_{\Sigma,V}$.
The last Section of our paper is concerned with describing the Poisson structures on $\IEAsymbol_{\Sigma,V}$ and their deformation quantization explicitly. The main example will be the moduli spaces of flat principal $G$-bundles for a reductive algebraic group $G$.  

We start by showing in Proposition~\ref{prop:MarkedFusionInternalEndAlgs} that gluing $\Sigma$ along two marked points $v_1,v_2$ on its boundary and adding one new marked point corresponds to applying $\widehat{\otimes}\colon \widehat{C}\boxtimes \widehat{C} \to \widehat{C}$ on the two factors in $\widehat{C}^{\boxtimes k}$ corresponding to $v_1$ and $v_2$ to $\IEAsymbol_{\Sigma,V}$. 
This allows us to compute $\IEAsymbol_{\Sigma,V}$ gluing $\Sigma$ together from disks with two marked points. In Section~\ref{sec: Fusion and Poisson} we apply this to ribbon $\mathsf{aP}_2$-categories and derive an explicit formula for the Poisson structure in Theorem~\ref{Thm: gluing Poisson}. Explicitly, we get that excision for factorization homology recovers fusion for Poisson structures~\cite{AKSM, SeveraCenters}. Using skein theory these results give generalizations of the Goldman Poisson structures defined by resolving crossings as we show in Proposition~\ref{prop:goldman}.    
Finally, in Section~\ref{sec:CS} we apply these result to character stacks.

 \vspace*{0.2cm}\textsc{Acknowledgments.} 
	 	We thank Adrien Brochier, Damien Calaque, Kevin Costello, Nick Gurski, David Jordan, 
        Lyne Moser, Pavel Safronov,  Alexander Schenkel, Pavol Ševera and Lukas Woike 
        for helpful discussions related to this project. We would also like to thank the referee for a careful reading and many useful remarks. We would like to thank Claudia Scheimbauer for hosting multiple meetings of the authors, facilitating work on this project.
        EK was supported by the SFB 1085: Higher invariants from the Deutsche Forschungsgemeinschaft (DFG). 
        CK has received funding from the European Research Council (ERC) under the European Union’s Horizon 2020 research and innovation program (grant agreement No. 768679).
	 	LM gratefully acknowledges support of the Simons Collaboration on Global Categorical Symmetries. Research at Perimeter Institute is supported in part by the Government of Canada through the Department of Innovation, Science and Economic Development and by the Province of Ontario through the Ministry of Colleges and Universities. The Perimeter Institute is in the Haldimand Tract, land promised to the Six Nations. 
        JP was supported by the Postdoc.Mobility grant 203065 of the SNSF. JP would like to thank the Perimeter Institute and the Department of Mathematics of HU Berlin for hospitality during the last stages of writing.

\part{Enriched skein categories}
\section{Background on oriented factorization homology}\label{Sec: FH}
The algebraic input data for factorization homology on oriented $n$-manifolds are framed $\E_n$-algebras with values in a symmetric monoidal $(\infty,1)$-category $\Ca$. Factorization homology then computes functorial invariants of oriented $n$-manifolds by `integrating' the local input data over a given manifold. In this work we will only be concerned with the case of surfaces ($n = 2$) and categorical framed $\E_2$-algebras. We will nevertheless first give the basic definitions for manifolds of any dimension and values in arbitrary symmetric monoidal \((\infty, 1)\)-categories before restricting to the 2-dimensional, categorical case.

\subsection{Basic definitions}
The geometric data for defining oriented factorization homology are oriented manifolds and the spaces of embeddings between them. More precisely, we work with the following $(\infty,1)$-category:

\begin{definition}
$\Man_n^{\ori}$ is the $(\infty,1)$-category whose
\begin{itemize}
    \item objects are oriented $n$-dimensional manifolds
    \item space of morphisms $\mathsf{Emb}^{\ori}(X,Y)$ is the space of oriented smooth embeddings of $X$ into $Y$ equipped with the compact-open topology. 
\end{itemize}
The disjoint union endows $\Man_n^{\ori}$ with the structure of a symmetric monoidal $(\infty,1)$-category. We denote by $\Disk_n^{\ori} \subset \Man_n^{\ori}$ the full symmetric monoidal subcategory whose objects are oriented Euclidean spaces and disjoint unions thereof.
\end{definition}

Throughout this section we fix an ``ambient'' symmetric monoidal $(\infty,1)$-category $(\Ca,\catprod)$, and assume that $\Ca$ is \(\otimes\)-sifted cocomplete~\cite{AFfh}. These assumptions ensure that factorization homology with values in $(\Ca,\catprod)$, introduced in Definition \ref{def:FH} below, is well-defined. First, we need to define the algebraic input data for factorization homology: 

\begin{definition}
A \emph{framed $\E_n$-algebra} in $\Ca$ is a symmetric monoidal functor 
$$
\Aa \colon \Disk_n^{\ori, \sqcup} \longrightarrow \Ca^{\catprod}   \ .
$$
\end{definition}

Since $\Disk_n^{\ori}$ is generated, as a symmetric monoidal category, by $\mathbb{D}^n$, we also denote by $\Aa$ the image of the generator $\mathbb{D}^n$ under the above functor.

\begin{example}
In Figure \ref{fig:diskalgebras} below we give a sketch of the generating morphisms in $\Disk_2^{\ori}$, i.e.\ for $n=2$, and the corresponding algebraic structures on the framed $\E_2$-algebra $\Aa \colon \Disk_2^{\ori,\sqcup} \to \Ca^{\catprod}$.
\end{example}
\begin{figure}[h!]
\centering
\begin{overpic}[scale=0.55,tics=10]{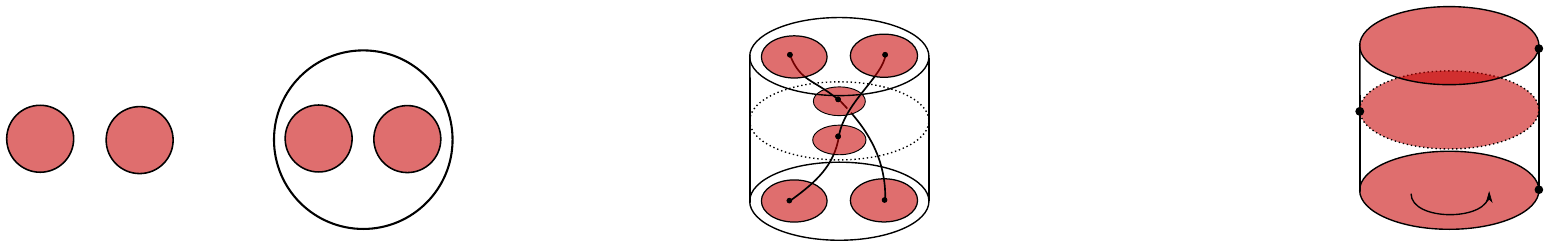}
 	 	\put(13,5.5){$\hookrightarrow$}
  		\put(7,-4){$ m \colon \Aa \catprod \Aa \to \Aa$}
  		\put(46,-4){$ \beta \colon m \Rightarrow m \circ \sigma$}
  		\put(86,-4){$\theta \colon \id_\Aa \Rightarrow \id_\Aa$}
\end{overpic}  
\vspace{20pt}
\captionsetup{format=plain}
\caption{\emph{Left-to-right}: Disk embeddings, or isotopies thereof, in $\Disk_2^{\ori}$ that give rise to the multiplication $m$ and the braiding $\beta$ in $\Aa = \Aa(\mathbb{D}^2)$ ($\sigma$ denotes the symmetry in $\Ca$). Loop in the space of disk embeddings coming from rotating the disk by $2\pi$. This gives rise to a natural isomorphism \(\theta\) satisfying certain properties, and is called a twist on \(\Aa\).}
\label{fig:diskalgebras}
\end{figure}
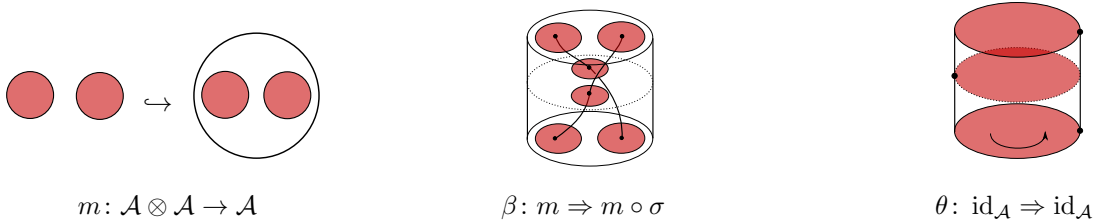

\begin{definition}\cite{AFfh}\label{def:FH}
\emph{Factorization homology $\int_{(-)} \Aa$ of the framed $\E_n$-algebra $\Aa$} is defined as the left Kan extension of \(\Aa\) along the inclusion \(\Disk_n^{\ori, \sqcup} \hookrightarrow \Man_n^{\ori, \sqcup}\), i.e.
\begin{center}
\begin{tikzcd}
\Disk_n^{\ori,\sqcup} \arrow[d,hook] \arrow[r,"\Aa"] & \Ca^{\catprod} \\
\Man_n^{\ori, \sqcup} \arrow[ur,"\int_{(-)}\Aa",dashed,swap]
\end{tikzcd}.
\end{center}
\end{definition}

The left Kan extension admits a pointwise formula: the value of factorization homology on a manifold $M$ is computed by the colimit
$$
\int_M \Aa = \colim \big( {\Disk_n^{\ori, \sqcup}}_{ / M} \xrightarrow{\fgt} \Disk_n^{\ori, \sqcup} \xrightarrow{\Aa} \Ca^{\catprod} \big)
$$
over all possible disk embeddings into $M$. The assumptions on $\Ca$ guarantee that the above colimit exists and makes factorization homology into a symmetric monoidal functor \cite[Proposition 3.7]{AFfh}. The value of factorization homology on any manifold $M$ is naturally pointed by the inclusion $\emptyset \hookrightarrow M$ of the empty manifold:
\begin{equation}\label{fh_pointing}
\int_\emptyset \Aa \cong 1_\Ca \longrightarrow \int_{M} \Aa  \ \ .
\end{equation}

\subsection{Excision}
Oriented factorization homology, being a homology theory for oriented manifolds, satisfies a certain gluing property called $\otimes$-excision. Excision allows to reconstruct the value of factorization homology from a certain decomposition of $M$, as we will explain here. We also give a list of two properties, of which excision is one, that completely characterise factorization homology. 

We first define the type of decomposition of an oriented manifold needed for excision. 
\begin{definition}
A \emph{collar-gluing} is a decomposition $M = M_- \bigcup_{M_0} M_+$, where $M_+$ and $M_-$ are open subsets of $M$, together with a direct product structure on $M_0 = M_- \cap M_+$, i.e.~a diffeomorphism $\theta \colon M_0 \xrightarrow{\cong} N \times (-1,1)$ of oriented manifolds. 
\end{definition}

\begin{example}
    Let \(M=\Sigma_{2,0}\), i.e.\ the genus 2 surface. Figure \ref{fig:collargluing} illustrates a collar-gluing of \(M\) where we have \(M_- \cong \Sigma_{1,1} \cong M_+\) and \(M_0 \cong S^1 \times (-1,1)\). 
\end{example}

\begin{figure}[h]
 \centering
  \begin{overpic}[scale=0.3,tics=10]{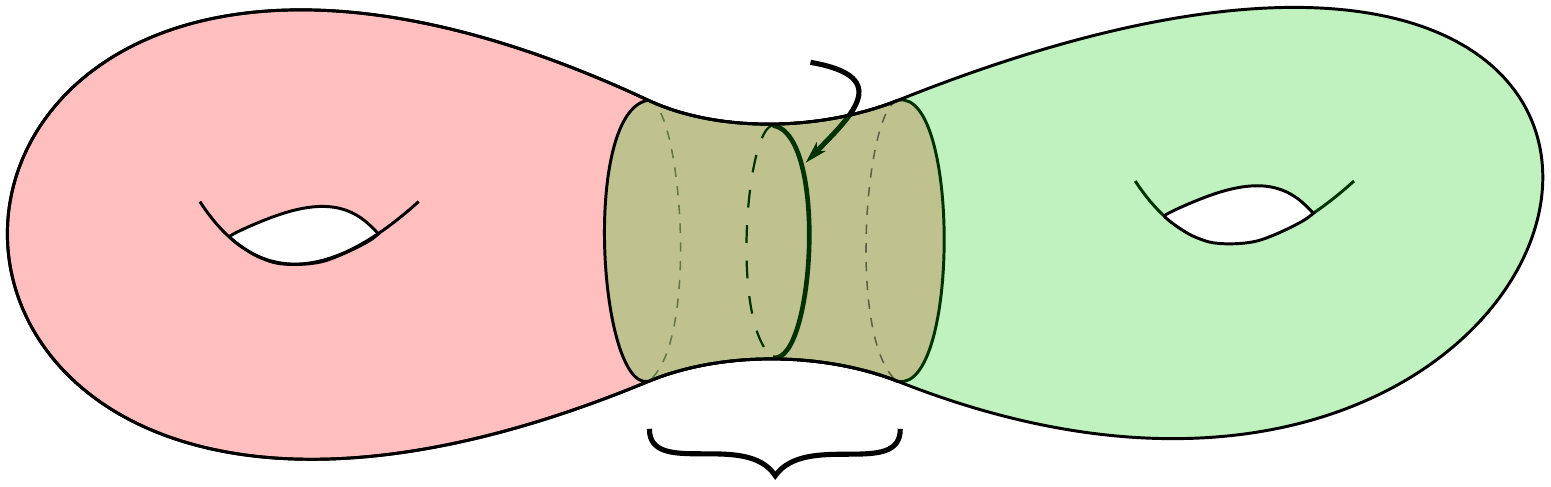}
  	\put(47,27){$N$}
  	\put(8,8){$M_-$}
  	\put(78,8){$M_+$}
  	\put(47,-5){$M_0$}
	\end{overpic}  
	\vspace{0.2cm}
	\caption{Example of a collar-gluing of the genus 2 surface.}
  \label{fig:collargluing}
\end{figure}
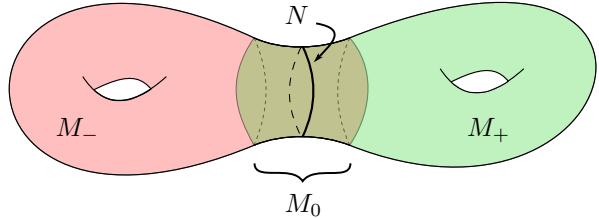 
Factorization homology $\int_{N \times (-1,1)} \Aa$ for the product manifold $N \times (-1,1)$ has a natural $\E_1$-algebra structure coming from embeddings of open intervals. This in turn gives rise to an $\E_1$-algebra structure on $\int_{M_0} \Aa$ under the equivalence $\theta \colon M_0 \cong N \times (-1,1)$. Fix oriented embeddings
$$
\mu_- \colon [-1, 1) \sqcup (-1, 1) \hookrightarrow [-1, 1)  \quad \text{and} \quad \mu_+ \colon (-1, 1) \sqcup (-1, 1] \hookrightarrow (-1, 1] 
$$
such that $\mu_- (-1) = -1$ and $\mu_+(1) = 1$. Under the diffeomorphism $\theta$ these maps lift to embeddings
$$
\act_- \colon M_- \sqcup M_0 \longrightarrow M_-  \quad \text{and} \quad \act_+ \colon M_0 \sqcup M_+ \longrightarrow M_+  \ \ .
$$
See Figure \ref{fig:modulestrExcision} for a sketch. The maps $\act_-$ and $\act_+$ induce right, respectively left $\int_{M_0} \Aa$-module structures 
on \(\int_{M_-} \Aa\), respectively \(\int_{M_+}\Aa\). 

\begin{lemma}\label{lma:excisionFH}\cite[Lemma 3.18]{AFfh}
Let $M = M_- \bigcup_{M_0} M_+$ be a collar-gluing of oriented $n$-manifolds and let $\Aa$ be a framed $\E_n$-algebra in $\Ca$. There is an equivalence of categories
$$
\int_M \Aa \simeq \int_{M_-} \Aa \underset{\int_{M_0} \Aa}{\catprod} \int_{M_+} \Aa \ \ ,
$$
where on the right hand side the relative tensor product is computed by the colimit of the 2-sided bar construction
\begin{equation}\label{eq:barconstruction}
\begin{tikzcd}[column sep = small]
\dots  \arrow[r,yshift=1.5ex] \arrow[r,yshift=0.5ex] \arrow[r,yshift=-0.5ex] \arrow[r,yshift=-1.5ex] & \Ma_- \catprod \Na \catprod \Na \catprod \Ma_+  \arrow[r] \arrow[r,yshift=1ex] \arrow[r,yshift=-1ex]& \Ma_- \catprod \Na \catprod \Ma_+ \arrow[r,yshift=0.5ex] \arrow[r,yshift=-0.5ex] & \Ma_- \catprod \Ma_+
\end{tikzcd}
\end{equation}
for $\Ma_- \coloneqq \int_{M_-}\Aa$, $\Ma_+ \coloneqq \int_{M_+} \Aa$ and $\Na \coloneqq \int_{M_0} \Aa$.
\end{lemma}

\begin{figure}[h]
 \centering
  \begin{overpic}[scale=0.3,tics=10]{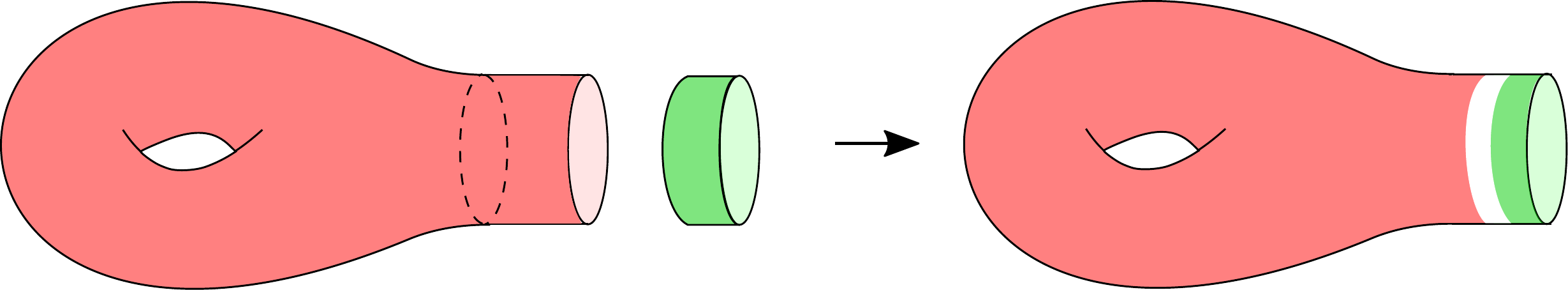}
	\end{overpic}  
	\captionsetup{format=plain}
	\caption{The map which induces the right $\int_{M_0} \Aa$-module structure on $\int_{M_-} \Aa$. Here, the green collar depicts the product manifold $N \times (-1, 1)$.}
  \label{fig:modulestrExcision}
\end{figure}

Conversely, factorization homology is characterised by two properties, of which \(\otimes\)-excision is typically the most involved to prove. More precisely, 
\begin{theorem} \cite{AFfh, AFT-stratified} 
\label{thm:properties-for-FH}
    Let \(\Aa\) be a framed \(\E_n\)-algebra in \(\Ca\). The functor \(\int_{\mhyphen} \Aa\) is characterised by the following properties:
    \begin{enumerate}
        \item Its restriction to the full subcategory $\Disk_n^{\operatorname{or}, \sqcup}$
        is equivalent to $\mathcal{A}$.
        \item \(\int_{\mhyphen}\Aa\) satisfies \(\otimes\)-excision. 
    \end{enumerate}
\end{theorem}

\section{Enriched category theory}\label{Sec: enriched categories}
The main goal of this section is to recall and establish some basic results from $\mathcal{V}$-enriched category theory, where throughout this section, $\mathcal{V}$ is always a closed symmetric monoidal category.  In Section \ref{sect:the-2-cat-VCat} we recall the definition of the 2-category $\mathcal{V}\mhyphen \Cat$ of $\Va$-enriched categories, as well as show that it is \(\otimes\)-sifted cocomplete making it a suitable target for factorization homology. Then we present a brief review of enriched coends in Section \ref{Sec:enriched coends}, since they will play an important role for computations later. Another important target will be the 2-category $\Va\mhyphen\Pres$ of enriched locally presentable categories, which will be discussed in Section \ref{sect:the-2cat-VPres}. 

As we have seen in the previous section, excision allows one to compute the value of factorization homology on more complicated manifolds as a relative tensor product of its value on simpler manifolds. We will give three explicit ways of computing these relative tensor products in Section \ref{sect:rel-tensor-products}. Finally, in Section \ref{sect:example-of-enriching-cat-complete-modules} we discuss our main example of enriching category, namely $\mathcal{V}=\widehat{\Ch\lmod}$, the category of complete modules over the formal power series ring $\Ch$. We also recall two important $\Va=\widehat{\Ch\lmod}$-enriched categories: the Drinfeld category and quantum group representations.

\paragraph{Conventions and notation.}

\begin{definition} \label{defn:underlying-ordinary-category}
Let $\Ca$ be a $\Va$-enriched category. We will denote by $\Ca_0$ the \emph{underlying (ordinary) category} with the same objects as $\Ca$, and with Hom-sets $\Ca_0(c, c') = \Va( 1_\Va, \Ca(c, c'))$. 
\end{definition}

If we have a morphism $f \in \Ca_0(c, c')$, \emph{precomposition} with $f$ is the morphism in $\Va$
\[  \Ca(c', c'') \xrightarrow{- \circ f} \Ca(c, c'') \]
given by the composition
\begin{equation}
\Ca(c', c'') \xrightarrow{\cong}  \Ca(c', c'')\monprod[\Va] 1_\Va \xrightarrow{\id\monprod[\Va] f} \Ca(c', c'')\monprod[\Va] \Ca(c, c')\xrightarrow{\circ^{\Ca}} \Ca(c, c''),
\end{equation}
where in the last step we use the composition of \(\Ca\). \emph{Postcomposition} with \(f\), i.e.\ $f\circ -$, is defined similarly. Moreover, if $\Ca$ is monoidal (see Section \ref{Sec:enriched ribbon cats} for the definition of a monoidal $\Va$-category), we analogously define
\[ \Ca(e, e') \xrightarrow{- \monprod[\Ca] f} \Ca(e\monprod[\Ca] c, e'\monprod[\Ca] c') \]
and
\[  \Ca(c'\monprod[\Ca] d', c''\monprod[\Ca] d'')\xrightarrow{- \circ (f\monprod[\Ca] g) } \Ca(c\monprod[\Ca] d, c''\monprod[\Ca] d''),\]
where $g \in \Ca_0(d, d')$.

\subsection{The 2-category \texorpdfstring{$\mathcal{V}\mhyphen \Cat$}{V-Cat}} \label{sect:the-2-cat-VCat}
In this section we introduce the 2-category of \(\Va\)-enriched categories. We also explain how one can change the enriching category, equip \(\Va\mhyphen \Cat\) with a monoidal product and under which assumptions (on \(\Va\)) we have that $\mathcal{V}\mhyphen \Cat$ is cocomplete. 

\begin{definition}
Let $(\Va,\monprod[\Va])$ be a monoidal category. We define $\VCat$ to be the 2-category whose objects are $\Va$-categories and 1-morphisms are $\Va$-functors $F \colon \Ca \to \Da$, i.e. a map $\Obj(\Ca) \to \Obj(\Da)$ together with morphisms 
$$
F_{c,c'} \colon \Ca(c,c') \to \Da(F(c),F(c'))
$$
in $\Va$. The 2-morphisms are $\Va$-natural transformations $\alpha \colon F \Rightarrow G$ between $F,G \colon \Ca \to \Da$, with components $\alpha_c \colon 1_\Va \to \Da(F(c),G(c))$, for all $c \in \Ca$, making the following diagram commute 
    \begin{center}
        \begin{tikzcd}
        \Ca(c,c') \arrow[r,"F_{c,c'}"] \arrow[d,"G_{c,c'}"] & \Da(F(c), F(c')) \arrow[d,"\alpha_{c'} \circ \mhyphen"] \\
        \Da(G(c),G(c')) \arrow[r,"\mhyphen \circ \alpha_c"] & \Da(F(c),G(c'))
        \end{tikzcd}
    \end{center}
The set of $\Va$-natural transformations $\alpha \colon F \Rightarrow G$ will be denoted by $\Va\mhyphen\text{Nat}(F,G)$. For $\Ca,\Da \in \VCat$ we denote the category of $\Va$-functors and $\Va$-natural transformations between them by $\Va\mhyphen\text{Fun}(\Ca,\Da)$. 
\end{definition}

A $\Va$-functor $F \colon \Ca \to \Da$ is an equivalence if and only if it is essentially surjective and fully-faithful. Note that in the enriched context, a functor is called fully-faithful if each $F_{c,c'}$ is an isomorphism in $\Va$ \cite[Section 1]{KellyVCat}.

\paragraph{Change of enrichment.}

Let $(\Va,\monprod[\Va],1_\Va)$ and $(\Ua,\monprod[\Ua],1_\Ua)$ be two monoidal categories and $\Phi \colon \Va \to \Ua$ a monoidal functor with structure maps 
$$
\mu_\Phi \colon \Phi(-) \monprod[\Ua] \Phi(-) \xrightarrow{\cong} \Phi(-\monprod[\Va]-), \quad \varepsilon_\Phi \colon 1_\Ua \xrightarrow{\cong} \Phi(1_\Va) \ .
$$
Then, there is an induced 2-functor  
$$
\Phi_\ast \colon \VCat \longrightarrow \Ua\mhyphen\Cat
$$
defined as follows: For a $\Va$-category $\Aa$, the $\Ua$-category $\Phi_* \Aa$ has the same objects as $\Aa$ and the morphisms are defined by
$$
\Phi_*\Aa(a,a') \coloneqq \Phi(\Aa(a,a')) \in \Ua 
$$
for all $a,a' \in \Aa$. For a $\Va$-functor $F \colon \Aa \to \Ba$ we define a $\Ua$-functor $\Phi_*F \colon \Phi_*\Aa \to \Phi_* \Ba$ such that on objects it agrees with $F$ and on morphisms it is
$$
(\Phi_*F)_{a,a'} \coloneqq \Phi(F_{a,a'}) \colon \Phi(\Aa(a,a')) \to \Phi(\Ba(F(a),F(a'))) \  .
$$
Finally, on a $\Va$-natural transformation $\alpha \colon F \Rightarrow G$, for $F,G \colon \Aa \to \Ba$ two $\Va$-functors, the components of $\Phi_*\alpha \colon \Phi_*F \Rightarrow \Phi_*G$ are 
$$
1_\Ua \xrightarrow{\varepsilon_\Phi} \Phi(1_\Va) \xrightarrow{\Phi(\alpha_a)} \Phi(\Ba(F(a),G(a))) \ .
$$

\paragraph{Tensor product of $\Va$-categories.}\label{sec:tensorprod_VCat}
{In this subsection we further assume that $\Va$ is symmetric monoidal, closed and complete. The symmetry of $\Va$ allows us to define a natural tensor product on $\Va\mhyphen\Cat$, and the additional assumptions ensure the existence of the internal Hom \eqref{eq:VCat is closed}.}

\begin{definition}\label{definition:tensor-product-of-Vcats}
For two $\Va$-categories $\Ca, \Da$ define $\Ca \catprod \Da$   to be the $\Va$-category whose objects are pairs $(c,d) \in \Obj(\Ca) \times \Obj(\Da)$ and its morphisms are defined by 
$$
\Ca \catprod \Da\, ((c,d),(c',d')) \coloneqq \Ca(c,c') \monprod[\Va] \Da(d,d') \in \Va \ .
$$
The composition in $\Ca \catprod \Da$ is defined using the symmetry of $\Va$.
\end{definition}
The 2-category $\VCat$ is closed under the above tensor product \cite[Section 2.3]{KellyVCat}, i.e.~there is an equivalence of categories
\begin{equation}\label{eq:VCat is closed}
\Hom_{\VCat}(\Ca \catprod \Da, \Ea) \cong \Hom_{\VCat}(\Ca, [\Da,\Ea]) \ ,
\end{equation}
which is 2-natural in the $\Va$-categories $\Ca,\Da,\Ea$. Here, $[\Da,\Ea]$ is the $\Va$-category whose objects are $\Va$-functors $F \colon \Da \to \Ea$ and morphisms are defined by the objects 
$$
\Va\mhyphen\textbf{Nat}(F,G) \coloneqq \int_{d \in \Da} \Ea(F(d),G(d)) ~\in \Va \ ,
$$
where the right hand side is a $\Va$-enriched end defined as the equalizer of 
\begin{align}
    \prod_{d\in \Da } \Ea(F(d),G(d)) \rightrightarrows  \prod_{d_1,d_2\in \Da} [\Ea(d_1,d_2),\Ea(F(d_1),G(d_2))]_{\Va}
\end{align}
where $[-,-]_{\Va}$ denotes the internal Hom of $\Va$ (see also the next section). \par 
From Equation \eqref{eq:VCat is closed} we conclude that the internal Hom is right adjoint to the natural tensor product of \(\VCat\), so the tensor product preserves colimits in \(\VCat\). 

\paragraph{Cocompleteness of $\VCat$.}

In the last paragraph we have seen that the tensor product of $\Va$-categories is a left adjoint. In order to compute factorization homology in $\VCat$, it thus remains to show that it is closed under bicolimits. To that end, we refer the interested reader to Appendix \ref{Sec:cocompleteness VCat}, where we discuss the existence of bicolimits of categories in the $\Va$-enriched setting. From Proposition \ref{Prp:bicolimits VCat} and the fact that the tensor product of \(\VCat\) is left adjoint we get: 

\begin{proposition} \label{prop:VCat-cocomplete-and-tensor-prod-commutes-with-colimits}
    If $\Va$ is a cocomplete and closed symmetric monoidal category, $\VCat$ admits all bicolimits and the tensor product preserves all bicolimits. 
\end{proposition}

\subsection{Enriched coends}\label{Sec:enriched coends}
In this section we recall the definition and some basic facts about enriched coends, which will play an important role for one of the enriched relative tensor products we will use. We mostly follow~\cite[Section II.7]{Riehl}, but~\cite{KellyVCat,CoendCalc} are also excellent references. 

We start by introducing weighted enriched colimits. We will focus on colimits, but clearly limits can be defined dually. 
\begin{definition}
Let $\mathcal{F}\colon \mathcal{D} \to \mathcal{M} $ and $W\colon \mathcal{D}^{\operatorname{op}}\to \mathcal{V}$ be two $\mathcal{V}$-functors. The \emph{colimit of $\mathcal{F}$ weighted by $W$} is an object $\colim^W\mathcal{F}$ of $\mathcal{M}$, together with $\mathcal{V}$-natural isomorphisms 
\begin{align}
\mathcal{M}(\colim^W\mathcal{F}, m) \cong \mathcal{V}^{\mathcal{D}^{\operatorname{op}}}(W, \mathcal{M}(\mathcal{F}(-),m)) \ . 
\end{align}
If $\mathcal{D}$ is the free $\Va$-category on an ordinary category $\mathcal{D}_0$ we can choose the weight $W$ to be induced by the constant functor at the monoidal unit $1_\Va \colon \mathcal{D}_0\to \Va$. A colimit weighted by a weight of this shape is called a \emph{conical colimit} and denoted by $\colim^* \mathcal{F}$.
\end{definition}
Note that if $\mathcal{M}$ is tensored over $\Va$, ordinary colimits are also conical colimits.

\begin{definition}
Let $\mathcal{F}(-,-)\colon \mathcal{D}^{\operatorname{op}}\catprod \mathcal{D}\to \mathcal{M}$ be a  $\mathcal{V}$-functor. We define the \emph{enriched coend $\int^{\mathcal{D}} \mathcal{F}$ of $\mathcal{F}$} as the weighted colimit $\colim^{\Hom_{\mathcal{D}}} \mathcal{F} $.
\end{definition}
Equivalently, if $\mathcal{M}$ is tensored over $\mathcal{V}$, we can define the enriched coend as the conical colimit 
\begin{align}\label{Eq: coend = coeq}
\operatorname{coeq}\left( \coprod_{d_1,d_2} \mathcal{D}(d_2,d_1)\otimes \mathcal{F}(d_1,d_2) \rightrightarrows  \coprod_{d} \mathcal{F}(d,d)  \right) \ .
\end{align}
Most results which hold for ordinary coends also hold for their enriched version. That is, we for example have the Fubini theorem
\begin{align}
\int^{d_1 \in D_1 }\int^{d_2\in D_2 } \mathcal{F}(d_1,d_2,d_1,d_2) \cong \int^{d_1\times d_2\in D_1 \times D_2 } \mathcal{F}(d_1,d_2,d_1,d_2)
\end{align}
and coYoneda lemma 
\begin{align}
\int^{c'\in C} \mathcal{C}(-,c') \otimes \mathcal{F}(c') \cong \mathcal{F}(-)  
\end{align}
for all $\mathcal{V}$-functors $\mathcal{F}\colon \mathcal{C}\to \mathcal{M}$, where $\mathcal{M}$ is tensored over $\mathcal{V}$.  

From Equation~\eqref{Eq: coend = coeq} we see that the enriched coend can also be characterized by the universal property that giving a map from it to $m\in \mathcal{M}$ is equivalent to giving a family of maps $\alpha_{d}\colon \mathcal{F}(d,d)\to m$ such that
\begin{equation}
    \begin{tikzcd}
      \mathcal{D}(d_2,d_1)\otimes \mathcal{F}(d_1,d_2) \ar[r] \ar[d] & \mathcal{F}(d_1,d_1) \ar[d] \\ 
      \mathcal{F}(d_2,d_2) \ar[r] &m 
    \end{tikzcd}
\end{equation}
commutes for all pairs of objects $d_1,d_2\in \mathcal{D}$.

\subsection{The 2-category \texorpdfstring{$\mathcal{V}\mhyphen\Pres$}{V-Pres}} \label{sect:the-2cat-VPres}
The notion of locally presentable categories was first introduced by Gabriel–Ulmer in \cite{Gabriel-Ulmer}. The extension to the enriched
setting is due to Kelly \cite{KellyLFP} and was further developed in \cite{BQR}. Here we recall some of the main definitions that allow generalizations of fundamental results for locally presentable categories to the enriched context. 

Throughout, let $\Va$ be a complete and cocomplete symmetric monoidal closed category. Moreover, following \cite{BQR}, we fix a regular cardinal $\alpha_0$ and assume that $\Va_0$ is locally $\alpha_0$-presentable and that the subcategory of $\alpha_0$-compact objects is closed under the tensor product and contains the monoidal unit. When in this situation, $\Va$ is called a \emph{locally $\alpha_0$-presentable base}. 

\begin{example}
The category $\Va = \widehat{\Ch\lmod}$ is a locally $\aleph_1$-presentable base (see Proposition \ref{prp:ChMod_presentable_base}).
\end{example}
An object $c$ in a $\Va$-category $\Ca$ is called \emph{$\alpha$-compact} (in the enriched sense) if the functor $\Ca(c,-) \colon \Ca \to \Va$ preserves $\alpha$-filtered $\Va$-colimits\footnote{For the definition of a filtered colimit in the $\Va$-enriched setting, we refer the reader to e.g.\ \cite[Definition 2.3]{BQR}}.

\begin{definition}
Let $\Va$ be a locally $\alpha_0$-presentable base and $\alpha \geq \alpha_0$. A $\Va$-category $\Ca$ is called \emph{locally $\alpha$-presentable} if it is $\Va$-cocomplete and admits a strong $\Va$-generator formed by $\alpha$-compact objects.
\end{definition}

It was shown in \cite{BQR}, based on the work of Kelly \cite{KellyLFP} for locally finitely presentable $\Va$-categories, that there is an equivalence between locally $\alpha$-presentable $\Va$-enriched categories and $\alpha$-continuous $\Va$-valued presheaves, generalizing this important result from ordinary category theory to the enriched setting. \par 
\begin{definition}
Fix a locally $\alpha_0$-presentable base $\Va$ and $\alpha \geq \alpha_0$. We define $\VPres$ to be the 2-category whose objects are locally $\alpha$-presentable $\Va$-enriched categories, 1-morphisms are compact cocontinuous $\Va$-functors and 2-morphisms are enriched natural transformations.
\end{definition}
The 2-category $\VPres$ admits a tensor product, which we now describe. 
 
\paragraph{Tensor product of locally presentable $\Va$-categories.}
To define the tensor product on \(\VPres\) it will be convenient to introduce the 2-category $\VRex$ whose objects are small $\Va$-enriched categories having all $\alpha$-small colimits, 1-morphisms are $\Va$-functors preserving $\alpha$-small colimits, and 2-morphisms are enriched natural transformations. Given a category $\Ca \in \VRex$, one can take its \emph{ind-completion} $\operatorname{ind}(\Ca)$, also known as \emph{free completion under $\alpha$-filtered colimits}, which lands in $\VPres$. Conversely, the subcategory of $\alpha$-compact objects of a locally presentable $\Va$-category is small and has all $\alpha$-small colimits. These operations extend to a 2-categorical equivalence $\VRex \cong \VPres$ \cite[Section 9]{Kelly2limits}.\par  
In \cite{KellyVCat}, Kelly introduced a tensor product $\boxtimes$ of small $\Va$-enriched categories with $\alpha$-small colimits, which is uniquely characterized by the following universal
property: for $\Ca$ and $\Da$ two categories in $\VRex$, their \emph{Kelly tensor product} is an object $\Ca \boxtimes \Da \in \VRex$ such that for any $ \Ea \in \VRex$ there is a natural equivalence
$$
\VRex(\Ca \boxtimes \Da, \Ea) \cong \VRex(\Ca, \Da ; \Ea)
$$
between the category of \(\Va\)-functors $\Ca \boxtimes \Da \to \Ea$ in $\VRex$ and the category of \(\Va\)-functors $\Ca \catprod \Da \to \Ea$ preserving $\alpha$-small colimits in each variable. Moreover, it is shown in \cite{KellyVCat} that we have a further equivalence:
$$
\VRex(\Ca \boxtimes \Da, \Ea) \cong \VRex(\Ca, \VRex(\Da, \Ea)) \ \ .
$$
Hence, $\Ca \boxtimes -$ commutes with colimits. One can transport the Kelly tensor product along the equivalence $\VRex \cong \VPres$. We denote the resulting tensor product for $\VPres$ again by $\boxtimes$. 

\paragraph{Cocompleteness of $\VPres$.}
Together with the previous discussion, the following proposition asserts that $(\VPres, \boxtimes)$ is a valid target for factorization homology. A proof is given in the appendix. 

\begin{proposition}[Proposition \ref{prp:VPres_colimits}]
$\VPres$ admits all bicolimits.
\end{proposition}

\newcommand{\Lan}[2]{\mathrm{Lan}_{#1}{#2}}
\subsection{Free cocompletions} \label{sect:Completions} 
Many of the locally presentable \(\Va\)-categories in this article are obtained by freely cocompleting small \(\Va\)-enriched categories. Hence, in this subsection we give the definition of the free cocompletion of a (small \(\Va\)-enriched) category together with some elementary results. Here \(\Va\) is a complete and cocomplete symmetric monoidal closed category.

\begin{definition}\label{defn:free-cocompletion}
    Let \(\Ca\) be a small \(\Va\)-enriched category. The \emph{free cocompletion} \(\Comp{\Ca}\) is the enriched functor category (see Equation \eqref{eq:VCat is closed})
    \begin{align}
        \Comp{\Ca} \coloneqq [\Ca^{\op}, \Va].
    \end{align}
\end{definition}

Since colimits in presheaf categories are computed pointwise it follows that \(\Comp{\Ca}\) has all \(\Va\)-colimits which justifies the name. Recall the enriched Yoneda embedding \(\yo\colon\Ca \ra \Comp{\Ca}\). The free cocompletion enjoys a universal property:

\begin{theorem}\cite[Theorem 4.51]{KellyVCat}\label{thm:CompletionFunctors}
    Let \(\Ca\) be a small \(\Va\)-enriched category and \(\Da\) a cocomplete \(\Va\)-category. There is an equivalence of \(\Va\)-categories
  	\begin{align}
		\Va\mhyphen\Fun^{c.c.}(\Comp{\Ca}, \Da) &\xrightarrow{\simeq} \Va\mhyphen\Fun(\Ca, \Da),   
		\\ 	F &\mapsto F\circ \yo
    	\end{align}
	where \(\Va\mhyphen\Fun^{c.c.}(\Comp{\Ca}, \Da) \subset \Va\mhyphen\Fun(\Comp{\Ca}, \Da)\) is the full subcategory of cocontinuous functors. The inverse sends \(G\colon \Ca \ra \Da\) to the left Kan extension \(\Lan{\yo}{G}\) of \(G\) along \(\yo\). 
\end{theorem}

Sometimes we want to employ the above result to a \(\Va\)-functor \(F \colon \Ca \ra \Da\) where the target category is not necessarily cocomplete. In this situation we use the composition \(\Va\mhyphen\Fun(\Ca, \Da) \hookrightarrow \Va\mhyphen\Fun(\Ca, \Comp{\Da}) \simeq \Va\mhyphen\Fun^{c.c.}(\Comp{\Ca}, \Comp{\Da})\), where the first arrows comes from post-composing with the Yoneda embedding. Explicitly, this gives that a functor \(F \colon \Ca \ra \Da\) admits a unique extension \( \Comp{F} \colon \Comp{\Ca}\ra \Comp{\Da}\) defined on objects as 
    \begin{align}\label{eq:CompletionFunctorEquation}
        \Comp{F}(\Comp{X})(d) \coloneqq \int^{c\in \Ca} \Da(d, F(c)) \monprod[\Va] \Comp{X}(c),
    \end{align}
where we have simply spelled out what the left Kan extension \(\Lan{\yo}{\yo\circ F}\) is following e.g.\ \cite[Equation 4.25]{KellyVCat}.

Another consequence of Theorem \ref{thm:CompletionFunctors} is that monoidal structures can be transferred to free cocompletions. 
Let \((\Ca, \otimes)\) be a monoidal \(\Va\)-category. It follows that \(\Comp{\Ca}\) is also monoidal by the so-called \emph{Day convolution product}, denoted by \(\Comp{\otimes}\). For two objects \(F, G \in \Comp{\Ca}\) we have
\begin{align}\label{eq:DayConvProd}
    (F\CompProd G)(x) = \int^{y_1, y_2 \in \Ca} \Ca(x, y_1 \otimes y_2) \monprod[\Va] F(y_1) \monprod[\Va] G(y_2).
\end{align}

For example in Section \ref{ssec:IntEndAlg} we define a particular algebra (the internal endomorphism algebra) in a free cocompletion. 
It will be convenient to have multiple ways of encoding the algebra structure in this situation. To that end we have the following lemma. 

\begin{lemma}\label{lemma:LaxMonVsAlgStructure}
    Let \(\Ca\in \VCat\) be a monoidal \(\Va\)-category. An algebra in \(\Comp{\Ca}\) is equivalent to a lax monoidal functor \(\Fa \colon \Ca^{\op} \ra \Va\). 
\end{lemma}
\begin{proof}
    Let \(F \in \Comp{\Ca}\) be an algebra. The corresponding lax monoidal structure is defined by
    \begin{align} \label{eq:AlgToLaxFunctor}
         F(y_1) \monprod[\Va] F(y_2) \xrightarrow{1_\Va\otimes \id} \Ca(y_1\otimes y_2, y_1\otimes y_2) \monprod[\Va] F(y_1) \monprod[\Va] F(y_2) \xrightarrow{\mu_{y_1\otimes y_2}^{y_1, y_2}} F(y_1\otimes y_2), 
    \end{align}
    where \(\mu_{c}^{y_1, y_2}\) is the multiplication of \(F\) given component-wise. 
    Conversely, let \(G\in \Comp{\Ca}\) be a lax monoidal functor. The corresponding algebra structure is defined by 
    \begin{align}\label{eq:LaxToAlg}
        G\CompProd G (c) = \int^{y_1, y_2\in \Ca} \Ca(c, y_1\otimes y_2) \monprod[\Va] G(y_1) \monprod[\Va] G(y_2) &
        \\ \xrightarrow{\psi_{y_1, y_2}} &\int^{y_1, y_2\in \Ca} \Ca(c, y_1\otimes y_2) \monprod[\Va] G(y_1 \otimes y_2) \simeq G(c),
    \end{align}
    where \(\psi\) is the lax monoidal structure, and in the last step we use the enriched co-Yoneda lemma. 
\end{proof}

As alluded to in the beginning of this subsection, it turns out that free cocompletions not only contain all (small) \(\Va\)-colimits, but they are even locally presentable.
\begin{proposition} \cite[Proposition A.2.5]{CorinaThesis} 
	The free cocompletion \(\Comp{\Ca}\) is locally presentable as a \(\Va\)-enriched category.
\end{proposition}

Moreover, taking the free cocompletion is a symmetric monoidal functor 
\begin{equation*}
    \widehat{(-)} \colon \VCat \longrightarrow \VPres
\end{equation*}
hence it sends (braided) monoidal categories to (braided) monoidal categories. 

\subsection{Relative tensor products} \label{sect:rel-tensor-products}
In this section we discuss the relative tensor product of $\Va$-enriched module categories $\Ma$ and $\Na$ over a monoidal $\Va$-category $\Aa$ (the notion of an enriched module category is given in Definition \ref{def:modulecat} below). Throughout this paper, we will use three definitions for the relative tensor product of enriched categories:
\begin{itemize}
\item Tambara relative tensor product: $\Ma \Tambaratens{\Aa} \Na$. \\
This is an enriched version of the relative tensor product defined in \cite{Tambara}, which is a categorical analogue of the definition of the relative tensor product of modules. It was already used in \cite{Cooke19} to prove excision for skein categories in the $\C$-linear setting.
\item Colimit of the truncated bar construction: $\Ma \relcatprod{\Aa} \Na$. \\
This is the definition used to formulate excision for factorization homology \cite{AFfh} (see also Lemma \ref{lma:excisionFH} of the background part on factorization homology).
\item Coend relative tensor product: $\Ma \coendtens{\Aa} \Na$. \\
This is a new model for the relative tensor product. In this paper we will mainly use it to compute the internal endomorphism algebras associated to the pointing of enriched skein categories. 
\end{itemize}
In Theorem \ref{thm:eqTambaraBar} and Proposition \ref{prp:CoendTensRelTens} of this section we show that these definitions are indeed equivalent. 

\begin{definition}\label{def:modulecat}
Let $\Aa$ be a monoidal $\Va$-category. A \emph{right $\Va$-enriched $\Aa$-module category} is a $\Va$-enriched category $\Ma$ equipped with a $\Va$-functor 
$$
\triangleleft \colon \Ma \catprod \Aa \to \Ma, \quad (m,a) \mapsto m \triangleleft a
$$ 
called the \emph{action}, endowed with natural \emph{associativity} and \emph{unit} \(\Va\)-isomorphisms
\begin{center}
\begin{tikzcd}[sep = large]
\Ma \catprod \Aa \catprod \Aa \arrow[r, bend left, "{\triangleleft \circ (\id_\Ma \catprod \monprod[\Aa])}"{name=A}]
   \arrow[r, swap, bend right, "\triangleleft \circ (\triangleleft \catprod \id_\Aa)"{name=B}]
 & \Ma \arrow[shorten <=5pt,shorten >=5pt, Rightarrow, from=A, to=B, "\beta"] & \Ma \arrow[r, bend left, "\triangleleft 1_\Aa "{name=C}] \arrow[r, swap, bend right, "\id_\Ma"{name=D}] & \Ma \arrow[shorten <=5pt,shorten >=5pt, Rightarrow, from=C, to=D, "\eta"]
\end{tikzcd}
\end{center}
satisfying the evident pentagon and triangle identities. The notion of a $\Va$-enriched left $\Aa$-module category is defined analogously. 
\end{definition}

\subsubsection{Enriched Tambara relative tensor product}\label{TambaraTensorSection}
We construct the enriched Tambara relative tensor product in terms of generators and relations. We refer the reader to Appendix \ref{appendix:VCatVGrph} and \ref{appendix:VCatGenRel} for details on how to present a $\Va$-category in terms of generators and relations. In a first step, we define the underlying $\Va$-graph $\Omega$ of the relative tensor product together with the free $\Va$-category $\Free(\Omega)$ generated by $\Omega$. In a second step we impose relations between morphisms in $\Free(\Omega)$ such that the resulting object, called the enriched Tambara relative tensor product $\Ma \Tambaratens{\Aa} \Na \coloneqq \Free(\Omega)/\sim$, has the following universal property. For any $\Va$-category $\Ca$ there is an equivalence (of ordinary categories)
\begin{equation} \label{eq:Univ-prop-tambara-rel-prod}
\Hom_{\VCat}(\Ma \Tambaratens{\Aa} \Na, \Ca) \simeq \BalFunA(\Ma \catprod \Na,\Ca)
\end{equation}
where $\BalFunA(\Ma \catprod \Na, \Ca)$ denotes the category of $\Aa$-balanced $\Va$-functors $\Ma \catprod \Na \to \Ca$ (see Definition \ref{df:BalancedFunctor} of Appendix \ref{app:reltens}).  \par 

Define the vertices of the $\Va$-graph $\Omega$ to be $\Obj(\Omega) \coloneqq \Obj(\Ma \catprod \Na)$, and the $\Va$-object of edges is 
$$
\Omega((m,n),(m',n')) \coloneqq \Ma \catprod \Na((m,n),(m',n')) \coprod \big( \coprod_{\substack{\{a \in \Aa~|~m = m' \triangleleft a, \\ a \triangleright n = n' \}}} 1_\Va \big) \coprod \big( \coprod_{\substack{\{a \in \Aa~|~m \triangleleft a = m', \\ n = a \triangleright n'\}}} 1_\Va \big) \ .
$$
We introduce the following notation 
\begin{align} \label{eq:iota}
\iota_{m' \triangleleft a,n}^{m', a \triangleright n} \colon 1_\Va \to \coprod_{\substack{\{a \in \Aa~|~m = m' \triangleleft a, \\ a \triangleright n = n' \}}} 1_\Va \to \Omega((m,n),(m',n'))
\end{align}
and 
\begin{align} \label{eq:iotaInverse}
(\iota_{m \triangleleft a,n'}^{m, a \triangleright n'})^{-1} \colon 1_\Va \to \coprod_{\substack{\{a \in \Aa~|~m \triangleleft a = m', \\ n = a \triangleright n'\}}} 1_\Va \to \Omega((m,n),(m',n'))
\end{align}
for the canonical maps into the $\Va$-object of edges for each $a \in \Aa$.

\begin{definition}\label{df:TambaraTensProd}
Let $\Aa$ be a $\Va$-enriched monoidal category, let $\Ma$ be an enriched right $\Aa$-module category, and let $\Na$ be an enriched left $\Aa$-module category. The \emph{$\Va$-enriched Tambara relative tensor product} is defined by
\begin{equation}
\Ma \Tambaratens{\Aa} \Na \coloneqq \Free(\Omega) / \sim \ \ ,
\end{equation}
where we impose the following relations between morphisms in $\Free(\Omega)$:
\begin{itemize}
\item Isomorphism: 
$$
(\iota^{m,a \triangleright n}_{m \triangleleft a, n})^{-1} \circ \iota_{m \triangleleft a, n}^{m,a \triangleright n} \sim \id_{(m \triangleleft a,n)}, \qquad \iota_{m \triangleleft a, n}^{m ,a \triangleright n} \circ (\iota_{m \triangleleft a, n}^{m ,a \triangleright n})^{-1} \sim \id_{(m, a \triangleright n)}
$$
\item Naturality of $\iota$
\item Compatibility of $\iota$ with associators and unitors
\end{itemize}
We have spelled out these relations in Equations \eqref{reln:iso}, \eqref{reln:naturality}, \eqref{reln:associatior} and \eqref{reln:unitor} of Appendix \ref{app:reltens}.
\end{definition} 

\subsubsection{Bar construction} \label{sect:bar-construction}
Here we explain what the 2-sided bar construction amounts to in our setting. That is, since \(\VCat\) is a 2-category we can truncate the general bar construction. 

Let $\Da$ be the 2-category whose objects and 1-morphisms are as below
\begin{center}
\begin{tikzcd}
A \arrow[r, yshift=3ex,"G_0"] \arrow[r, "G_1"] \arrow[r,yshift=-3ex,"G_2"] & B
\arrow[r,yshift=1.5ex,"F_0"] \arrow[r,yshift=-1.5ex,"F_1"] & C \ ,
\end{tikzcd}
\end{center}
and with 2-morphisms
$$
\kappa_1 \colon F_1 \circ G_0 \Rightarrow F_0 \circ G_2, \quad \kappa_2 \colon F_0 \circ G_0 \Rightarrow F_0 \circ G_1, \quad \kappa_3 \colon F_1 \circ G_2 \Rightarrow F_1 \circ G_1 \ \ .
$$
A strict 2-functor from $\Da$ to $\VCat$ is called a \emph{diagram of shape $\Da$ in $\VCat$}. The relative tensor product used in \cite{AFfh} for proving excision for factorization homology (with values in \(\VCat\)) is the colimit of the truncated bar construction, which is the following diagram of shape $\Da$. 

\begin{definition}\label{df:BarTensorProduct}
Let $(\Aa,\otimes)$ be a monoidal $\Va$-category and $\Ma, \Na$ right, respectively left, $\Aa$-module categories. The \emph{truncated bar construction} is the diagram
\begin{center}
\begin{tikzcd}[column sep = huge]
\Ma \catprod \Aa \catprod \Aa  \catprod \Na \arrow[r, yshift=3ex,"\triangleleft \catprod \id_{\Aa \catprod \Na}"] \arrow[r, "\id_\Ma \catprod \otimes \catprod \id_\Na"] \arrow[r,yshift=-3ex,"\id_{\Ma\catprod \Aa} \catprod \triangleright"] & \Ma \catprod \Aa \catprod \Na
\arrow[r,yshift=1.5ex,"\triangleleft \catprod \id_\Na"] \arrow[r,yshift=-1.5ex,"\id_\Ma \catprod \triangleright"] & \Ma \catprod \Na
\end{tikzcd}
\end{center}
in $\VCat$ together with the $\Va$-natural isomorphisms
\begin{align*}
& \id \colon (\id_\Ma \catprod \triangleright) \circ (\triangleleft \catprod \id_{\Aa \catprod \Na})\Rightarrow (\triangleleft \catprod \id_\Na)\circ (\id_{\Ma\catprod \Aa} \catprod \triangleright) \\
& \beta_\Ma \colon  (\triangleleft \catprod \id_\Na) \circ (\triangleleft \catprod \id_{\Aa \catprod \Na}) \Rightarrow (\triangleleft \catprod \id_\Na) \circ (\id_\Ma \catprod \otimes \catprod \id_\Na)   \\
& \beta_\Na \colon (\id_\Ma \catprod \triangleright) \circ (\id_{\Ma\catprod \Aa} \catprod \triangleright) \Rightarrow (\id_\Ma \catprod \triangleright) \circ (\id_\Ma \catprod \otimes \catprod \id_\Na)
\end{align*}
where $\beta_\Ma$ and $\beta_\Na$ are the associators for the right and left $\Aa$-action.
\end{definition}

\begin{definition}
We define the \emph{relative tensor product arising from the (truncated) bar construction} to be given by the bicolimit over the above diagram, i.e.
\begin{center}
\begin{tikzcd}[sep = small]
\Ma \relcatprod{\Aa} \Na \coloneqq \colim \Big(\Ma \catprod \Aa \catprod \Aa \catprod \Na \arrow[r,yshift=1ex] \arrow[r] \arrow[r,yshift=-1ex] & \Ma \catprod \Aa \catprod \Na \arrow[r,yshift=0.5ex] \arrow[r,yshift=-0.5ex] & \Ma \catprod \Na \Big)
\end{tikzcd}
\end{center}
\end{definition}

When we later prove that the enriched skein category $\mathbf{Sk}_\Ca(-)$ computes factorization homology we need that the relative tensor product coming from the bar construction agrees with the enriched Tambara relative tensor product defined in Section \ref{TambaraTensorSection}. This has been proven in \cite[Theorem 2.27]{Cooke19} for linear categories and is here extended to the $\Va$-enriched setting:

\begin{theorem}\label{thm:eqTambaraBar}
The $\Va$-enriched Tambara relative tensor product $\Ma \Tambaratens{\Aa} \Na$ is equivalent to the relative tensor product arising from the truncated bar construction. Explicitly, we have an equivalence of \(\Va\)-categories
\begin{align*}
    \Ma \Tambaratens{\Aa} \Na \simeq \Ma \relcatprod{\Aa} \Na \ \ .
\end{align*}
\end{theorem}
For a proof of Theorem \ref{thm:eqTambaraBar} we refer the reader to Appendix \ref{app:eqTambaraBar}.

\subsubsection{The coend relative tensor product}\label{Sec:coend_tensor_product}
We now introduce a new model for the relative tensor product which is defined using enriched coends. 

\begin{definition}
Let \(\Aa\) be a monoidal \(\Va\)-category and \(\Ma\), \(\Na\) right, respectively left \(\Aa\)-module \(\Va\)-categories.
We define the \emph{coend relative tensor product}, denoted \(\Ma\coendtens{\Aa} \Na\), to be the \(\Va\)-category whose
objects are given by pairs $(m , n)$, for $m\in \Ma$ and $n\in \Na$, and denoted by $m\otimes_{\Aa}n$.
The $\Va$-object of morphisms from $m\otimes_{\Aa}n$ to $m'\otimes_{\Aa}n'$ is defined by the following enriched coend: 
\begin{equation} \label{eq:CoendRelProd}
\mathcal{M}\coendtens{\Aa} \Na \left( m\otimes_{\Aa}n , m'\otimes_{\Aa}n' \right) \coloneqq \int^{a\in \Aa} \Ma(m,m'\triangleleft a ) \monprod[\Va] \Na(a\triangleright n,n') \in \Va \ \ . 
\end{equation}

Composition is defined by 
\begin{align}\label{eq:CompositionCoendRelProd}
\begin{tikzcd}
	{\left( \displaystyle\int^{a'\in \Aa} \Ma(m',m''\triangleleft a' ) \monprod[\Va] \Na(a'\triangleright n',n'') \right) \monprod[\Va] \left( \displaystyle\int^{a\in \Aa}  \Ma(m,m'\triangleleft a ) \monprod[\Va] \Na(a\triangleright n,n') \right) } \\
	{\displaystyle\int^{(a, a') \in \Aa\catprod \Aa }  \Ma(m',m''\triangleleft a' ) \monprod[\Va] \Ma(m,m'\triangleleft a ) \monprod[\Va] \Na(a'\triangleright n',n'') \monprod[\Va] \Na(a\triangleright n,n')} \\
	{\displaystyle\int^{(a, a') \in \Aa\catprod \Aa }  \Ma(m' \triangleleft a,m''\triangleleft a' \triangleleft a ) \monprod[\Va] \Ma(m,m'\triangleleft a ) \monprod[\Va]   \Na(a'\triangleright n',n'') \monprod[\Va] \Na(a'\triangleright a\triangleright n, a'\triangleright n')} \\
	{\displaystyle\int^{(a, a')\in \Aa \catprod \Aa}  \Ma(m,m''\triangleleft a'\triangleleft a ) \monprod[\Va] \Na(a'\triangleright a\triangleright n,n'') } \\
	{\displaystyle\int^{(a, a')\in \Aa \catprod \Aa }  \Ma(m,m''\triangleleft (a'\otimes a) )\monprod[\Va] \Na((a'\otimes a)\triangleright n,n'') } \\
	{\displaystyle\int^{ c \in \Aa  }  \Ma(m,m''\triangleleft c ) \monprod[\Va] \Na(c\triangleright n,n''),}
	\arrow["\cong"', from=1-1, to=2-1]
	\arrow["{(-) \triangleleft\, a\, \monprod[\Va] \id \monprod[\Va] \id  \monprod[\Va] a' \triangleright (-) }"', from=2-1, to=3-1]
	\arrow["{\circ_{\Ma} \monprod[\Va] \circ_{\Na}} "', from=3-1, to=4-1]
	\arrow["\cong"', from=4-1, to=5-1]
	\arrow[from=5-1, to=6-1]
\end{tikzcd}
\end{align}
where the last map sends $a'\otimes a$ to $c$. The unit for this composition is induced by the unit for the composition in $\Ma$ and $\Na$, respectively. 
\end{definition}

Observe that there is a canonical \(\Va\)-functor $F\colon \mathcal{M} \catprod \mathcal{N} \to \mathcal{M}\coendtens{\Aa} \Na$ which sends an object $(m, n)$ to $m\otimes_{\Aa} n$. On Hom-objects the functor \(F\) is given by the component of the $\Va$-dinatural transformation 
\begin{align}
\Ma(m,m')\monprod[\Va]\Na(n,n')\cong \Ma(m,m'\triangleleft 1_{\Aa})\monprod[\Va] \Na( 1_{\Aa} \triangleright  n,n') \longrightarrow \int^{a\in \Aa} \Ma(m,m'\triangleleft a ) \monprod[\Va] \Na(a\triangleright n,n')
\end{align} 
for $a=1_{\Aa}$. 
Furthermore, $F$ is canonically balanced (as in Definition \ref{df:BalancedFunctor}). To see this we need to construct a $\Va$-natural transformation $\beta_{m,a,n}\colon F(m\triangleleft a, n)\to F(m,a\triangleright n)$, which is given by 
\begin{align}
1_{\Va} \xrightarrow{\id\otimes \id} \Ma(m\triangleleft a,m\triangleleft a ) \monprod[\Va] \Na(a\triangleright n,a\triangleright n)\longrightarrow  \int^{b\in \Aa} \Ma(m\triangleleft a,m\triangleleft b ) \monprod[\Va] \Na(b\triangleright n,a\triangleright n)
\end{align}
where the unlabeled arrow is the map into the coend for $b=a$. In general $\beta$ is not a natural isomorphism. However, similarly  to~\cite[Lemma~2.12]{AraujoGuuHudson2025}, it is easy to verify that $\beta$ is a natural isomorphism if $\mathcal{A}$ is an enriched ribbon category, as defined in Section~\ref{Sec:enriched ribbon cats}. If $\beta$ is an isomorphism, the condition making it a balancing follows from a straightforward computation.   

\begin{proposition}\label{prp:CoendTensRelTens}
If the natural transformation $\beta$ is a natural isomorphism, then the balanced $\cat{V}$-functor $F\colon \Ma\catprod \Na \to \mathcal{M}\coendtens{\Aa} \Na$ is universal. In particular, $\mathcal{M}\coendtens{\Aa} \Na$ is a model for the relative tensor product. 
\end{proposition}
\begin{proof}
Recall the universal property of the Tambara relative tensor product from Equation \eqref{eq:Univ-prop-tambara-rel-prod}. If we can show the same universal property for the coend tensor product we are done: i.e.\ we need to show that for any $\Va$-category $\Ca$ the functor $F^*\colon \Va\text{-}\Cat (\mathcal{M}\coendtens{\Aa} \Na,\Ca) \longrightarrow \BalFunA(\Ma \catprod \Na, \Ca) $ is an equivalence (of ordinary categories). We do this by constructing an explicit inverse $G$. Let $(H,\beta)$ be an \(\Aa\)-balanced functor 
$\Ma\catprod \Na \longrightarrow \Ca$. We define $G(H,\beta)\colon \Ma\coendtens{\Aa} \Na \to \Ca $ on objects by $G(H,\beta)[m\otimes_{\Aa} n] \coloneqq H(m,n)$. On Hom-objects it is induced by the maps 
\begin{equation}
\begin{tikzcd}
 \Ma(m, m'\triangleleft a ) \monprod[\Va] \Na(a\triangleright n,n') \ar[d,"{ \Ma(m, m'\triangleleft a ) \monprod[\Va] \id_n \monprod[\Va] \id_{m'} \monprod[\Va] \Na(a\triangleright n,n') }"] \\  \big( \Ma(m, m'\triangleleft a ) \monprod[\Va] \Na(n, n) \big)  \monprod[\Va] \big( \Ma(m',m') \monprod[\Va] \Na(a\triangleright n,n') \big) \ar[d,"{H\otimes^{\Va} H}"] \\ 
 \Ca(H(m, n),H(m'\triangleleft a, n )) \monprod[\Va] \Ca(H(m', a\triangleright n),H(m', n' )) \ar[d,"{{\beta_{m',a,n}}_*}"] \\ 
 \Ca(H(m, n),H(m', a\triangleright n )) \monprod[\Va] \Ca(H(m', a\triangleright n),H(m', n' )) \xrightarrow{\circ_\Ca} \Ca(H(m, n),H(m', n' ))  
\end{tikzcd} \ \ .
\end{equation}
It is tedious but straightforward to check that this defines a $\Va$-functor; we did indeed check this! 
The value of $G$ on \(\Aa\)-balanced natural transformations $\varphi\colon H\to H'$ is the natural transformation with the same components as $\varphi$. The fact that $\varphi$ is balanced implies that this is also a natural transformation $G(H)\to G(H')$. 
This finishes the definition of the functor $G$. A direct computation shows that $F^* \circ G= \id$, hence $F^*$ is essentially surjective. Since, by construction, $F^*$ is also injective on morphisms, it is fully faithful.
\end{proof}

\subsection{Example of enriching category: \texorpdfstring{$\Va = \widehat{\Ch\lmod}$}{V = complete C[[h]] modules}} \label{sect:example-of-enriching-cat-complete-modules}
We now give our main example of enriching category, namely the category of complete \(\Ch\)-modules. 
Let $M$ be a left module over $\Ch$. Consider the submodules $(\hbar^n M)_{n \in \mathbb{N}}$ and denote $M_n \coloneqq M / \hbar^n M$. There are canonical projections 
$$
p_n \colon M_n \to M_{n-1} \ ,
$$
and $(M_n, p_n)_{n \in \mathbb{N}}$ is an inverse system of $\Ch$-modules. Hence, we can consider the inverse limit
$$
\widehat{M} \coloneqq \varprojlim_n M_n = \Big\{ (x_n) \in \prod_n M_n~|~p_n(x_n) = x_{n-1} \Big\} \ .
$$
The left $\Ch$-module $\widehat{M}$ is called the \emph{$\hbar$-adic completion of $M$}. 

\begin{definition}
The $\Ch$-module $M$ is called \emph{complete} if the canonical map $M \to \widehat{M}$ is an isomorphism.
We denote the category of complete \(\Ch\)-modules by $\widehat{\Ch\lmod}$.
\end{definition}

\begin{proposition}\label{prp:cmplModreflectivesubcat}
$\widehat{\Ch\lmod} \hookrightarrow \Ch\lmod$ is a reflective subcategory, where the left adjoint to the inclusion functor is the completion functor $\widehat{(-)}$. In particular, $\widehat{\Ch\lmod}$ is complete and cocomplete. Limits are calculated in $\Ch\lmod$, and colimits are calculated by completing the colimits in $\Ch\lmod$.
\end{proposition}

\begin{proof}
A full proof can be found for example in \cite[Theorem 5.8]{Po20}, here we give a brief outline of the proof. Let $M \in \Ch\lmod$ and $C \in \widehat{\Ch\lmod}$. We have the following sequence of bijections
\begin{align*}
    \Hom_{\Ch\lmod}(M, C) & \cong \{\Hom_{\Ch\lmod}(M/\hbar^nM, C/\hbar^nC) \}_{n \in \mathbb{N}}\\ 
    & \stackrel{(\ast)}{\cong} \{\Hom_{\Ch\lmod}(\widehat{M}/\hbar^n\widehat{M}, C/\hbar^n C ) \}_{n \in \mathbb{N}}\\ 
    & \cong \Hom_{\widehat{\Ch\lmod}}(\widehat{M}, C) \  ,
\end{align*}
since giving a map $M \rightarrow C$ is equivalent to specifying a family of maps $\{M \rightarrow C/\hbar^nC\}_{n \in \mathbb{N}}$ because $C \cong \widehat{C}$. For $(\ast)$ one uses \cite[Lemma 10.96.3, \href{https://stacks.math.columbia.edu/tag/00M9}{Tag 00M9}]{stacks-project}, stating that $\hbar^n \widehat{M} = \textrm{Ker}(\widehat{M}\to M/\hbar^nM)$.
\end{proof}

Write $\otimes = \otimes_{\Ch}$. Let $M$ and $N$ be two $\Ch$-modules. We define their tensor product as the $\hbar$-adic completion of $M \otimes N$:
\begin{align*}
M\widehat{\otimes} N & \coloneqq \widehat{M\otimes N } \\
& = \varprojlim_{n >0} (M\otimes N)/ \hbar^n (M\otimes N)  \ .
\end{align*}
The standard associativity and commutativity constraints (of \(\Ch\lmod\)) induce \(\Ch\)-linear isomorphisms
\begin{equation}
    (M\widehat{\otimes} N) \widehat{\otimes} P \cong M\widehat{\otimes} (N \widehat{\otimes} P) \ \ \text{and}  \ \ M \widehat{\otimes} N \cong N \widehat{\otimes} M,
\end{equation}
which gives \(\widehat{\Ch\lmod}\) the structure of a symmetric monoidal category. The monoidal unit is given by \(\Ch\) as can be seen from the \(\Ch\)-linear isomorphisms \(\Ch \widehat{\otimes}M \cong \widehat{M} \cong M\widehat{\otimes} \Ch \). 

\begin{proposition}
    The symmetric monoidal category $(\widehat{\Ch\lmod}, \widehat{\otimes})$ is closed. 
\end{proposition}

\begin{proof}
We first show that if $C$ is a $\hbar$-adically complete module then $\Hom_{\Ch\lmod}(M,C)$ is $\hbar$-adically complete. To that end, first assume that $M$ is a finitely presented $\Ch$-module. Choose a representation $\Ch^m \to \Ch^n \to M \to 0$, applying $\Hom_{\Ch\lmod}(-,C)$ we get an exact sequence
$$
0 \to \Hom_{\Ch\lmod}(M,C) \to C^n \to C^m \  .
$$
Since $\widehat{\Ch\lmod}$ is closed under products and kernels we get that $\Hom_{\Ch\lmod}(M,C)$ is complete. Moreover, every module is a colimit over its finitely presented submodules, and so we have that 
$$
\Hom_{\Ch\lmod}(N,C) = \Hom_{\Ch\lmod}(\colim_i N_i,C) = \lim_i \Hom_{\Ch\lmod}(N_i,C)
$$
is complete. \par 
By Proposition \ref{prp:cmplModreflectivesubcat} we then have for $M,N,C \in \widehat{\Ch\lmod}$
\begin{align*}
    \Hom_{\widehat{\Ch\lmod}}(M \widehat{\otimes} N, C) & \cong \Hom_{\Ch\lmod}(M \otimes N, C) \\
    & \cong \Hom_{\Ch\lmod}(M, \Hom_{\Ch\lmod}(N,C))\\
    & \cong \Hom_{\widehat{\Ch\lmod}}(M,\Hom_{\Ch\lmod}(N,C)) \  . \qedhere
\end{align*}
\end{proof}

\begin{proposition}\label{prp:ChMod_presentable_base}
$\widehat{\Ch\lmod}$ is locally $\alpha$-presentable (as an ordinary category) for $\alpha > \aleph_0$.
\end{proposition}

\begin{proof}
For lighter notation, we write $K \coloneqq \Ch$. The functor $\Hom_{K\lmod}(K,-)$ is conservative since an isomorphism in $K\lmod$ is an isomorphism of the underlying vector space which is compatible with the $K$-action. By Proposition \ref{prp:cmplModreflectivesubcat} we have that $\Hom_{\widehat{K\lmod}}(K, -) \cong \Hom_{K\lmod}(K, -) \circ \iota$. Since $\iota$ is fully faithful the composite is conservative and $K$ is a strong generator in $\widehat{K\lmod}$.\par
Next, we show that $K$ is $(\alpha > \aleph_0)$-compact. For any $\alpha$-filtered diagram $F \colon \Da \to \widehat{K\lmod}$ of complete $K$-modules the following holds in $K\lmod$:
\begin{align*}
\widehat{\colim F} & \cong \lim_{n \in \mathbb{N}} \colim F / (\hbar^n \colim F )\\ 
& \cong \colim \lim_{n \in \mathbb{N}} F / \hbar^n F \\
& \cong \colim F \ .
\end{align*}
We used that the completion functor is idempotent and that $\alpha$-filtered colimits commute with $\alpha$-small limits. Hence for $\alpha > \aleph_0$ we have
\begin{align*}
\Hom_{\widehat{K\lmod}}(K,\widehat{\colim F}) & \cong \Hom_{K\lmod}(K,\colim F) \\
& \cong \colim F \\
& \cong \colim \Hom_{\widehat{K\lmod}}(K,F)
\end{align*}
showing that $K$ is $\alpha$-compact.
\end{proof}

\subsubsection{The Drinfeld category and quantum group representations}\label{Drinfeld category}
Our main examples of categories enriched in $\Va = \widehat{\Ch\lmod}$ comes from the representation theory of topological quasi-Hopf algebras. More precisely, we consider on one hand the category of representations of Drinfeld-Jimbo type quantum groups $U_\hbar(\mathfrak{g})$. On the other hand we have Drinfeld's category of representations of the quantized universal enveloping algebra $U(\mathfrak{g})[[\hbar]]$ with non-trivial associator. In the following we describe these two categories in more detail and recall a result by Drinfeld giving a balanced braided equivalence between the two.\par

Let $\mathfrak{g}$ be a complex semi-simple Lie algebra of rank $n$. Its Casimir element $C \in U\mathfrak{g}$ is defined as $C= \sum_i x_i x^i$, where $\{x_i\}$ and $\{x^i\}$ are two bases of $\mathfrak{g}$ that are dual with respect to the Killing form, i.e. $ \operatorname{tr}_{\mathfrak{g}} (\ad_{x_i} \ad_{x^j}) = \delta_i^j$. Let  $t \in (\mathsf{Sym}^2\mathfrak{g})^\mathfrak{g}$ be the $\ad$-invariant symmetric tensor defined by $t=(\Delta C - C \otimes 1 - 1 \otimes C)/2 = \sum_i (x_i\otimes x^i + x^i \otimes x_i)/2$. The \emph{Drinfeld-Jimbo quantum group} $U_\hbar(\mathfrak{g})$ is the $\Ch$-algebra topologically generated by the $3n$ symbols $(H_i, X_i, Y_i)_{1 \leq i \leq n}$ and a set of relations, which reduce for $\hbar = 0$ to the Chevalley-Serre relations in $U(\mathfrak{g})$. For a detailed definition of $U_\hbar(\mathfrak{g})$ see for example \cite[Definition XVII.2.3.]{Kassel}. The topological braided Hopf algebra 
\begin{equation} \label{eq:Uhbar-of-g}
U_\hbar(\mathfrak{g}) \coloneqq (U_\hbar(\mathfrak{g}), \Delta_\hbar, \varepsilon_\hbar, \Phi_\hbar = 1 \widehat{\otimes} 1 \widehat{\otimes} 1, \Ra_\hbar, S_\hbar)
\end{equation}
is a quantum enveloping algebra (QUEA) \cite{DrinQG}. In particular, $U_\hbar(\mathfrak{g})$ is topologically free as a $\Ch$-module. The universal R-matrix is of the form $\Ra_\hbar = 1 \otimes 1 + \hbar r + \Oa(\hbar^2)$, where $r$ is a classical r-matrix with symmetric part $t$. 
\begin{notation}
    We write $U_\hbar(\mathfrak{g})\lmod$ for the \(\Va\)-category of topologically free (i.e. torsion free and complete), finite rank modules over the quantum enveloping algebra $U_\hbar(\mathfrak{g})$ from \eqref{eq:Uhbar-of-g}. 
\end{notation}
Moreover, $U_\hbar(\mathfrak{g})\lmod$ is a ribbon \(\Va\)-category (a notion explained in Definition \ref{defn:ribbon-V-cat}). 
The braiding comes from the action of the universal R-matrix $\Ra_\hbar = \Ra_\hbar^1 \otimes \Ra_\hbar^2$ and the ribbon twist is given by the action of $\theta_\hbar = e^{-\hbar \rho}u$, where $\rho$ is the half-sum of positive roots and $u = S(\Ra_\hbar^2)\Ra_\hbar^1$.  \par 

In \cite{DrinKZ}, Drinfeld showed that there exist a set of quantum enveloping algebras $A_{\mathfrak{g},t}$, with coproduct and counit $\Ch$-linearly extended from $U(\mathfrak{g})$, by proving existence of non-trivial associators $\Phi$, called \emph{Drinfeld associators}. Namely,
$$
A_{\mathfrak{g},t} \coloneqq (U(\mathfrak{g})[[\hbar]],\Delta,\varepsilon, \Phi, \Ra = e^{\frac{\hbar}{2}t}) \ \ ,
$$
is a quasi-triangular quasi-Hopf algebra, with a Drinfeld associator $\Phi = 1 + \frac{\hbar^2}{24} [t_{12},t_{23}] + \Oa(\hbar^3) \in U(\mathfrak{g})^{\otimes 3}[[\hbar]]$. See \cite[Theorem XIX.4.2]{Kassel} and references therein for more details on the QUEA $A_{\mathfrak{g},t}$. 
\begin{definition}
    The \emph{Drinfeld category}, denoted $U(\mathfrak{g})\lmod^\Phi[[\hbar]]$, is defined as the braided tensor \(\Va\)-category of topologically free $U(\mathfrak{g})[[\hbar]]$-modules of finite rank with associator induced by $\Phi$ and non-trivial braiding coming from the action of $e^{\frac{\hbar}{2}t}$. 
\end{definition}
The Drinfeld category is even a ribbon \(\Va\)-category. The evaluation and coevaluation are twisted using invariant elements $\alpha, \beta \in U(\mathfrak{g})[[\hbar]]$ as in \cite[Eq.~1.18]{DrinfeldQuasiHopf}, \cite[Eq.~(16 bis)]{Cartier1993}; one can choose e.g. $\alpha = \beta = \nu^{1/2}$ \cite{LeMurakamiParallel}. The ribbon twist comes from the action by the element $\theta = e^{\frac{\hbar C}{2}}$, where $C \in U(\mathfrak{g})$ is the quadratic Casimir element; See \cite{DrinfeldGal1990, Cartier1993} and \cite{KasselTuraev1998}. \par  
Recall that a gauge transformation of a braided quasi-bialgebra $(H, \Delta, \varepsilon,\Phi,\Ra)$ is an invertible element $F \in H \otimes H$, satisfying a normalization condition, that allows to twist $H$ to a new braided quasi-bialgebra $H_F = (H, \Delta_F = F \Delta F^{-1},\varepsilon,\Phi_F, \Ra_F = F_{21}^{-1}\Ra F)$, a formula for the twisted associator $\Phi_F$ can for example be found in \cite[Section XV.3]{Kassel}. Two braided quasi-bialgebras $H$ and $H'$ are then called gauge equivalent if there exist a gauge transformation $F$ and an isomorphism $\alpha: H \xrightarrow{\cong} H'_F$. Gauge equivalent braided quasi-bialgebras have equivalent braided monoidal categories of modules, explicitly the braided tensor equivalence is given by \cite[Theorem XV.3.9]{Kassel}:
$$
H'\lmod \xrightarrow{\alpha^*} H\lmod, \quad \alpha^* V \otimes \alpha^* W \xrightarrow{F^{-1} \triangleright} \alpha^*(V \otimes W)
$$
for all $V, W \in H'\lmod$, where by abuse of notation we wrote $\otimes$ for the tensor structure in both $H\lmod$ and $H'\lmod$. \par 
Twist equivalence of the QUEAs $U_\hbar(\mathfrak{g})$ and $A_{\mathfrak{g},t}$ is due to Drinfeld (\cite[{Prop.~3.16}]{DrinfeldQuasiHopf}, \cite{DrinfeldGal1990}), who showed that there exists a gauge transformation $F \in A_{\mathfrak{g},t} \widehat{\otimes} A_{\mathfrak{g},t}$ and an isomorphism $\alpha \colon U_\hbar (\mathfrak{g}) \xrightarrow{\cong} (A_{\mathfrak{g},t})_F$ of topological braided quasi-bialgebras. On the ribbon-categorical level we then have the following \cite[Section XIX.5]{Kassel},  \cite[Theorem~A.1]{KasselTuraev1998}.

\begin{theorem}\label{Thm: equivalence}
There is an equivalence of ribbon \(\Va\)-categories
$$
U_\hbar(\mathfrak{g})\lmod \xrightarrow{\alpha^*} U(\mathfrak{g})\lmod^\Phi[[\hbar]] \ \ ,
$$
between the representation category of the quantum group $U_\hbar(\mathfrak{g})$ and the Drinfeld category.
\end{theorem}

\section{Enriched skein theory}\label{Sec: Enriched skein}
In this section we will first introduce an enriched version of skein categories. 
For this we define enriched ribbon categories in Section \ref{Sec:enriched ribbon cats}.
Afterwards, we define in Section \ref{sec:enrichedskeincat} enriched skein categories. 
The main differences, compared to the usual definition, e.g.\ as given in \cite{Cooke19}, is that we work with a ribbon category \(\Ca\) 
which is \(\Va\)-enriched and not necessarily strict. 
Then in Section \ref{sec:excisionenrichedsetting} we prove that enriched skein categories satisfy excision. 
As a formal consequence we get the main theorem of this section (Theorem \ref{Thm: FH=Sk}) which verifies that enriched skein categories compute factorization homology.

\subsection{Enriched ribbon categories}\label{Sec:enriched ribbon cats}
In this section we set up some basic definitions for rigid and ribbon categories in the enriched setting. 

Recall that a \emph{monoidal $\Va$-category} $\Ca$ is a pseudomonoid in the symmetric monoidal 2-category $\Va\mhyphen\Cat$, see \cite[Section~4]{DayStreet1997}. In other words, it consists of a $\Va$-category $\Ca$ together with a \emph{unit} object $1_\Ca \in\Ca$, a $\Va$-functor $\monprod[\Ca] \colon \Ca \,\catprod\, \Ca \to \Ca$, \emph{unitors} $\rho, \lambda$ and an \emph{associator} $\alpha$ which are all $\Va$-enriched natural isomorphisms satisfying the usual axioms. Similarly, a \emph{braiding} or \emph{symmetry} is a suitable $\Va$-natural transformation $\beta \colon \monprod[\Ca] \to (\monprod[\Ca])^\mathrm{op}$. If $\Ca$ is (braided/symmetric) monoidal, so is $\Ca_0$ \cite[Theorem 5.7.1]{CruttwellThesis}.

\begin{definition}\label{def:enrribboncat}
    A monoidal $\Va$-category $\Ca$ is called \emph{left-rigid} if for each object $c \in \Ca$, there exists an object $c^* \in \Ca$, called its \emph{left dual}, and morphisms 
    \begin{align*}
        \operatorname{ev}_c \in \Ca_0( c^* \monprod[\Ca] c, 1_\Ca ) \qquad \mathrm{and}  \qquad \operatorname{coev}_c \in \Ca_0( 1_\Ca, c \monprod[\Ca] c^* )
    \end{align*}
   satisfying the snake identities: i.e.\ the following two diagrams (in $\Va$) commute
\begin{equation} \label{eq:snake-relations-enriched-rigid-cat}
	\begin{tikzcd}[column sep=7, row sep=20.0]
		{\Ca(c^*, c^*)} && {\Ca(c^* \otimes^\Ca 1_\Ca, 1_\Ca\otimes^\Ca c^*)} && {\Ca(c,c)} && {\Ca( 1_\Ca \otimes^\Ca c, c \otimes^\Ca 1_\Ca)} \\
		&& {\Ca(c^* \otimes^\Ca 1_\Ca, (c^* \otimes^\Ca c)\otimes^\Ca c^*)} &&&& {\Ca( 1_\Ca \otimes^\Ca c, c\otimes^\Ca (c^* \otimes^\Ca c))} \\
		&& {\Ca(c^* \otimes^\Ca 1_\Ca, c^* \otimes^\Ca (c\otimes^\Ca c^*))} &&&& {\Ca( 1_\Ca \otimes^\Ca c, (c\otimes^\Ca c^*) \otimes^\Ca c)} \\
		{1_\Va} && {\Ca(c^* \otimes^\Ca 1_\Ca, c^* \otimes^\Ca 1_\Ca)} && {1_\Va} && {\Ca( 1_\Ca \otimes^\Ca c, 1_\Ca \otimes^\Ca c)}
		\arrow["{\lambda_{c^*}\circ - \circ \rho_{c^*}^{-1}}"' {yshift=3pt, xshift=3pt} , from=1-3, to=1-1]
		\arrow["{\rho_c \circ -  \circ \lambda_c^{-1}}"' {yshift=3pt, xshift=3pt}, from=1-7, to=1-5]
		\arrow["{(\operatorname{ev}_c \otimes^\Ca \id) \circ -}"', from=2-3, to=1-3]
		\arrow["{(\id \otimes^\Ca \operatorname{ev}_c) \circ -}"', from=2-7, to=1-7]
		\arrow["{\alpha^{-1}_{c^*, c, c^*} \circ -} "', from=3-3, to=2-3]
		\arrow["{\alpha_{c, c^*, c} \circ -} "', from=3-7, to=2-7]
		\arrow["\id", from=4-1, to=1-1]
		\arrow["\id"', from=4-1, to=4-3]
		\arrow["{(\id\otimes^\Ca \operatorname{coev}_c) \circ -}"', from=4-3, to=3-3]
		\arrow["\id", from=4-5, to=1-5]
		\arrow["\id" ', from=4-5, to=4-7]
		\arrow["{(\operatorname{coev}_c \monprod[\Ca] \id ) \circ -}"', from=4-7, to=3-7]
        \end{tikzcd}
\end{equation}
A \emph{right-rigid} category is defined analogously; a category is \emph{rigid} if it is left-rigid and right-rigid.
\end{definition}

The enriched version of the snake identities above can be seen as the usual equality of the identity and snake diagrams 
by using the enriched Reshetikhin-Turaev functor -- see Proposition \ref{prop:enrRT}.

\begin{remark}
The category $\Ca$ is (left-)rigid iff $\Ca_0$ is, as the snake identities can be formulated purely in terms of the composition in $\Ca_0$. Thus, if $\Ca$ is braided and left-rigid, it is rigid \cite[Proposition 7.2]{JoyalStreet1993}.
\end{remark}

\begin{definition}
    Let $c^*$, $d^*$ be left duals of $c, d \in \Ca$. A \emph{transpose} is the morphism $\Ca(c, d) \to \Ca(d^*, c^*)$ defined as
    \[ \Ca(c, d) \xrightarrow{\id \monprod[\Ca] (- \monprod[\Ca] \id)} \Ca(d^* \monprod[\Ca] (c \monprod[\Ca] c^*), d^* \monprod[\Ca] (d \monprod[\Ca] c^*)) \xrightarrow{ \lambda_{c^*} \circ (\operatorname{ev}_d\monprod[\Ca] \id) \circ \alpha_{d^*, d, c^* }^{-1} \circ - \circ (\id \monprod[\Ca] \operatorname{coev}_c) \circ \rho_{d^*}^{-1}} \Ca(d^*, c^*).\]
\end{definition}

\begin{definition}\label{defn:ribbon-V-cat}
    A \(\Va\)-category $\Ca$ is called \emph{ribbon} if it is rigid, braided monoidal and is equipped with a \emph{twist}; i.e.\ a $\Va$-natural automorphism $\theta$ of the identity functor $\id_\Ca$ which satisfies
    \[ \theta_{c\monprod[\Ca] c'} = \beta_{c', c} \circ \beta_{c, c'} \circ (\theta_c \monprod[\Ca] \theta_{c'}) \]
    in $\Ca_0$, and for which the following diagram (in \(\Va\)) commutes
    \[ 
    \begin{tikzcd}
    1_\Va \arrow[r, "\theta_a"] \arrow[rd, "\theta_{a^*}"'] & \Ca(a, a) \arrow[d, "\mathrm{transpose}"] \\
    & \Ca(a^*, a^*) \ .
    \end{tikzcd}
    \]
\end{definition}

\subsection{Enriched skein categories}\label{sec:enrichedskeincat}
An (ordinary) skein category is a categorical analogue of a skein algebra. 
This was first defined by Walker and Johnson-Freyd in \cite{Walker, Theo}, respectively. 
The latter definition is also used in \cite{Cooke19}. 
In this section we define enriched skein categories, which is a generalization of the usual 
notion allowing for more general enriching categories than \(R\lmod\). Moreover, we also explain how enriched skein categories in certain situations have a monoidal or module-structure.

\subsubsection{Definition of enriched skein categories}
Recall from \cite[Definition~1.1]{Cooke19} that a \emph{ribbon graph} is a collection of oriented ribbons\footnote{Each ribbon has two possible orientations. For notational convenience, \cite{Cooke19} remembers the orientation of the interval $[0,1]$ \textit{and} the additional choice $\pm$, our orientation should be thought of as the product of these two, i.e. it corresponds to the arrow in \cite[Figure~1]{Cooke19}.} (strands) and coupons, with some ribbon ends attached to coupon bases or joined to form annuli. See Figure \ref{fig:qRibbonGraph} for an example.  A ribbon graph on a surface $\Sigma$ is an embedding of a ribbon graph to $\Sigma\times [0, 1]$ such that the free ends of the ribbons are sent to $\Sigma\times \{0, 1\}$ and the rest of the ribbon graph lies in $\Sigma \times (0, 1)$.

We will use ribbon graphs equipped with parenthesizations of coupons, which we will call ribbon q-graphs after \cite{LeMurakamiCategory}.
\begin{definition}
A \emph{ribbon q-graph} is a ribbon graph together with a choice, for each coupon, of parenthesization of the ribbon endpoints on the top and on the bottom of the coupon.
\end{definition}

Let $\Ca$ be a $\Va$-enriched ribbon category, where $\Va$ is tensor cocomplete. A \emph{$\Ca$-coloring} of a ribbon q-graph is a choice of an object $m_r\in \Ca$ for each ribbon $r$, c.f.\ \cite[Definition 1.2]{Cooke19}. A \emph{\(\Ca\)-colored ribbon q-graph on $\Sigma$} is defined analogously to \cite[Definition~1.3.]{Cooke19}. That is, it is a \(\Ca\)-colored ribbon q-graph together with an embedding into \(\Sigma\times [0,1]\) such that unattached bases of ribbons are sent to \(\Sigma\times \{0,1\}\), otherwise the image is in \(\Sigma \times (0,1)\) and coupons are oriented upwards.  

In the usual definition, a \(\Ca\) coloring of a graph also attaches a suitable morphism of \(\Ca\) to each coupon. 
In the $\Va$-enriched setting, each \(\Ca\)-colored q-graph instead carries an object of $\Va$ of \emph{all} possible colorings of coupons. 
We now properly define this ``object of morphisms''. 

\begin{definition}
Let  $\Gamma$ be a $\Ca$-colored ribbon q-graph. If $e$ is the bottom or the top edge of a coupon, let 
\[c_e = \otimes_{r \in e} m_r^{\text{or}(r)} \in \Ca \]
be the tensor product of the objects corresponding to the decorations of the incident ribbons, parenthesized using the parenthesization of the coupon edge. As in \cite[Definition~1.2]{Cooke19}, we use $m_r^{\text{or}(r)}$ to mean $m_r$ or $m_{r}^*$ for ribbons oriented up or down.

Define $\Ca((\Gamma)) \in \Va$ to be the unordered\footnote{That is, for each ordering of the $n$ coupons, a tensor product of the objects in the corresponding order; modulo equivalence obtained by permuting the tensor factors and the ordering. This quotient is the coequalizer of the diagram with $n!$ objects, each being a tensor product in one order, and morphisms given by permutations. See \cite[Def~II.1.58]{MSSOperads} for more details.} tensor product
\begin{equation}
    \Ca((\Gamma)) := \bigotimes_{c \textit{ Coupon of } \Gamma} \Ca(c_\text{in}, c_\text{out}) \ , 
\end{equation}
where $c_\text{in}$ and $c_\text{out}$ are the objects of $\Ca$ given by the tensor product of the objects of $\Ca$ according to the decoration and parenthesization of the bottom and top of the coupon $c$. 
\end{definition}

\begin{example}
    Let \(\Gamma\) be the \(\Ca\)-colored ribbon q-graph from Figure \ref{fig:qRibbonGraph}. 
    The corresponding ``$\Va$-object of morphisms'' is given by
    \begin{equation} 
        \Ca((\Gamma)) = \Ca((m_1\otimes  m_5) \otimes m_6, (m_3^* \otimes(m_4\otimes m_4^*))\otimes m_5)) \monprod[\Va] \Ca(m_7^*, m_7^*) \ ,
    \end{equation}
 where the order of the two coupons is chosen as on Figure \ref{fig:qRibbonGraph} when writing down the unordered tensor product.
\end{example}

\begin{figure}[H]
\centering
    \begin{overpic}[tics=10]{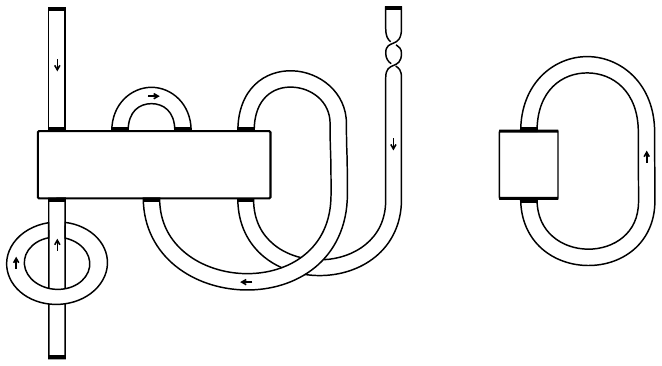}
         \put(6.7,26.5){$(\bullet$}
         \put(22.15,26.5){$\bullet)$}
          \put(36.4,26.5){$\bullet\phantom{)}$}

        \put(6.7, 32.5){$(\bullet$}
        \put(16, 32.5){$(\bullet$}
        \put(26.8, 32.5){$\bullet))$}
        \put(36.4, 32.5){$\bullet\phantom{(}$}

        \put(79.2, 32.5){$\bullet\phantom{(}$}
        \put(79.2,26.5){$\bullet\phantom{)}$}

        \put(10.5,5){$m_1$}
        \put(17,14.7){$m_2$}
        \put(10.5,45.5){$m_3$}
        \put(21.4,42.8){$m_4$}
        \put(43,45.5){$m_5$}
        \put(61.5,45.5){$m_6$}
        \put(87,47.5){$m_7$}
        
        \put(3.7,29){$1$}
        \put(73.5,29){$2$}
    \end{overpic}
	\caption{An example of a \(\Ca\)-colored ribbon q-graph \(\Gamma\).}
	\label{fig:qRibbonGraph}
\end{figure}

\begin{definition}
The category $\mathbf{Ribbon}_\Ca(\Sigma)$ is a $\Va$-category with objects \(m\) given by finite collections of points \(x_i\) on $\Sigma$ with framing, orientation and a decoration by objects \(m_i\) of $\Ca$.
That is, an object is \(m=\{{m_i}_{x_i}^{\varepsilon_i}\}_{i\in I_m}\), where the \(x_i\)'s are pairwise disjoint framed points of \(\Sigma\), \(\varepsilon_i = \pm\) is the corresponding orientation and \(m_i\) denotes the object of \(\Ca\) decorating the point. 
We often abbreviate this by simply saying a finite collection of \emph{decorated points} on \(\Sigma\), and only write \(m=\{m_i\}_{i\in I_m}\) with the rest of the data left implicit. 

Let $m, n$ be two such objects. We define the corresponding Hom-object as
\begin{equation}\label{eq:EnrichedSkeinHom}
    \mathbf{Ribbon}_\Ca(\Sigma)(m, n) := \coprod_{[\Gamma]\colon m \to n} \Ca((\Gamma)) 
\end{equation}
where the coproduct is over all isotopy classes of $\Ca$-colored ribbon q-graphs on \(\Sigma\) compatible with the coloring, framing and orientation of $m$ and $n$ on $\Sigma\times \{0\}$ and $\Sigma\times \{1\}$, respectively.

The composition is defined by stacking of graphs using that $\Ca((\Gamma\circ\Gamma'))\cong\Ca((\Gamma))\otimes\Ca((\Gamma'))$ and that the monoidal product of $\Va$ distributes over coproducts. 
\end{definition}

\begin{example}
For $\Va = (\operatorname{Set}, \times)$, an element of the set $\mathbf{Ribbon}_\Ca(\Sigma)(m, n)$ is an isotopy class of $\Ca$-colored ribbon graphs with coupons decorated by morphisms in $\Ca$ compatible with the decoration of ribbons. For $\Va = (R\lmod, \otimes_R)$, an element of the $R$-module $\mathbf{Ribbon}_\Ca(\Sigma)(m, n)$ is a formal $R$-linear combination of isotopy classes of $\Ca$-colored ribbon graphs, again with coupons decorated by suitable morphisms in $\Ca$. For example, a ribbon graph with a coupon decorated by $f+g$ is equal to the sum of ribbon graphs with coupons decorated with $f$ and $g$. We note that this property is not present in \cite[Definition 1.5]{Cooke19} -- the category defined there is rather $F_*(\mathbf{Ribbon}_{U_*\Ca}(\Sigma))$, where $U\colon \Vect \rightleftarrows \operatorname{Set}\colon F$ is the usual free--forgetful adjunction. 
After passing to the skein category this distinction vanishes.
\end{example}

\begin{definition}
For $\Sigma = [0, 1]^{\times 2}$, define $\mathbf{qRibbon}_\Ca([0, 1]^{\times 2})$ to be the $\Va$-category with objects are given by \emph{parenthesized words} $m_W$ with letters $m_i^{\pm}$ for $m_i \in \Ca$ and \(\pm\) denoting the corresponding orientation. 
Morphisms are given by
\begin{equation}\label{eq:EnrichedSkeinHomqRibbon}
    \mathbf{qRibbon}_\Ca([0, 1]^{\times 2})(m_W, n_{V}) := \mathbf{Ribbon}_\Ca([0, 1]^{\times 2})(m_W \times\{0\},  n_{V} \times\{0\})
\end{equation}
where $m_W \times\{0\}$ is the square with the objects $m_i^\pm$ uniformly distributed on the interval $[0,1]\times\{0\} \subset [0,1]^{\times 2}$, with framing in the direction of 
$[0,1]\times\{0\}$.
\end{definition}
In other words, $\mathbf{qRibbon}_\Ca([0, 1]^{\times 2})$ consists of (\(\Ca\)-colored) ribbon q-graphs in the cube ending on the bottom front and top front edges of the cube, together with parenthesizations. This category comes with a ribbon structure and a straightforward generalization of the ordinary Reshetikhin-Turaev functor \cite{RT, Turaev} to the enriched setting, as we will now see.

\begin{proposition}\label{prop:enrRT}
The category $\mathbf{qRibbon}_\Ca([0, 1]^2)$ is a $\Va$-enriched ribbon category. There is a ribbon\footnote{An obvious generalization of a ribbon functor, i.e. a braided, strong monoidal \(\Va\)-functor preserving twists.} $\Va$-functor
\begin{equation*}
T \colon \mathbf{qRibbon}_\Ca([0, 1]^2) \to \Ca 
\end{equation*}
such that on generators we have $T(m, +)= m$, $T(m, -)=m^*$, $T$ is strict monoidal and such that on morphisms it satisfies the (enriched analogue of the) conditions of \cite[Theorem 2.5]{Turaev}.
The functor \(T\) is called the \emph{enriched Reshetikin-Turaev functor}. 
\end{proposition}
\begin{proof}
In this proof we abbreviate objects \(m_W\) of \(\mathbf{qRibbon}_\Ca([0, 1]^2)\) by \(m\) for lighter notation. 
The tensor product of $\mathbf{qRibbon}_\Ca([0, 1]^2)$ is given by juxtaposition of parenthesizations.
On morphisms, the tensor product is induced by the morphism 
\begin{align*}
\coprod_{[\Gamma]\colon m \to n} \Ca((\Gamma)) \otimes  \coprod_{[\Gamma']\colon {m'} \to {n'}} \Ca((\Gamma'))
\cong  \coprod_{[\Gamma]\colon m \to n, [\Gamma']\colon {m'}\to {n'} } \Ca((\Gamma))\otimes \Ca((\Gamma')) \to 
\coprod_{[\Gamma'']\colon {mm'} \to {nn'}}\Ca((\Gamma''))
\end{align*}
given by the canonical inclusions $\Ca((\Gamma))\otimes \Ca((\Gamma'))\cong \Ca((\Gamma\Gamma')) \xrightarrow{i_{\Gamma\Gamma'}} \coprod_{[\Gamma'']\colon mm' \to nn'}\Ca((\Gamma''))$. Here, $\Gamma\Gamma'$ is just the juxtaposition of the two ribbon graphs.

The Reshetikhin-Turaev functor, being strict monoidal, sends a parenthesized sequence $m$ into the parenthesized tensor product of the objects (or their duals). On morphisms, it is defined as follows: As a map from the coproduct $\coprod_{[\Gamma]\colon m \to n} \Ca((\Gamma)) \to \Ca(T(m), T(n))$, it is given by the sequence of maps $T_{[\Gamma]}\colon \Ca((\Gamma)) \to \Ca(T(m), T(n))$. These maps are defined by decomposing the ribbon graph $\Gamma$ into compositions of generators of the category $\mathbf{qRibbon}_\Ca([0, 1]^2)$ as in \cite[Theorem 2.5, Lemma~3.4]{Turaev}, with suitable associators inserted in between. For example for the ribbon graph $\beta\colon (m,m') \to (m', m)$ given by the overcrossing of two strands going up, we define $T_\beta$ as the component $\beta_{m, m'} \colon 1_\Va \ra \Ca(m\otimes m', m'\otimes m)$; and similarly for other basic ribbon graphs such as the twist, associator, cap or cup. For a ribbon graph $\Gamma_c$ given by a single coupon \(c\), the $\Va$-morphism $T_c$ is the identity on $\Ca((\Gamma_c))$. 

For $T_{[\Gamma]}$ to be well-defined, it needs to satisfy the relations listed in \cite[Lemma~3.4]{Turaev}, which follows from the fact that $\Ca$ is a ribbon \(\Va\)-category. Note that, in contrast to loc. cit., the images of the relations in $\Ca$ will have associators inserted between the images of generating morphisms, since $\Ca$ is not assumed to be strict. For example,  the snake identities (in the ribbon category \(\Ca\)) from Equation \eqref{eq:snake-relations-enriched-rigid-cat} have an additional associator compared to \cite[((3.2.b,c)]{Turaev}. 
\end{proof}

Using the Reshetikhin-Turaev functor, we can define the enriched skein category. 
Here, a suitable coequalizer will replace the usual quotient by skein relations. 
\begin{definition}[Enriched skein category] \label{defn:EnrichedSkeinCatnew}
The \emph{enriched skein category of \(\Sigma\)}, $\mathbf{Sk}_\Ca(\Sigma)$, is the \(\Va\)-category with the same objects as $\mathbf{Ribbon}_\Ca(\Sigma)$. To define morphisms $m\to n$, consider the set of all triples $(\Gamma, cb, p)$ where $\Gamma\colon m\to n$ is a \(\Ca\)-colored ribbon q-graph, 
$cb \colon [0, 1]^3\hookrightarrow \Sigma\times[0, 1]$ an orientation-preserving embedding of the cube such that $\Gamma$ intersects the cube transversely and only at the bottom and the top front edges, 
and $p$ is parenthesizations of these intersection points on the bottom and top. From this we form a diagram in \(\Va\) with two sets of objects: $\{s_\Gamma = \Ca((\Gamma))\}$ for all ribbon graphs $\Gamma\colon m\to n$ and $\{r_{(\Gamma, cb, p)} = \Ca((\Gamma))\}$ for all triples as above. The arrows in this diagram are given as follows: for each triple $(\Gamma, cb, p)$, there is an arrow $r_\Gamma\xrightarrow{\text{id}} s_\Gamma$, and an arrow $r_\Gamma\xrightarrow{T_{\Gamma\cap cb}} s_{\Gamma_{cb, p}}$, where $\Gamma_{cb, p}$ is the ribbon graph obtained by replacing the cube in $\Gamma$ by the single parenthesized coupon given by the cube. The arrow is given by the component of the Reshetikhin-Turaev functor $T_{\Gamma\cap cb}$ on the interior of the cube tensored with the identity on the exterior of the cube:
\begin{equation} \label{eq:SkCoequalizer}
\begin{tikzcd}
	{\{r_{(\Gamma, cb, p)} = \Ca((\Gamma))\}_{(\Gamma, cb, p)}} && {\{s_{\Gamma'} = \Ca((\Gamma'))\}_{\Gamma'}}
	\arrow["{\{\id_{\Ca((\Gamma))}\}_{(\Gamma, cb, p)}}", shift left=2, from=1-1, to=1-3]
	\arrow["{\{T_{\Gamma\cap cb}\}_{(\Gamma, cb, p)}}"', shift right=2, from=1-1, to=1-3]
\end{tikzcd}\end{equation}
Then, $\mathbf{Sk}_\Ca(\Sigma)(m, n)\in \Va$ is defined as the colimit over the diagram \eqref{eq:SkCoequalizer}. 
\end{definition}
\begin{remark}\label{rem:isotopy-classes-automatic-from-RT-functor}
    We do not impose invariance under isotopies of $\Gamma$ as an additional relation above because isotopy invariance already follows from the fact that the Reshetikhin-Turaev functor is isotopy invariant and that each isotopy can be seen as a composition of isotopies inside boxes \cite[Corollary 1.3]{EK}.  For $\Va = \operatorname{Set}, R\lmod$, we get back the definition of the skein category (with parenthesizations). That is, the colimit \eqref{eq:SkCoequalizer} quotients by the usual skein relations of \cite[Definition 1.9.]{Cooke19}. Note that, using the skein relations, we can change the parenthesizations arbitrarily: 
    
\begin{equation}
  \begin{overpic}[tics=10]{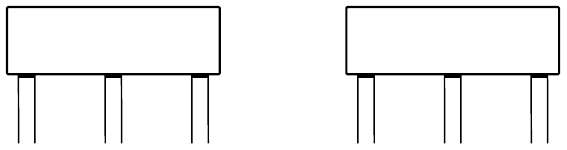}
        \put(48.5,18){$=$}
        \put(19.3, 19){$f$}
        \put(69.0, 19){$f\circ \alpha^{-1}_{m_1, m_2, m_3}$}
        \put(2.7,14.8){$(\bullet$}
        \put(19.1,14.8){$\bullet)$}
        \put(34.4,14.8){$\bullet\phantom{)}$}
        \put(63.8,14.8){$\bullet$}
        \put(77.5,14.8){$(\bullet$}
        \put(94.4,14.8){$\bullet)$}
        \put(7, 6){$m_1$}
        \put(22.1, 6){$m_2$}
        \put(37.7, 6){$m_3$}
        \put(67, 6){$m_1$}
        \put(82.1, 6){$m_2$}
        \put(97.7, 6){$m_3$}
  \end{overpic}
\end{equation}
    Thus, if $\Ca$ is strict, the functor from the parenthesized skein category to the non-parenthesized skein category which forgets the parenthesizations of coupons is an isomorphism of categories.
\end{remark}

We now investigate an example of an enriched skein category, i.e.\ for \(\Sigma=S^2\). In \cite[Proposition 4.4]{GJS} the free cocompletion of the (ordinary/unenriched) skein category of the sphere is computed to be the Müger center of the free cocompletion of the corresponding ribbon category. It would be interesting to understand in what sense this carries over to the enriched setting. 

\begin{example}
For a $\Va$-enriched ribbon category $\Ca$ we describe the skein category associated to the 2-sphere ${S}^2$. 
We begin by looking at the functor $\Ca \cong \SkC{\mathbb{D}^2}\to \SkC{S^2}$ corresponding to an embedding of $\mathbb{D}^2$ into $S^2$ which we assume to avoid the north pole of $S^2$. This functor is clearly essentially surjective and full, hence $\SkC{S^2}$ is a quotient of $\SkC{\mathbb{D}^2}$. The relations not already holding in $\Ca$ are induced by moving a ribbon through the north pole. The new relation is given by
\begin{equation}\label{eq:sphererelation}
    \begin{overpic}[tics=10]{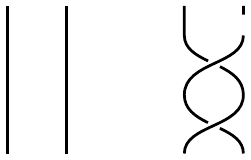}
        \put(48, 28){$\sim$}
        \put(95.5, 51){$\theta^2$}
    \end{overpic},
\end{equation}
where $\theta$ is the twist. This relation is imposed by quotienting by the categorical ideal generated by it; \emph{not} by the monoidal ideal, i.e. it doesn't hold fully locally. The skein category of the sphere could be therefore called the \emph{twisted Müger cocenter} of $\Ca$. 
\end{example}

\begin{example}[Skein categories introduce torsion]\label{Ex: torsion}
Let $\Va=\widehat{\Ch\lmod}$ and $\Ca=  U(\mathfrak{g})\lmod^\Phi[[\hbar]]$ 
and consider the corresponding enriched skein category of the sphere, i.e.\ $\mathbf{Sk}_{U(\mathfrak{g})\lmod^\Phi[[\hbar]]}{(S^2)}$. The relation \eqref{eq:sphererelation}, with trivial left strand, implies 
\begin{align}
0= \theta^2_{X}- \id_X = \hbar t_{X, X} + O(\hbar^2) \ \ . 
\end{align}
Hence the element $t_{X,X}+O(\hbar)$ is torsion in $\mathbf{Sk}_{U(\mathfrak{g})\lmod^\Phi[[\hbar]]}{(S^2)}$. It is also non-zero because its image under the canonical map to $\mathbf{Sk}_{U(\mathfrak{g})\lmod}{(S^2)}$ is $t_{X,X}$ which is non-zero\footnote{The skein category of a symmetric monoidal category $\Ca_\textnormal{SMC}$ over a sphere is $\Ca_\textnormal{SMC}$.}. This example shows that even for a topologically free category $\Ca$ the skein category is not topologically free in general. 
\end{example}

Note that in Proposition \ref{prop:freeskeinmodules} we prove that given a surface with non-trivial boundary it does follow that the enriched skein category is topologically free.

\subsubsection{Monoidal structure on \(\SkC{C\times [0,1]}\)}
We here explain how, for a 1-manifold \(C\), the \(\Va\)-category \(\SkC{C\times [0,1]}\) is monoidal. This follows from the functoriality of the construction with respect to embeddings, which we briefly explain first. 
\begin{observation}
    An embedding of surfaces \(p\colon \Sigma \ra \Pi\) induces a \(\Va\)-functor 
    \begin{equation}
    P\colon \SkC{\Sigma} \ra \SkC{\Pi}
    \end{equation}
    of enriched skein categories. On objects, \(P\) sends points \(m_i\) on \(\Sigma\) with framing, orientation and decoration to the point \(p(m_i)\) in \(\Pi\) decorated by the same object of \(\Ca\) and framing and orientation induced by the embedding \(p\). 
    To define \(P\) on morphisms, let \(m,n\) denote two objects of \(\SkC{\Sigma}\). The $\Va$-morphism \(\SkC{\Sigma}(m, n) \to \SkC{\Pi}(P(m), P(n))\) is specified using the universal property of the colimit \eqref{eq:SkCoequalizer}: it corresponds to the collections of inclusions of  \(\Ca((\Gamma)) = \Ca((p(\Gamma)))\) into \(\SkC{\Pi}(P(m), P(n))\).
\end{observation}

\begin{definition}
\label{ex:MonStruCTimesInt} 
For $\Sigma = C \times [0, 1]$ as above, the \(\Va\)-enriched skein category \(\SkC{C\times [0,1]}\) has a monoidal structure defined as follows. Consider the two embeddings \(L, R \colon C\times [0,1] \ra C\times [0,1]\) given by \((c, s) \stackrel{L}{\mapsto} (c, \frac{s}{2})\) and \((c, s) \stackrel{R}{\mapsto} (c, \frac{1}{2} + \frac{s}{2})\). That is, \(L\) embeds \(C\times [0,1]\) into the left half of \(C\times [0,1]\) and \(R\) embeds it into the right half. For later reference let \(l\) respectively \(r\) denote the isotopies between the identity embeddings of \(C\times [0,1]\) into itself and \(L\) respectively \(R\). The monoidal product on \(\SkC{C\times [0,1]}\) is induced by the embedding
    \begin{align*}
        I \colon C\times [0,1] \sqcup C\times [0,1] \xhookrightarrow{L \sqcup R} C\times [0,1].
    \end{align*}
    We denote the monoidal product of two objects \(a\) and \(b\) by \(a\ast b\). The monoidal unit is the empty set. 
\end{definition}

\subsection{Excision}\label{sec:excisionenrichedsetting}
In this section we prove excision for enriched skein categories. The general ideas behind the proof follow those of \cite{Cooke19}, where excision is proven for \(\Vect\)-enriched skein categories. 

Since the proof of excision mainly exploits the topological features of skein categories, many arguments directly carry over from the usual setting. Because of this we will refrain from giving equally explicit descriptions of some of the topological arguments (i.e.~spelling out all isotopies) and rather explain the arguments pictorially. We mostly follow the notation in \cite{Cooke19}, so the interested reader can easily fill in the missing details. 

A formal consequence of the excision property, and the main result of this section, is:
\begin{theorem}\label{Thm: FH=Sk} 
Let $\Ca$ be a $\Va$-enriched ribbon category.
The skein category computes factorization homology 
\begin{align}
    \SkC{\Sigma}\cong \int_\Sigma \Ca \in \Va\mhyphen \Cat
\end{align}
\end{theorem}
\begin{proof}
    By Theorem \ref{thm:properties-for-FH} this is a direct consequence of Theorem \ref{thm:Excision} and \(\SkC{\mathbb{D}^2}\cong \Ca\). 
\end{proof}

\subsubsection{Module structure on (enriched) skein categories}
Let \(C\) be a 1-manifold, and \(M\) be a surface with boundary. The first building block for the excision-proof is to understand how \(\SkCM\) can be equipped with a \(\SkC{C\times [0,1]}\)-module structure through the data of an appropriate embedding of \(C\) into \(\partial M\).

\begin{definition} \cite[Definition 1.17]{Cooke19} \label{defn:ThickEmbedding}
    A \emph{thickened right embedding} of \(C\) into the boundary of \(M\) consists of 
    \begin{enumerate}
        \item An embedding \(\Xi\colon C\times (-\varepsilon, 1] \hookrightarrow M\), for some \(\varepsilon >0\), such that its restriction to \(C\times \{1\}\) gives an embedding \(\xi\colon C\hra \partial M \). We define the restriction \(\Phi:=\Xi_{|_{[0,1]}}\) and the restriction \(\mu:=\Xi_{|_{C\times \{0\}}}\).
        \item An embedding \(E\colon M \ra M\) such that \(\textrm{Im}(E)\) is disjoint from \(\trm{Im}(\Phi)\).
        \item An isotopy \(\lambda\colon M\times [0,1] \ra M\) from \(\trm{Id}_{M}\) to \(E\) which is trivial outside of \(\trm{Im}(\Xi)\).
    \end{enumerate}
\end{definition}
Similarly one defines a \emph{thickened left embedding} with the only alteration being that \(\Xi\) is an embedding \(\Xi \colon C\times [0, 1+\varepsilon)\hra M\), for some \(\varepsilon>0\), such that restricted to \(C\times \{0\}\) it is an embedding \(\xi \colon C\hra \partial M\).

\begin{example} \label{ex:Delta}
\normalfont
Figure \ref{fig:Excision_Fig1} illustrates the data of a thickened right embedding of \(C=S^1\) into the boundary of the once-punctured torus \(M=\Sigma_{1,1}\). Condition (2) and (3) from Definition \ref{defn:ThickEmbedding} together imply that there exists a \(\delta>0\) such that \(-\varepsilon < -\delta < 0 \) as illustrated in Figure \ref{fig:Excision_Fig1}. 
\end{example}

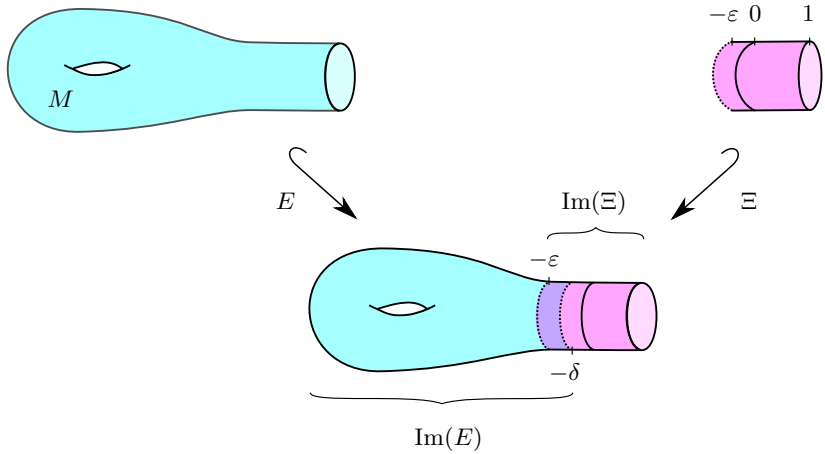
\begin{figure}[h!]
\centering
\begin{overpic}[scale=.2,tics=10]{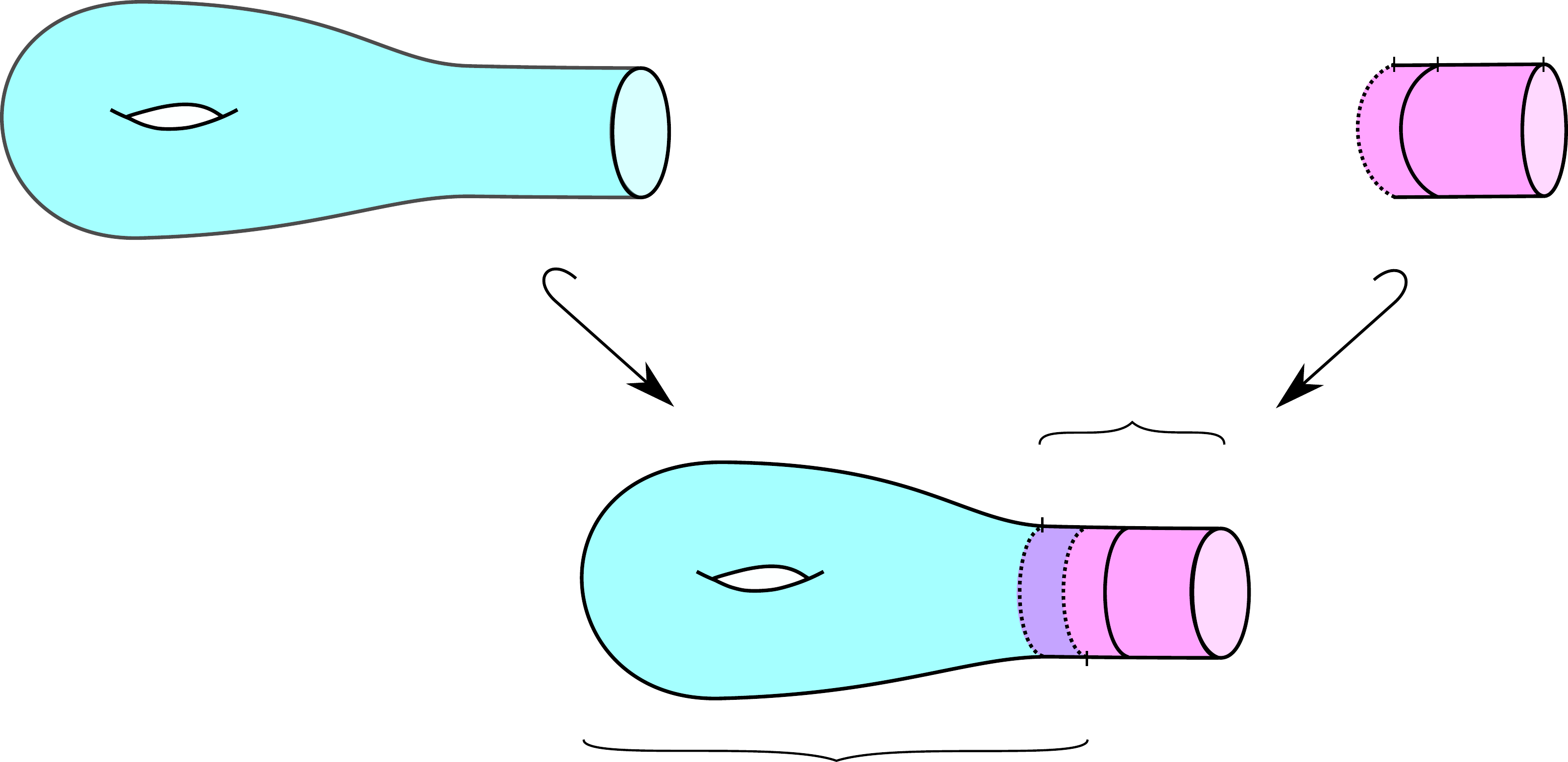}
\put(5,36.5){\small{$M$}}
\put(86,47){\small{$-\varepsilon$}}
\put(91,47){\small{$0$}}
\put(97.5,47){\small{$1$}}
\put(64,17){\small{$-\varepsilon$}}
\put(66.5,3){\small{$-\delta$}}
\put(68,24){\small{Im($\Xi$)}}
\put(50,-5){\small{Im($E$)}}
\put(33,24){\small{$E$}}
\put(90,24){\small{$\Xi$}}
\end{overpic}
\vspace{0.5cm}
	\caption{Illustration of a thickened right embedding of \(S^1\) into the once-punctured torus \(\Sigma_{1,1}\).}
	\label{fig:Excision_Fig1}
\end{figure}

\begin{definition}
\label{rem:IsotopyGivesRibbon} 
    Let \(f,g: M \ra M\) be two embeddings and let \(\sigma\colon M\times[0,1] \ra M\) be an isotopy from \(f\) to \(g\). For any object \(m\in \SkCM\) the isotopy \(\sigma\) induces a map
    \begin{align*}
        r_{\sigma, m} \colon 1_\Va \ra \SkCM(f(m), g(m)).
    \end{align*}
    That, is, let \(\Gamma_\sigma\) denote the actual ribbon traced out by the isotopy \(\sigma\). The map \(r_{\sigma, m}\) is induced by sending \(1_\Va \to \Ca((\Gamma_\sigma)) = 1_\Va\). We talk about both \(\Gamma_\sigma\) and the induced map \(r_{\sigma, m}\) as \emph{the ribbon corresponding to the isotopy \(\sigma\)}.  
\end{definition}

There are two instances of isotopies giving rise to ribbons that will play an important role later. 
    
\begin{example} \label{ex:RibbonTangleFromLambda} \normalfont
    Given a thickened right embedding \((\Xi, E, \lambda)\) and an arbitrary object \(m\in \SkCM\) the isotopy \(\lambda\) will give a ribbon tangle \(r_{\lambda,m}\colon 1_\Va \ra \SkCM(m, E(m))\). This is illustrated in Figure \ref{fig:Excision_Fig2} below for \(C=S^1\), \(M=\Sigma_{1,1}\) and the object \(m=\{m_1, m_2, m_3, m_4\}\). 
    The isotopy \(\lambda\) traces out identity ribbons for the (blue) objects \(m_1\) and \(m_2\), while it gives ribbons that continuously move the (red) objects \(m_3\) and \(m_4\) into \(\im(E)\cap \im(\Xi)\).  
\end{example}

\begin{example} \label{ex:UnitorRibbon}
From the monoidal structure on \(\SkC{C\times[0,1]}\), introduced in Definition \ref{ex:MonStruCTimesInt}, we have two isotopies \(l\) and \(r\) retracting \(C\times [0,1]\) to the left respectively right half of \(C\times [0,1]\). The ribbon corresponding to the isotopy \(l\) for the object \(a=\{a_1, a_2\}\) is depicted in Figure \ref{fig:Excision_Fig3}. The ribbon coming from the isotopy \(r\) is analogous, only retracting to the right instead of left. 
\end{example}

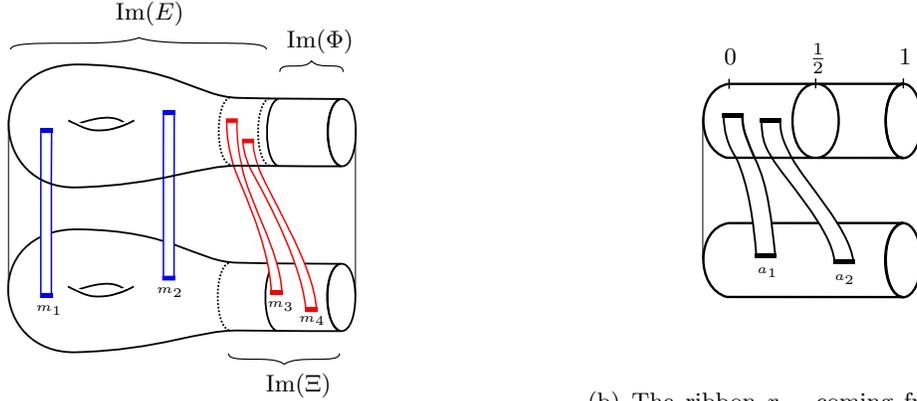
\begin{figure}[H]
    \centering
    \vspace{0.25cm}
    \begin{subfigure}[b]{0.5\textwidth}
         \centering
         \begin{overpic}[scale=.2,tics=10]{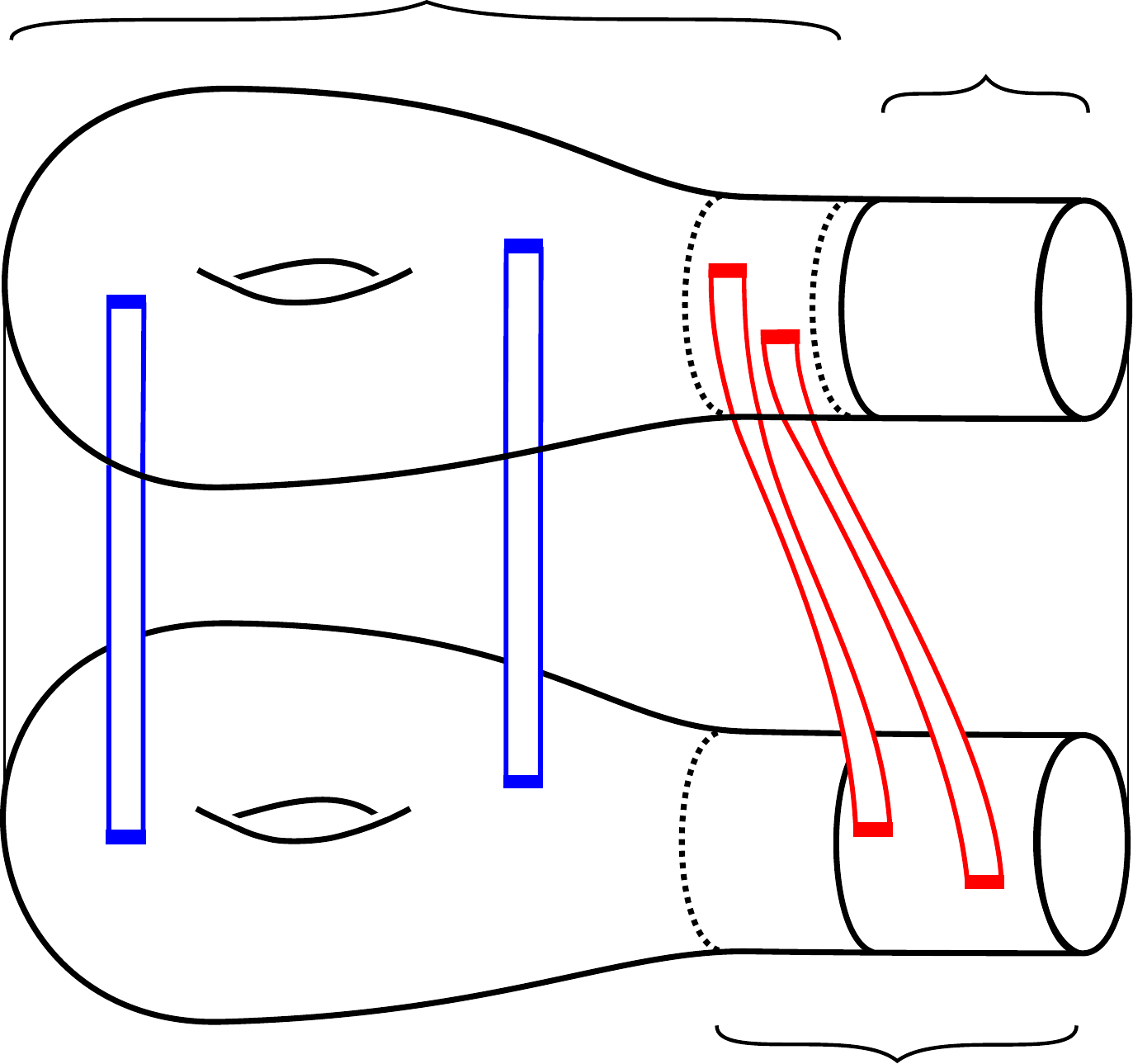}
         \put(74,-7.5){\small{Im($\Xi$)}}
         \put(80,91){\small{Im($\Phi$)}}
         \put(31,99){\small{Im($E$)}}
         \put(8.25,15.5){\tiny{$m_1$}}
         \put(42.75,20.5){\tiny{$m_2$}}
         \put(74.5,16.5){\tiny{$m_3$}}
          \put(83.5,12.5){\tiny{$m_4$}}
         \end{overpic}
         \vspace{0.5cm}
         \caption{The ribbon \(r_{\lambda, m}\) for $m = \{m_1,  m_2, m_3, m_4\}$ coming from the isotopy \(\lambda\).}
         \label{fig:Excision_Fig2}
     \end{subfigure}
     \hfill
     \begin{subfigure}[b]{0.4\textwidth}
         \centering
         \begin{overpic}[scale=.25,tics=10]{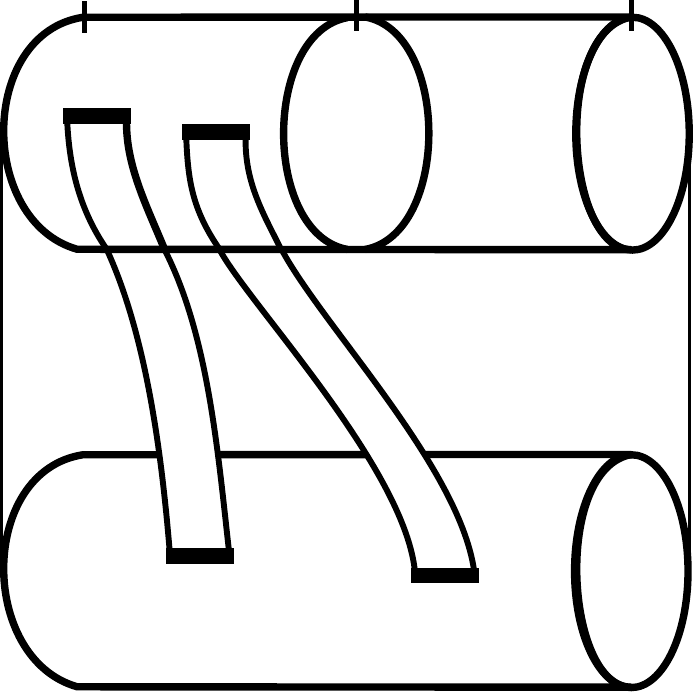}
         \put(10,106){\small{$0$}}
         \put(48.75,106){\small{$\frac{1}{2}$}}
         \put(89.25,106){\small{$1$}}
         \put(25,11.25){\tiny{$a_1$}}
         \put(60,8.5){\tiny{$a_2$}}
         \end{overpic}
         \vspace{1cm}
         \caption{The ribbon \(r_{l,a}\) coming from the left retraction, i.e.~multiplying an object \(a = \{a_1, a_2\}\) with the empty set \(\emptyset\).}
         \label{fig:Excision_Fig3}
     \end{subfigure}
        \caption{Two examples of ribbons arising from an isotopy.}
        \label{fig:Excision_Fig2And3}
\end{figure}

\begin{definition} \cite[Definition 1.19]{Cooke19} \label{defn:embedding_module}
    Given a thickened right embedding \((\Xi, E, \lambda)\) of \(C\) into the boundary of \(M\), \(\SkCM\) is a right \(\SkC{C\times [0,1]}\)-module with action 
    \begin{align*}
        \lhd \colon \SkCM \catprod \SkC{C\times [0,1]} \longrightarrow \SkCM
    \end{align*}
    induced from the embedding of surfaces 
    \begin{align*}
        M \sqcup \left( C\times [0,1] \right) &\xrightarrow{E\sqcup \Phi} M
    \end{align*}
    The associator is a \(\Va\)-natural transformation, i.e.~we need to specify maps \(\beta_{a, m, b} \colon 1_\Va \rightarrow \SkC{M}((m\lhd a) \lhd b, m \lhd (a\ast b))\) for all \(a, b \in \SkC{C\times [0,1]}\) and \(m\in\SkC{M}\). These maps are given by sending \(1_\Va \to \Ca((\Gamma_\beta)) = 1_\Va\), where \(\Gamma_\beta\) is the ribbon in Figure \ref{fig:Excision_Fig4}. Similarly, the unitor \(\eta_m \colon 1_\Va \rightarrow \SkC{M}(m\lhd \emptyset, m)\) is defined through the map sending \(1_\Va \to \Ca((\Gamma_\eta)) = 1_\Va\) where \(\Gamma_\eta\) is the inverse of the ribbon from Example \ref{ex:RibbonTangleFromLambda}. 
\end{definition}

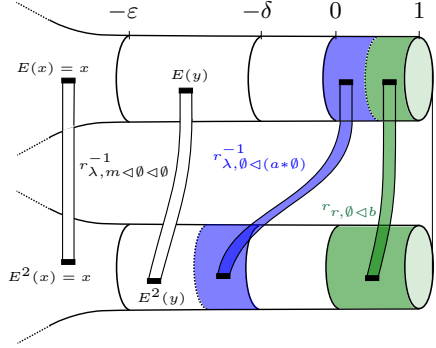
\begin{figure}[H]
 \centering
 \begin{overpic}[scale=.15,tics=10]{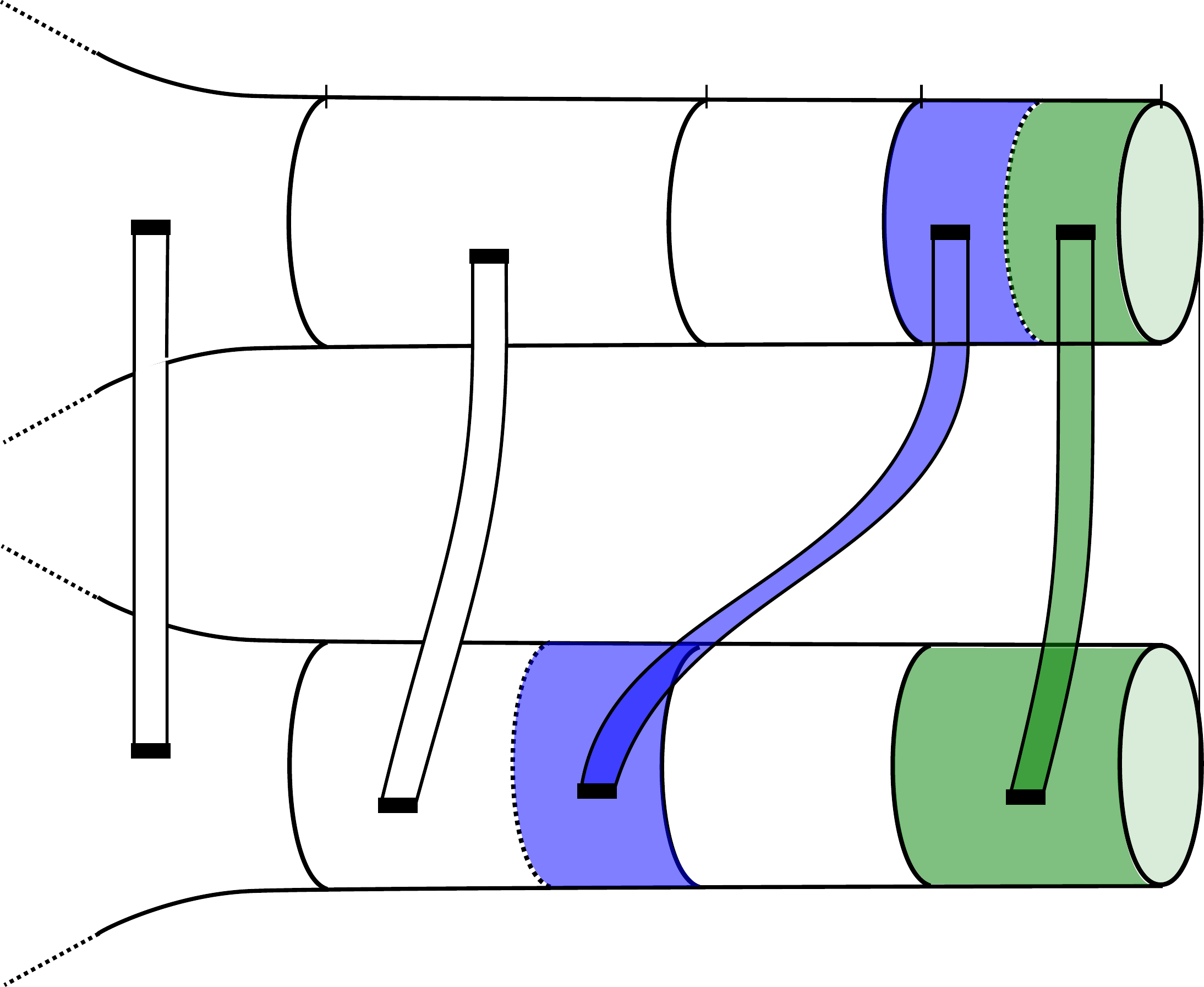}
         \put(22,78){\small{$-\varepsilon$}}
         \put(54,78){\small{$-\delta$}}
         \put(75,78){\small{$0$}}
         \put(95,78){\small{$1$}}
        \put(0,65){\tiny{$E(x) = x$}}
        \put(-2,15){\tiny{$E^2(x) = x$}}
        \put(29,10){\tiny{$E^2(y)$}}
        \put(37,63.5){\tiny{$E(y)$}}
        \put(15,42){\tiny{$r^{-1}_{\lambda, m\lhd \emptyset \lhd \emptyset}$}}
        \put(47,44){\textcolor{blue}{\tiny{$r^{-1}_{\lambda,\emptyset \lhd(a \ast \emptyset)}$}}}
        \put(73,32){\textcolor{darkspringgreen}{\tiny{$r_{r,\emptyset \lhd b}$}}}
 \end{overpic}
	\caption{The ribbon \(\Gamma_\beta\) corresponding to the associator \(\beta_{m,a,b}\) for \(m= x\otimes y\).}
	\label{fig:Excision_Fig4}
\end{figure}

Equivalently, the data of a thickened left embedding \((\Xi, E, \lambda)\) of \(C\) into the boundary of a surface \(N\) defines a left \(\SkC{C\times [0,1]}\)-module structure on \(\SkC{N}\).

\subsubsection{Excision of enriched skein categories}
Excision of ordinary skein categories was conjectured by Johnson-Freyd \cite{Theo}, based on ideas of Walker \cite{Walker, MW11}. 
A similar excision result for universal braid categories in \(\mathrm{Set}\) is proven by Yetter in \cite{Yetter92}. 
The topological parts of this proof carries over to the proof of excision for ordinary skein categories in \cite[Theorem 1.22]{Cooke19}. 
We also use the same topological arguments here when proving excision of enriched skein categories. 

\begin{theorem}[Excision of enriched skein categories] \label{thm:Excision}
    Let \(C\) be a 1-manifold with a thickened right embedding \((\Xi_M, E_M, \lambda_M)\) into the boundary of the surface \(M\) and a thickened left embedding \((\Xi_N, E_N, \lambda_N)\) into the boundary of the surface \(N\). Denoting \(A \coloneqq C\times [0,1]\), let \(\SkC{M} \relcatprod{\SkC{A}} \SkC{N}\) be the \(\Va\)-enriched relative Tambara tensor product (defined in Section \ref{TambaraTensorSection}), and let 
    \begin{align*}
        M\sqcup_A N \coloneqq M\sqcup N \bigg/\bigg\{ \Phi_M(c,i) \sim \Phi_N(c,i) \quad\forall\, c\in C,\,\,i\in [0,1]\bigg\} \ \ .
    \end{align*} 
    The thickened embeddings define a \(\Va\)-functor 
    \begin{align}
        F \colon \SkC{M} \relcatprod{\SkC{A}} \SkC{N} \stackrel{\simeq}{\longrightarrow} \SkC{M\sqcup_A N}
    \end{align}
    which is an equivalence of \(\Va\)-categories.
\end{theorem} 

\begin{notation}
	For easier referencing of the different regions of \((M\sqcup_A N)\times [0, 1]\), we will call \( (\im \,\Xi_M \cup \im\,\Xi_N) \times [0,1]\) the \emph{middle region}. 
\end{notation}

\begin{remark}
    We will for notational convenience work with \(\delta\) as in Example \ref{ex:Delta} being the same for both the thickened left and right embedding.
    We can do this without loss of generality by e.g.\ setting \(\delta\) to be the smaller of the two a priori different numbers.
\end{remark} 

\begin{proof}
We start by defining the functor \(F\). Recall from the definition of the Tambara tensor product that objects of \(\SkC{M}\relcatprod{\SkC{A}} \SkC{N}\) are given by pairs of objects \((m, n)\), where \(m\in \SkC{M}\) and \(n\in \SkC{N}\). Explicitly, \(m\) will be a finite collection of disjoint, framed, oriented and decorated (by objects of \(\Ca\)) points in \(M\), and similarly for \(n\). \par 
On objects, we define the functor by
\begin{align*}
	F(m, n) \coloneqq E_M(m) \sqcup E_N(n) \ ,
\end{align*}
which is a finite set of disjoint, framed, oriented and decorated points in \(M\sqcup_A N\), and hence an object of \(\SkC{M\sqcup_A N}\). 

Next, we define $F$ on $\Va$-objects of morphisms: let \((m, n)\) and \((m', n')\) be two objects of \(\SkC{M} \relcatprod{\textbf{Sk}_\Ca(A)} \SkC{N}\). Recall that the \(\Va\)-object of edges of the graph \(\Omega\) underlying the Tambara tensor product \(\Va\)-category is
\begin{align}
\Omega((m&,n),(m',n')) = \\
&\SkC{M} \catprod \SkC{N}\big((m,n),(m',n')\big) \coprod \big( \coprod_{\substack{\{a~|~m = m' \triangleleft a, \\ a \triangleright n = n' \}}} 1_\Va \big) \coprod 							\big( \coprod_{\substack{\{a~|~m \triangleleft a = m', \\ n = a \triangleright n'\}}} 1_\Va \big) \ \ ,
\label{eq:OmegaTambaraSkein}
\end{align}
where \(a\in \SkC{A}\).
From this we need to produce a \(\Va\)-object of edges from \(E_M(m) \sqcup E_N(n)\) to \(E_M(m') \sqcup E_N(n')\). Consider \(\Gamma_M\) a ribbon graph from \(m\) to \(m'\) in \(M\) and \(\Gamma_N\) a ribbon graph from \(n\) to \(n'\) in \(N\). The map on Hom-objects is induced by sending \(\Gamma_M \mapsto E_M(\Gamma_M)\) resp.~\(\Gamma_N \mapsto E_N(\Gamma_N)\). In addition the copies of \(1_\Va\) for \(a\in \SkC{A}\) are sent to \(\Ca((\Gamma)) = 1_\Va\), where \(\Gamma\) is the ribbon in Figure \ref{fig:Excision_Fig5} for the first coproducts of \(1_\Va\) corresponding to \(\iotaman\), respectively the inverse ribbon for the second coproducts of \(1_\Va\) corresponding to \(\iotamanInv\). 

\begin{figure}[H]
\centering
    \begin{overpic}[scale=.12,tics=10]{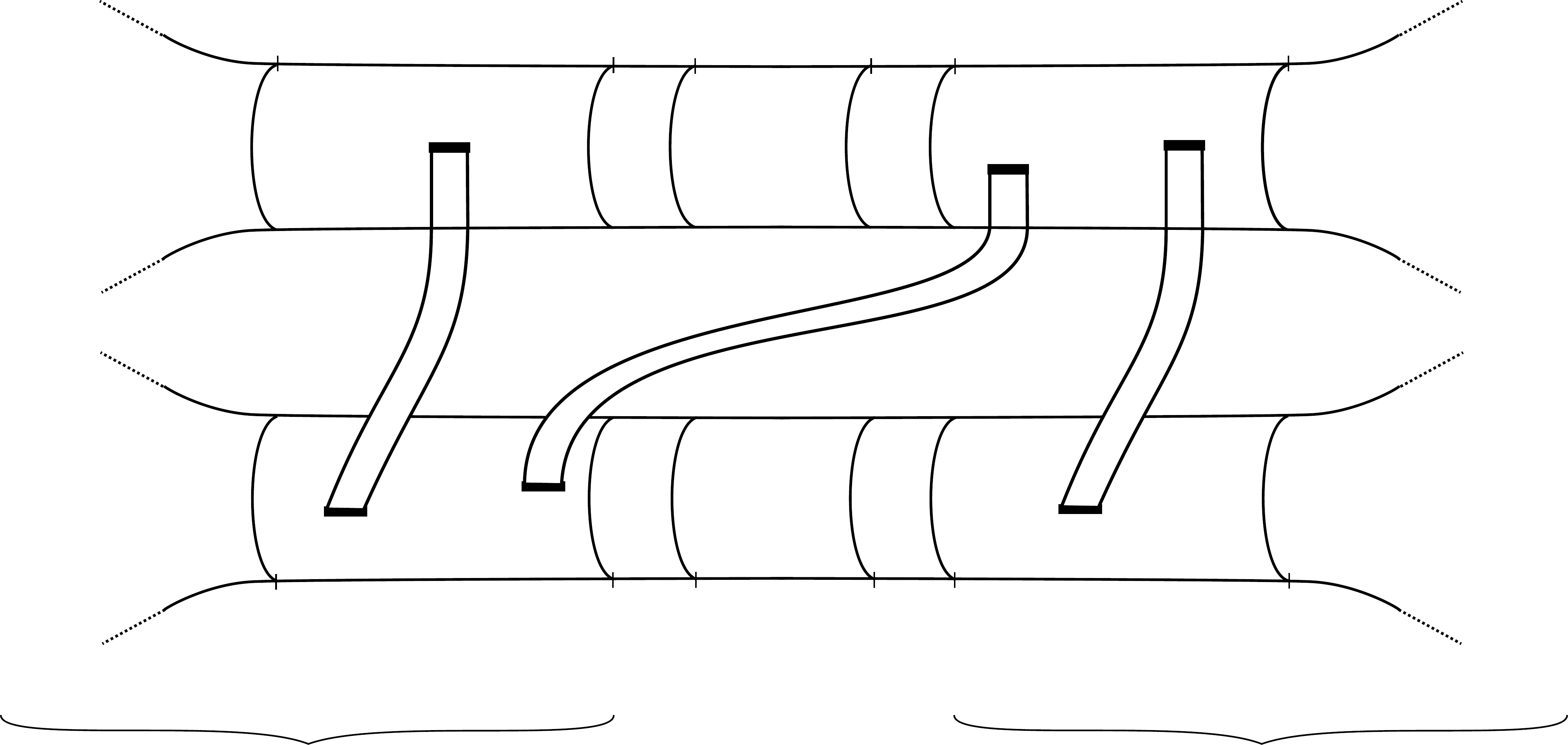}
    \put(5,10){\small{$M$}}
    \put(90,10){\small{$N$}}
    \put(10,-5){\small{Im($E_M$)}}
    \put(72,-5){\small{Im($E_N$)}}
    \put(15,5.5){\footnotesize{$-\varepsilon$}}
    \put(35,5.5){\footnotesize{$-\delta$}}
    \put(43.5,5.5){\footnotesize{$0$}}
    \put(54.5,5.5){\footnotesize{$1$}}
    \put(59,5.5){\footnotesize{$1+\delta$}}
    \put(80,5.5){\footnotesize{$1+\varepsilon$}}
    \put(17,12){\tiny{$E^2_M(m')$}}
    \put(27.5,14){\tiny{$E_M(a)$}}
    \put(65,12.5){\tiny{$E_N(n)$}}
    \put(23,40){\tiny{$E_M(m')$}}
    \put(60,38.5){\tiny{$E_N(a)$}}
    \put(70,40){\tiny{$E_N^2(n)$}}
\end{overpic}
\vspace{0.5cm}
	\caption{The ribbon corresponding to \(F\left(\iota_{m' \triangleleft a,n}^{m', a \triangleright n} \right) \colon 1_\Va \rightarrow \SkC{M\sqcup_A N}(F(m'\triangleleft a, n), F(m', a\triangleright n) )\). Here we have \(m=m'\lhd a\) and \(n'=a \rhd n\).}
 	\label{fig:Excision_Fig5}
\end{figure}

\paragraph{\(F\) is well-defined.} Recall that the Hom-objects of \(\SkC{M} \relcatprod{\SkC{A}} \SkC{N}\) are constructed by  \(\Free(\Omega((m, n), (m', n')))\), where \(\Omega((m, n), (m', n'))\) is given in Equation \eqref{eq:OmegaTambaraSkein}, and then imposing the list of relations in Section \ref{TambaraTensorSection} by taking coequalizers. We need \(F\) to respect the conditions imposed on the Tambara tensor product. For example, the first condition, which is spelled out in detail in Equation \eqref{reln:iso}, ensures that \(\iotaman\) is an isomorphism, with inverse given by \(\iotamanInv\). Pictorially, this is exactly the fact that composing the ribbon in Figure \ref{fig:Excision_Fig5} corresponding to \(F(\iotaman)\) with the ribbon corresponding to \(F(\iotamanInv\) is isotopic to the identity ribbon on \(F(m \triangleleft a, n)\).  By similar topological arguments the other relations can be checked and the functor \(F\) is indeed well-defined.

\paragraph{\(F\) is essentially surjective.} Observe that any point in \(E_M(M) \sqcup E_N(N) \subset M\sqcup_A N\) is in the image of \(F\) by construction. Secondly, if we have a decorated point \(x\) that is not in this region there exists a ribbon which moves \(x\) to the left and into the image of \(E_M\)\footnote{or similarly, to the right into the image of \(E_N\).}. This ribbon is an isomorphism. Hence we have that every decorated point in \(M\sqcup_A N\) is isomorphic to any point in the image of the functor \(F\), which gives that \(F\) is essentially surjective. 

\paragraph{\(F\) is fully faithful.} For \(F\) to be fully faithful we need to prove that \(F_{(m, n), (m', n')}\) provides isomorphisms in \(\Va\) between the Hom-objects 
\begin{align}\label{eq:F-on-hom-objects}
        F_{(m, n), (m',n')} \colon &\SkC{M} \relcatprod{\SkC{A}} \SkC{N}((m, n), (m', n')) \stackrel{\cong}{\longrightarrow}  \\ &\SkC{M\sqcup_A N}(F(m,n), F(m',n'))
\end{align}
where \((m,n)\) and \((m',n')\) are any two objects of \(\SkC{M} \relcatprod{\SkC{A}} \SkC{N}\). We divide this into two parts: first we construct the inverse, and then we prove that the inverse is indeed well-defined. 

\paragraph{Construction of inverse.} Let \(\bar{\Gamma}\) denote a ribbon diagram from \(E_M(m) \sqcup E_N(n)\) to \(E_M(m')\sqcup E_N(n')\) in \(M\sqcup_A N\). That is, we have a map
\begin{align*}
\Ca((\bar{\Gamma})) \longrightarrow \SkC{M\sqcup_A N}\left( F\left( m, n\right), F\left( m', n' \right)\right) \ \ .
\end{align*}
We are going to construct a ribbon \(\Psi\) arising in the colimit of the Hom-object of the Tambara tensor product, i.e.~with a map
\begin{align*}
    \Ca((\Psi)) \longrightarrow \SkC{M} \relcatprod{\textbf{Sk}_\Ca(A)} \SkC{N} \left( \left( m, n\right), \left( m', n' \right) \right)
\end{align*}
such that \(F\) sends \(\Psi\) to \(\Gamma\) for some \(\Gamma\) that is isotopic to \(\bar{\Gamma}\). With this, the assignment \(\bar{\Gamma} \mapsto \Psi\) induces an inverse \(F^{-1}_{(m, n), (m', n')}\).

The overall idea of the subsequent argument is the following. The ribbon \(\bar{\Gamma}\) might be non-trivial in the middle region. Meanwhile, the only strands of ribbons in the image of \(F\) that crosses the middle region are of the form \(F(\iotaman)\) and its inverse. That is, a ribbon in the image of \(F\) has no (non-trivial) coupons or twists in the middle region. Hence, we need to provide some ribbon \(\Gamma\) that, up to isotopy, captures the parts of the ribbon \(\bar{\Gamma}\) that are non-trivial across the middle region.  This is done by reverse-engineering \(\Gamma\) from \(\bar{\Gamma}\) by a sequence of isotopies explained in the steps below\footnote{To simplify notation we will call the ribbon obtained from isotoping \(\bar{\Gamma}\) for \(\Gamma\) in all the steps below, even though we are really considering several isotopic versions of the ribbon \(\Gamma\) throughout the construction. }. 

\begin{enumerate}
\setlength\itemsep{0.5em}
    \item \emph{Transverse intersections.} \par 
    By an isotopy that is fixed outside of the middle region we assume that the strands of \(\Gamma\) only intersect \(C\times \{0\} \times [0,1]\) by a finite number of transverse strands. Let \(t_i \in [0,1], i\in I\) be these different levels. In addition we ask for the isotopy to ensure that we only have decorated points and no coupons of the ribbon \(\Gamma\) intersecting the hyperplane at \(t_i\). We denote these decorated points by \(\Gamma_{t_i}\). This is illustrated for a small part of the ribbon \(\Gamma\) in Figure \ref{fig:excisionprf1}. The small blue coupons depict the decorated points \(\Gamma_{t_i}\) at each level \(t_i\).
    \item \emph{Move points of \(\Gamma_{t_i}\) to \( \im E_M \sqcup \im E_N \).} \par 
     Choose an isotopy in the \(t\)-direction that is fixed outside of the middle region such that each \(\Gamma_{t_i}\) only has points lying in \(\im E_M \sqcup \im E_N\), as illustrated in Figure \ref{fig:excisionproof2}. 
    \item \emph{Choose small \(\varepsilon_i\) for each \(t_i\).} \par  
    For every \(t_i\), choose a small \(\varepsilon_i >0\) such that, in the region \(C\times (-\delta, 1+\delta) \times [0,1]\), $\Gamma$ lies only between the hyperplanes at \(t_i\) and \(t_i + \varepsilon_i\). In addition we ask for the levels to satisfy \(t_i+\varepsilon_i < t_{i+1}\) for all \(i\), and also that \(\Gamma_{t_i}\) and \(\Gamma_{t_i+\varepsilon_i}\) only consists of decorated points and no coupons. This is illustrated in Figure \ref{fig:excisionproof3}. 
    \item \emph{Horizontally moving points of \(\Gamma_{t_i}\) and \(\Gamma_{t_i + \varepsilon_i}\).}   \par
    By an isotopy fixed in \(C\times (0,1) \times [0,1]\) we further assume that each \(\Gamma_{t_i}\) and \(\Gamma_{t_i +\varepsilon_i}\) only has points in  to \(\im E_{M} \sqcup \im \Phi_{M} \sqcup \im \Phi_N \sqcup \im E_N\). This means that we can assume \(\Gamma_{t_i} = (m\triangleleft (a\sqcup \bar{b}), (\bar{c} \sqcup d)\triangleright n)\), where only the points \(\bar{b}\) and \(\bar{c}\) intersect \(C\times \{-\delta\}\times[0,1]\) respectively \(C\times \{1+\delta\}\times[0,1]\) in the time-interval \([t_i, t_i + \varepsilon_i]\). Analogously we can write \(\Gamma_{t_i+\varepsilon_i} = (m'\triangleleft (a'\sqcup \bar{b}'), (\bar{c}' \sqcup d')\triangleright n')\). 
    \item \emph{Isotopy in \(t\)-direction for non-crossing ribbons.} \par 
    Isotope in the \(t\)-direction such that any ribbon not crossing the middle region is only the identity in the intervals \([t_i, t_i + \varepsilon_i]\). Denote by \(\Gamma_{[t_i, t_i+\varepsilon_i]}\) the part of the ribbon \(\Gamma\) in the time-interval \([t_i, t_i+\varepsilon_i]\). This step ensures that \(\Gamma_{[t_i, t_i+\varepsilon_i]} =\id_{m \triangleleft a, d \triangleright n} \sqcup \Theta_{[t_{i}, t_i + \varepsilon_i]}\), where \(\Theta_{[t_{i}, t_i + \varepsilon_i]} \colon E_M(\bar{b}) \sqcup E_N(\bar{c}) \ra E_M(\bar{b}') \sqcup E_N(\bar{c}')\) is the part of the ribbon straddling the middle region.
    \item \emph{Move straddling coupons to the side.} \par  
    The middle region is homeomorphic to \(C\times [0,1]\) (at any fixed time \(t\in [0,1]\)), and since \([0,1]\) is topologically trivial there exists a ribbon tangle \(\bar{\Theta}_{[t_{i}, t_i + \varepsilon_i]} \colon \bar{b}\triangleright \bar{c} \ra \bar{b}'\triangleright \bar{c}'\) in \(\SkC{N}\) such that \(\Theta\) is isotopic to \(\bar{\Theta}\) pre- and postcomposed by ribbons obtained from dragging \(\Theta\)  to the right into \(\im E_N\). This is illustrated in \cite[Figure 7]{Cooke19}. 
\end{enumerate}

\begin{figure}[t]
     \centering
     \begin{subfigure}[b]{0.3\textwidth}
         \centering
         \begin{overpic}[scale=.06,tics=10]{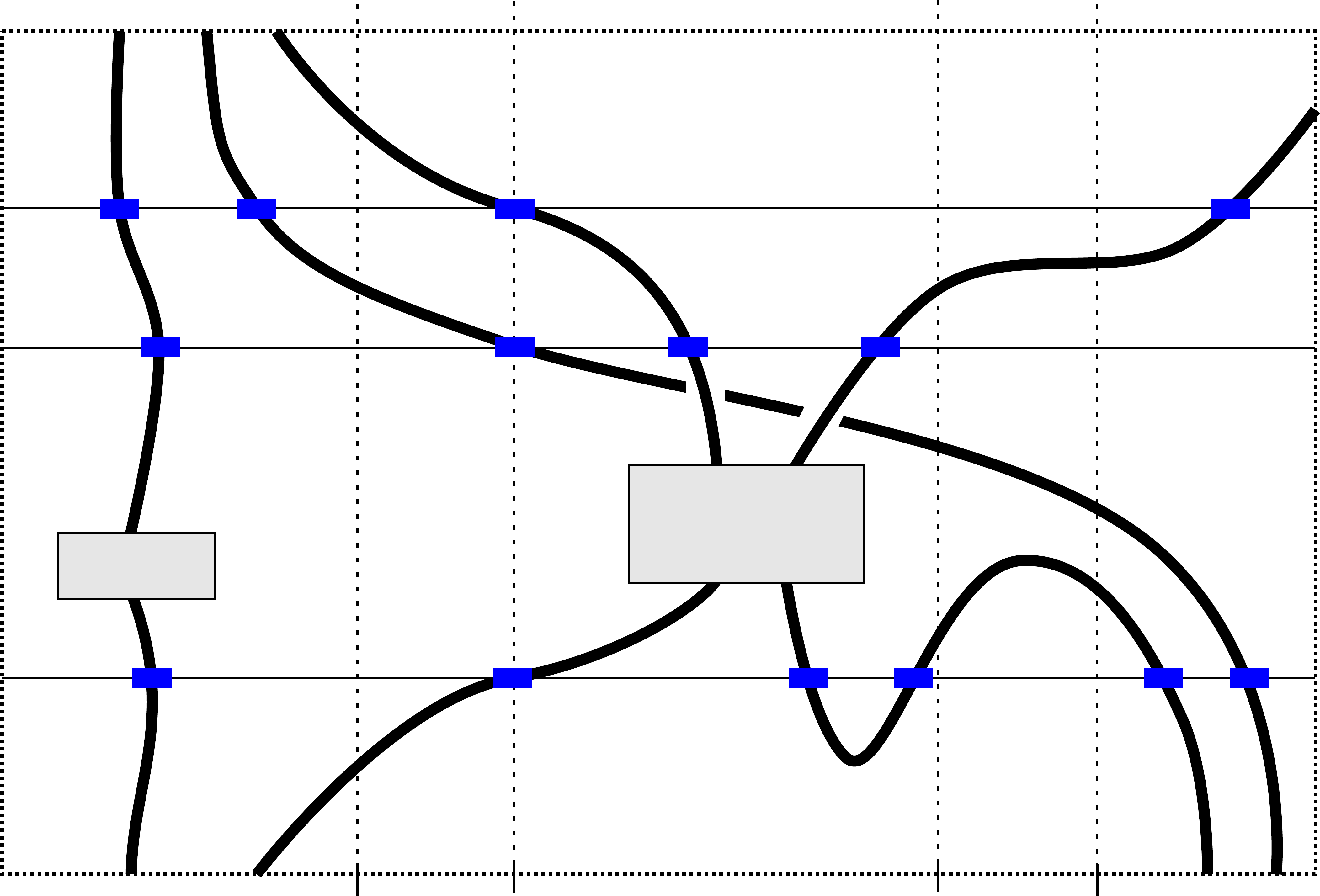}
         \put(25,-5){\scriptsize{-$\delta$}}
         \put(37,-5){\scriptsize{$0$}}
         \put(70,-5){\scriptsize{$1$}}
         \put(78,-5){\scriptsize{$1$+$\delta$}}
         \put(-7,15){\scriptsize{$t_i$}}
         \put(-12,41){\scriptsize{$t_{i+1}$}}
         \put(-12,51){\scriptsize{$t_{i+2}$}}
         \end{overpic}
         \vspace{0.2cm}
         \caption{Step 1}
         \label{fig:excisionprf1}
     \end{subfigure}
     \hfill
     \begin{subfigure}[b]{0.3\textwidth}
         \centering
          \begin{overpic}[scale=.06,tics=10]{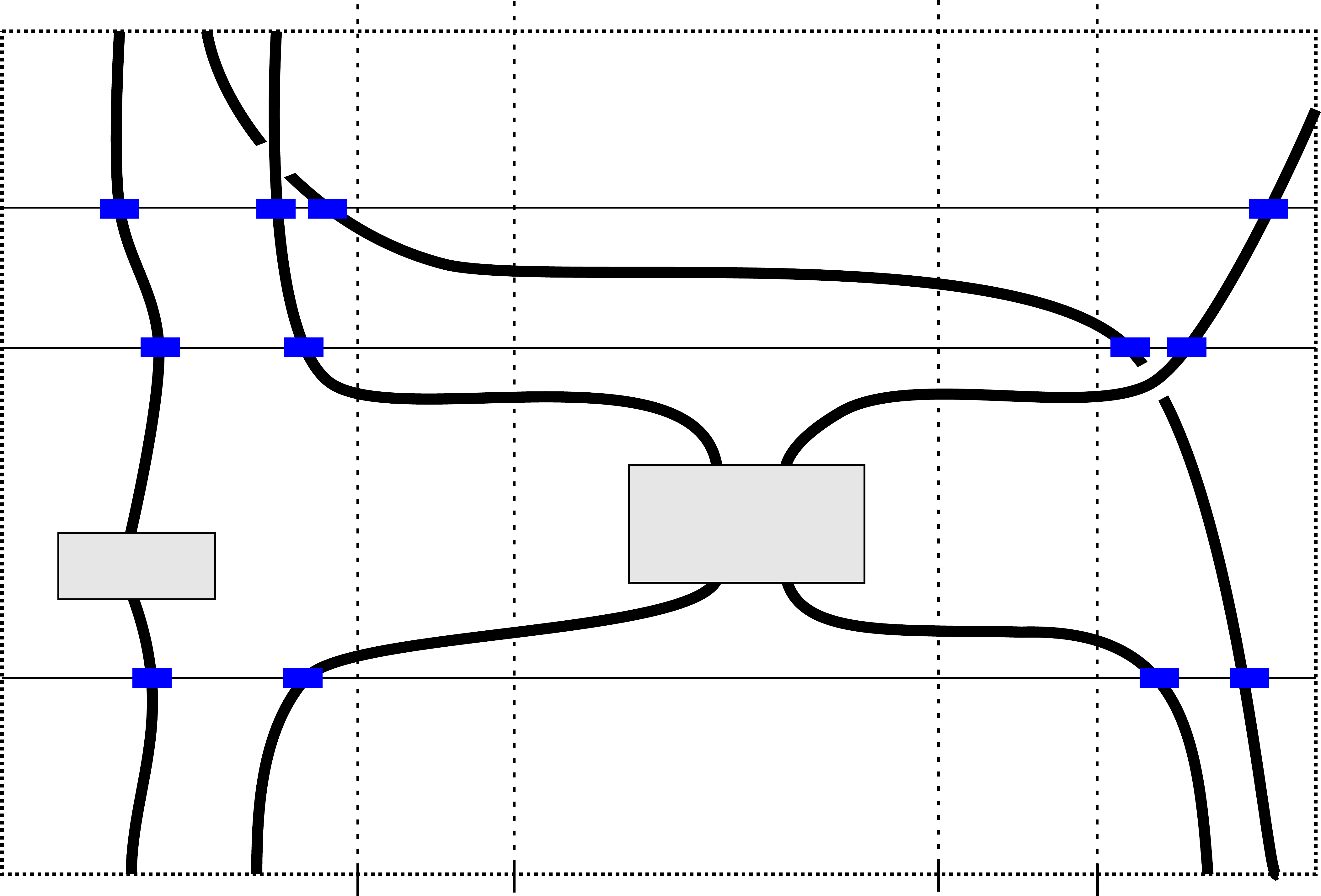}
          \put(25,-5){\scriptsize{-$\delta$}}
         \put(37,-5){\scriptsize{$0$}}
        \put(70,-5){\scriptsize{$1$}}
         \put(78,-5){\scriptsize{$1$+$\delta$}}
          \end{overpic}
          \vspace{0.2cm}
         \caption{Step 2}
         \label{fig:excisionproof2}
     \end{subfigure}
     \hfill
     \begin{subfigure}[b]{0.3\textwidth}
         \centering
          \begin{overpic}[scale=.06,tics=10]{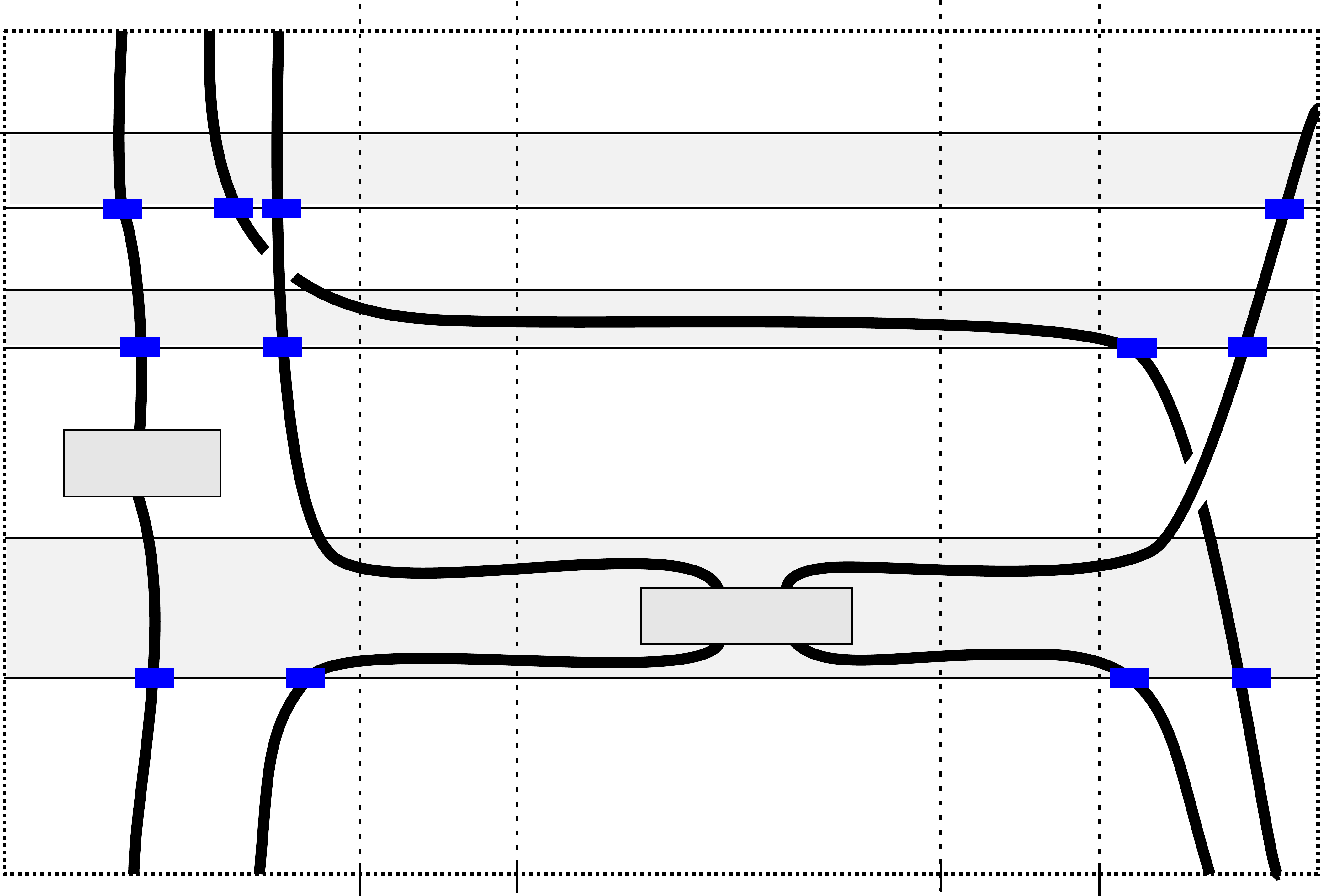}
          \put(25,-5){\scriptsize{-$\delta$}}
         \put(37,-5){\scriptsize{$0$}}
        \put(70,-5){\scriptsize{$1$}}
         \put(78,-5){\scriptsize{$1$+$\delta$}}
          \put(102,15){\scriptsize{$t_i$}}
         \put(102,26){\scriptsize{$t_i$+$\varepsilon_i$}}
          \end{overpic}
          \vspace{0.2cm}
         \caption{Step 3}
         \label{fig:excisionproof3}
     \end{subfigure}
        \caption{}
        \label{fig:excisionproof}
\end{figure}

In total we have constructed a ribbon diagram \(\Gamma\) isotopic to \(\bar{\Gamma}\) that consists of ribbons in \(E_M(M)\) and \(E_N(N)\) composed with ribbons corresponding to \(F(\iotaman)\) and \(F\left(\iotamanInv\right)\).  In other words there exists some ribbon diagram \(\Psi\) such that \(F\) sends \(\Psi\) to \(\Gamma\). The assignment \(\bar{\Gamma} \mapsto \Psi\) induces the inverse \(F_{(m,n), (m',n')}^{-1}\). 

\paragraph{The inverse is well-defined.} Here we will argue that the assignment \(\bar{\Gamma} \mapsto \Psi\) gives a well-defined inverse to \eqref{eq:F-on-hom-objects}. Since \(F(\Gamma_M,\Gamma_N)\mapsto (\Gamma_M,\Gamma_N)\) and \(F(\iotaman) \mapsto \iotaman\), this map will automatically be the inverse. What we have to check is that it is indeed well-defined, i.e.~sends equivalent ribbons to equivalent ribbons. 
Pick some open cover of \((M\sqcup_A N)\times [0,1]\) consisting of (opens homeomorphic to) boxes. 
Any equivalence of ribbon diagrams in \(\SkC{M\sqcup_A N}\) can be decomposed into equivalences which are fixed outside one of the opens of the cover, c.f.\ Remark \ref{rem:isotopy-classes-automatic-from-RT-functor}. Hence, we only need to consider the different local cases which makes ribbon diagrams equivalent separately. Morally, this directly follows from the above six steps being compatible with identifying ribbons locally. The interested reader is once again referred to \cite[Section 1.4]{Cooke19} if they want to fill in the details. 
\end{proof}

\begin{remark}
    Observe that we never actually need to consider the parenthesization of the coupons above because we are only manipulating ribbons through isotopies. That is, we are only moving the coupons around, but their parenthesization remains the same. 
\end{remark}

\part{Categorical deformation quantization}

\section{Deformation quantization of categories}
\label{sect:DefQuant-Categories}
We start this section by motivating and giving naive definitions of almost Poisson structures
on linear categories as well as their quantization. Then we give some (mostly well-known) examples of such structures in Section \ref{sect:examples-of-Pois-and-BD-cats}. Lastly, in Section \ref{sect:dunns-additivity-for-pois-and-bd-cats}, we prove that these categories afford a version of Dunn's additivity which is important for later computations.

\subsection{Motivation and definitions} \label{sec:DefQuantPullback}
We start by recalling deformation quantization of algebras in a way which motivates
our definitions of deformation quantization of categories~\cite{BFFLS,Deligne95}\cite[Section 10]{GuttRawnsley}. Let $(A_0,\cdot,  \{-,-\})$ be a Poisson algebra over $\C$. 
A \emph{deformation quantization} of $A_0$ is a topologically-free $\C[[\hbar]]$-algebra $(A_\hbar, \star)$ 
together with an algebra isomorphism $\psi \colon A_\hbar/ \hbar A_\hbar \cong  A_\hbar\otimes_{\C[[\hbar]]} \C \longrightarrow A_0 $ such that 
\begin{align}\label{Eq: DQ}
\psi\left( \tfrac{a\star b - b\star a}{\hbar}  \right) = \{ \psi(a),\psi(b) \} ,
\end{align}
where the division by $\hbar$ is well defined since $[a\star b - b\star a]$ is zero in $A_\hbar/ \hbar A_\hbar$.
Keeping only the structure modulo $\hbar^2$, i.e.\ the first order deformation, we get an \emph{almost Poisson algebra}, 
i.e.\ a flat $\mathbb C_\varepsilon$-algebra\footnote{We denote $\Ce = \C[\varepsilon]/(\varepsilon^2)$.} $A_\varepsilon$ together with an algebra isomorphism $\psi_\varepsilon \colon A_\varepsilon\otimes_{\mathbb C_\varepsilon} \mathbb C \xrightarrow{\cong} A_0$. In summary, this motivates that almost Poisson algebras and their deformation quantizations can be defined as the following pullbacks of categories
\[\begin{tikzcd}
	\mathsf{aPoisAlg} & {\mathsf{Alg}(\Ce\lmod^{\mathsf{flat}})} && \mathsf{DefQuantAlg} & {\mathsf{Alg}(\Ch\lmod^\textsf{tf})} \\
	{\mathsf{CommAlg}{(\C\lmod)}} & {\mathsf{Alg}(\C\lmod)} && {\mathsf{CommAlg}(\C\lmod)} & {\mathsf{Alg}(\C\lmod)}
	\arrow[hook, from=2-1, to=2-2]
	\arrow["{\otimes_{\Ce} \C}", from=1-2, to=2-2]
	\arrow[from=1-1, to=1-2]
	\arrow[from=1-1, to=2-1]
    \arrow[rd, "{\mbox{\fontsize{13}{15}\selectfont\( \lrcorner \)}}" description, very near start, phantom, start anchor={[xshift=-1.3ex, yshift=0.1ex]}, from=1-1, to=2-2]
	\arrow["{\otimes_{\Ch}\C}", from=1-5, to=2-5]
	\arrow[hook, from=2-4, to=2-5]
	\arrow[from=1-4, to=2-4]
	\arrow[from=1-4, to=1-5]
    \arrow[rd, "{\mbox{\fontsize{13}{15}\selectfont\( \lrcorner \)}}" description, very near start, phantom, start anchor={[xshift=-1.3ex, yshift=0.1ex]}, from=1-4, to=2-5]
\end{tikzcd}\]
The right vertical arrows in both pullbacks come from the functors $\otimes_{\Ch}\C\colon \Ch\lmod \to \C\lmod$ and $\otimes_{\Ce}\C\colon \Ce\lmod\to\C\lmod$. These are symmetric monoidal, due to being induction functors\footnote{If $\varphi \colon R \to S$ is a morphism of commutative rings, then the induction functor $\otimes_R S \colon R\lmod \to S\lmod$ is symmetric monoidal via the map $(M\otimes_R S) \otimes_S (N\otimes_R S) \xrightarrow{\cong} ( M\otimes_R N )\otimes_R S $ given by $(m, s, n, s') \mapsto m \otimes n \otimes ss'$.} assigned to the  algebra morphism $\Ch\twoheadrightarrow \C$ and $\Ce\twoheadrightarrow\C$, respectively.

\begin{remark}[Poisson algebras as almost Poisson algebras]
    Any antisymmetric biderivation $\{-,-\}$ on a commutative algebra $A_0$ defines an almost Poisson algebra $A_\varepsilon = A_0 \otimes \mathbb C_\varepsilon$ via
    \begin{equation} 
        a \star b = a \cdot b + \frac \varepsilon2 \{a, b\} 
    \end{equation}
    and if this deformation can be extended to $\varepsilon^2$, then the Jacobi identity for $\{-,-\}$ holds. 
    
    In general, not all almost Poisson algebras come from antisymmetric biderivations. Choosing a splitting $A_\varepsilon \cong A_\varepsilon/(\varepsilon A_\varepsilon)\otimes_\C \Ce$, we get 
    \begin{equation} 
        a*b = a\cdot b + \varepsilon\, p(a, b), 
    \end{equation}
    where $p$ is a Hochschild 2-cocycle for $A_0 := A_\varepsilon/(\varepsilon A_\varepsilon)$, and different splittings change $p$ by a coboundary. Thus, almost Poisson algebras $A_\varepsilon$ naturally induce elements of the second Hochschild cohomology of $A_0$. The antisymmetrization of $p$ is a biderivation not depending on the choice of splitting, while the symmetrization of $p$ is a symmetric Hochschild 2-cocycle. Thus, the almost Poisson algebra $A_\varepsilon$ comes from an antisymmetric biderivation if e.g.\ the second symmetric Hochschild cohomology vanishes (such algebras were investigated by Knudson in \cite{Knudson}).  
    
    A notable case when this holds is in the setting of the HKR theorem \cite{HKR}, where, over a field, the assumption of smoothness implies $H^2_\text{sym}(A_0, A_0)=0$ \cite[Section 9.3 \& 9.4]{Weibel}. An important example is $A_0=C^\infty(M)$ with $p$ a bidifferential operator (see e.g.\ \cite[Remark~2.17]{GuttRawnsley}). In this case our notion of an almost Poisson algebra agrees with the usual one \cite[Section~3.3]{daSilvaWeinstein}.
\end{remark}

Generalizing to linear categories, we can define their deformations in a similar fashion. Motivated by 
factorization homology, we will consider deformations of $\E_\infty$-algebras into not only $\E_1$-algebras,
but to general $\E_i$-algebras. Heuristically, these can be defined as the following ``pullbacks'':

\begin{equation} \label{eq:DefnPCatBDnCat}
\begin{tikzcd}
	\PCat{i} & {\E_i(\Ce\mhyphen\mathsf{Cat})} && \BDCat{i} & {\E_i(\mathbb C[[\hbar]]\mhyphen \Cat}) \\
	{\E_\infty(\C \mhyphen\mathsf{Cat})} & {\E_i(\mathbb C\mhyphen\mathsf{Cat})} && {\E_\infty(\C \mhyphen \mathsf{Cat})} & {\E_i(\mathbb C \mhyphen\mathsf{Cat})}
	\arrow[hook, from=2-1, to=2-2]
	\arrow["{\otimes_{\Ce} \C}", from=1-2, to=2-2]
	\arrow[from=1-1, to=1-2]
	\arrow[from=1-1, to=2-1]
    \arrow[rd, "{\mbox{\fontsize{13}{15}\selectfont\( \lrcorner \)}}" description, very near start, phantom, start anchor={[xshift=-1.3ex, yshift=0.1ex]}, from=1-1, to=2-2]
	\arrow["{\otimes_{\Ch}\C}", from=1-5, to=2-5]
	\arrow[hook, from=2-4, to=2-5]
	\arrow[from=1-4, to=2-4]
	\arrow[from=1-4, to=1-5]
    \arrow[rd, "{\mbox{\fontsize{13}{15}\selectfont\( \lrcorner \)}}" description, very near start, phantom, start anchor={[xshift=-1.3ex, yshift=0.1ex]}, from=1-4, to=2-5]
\end{tikzcd}
\end{equation}
\begin{remark} \label{rem:PullbackWhere?}
We would like to think of the above pullbacks as being taken in a 3-category of symmetric monoidal bicategories. However, to the best of our knowledge, such pullbacks were not yet studied in the literature, neither for symmetric monoidal bicategories nor for bicategories. We do not develop such a theory here, and instead we define the above pullback bicategories explicitly, relaxing the usual pullback of categories -- see Appendix \ref{appendix:pullback} and in particular Definition \ref{def:pullback} for details. A similar construction appears for example in \cite[Section 2.3]{CHR2014HomotopyPullback}.
\end{remark}

\begin{definition}\label{def:pullback_Poisi}
	An \emph{$\Pois_i$-category} $\Ca$ is a $\Ce$-linear $\E_i$-category $\Ca_\varepsilon$, a $\C$-linear symmetric monoidal category $\Ca_\text{SMC}$ and an $\E_i$-equivalence
	$\Phi\colon (\Ca_\varepsilon)_0  \to \Ca_\text{SMC}$, where we denote $(-)_0:= \otimes_\Ce \C$ for brevity. In this case we also say that \(\Ca_\varepsilon\) is an \emph{almost Poisson-$i$ category}. 
 
	An \emph{$\Pois_i$-functor} $F\colon \Ca \to \Da$ is a $\Ce$-linear functor of $\E_i$-categories $F_\varepsilon \colon \Ca_\varepsilon \to \Da_\varepsilon$, 
	a $\C$-linear symmetric monoidal functor $F_\text{SMC}\colon \Ca_\text{SMC} \to \Da_\text{SMC}$ and a natural isomorphism  $\phi$ of $\E_i$-functors
\begin{equation}
\begin{tikzcd}
	{(\Ca_\varepsilon)_0} & {\Ca_\text{SMC}} \\
	{(\Da_\varepsilon)_0} & {\Da_\text{SMC}}
	\arrow["{(F_\varepsilon)_0}"', from=1-1, to=2-1]
	\arrow["{\Phi^\Ca}", from=1-1, to=1-2]
	\arrow["{F_\text{SMC}}", from=1-2, to=2-2]
	\arrow["{\Phi^\Da}"', from=2-1, to=2-2]
	\arrow["\phi"{description}, shorten <=4pt, shorten >=4pt, Rightarrow, from=2-1, to=1-2]
\end{tikzcd} \ . 
\end{equation}

    An \emph{$\Pois_i$-natural transformation} $\alpha \colon F \Rightarrow G$ is an $\E_i$-natural transformation $\alpha_\varepsilon \colon F_\varepsilon \Rightarrow G_\varepsilon$ and  a monoidal natural transformation $\alpha_\text{SMC} \colon F_\text{SMC} \Rightarrow G_\text{SMC}$ such that the two composite natural transformations below are equal
\begin{equation}\label{eq:pullback2cell}\begin{tikzcd}[sep=large]
	{(\Ca_\varepsilon)_0} & {\Ca_\text{SMC}} \\
	{(\Da_\varepsilon)_0} & {\Da_\text{SMC}}
	\arrow["{\Phi^\Ca}", from=1-1, to=1-2]
	\arrow["{\Phi^\Da}"', from=2-1, to=2-2]
	\arrow[""{name=0, anchor=center, inner sep=0}, "{(G_\varepsilon)_0}"{pos=0.4}, curve={height=-18pt}, from=1-1, to=2-1]
	\arrow["{G_\text{SMC}}", curve={height=-18pt}, from=1-2, to=2-2]
	\arrow[""{name=1, anchor=center, inner sep=0}, "{(F_\varepsilon)_0}"'{pos=0.4}, curve={height=18pt}, from=1-1, to=2-1]
	\arrow["{\phi^G}"{description}, shift right=3, shorten <=7pt, Rightarrow, from=2-1, to=1-2]
	\arrow["{(\alpha_\varepsilon)_0}", shorten <=7pt, shorten >=7pt, Rightarrow, from=1, to=0]
\end{tikzcd} \quad = 
\quad
\begin{tikzcd}[sep=large]
	{(\Ca_\varepsilon)_0} & {\Ca_\text{SMC}} \\
	{(\Da_\varepsilon)_0} & {\Da_\text{SMC}}
	\arrow["{\Phi^\Ca}", from=1-1, to=1-2]
	\arrow["{\Phi^\Da}"', from=2-1, to=2-2]
	\arrow[""{name=0, anchor=center, inner sep=0}, "{G_\text{SMC}}"{pos=0.6}, curve={height=-18pt}, from=1-2, to=2-2]
	\arrow["{(F_\varepsilon)_0}"', curve={height=18pt}, from=1-1, to=2-1]
	\arrow[""{name=1, anchor=center, inner sep=0}, "{F_\text{SMC}}"'{pos=0.6}, curve={height=18pt}, from=1-2, to=2-2]
	\arrow["{\phi^F}"{description, pos=0.4}, shift left=3, shorten >=7pt, Rightarrow, from=2-1, to=1-2]
	\arrow["{\alpha_\text{SMC}}", shorten <=7pt, shorten >=7pt, Rightarrow, from=1, to=0]
\end{tikzcd} \ .
\end{equation}
Furthermore, \(\Pois_i\)-categories,  \(\Pois_i\)-functors and \(\Pois_i\)-natural transformations 
assemble into a 2-category denoted \(\PCat{i}\). 
\end{definition}

\begin{definition}
    We define $\BD_i$-categories, $\BD_i$-functors and $\BD_i$-natural transformations by replacing \(\Ce\) with \(\Ch\) throughout Definition \ref{def:pullback_Poisi}. These assemble into a 2-category denoted $\BDCat{i}$. 
\end{definition}

 There is an obvious functor $\otimes_{\C[[\hbar]]} \C_\varepsilon \colon \BDCat{i} \to \PCat{i}$ coming from the pair of ring morphism $\Ch \twoheadrightarrow \Ce \twoheadrightarrow \C$, i.e.\ from the diagram: 
\begin{equation}
    \begin{tikzcd}
	& {\E_i(\Ch\mhyphen\mathsf{Cat})} \\
	{\E_\infty(\C\mhyphen\mathsf{Cat})} & {\E_i(\C\mhyphen\mathsf{Cat})} & {\E_i(\Ce\mhyphen\mathsf{Cat})} \\
	& {\E_\infty(\C\mhyphen\mathsf{Cat})} & {\E_i(\C\mhyphen\mathsf{Cat})}
	\arrow[hook, from=2-1, to=2-2]
	\arrow["{\otimes_\Ch \C}"', from=1-2, to=2-2]
	\arrow["{=}"{description}, from=2-2, to=3-3]
	\arrow["{\otimes_\Ch \Ce}", from=1-2, to=2-3]
	\arrow["{\otimes_\Ce \C}", from=2-3, to=3-3]
	\arrow["{=}"{description}, from=2-1, to=3-2]
	\arrow[hook, from=3-2, to=3-3]
\end{tikzcd} \ . 
\end{equation} 
Explicitly, the functor sends $(\Ca_\hbar, \Ca_\text{SMC}, \Phi\colon \Ca_\hbar\otimes_\Ch \C \to \Ca_\text{SMC})\in\BDCat{i}$ to \begin{equation}\label{eq:pullbackofPoissFunc}
    (\Ca_\hbar\otimes_\Ch \Ce, \Ca_\text{SMC}, \Phi' \colon \Ca_\hbar\otimes_\Ch \Ce\otimes_\Ce \C \cong \Ca_\hbar\otimes_\Ch \C 
        \xrightarrow{\Phi} \Ca_\text{SMC})\in\PCat{i} \ , 
\end{equation}
and similarly for 1-cells and 2-cells.

\begin{remark}
    One should be able to give more sophisticated but harder to work with definitions using the $\operatorname{P}_n$ and $\BD_n$ operad. However, we do not pursue such a definition here since the above (more naive) approach is sufficient for our purposes. 
\end{remark}

\begin{definition}\label{Def: qunatisation Cat}
A \emph{quantization} of an almost Poisson category $\Ca_\varepsilon$ is an element
of the (homotopy) fiber of $-\otimes_{\C[[\hbar]]} \C_{\varepsilon}$ at $\Ca_\varepsilon$. 
\end{definition}

\begin{remark}
For us it is important that factorization homology commutes with taking the classical limit. This requires us to not only work with flat or topologically-free modules but allow for more general modules. Being a flat deformation hence becomes a property which has to be checked independently for examples of interest. Even if the input is flat, the output of factorization homology might not be in general -- see Example \ref{Ex: torsion}.
\end{remark}

\begin{lemma}\label{lma:tensorprod_aPn_BDn}
The naive tensor product of $\C[[\hbar]]$-, and $\mathbb{C}_\varepsilon$-enriched categories equips the bicategories $\BD_n\mhyphen\Cat$ and $\Pois_n\mhyphen\Cat$ with a symmetric monoidal structure, and the functor $\otimes_{\C[[\hbar]]} \C_\varepsilon \colon \BDCat{i} \to \PCat{i}$ is a symmetric monoidal homomorphism. 
\end{lemma}
The above lemma follows from more general statements in Appendix \ref{appendix:pullback} about pullbacks, see Proposition \ref{prop:SMBCpullback} and Proposition \ref{prop:SMBCpullbackmap}.

\begin{remark}
For a $\BD_0$-category $\mathcal{C}$ we define the algebra $\mathcal{O}^{\hbar}_{\Ca}\coloneqq \End_{\Ca}(\star)$ where $\star$ is the pointing of $\Ca$. If $\mathcal{O}^{\hbar}_{\Ca}$ is topologically-free this defines a deformation quantization of algebras as defined at the beginning of this section. 
\end{remark}

\subsection{Examples of \texorpdfstring{$\Pois_i$}{Pois-i}- and \texorpdfstring{$\BD_i$}{BD-i}-categories}
\label{sect:examples-of-Pois-and-BD-cats}
We now make the definitions from the previous section more concrete by providing a number of examples. 
First out are some well-known examples of $\Pois_0$-categories. 

\begin{example}
Consider a linear category $*// A$ with one object $*$ and an algebra $A$ as endomorphisms. From an Eckmann-Hilton argument it follows that this category 
admits a (unique) symmetric monoidal structure if and only if $A$ is commutative. The monoidal structure is the multiplication in $A$. Extending $*// A$ to an $\Pois_0$-category corresponds to specifying a $\Ce$-algebra $A_\varepsilon$ and an isomorphism $A_\varepsilon\otimes_{\C_\varepsilon} \C \cong A$. We conclude that $\Pois_0$-structures on $*// A$ are exactly almost Poisson structure on $A$. Furthermore, a quantization of the $\Pois_0$-structure as in Definition~\ref{Def: qunatisation Cat} corresponds to a quantization of the almost Poisson algebra.  
\end{example}
\begin{example}
Let $\cat{C}$ be a symmetric monoidal $\C$-linear category. We consider compatible $\Pois_0$-structures on\footnote{Here, $\Ca[\varepsilon]$ has the same objects as $\Ca$ and the morphism spaces are obtained via extension of scalars:
\begin{equation}
    \Map_{\Ca[\varepsilon]}(X,Y) = \Map_\Ca(X,Y)\otimes_\C\Ce  \ \ .
\end{equation}}
$\cat{C}[\varepsilon]$ pointed by the monoidal unit in $\cat{C}$. They are described by deformations of the composition 
\begin{align}
\circ_\varepsilon \colon \Hom_{\cat{C}}(B,C)[\varepsilon] \otimes_{\C[\varepsilon]} \Hom_{\cat{C}}(A,B)[\varepsilon] &\longrightarrow \Hom_{\cat{C}}(A,C)[\varepsilon] \\ 
g\otimes f & \longmapsto g\circ_{\cat{C}}f + \varepsilon\cdot \circ_1(f,g)
\end{align}
where $\circ_1$ is a linear map $\Hom_{\cat{C}}(B,C)\otimes \Hom_{\cat{C}}(A,B)\to  \Hom_{\cat{C}}(A,C)$ 
satisfying some conditions ensuring that $\circ_\varepsilon$ defines a composition. To describe these condition
more explicitly recall that the Hochschild cochain complex~\cite{LVdB2} is formed by the vector spaces
\begin{align}
HC^n(\cat{C}) \coloneqq \prod_{a_0,\dots , a_n\in \cat{C}} \Hom_\C (\Hom_{\cat{C}}(a_{n-1},a_n)\otimes \dots \otimes \Hom_{\cat{C}}(a_{0},a_1), \Hom_{\cat{C}}(a_{0},a_n))
\end{align}
with the differential $HC^n(\cat{C}) \to HC^{n+1}(\cat{C}) $ induced by the composition in $\cat{C}$. 
The collection of all the maps $\circ_1$ defines an element of $HC^2(\cat{C})$. It defines
a deformation of the categorical structure if and only if it is a cycle. It is easy to see that two $\Pois_0$-structures are equivalent if the corresponding elements of $HC^2(\cat{C})$ lie in the same cohomology class. 
In total we have seen that $\Pois_0$-structures on $\cat{C}[\varepsilon]$ are classified by the second Hochschild cohomology group $HH^2(\cat{C})$. 

If $\cat{C}$ is the category of \emph{injective} modules over a commutative algebra $A$ the Hochschild 
cohomology of $A$ and $\cat{C}$ agree~\cite{LVdB2}. If furthermore $A$ is smooth then the HKR theorem~\cite{HKR} implies
that $H^2(A)$ agrees with the vector space of bivector fields on the algebraic variety $\operatorname{Spec}A$.
These also classify almost Poisson structures on $A$. 
\end{example}

\begin{example}[Poisson-Hopf algebras]
Recall that a Poisson-Hopf algebra is a commutative Hopf algebra $(H, m, u, \Delta, \eta)$ together with a Poisson bracket $p \colon H \otimes H \to H$ such that the coproduct is a morphism of Poisson algebras. This condition is equivalent to $(H_\varepsilon, m_\varepsilon, u, \Delta, \eta)$ being a $\Ce$-linear Hopf algebra, where $m_\varepsilon := m + \varepsilon p$. The category of comodules of $H_\varepsilon$ is then an $\Pois_1$-category, with the reduction at $\varepsilon=0$ being the symmetric monoidal category of comodules of $H$. The case of coPoisson-Hopf algebras and their modules is obtained by considering the previous example in $\Vect^\text{op}$.
\end{example}

\begin{example}
    The deformations of monoidal categories and functors classified by the Davidov-Yetter cohomology \cite{Davidov, CY, FGS} give a specific class of examples in $\Pois_1$. Namely,  2-cocycles in the DY complex of a symmetric monoidal functor of $F\colon \Ca \to \Da$ give $\Pois_1$ functor structures on the trivial $\Ce$-linear extension of $F$. Similarly, $3$-cocycles for the identity functor on $\Ca$ are given by deformations of the associator on a symmetric monoidal category $\Ca$, i.e. they give deformations of $\Ca$ into a $\Pois_1$-category.
\end{example}

We now turn to some important examples of $\Pois_2$-categories. 

\begin{example}[Infinitesimally braided categories]\label{ex:iBMC}
Following {\cite[Section~4]{Cartier1993}}, given a $\mathbb K$-linear monoidal category $(\Ca, \otimes)$ with symmetry $\sigma$, an infinitesimal braiding on $\Ca$ is a natural transformation
\begin{equation}
    t_{X,Y} \colon X \otimes Y \to X \otimes Y, \quad \quad X,Y \in \Ca
\end{equation}
such that $\sigma_{X,Y}t_{X,Y}=t_{Y,X}\sigma_{X,Y}$ and such that $\beta = \sigma \circ (1 + \varepsilon t)$ is a braiding in the $\Ce$-linear category $\Ca[\varepsilon]$. 
An important example for us will be  $\Ca = U(\mathfrak{g})\lmod$ for some Lie algebra $\mathfrak{g}$. Let $t \in (\mathsf{Sym}^2\mathfrak{g})^\mathfrak{g}$ be a symmetric $\mathfrak{g}$-invariant tensor. Then, $U(\mathfrak{g})\lmod$ is infinitesimally braided via
\begin{equation}
    t_{X,Y}(x \otimes y) = \sum_{i,j \in I}t^{ij} e_i \triangleright x \otimes e_j \triangleright y, \quad \quad x \in X, y \in Y 
\end{equation}
in a basis $(e_i)_{i \in I}$ of $\mathfrak{g}$ where $t = \sum_{ij} t^{ij} e_i \otimes e_j$.

An infinitesimal braiding should be understood as the categorical analogue to a Poisson structure on a commutative algebra. A quantization of the infinitesimal braided monoidal category $(U(\mathfrak{g})\lmod,t)$ is given by the Drinfeld category $U(\mathfrak{g})\lmod^\Phi[[\hbar]]$.

Infinitesimally braided categories were also called ``infinitesimal symmetric categories'' by Kassel \cite[Sec.~XX.4]{Kassel} and ``Cartier categories'' by Heckenberger and Vendramin \cite{HV}. Deformations by $t$ which is not necessarily symmetric were studied by \cite{ABSV} as ``(symmetric) pre-Cartier categories'', but note that antisymmetric $t_{X, Y}$ results in the deformed braiding $\sigma \circ (1 + \varepsilon t)$ being symmetric.
\end{example}

\begin{example}[Infinitesimal Drinfeld center]
A special case of the above construction can be obtained using an infinitesimal version of the Drinfeld center/Drinfeld double as described in \cite[Section~7.2]{PSVQv2}. Given a Poisson-Hopf algebra $H$, a dimodule for $H$ is a left $H$-comodule $X$ and a morphism $\rho\colon X\otimes H \to X$ such that
\begin{equation} 
    \rho(\rho(x, h), h') - \rho(\rho(x, h'), h) = \rho(x, p(h, h')), \quad \rho(x, hh') = \rho(x, h)\eta(h') + \rho(x, h') \eta(h) 
\end{equation}
and such that
\begin{equation} 
    x_{(0)}h_{(1)}\otimes \rho(x_{(1)}, h_{(2)}) - h_{(2)}\rho(x, h_{(1)})_{(0)} \otimes \rho(x, h_{(1)})_{(0)} = p(h, x_{(0)})x_{(1)} \ ,  
\end{equation}
for all $x\in X$, $h, h'\in H$. Here we use the Sweedler notation for the coaction $x\mapsto x_{(0)}\otimes x_{(1)} \in H\otimes X$ and $\Delta h = h_{(1)} \otimes h_{(2)}$. The category of $H$-dimodules and comodule morphisms is then infinitesimally braided: $\rho(x\otimes y, h) = \rho(x, h)\otimes y + x\otimes \rho(y, h)$ defines a dimodule structure on $X\otimes Y$, and the infinitesimal braiding is given by
\begin{equation} 
    t_{X, Y} = r_{X, Y} + \sigma_{Y, X} \circ r_{Y, X} \circ \sigma_{X, Y}  
\end{equation}
where $\sigma$ is the symmetry and $r_{X, Y}\colon X\otimes Y \to X\otimes Y$ is defined as
\begin{equation} 
    r_{X,Y}(x\otimes y) = - \rho(x, y_{(0)}) \otimes y_{(1)} \ . 
\end{equation}
The Hopf algebra $H$ is an $H$-dimodule with coaction given by $\Delta$ and $\rho(h, h') = p(h, h'_{(1)}) h'_{(2)}$.
\end{example}

\begin{example}\label{Ex: P_2 from r}
Let $U_\hbar(\mathfrak g)\lmod$ be the category of representations of the quantum group $U_\hbar(\mathfrak g)$ which are topologically free as $\Ch$-modules -- see Section \ref{Drinfeld category} for details. This forms a $\BD_2$-category, with reduction modulo $\hbar$ given by $U(\mathfrak g)\lmod$. The equivalence $U_\hbar(\mathfrak g)\lmod\otimes_\Ch \C \cong U(\mathfrak g)\lmod$ follows from the fact that any representation of $U_\hbar(\mathfrak g)$ on $V[[\hbar]]$ is equivalent to a trivial extension of an action of $U(\mathfrak g)$ on $V$, see \cite[Theorem~2.6.8]{JacobsThesis}. This theorem also implies\footnote{In fact, this holds in greater generality: $\hbar$-linear $A$-module morphisms $V \to W$ form a topologically free $\Ch$-module if the algebra $A$ and the modules $V, W$ are topologically free as $\Ch$-modules.} that the Hom-spaces in $U_\hbar(\mathfrak g)$ are topologically free $\Ch$-modules, i.e.\ the triple $(U_\hbar(\mathfrak g)\lmod, U(\mathfrak g)\lmod, \cong)$ is a $\BD_2$-category with topologically free Hom-spaces.

This $\BD_2$-category is a deformation quantization of the $\Pois_2$-category given as follows. Let $U_\varepsilon(\mathfrak{g})$ be the semi-classical limit of the quantum group $U_\hbar(\mathfrak{g})$, meaning that $U_\varepsilon(\mathfrak{g})$ is a Hopf algebra which is isomorphic to $U(\mathfrak{g})[\varepsilon]/(\varepsilon^2)$ as a $\mathbb{C}_\varepsilon = \mathbb{C}[\varepsilon]/(\varepsilon^2)$-module with coproduct $\Delta_\varepsilon$ satisfying 
\begin{equation}\label{eq:almostCocom}
    \Delta_\varepsilon^{\op}(-) = (1 - \varepsilon r) \Delta_\varepsilon(-) (1 + \varepsilon r) \ \ ,
\end{equation}
where $r$ is the standard classical r-matrix for the simple Lie algebra $\mathfrak{g}$. Let $U_\varepsilon(\mathfrak{g})\lmod$ be the category of $U_\varepsilon(\mathfrak{g})$-modules which are free as $\mathbb{C}_\varepsilon$-modules. Then, $U_\varepsilon(\mathfrak{g})\lmod$ is an $\Pois_2$-category with braiding given by 
\begin{equation}
    \beta_{X,Y} = \tau \circ (1 + \varepsilon r~\triangleright) \colon X \otimes Y \to Y \otimes X \ \ ,
\end{equation}
where $\tau(x \otimes y) = y \otimes x$ is the flip map. The above is a map of $U_\varepsilon(\mathfrak{g})$-modules due to Equation \eqref{eq:almostCocom}. 
\end{example}

\subsection{Additivity} \label{sect:dunns-additivity-for-pois-and-bd-cats}
In this section we state a version of Dunn's additivity for $\Pois_n$- and $\BD_n$-categories:

\begin{theorem}\label{prop:additivity}
For $n = 0,1,2, \infty$ the following holds:
\begin{itemize}
    \item $\E_n(\Pois_0\mhyphen\Cat) \cong \Pois_n\mhyphen\Cat$, 
    \item $\E_n(\BD_0\mhyphen\Cat) \cong \BD_n\mhyphen\Cat$. 
\end{itemize}
\end{theorem}
We first explain the heuristics behind the above additivity-result. Recall from Equation \eqref{eq:DefnPCatBDnCat} and Remark \ref{rem:PullbackWhere?} that \(\Pois_i\mhyphen \Cat\) and \(\BD_i\mhyphen\Cat\) should be thought of as pullbacks of certain (symmetric monoidal) bicategories. If this claim was established, then Proposition \ref{prop:additivity} would follow from the universal property of pullbacks. 
In some more detail, an \(\E_n\)-algebra \(\Ea\) with values in a bicategory \(\Aaa\) is a symmetric monoidal functor \(\Ea \colon \mathrm{Disk}_n \rightarrow \Aaa\). If the bicategory \(\Aaa\) is given by a pullback, i.e.\
\begin{center}
\begin{tikzcd} 
    \mathrm{Disk}_n \arrow[rd, "\Ea"] \arrow[bend left=30, dashed]{rrd} \arrow[bend right=30, dashed]{rdd}
    \\ &\Aaa \arrow[r] \arrow[d]  \ar[rd, "{\mbox{\fontsize{13}{15}\selectfont\( \lrcorner \)}}" description, very near start, phantom, start anchor={[xshift=-1.3ex, yshift=0.1ex]} ]  &\Caa_1  \arrow[d]
    \\ &\Caa_2 \arrow[r] & \Daa
\end{tikzcd} \ , 
\end{center}
it follows that providing a map into \(\Aaa\) is equivalent to providing the two dashed maps into \(\Caa_1 \) and \(\Caa_2\) such that the diagram commutes (up to natural isomorphism). That is, providing an \(\E_n\)-algebra in \(\Aaa\) is equivalent to providing an \(\E_n\)-algebra in \(\Caa_1 \) and \(\Caa_2\) such that they agree as \(\E_n\)-algebras in \(\Daa\). Thus, \(\E_n(\Aaa)\) is given exactly by the pullback-diagram of \(\E_n\)-algebras in \(\Caa_1 \), \(\Caa_2\) and \(\Daa\), i.e. taking pullbacks commutes with applying $\E_n$.

Specializing the above to e.g.~\(\Aaa = \Pois_0\mhyphen\Cat\) we get
    \begin{center}
    \begin{tikzcd}
        \E_n\left(\Pois_0\mhyphen \Cat\right)  \ar[rd, "{\mbox{\fontsize{15}{20}\selectfont\( \lrcorner \)}}" description, very near start, phantom, start anchor={[xshift=-1.3ex, yshift=0.1ex]} ]
& \E_n\left(\E_0\left(\mathbb C_\varepsilon\mhyphen\Cat \right) \right)  \\
	\E_n\left(\E_\infty\left(\mathbb{C} \mhyphen \Cat\right)\right) & \E_n\left(\E_0(\mathbb{C} \mhyphen\Cat) \right)
	\arrow[hook, from=2-1, to=2-2]
	\arrow["{\otimes_{\mathbb C_\varepsilon} \mathbb C}", from=1-2, to=2-2]
	\arrow[from=1-1, to=1-2]
	\arrow[from=1-1, to=2-1]
    \end{tikzcd}.
    \end{center}
Dunn's additivity tells us that \(\E_{n+m}(\Va^{\catprod}) \simeq \E_n\left( \E_m \left(\Va^{\catprod} \right) \right) \), for \(\Va^{\catprod}\) any symmetric monoidal \infy-category. In particular, we have equivalences 
    \begin{align*}
        \E_n(\E_0(\mathbb{C}_\varepsilon\mhyphen\Cat)) \simeq \E_{n}(\mathbb{C}_\varepsilon \mhyphen \Cat),\qquad  \E_n(\E_\infty(\mathbb{C}\mhyphen\Cat)) \simeq \E_{\infty}(\mathbb{C} \mhyphen \Cat)\quad \mathrm{and} \quad \E_n(\E_0(\mathbb{C}\mhyphen \Cat)) \simeq \E_n(\mathbb{C}\mhyphen \Cat)).
        \end{align*}
Hence the commuting diagram above becomes exactly that of the (heuristic) definition of \(\Pois_n\mhyphen\Cat\), providing the wanted equivalence. The same argument works ad verbum for \(\BD_n\mhyphen\Cat\). 

Since we do not have a description of $\Pois_n\mhyphen\Cat$ as a pullback, we prove the commutativity of pullbacks with taking $\E_n$ directly in Appendix \ref{appendix:pullback}. The rest of the proof goes through as described above, and we get:

\begin{proof}[Proof of Theorem \ref{prop:additivity}]
Applying Proposition \ref{prop:Eiandpullback} to $\Caa_1  = \E_0(\Ce\mhyphen \Cat)/ \E_0(\Ch \mhyphen \Cat)$,  $\Caa_2 = \E_\infty(\C \mhyphen\mathsf{Cat})$ and $\Daa = \E_0(\mathbb C \mhyphen\mathsf{Cat})$ gives the two claims of Theorem \ref{prop:additivity}.
\end{proof}

\section{Compatibility between factorization homology and quantization} \label{sect:Comp-FH-Quantization}
We start by proving that the 2-categories \(\PCat{0}\) and \(\BDCat{0}\) are suitable targets for factorization homology in Section \ref{sect:suitable-targets-for-FH-P0-and-BD0}. This allows us to use factorization homology in combination with Theorem \ref{prop:additivity} to relate local and global quantum observables in Section \ref{sec:GlobalQObsFH}. 
Lastly, we define the internal endomorphism algebra in our enriched setting and give some concrete examples of this notion in 
Section \ref{ssec:IntEndAlg}.

\subsection{Target categories for factorization homology} \label{sect:suitable-targets-for-FH-P0-and-BD0}
To ensure that \(\PCat{0}\) and \(\BDCat{0}\) are suitable target categories for factorization homology we need to examine (sifted) colimits therein. 
For this we first consider the following general result regarding colimits computed in what morally is a pullback 2-category.

\begin{lemma} \label{lemma:ColimitsInPullbackBicat}
    Let \(\Caa_1, \Caa_2, \Daa,\) and \(\Eaa\) be 2-categories, and let the diagram 
    \begin{equation} \label{eq:PullbackBicatsColimits}
    \begin{tikzcd}
        \Caa_1\times_\Daa \Caa_2 \arrow[r, "\mathrm{pr}_2"] \arrow[d, swap, "\mathrm{pr}_1"] &\Caa_2 \arrow[d, "H_2"]
        \\ \Caa_1 \arrow[r, swap, "H_1"] &\Daa
    \end{tikzcd}
    \end{equation}
    be a pullback diagram of 2-categories in the sense spelled out in the proof of Proposition \ref{prop:Eiandpullback}. If \(\Caa_1\) and \(\Caa_2\) have colimits of shape \(\Eaa\) which are preserved by \(H_1\) and \(H_2\), then \(\Caa_1\times_\Daa \Caa_2\) also has colimits of shape \(\Eaa\). 
\end{lemma}
\begin{proof}
Let \(J\colon \Eaa \ra \Caa_1\times_\Daa \Caa_2\) be a diagram, and denote by \(J_i=\mathrm{pr}_i \circ J \colon \Eaa \ra \Caa_i\) for \(i\in \{1,2\}\). Since \(\Caa_1\) and \(\Caa_2\) both have colimits of shape \(\Eaa\), and \(H_1, H_2\) preserve those colimits we have an isomorphism \( \Psi \colon H_1(\colim_\Eaa J_1) \cong \colim_\Eaa H_1\circ J_1 \cong \colim_\Eaa H_2\circ J_2 \cong H_2(\colim_\Eaa J_2) \) in \(\Daa\). We claim that \(\colim_\Eaa J := (\colim_\Eaa J_1, \colim_\Eaa J_2, \Psi)\) computes the colimit of \(J\) in \(\Caa_1 \times_\Daa \Caa_2\). 

We first explain how the object \(\colim_\Eaa J\) comes with structure maps making it a cocone of the diagram \(J\). 
For every object \(e\in \Eaa\) we have 1-morphisms
\begin{equation}
    \chi_e = (\chi_{e}^1, \chi_{e}^2, \varepsilon_{e}) \colon (J_1(e), J_2(e), \psi_{e})  \ra (\colim_\Eaa J_1, \colim_\Eaa J_2, \Psi) \ , 
\end{equation}
where \(\chi_{e}^i\) are the 1-morphisms into \(\colim_\Eaa J_i\) in \(\Caa_i\), for \(i\in \{1,2\}\), and the 2-isomorphism \(\varepsilon_e\) is the composition of the 2-isomorphisms filling the diagram below
\begin{equation}
\begin{tikzcd}[row sep=28pt]
    H_1 J_1 e \arrow[rrr, "\psi_e"] \arrow[d, swap, "H_1\chi_e^1"] \arrow[dr, "\chi^{11}_{H_1J_1(e)}", ""{name=U,inner sep=1pt,below}]& & &H_2 J_2 e \arrow[d, "H_2 \chi_e^2"] \arrow[dl, swap, "\chi^{22}_{H_2 J_2(e)}", ""{name=W, inner sep=2pt, below}]
    \\ H_1 \colim_\Eaa J_1 \arrow[r, "\sim"] & \colim_\Eaa H_1 J_1 \arrow[r, "\sim"]  & \colim_\Eaa H_2 J_2 \arrow[r, "\sim"] & H_2\colim_\Eaa J_2
    \arrow["{\alpha_e}", shorten >=3pt,  Rightarrow, from=2-1, to=U]
    \arrow["{\beta^{-1}_{e}}", shorten >=2pt,  Rightarrow, from=W, to=2-4]
    	\arrow[start anchor={[yshift=+1.5ex]}, end anchor={[xshift=1.5ex, yshift=1.5ex]}, shorten <=5pt, shorten >=80pt, Rightarrow, from=2-2, to=1-4, "\gamma_e" near start]
\end{tikzcd} \ .  
\end{equation}
The above left and right triangles have 2-isomorphisms filling them coming from \(H_1\), respectively \(H_2\) preserving colimits of shape \(\Eaa\). In addition, since \(H_1 J_1 \cong H_2 J_2\) in \(\Daa\) we get compatibility data between \(\colim_\Eaa H_1 J_1\) and \(\colim_\Eaa H_2 J_2\), which in particular means that there is a 2-iso \(\gamma_e\) filling the middle square above. 
For any 1-morphism \(g\colon e\ra e'\) in \(\Eaa\) we also have a 2-morphism \(\phi_g = (\phi_g^1, \phi_g^2) \colon \chi_{e'} \circ J(g) \Rightarrow \chi_e \). The compatibility condition for this to be a 2-morphism in \(\Caa_1 \times_\Daa \Caa_2\) again follows from \(H_1\) and \(H_2\) preserving colimits of shape \(\Eaa\). In total this defines the cocone \(\colim_\Eaa J\). 

We need to check that the cocone \(\colim_\Eaa J\) is universal. That is, any other cocone needs to factor through \(\colim_\Eaa J\), and any morphism between two cocones over some fixed category \(\Ca\) also needs to factor through \(\colim_\Eaa J\). Both properties follow, in a similar way to above, from using the 1- and 2-morphisms of \(\colim_\Eaa J_1\) and \(\colim_\Eaa J_2\) in \(\Caa_1\), respectively \(\Caa_2\), together with the compatibility data in \(\Daa\) arising from \(H_1\) and \(H_2\) preserving colimits and the diagrams \(H_1 J_1 \cong H_2 J_2\) being isomorphic. 
\end{proof}

We are now ready to establish the main statement of this subsection. 

\begin{lemma}\label{lma:targetFH}
    The categories $\Pois_0\mhyphen\Cat$ and $\BD_0\mhyphen\Cat$ are tensor sifted cocomplete (as (2,1)-categories).
\end{lemma}
\begin{proof}
We consider \(\Pois_0\mhyphen\Cat\) in detail. The proof for \(\BD_0\mhyphen\Cat\) is analogous. Let \(J \colon \Eaa \ra \Pois_0\mhyphen\Cat \) be a sifted diagram, and denote by \(J_1 = \mathrm{pr}_1 \circ J \colon \Eaa \ra \E_\infty(\mathbb{C}\mhyphen\Cat) \) and \(J_2 = \mathrm{pr}_2 \circ J \colon \Eaa \ra \E_0(\mathbb{C}_\varepsilon\mhyphen\Cat)\). We want to employ Lemma \ref{lemma:ColimitsInPullbackBicat} to get sifted cocompleteness, so we first need to establish its assumptions. 

Recall that \(\Va\mhyphen\Cat\) is cocomplete, so in particular both \(\mathbb{C}\mhyphen\Cat\) and \(\mathbb{C}_\varepsilon\mhyphen\Cat\) are cocomplete. Moreover, the category of \(\E_0\)-algebras in an arbitrary category \(\Ca\) is equivalent to the undercategory \(\Ca_{\emptyset_/}\) of objects under the distinguished object \(\emptyset\). The dual (and truncated) statement of \cite[Proposition 4.6]{LimitNlab} says that the (2,1)-colimit of a functor \(F\) in \(\Ca_{\emptyset_/}\) coincides with the (2,1)-colimit of the functor under the distinguished object \(F_{\emptyset_/}\). Explicitly, this means that the (2,1)-colimit \(\colim_\Eaa J_2\) in \(\E_0(\mathbb{C}_\varepsilon\mhyphen\Cat)\) coincides with the (2,1)-colimit \(\colim_\Eaa {J_2}_{\emptyset_/}\) in \(\mathbb{C}_\varepsilon\mhyphen\Cat\). Since the latter is cocomplete it follows that \(\E_0(\mathbb{C}_\varepsilon\mhyphen\Cat)\) is cocomplete. By the same argument we also get that \(\E_0(\mathbb{C}\mhyphen\Cat)\) is cocomplete. In addition we have that the functor \(\otimes_{\mathbb{C}_\varepsilon} \mathbb{C}\) preserves colimits because it is a left adjoint.

Now consider the category \(\E_\infty(\mathbb{C}\mhyphen\Cat)\). Since the tensor product of \(\VCat\) is left adjoint we know it commutes with all colimits, so in particular sifted ones. This allows us to employ \cite[Corollary 2.7.2 (1)]{Dag3Lurie} which ensures that \(\E_\infty(\mathbb{C}\mhyphen\Cat)\) admits all sifted colimits.

From Dunn's additivity we have \(\E_\infty (\mathbb{C}\mhyphen \Cat) \cong \E_\infty (\E_0(\mathbb{C}\mhyphen \Cat))\). Let \(K\colon \Eaa' \ra \E_0(\mathbb{C}\mhyphen \Cat) \) be some sifted diagram and let \(\Ca\in \E_0(\mathbb{C}\mhyphen \Cat)\). As above, the colimit \(\colim_{\Eaa'} (K\otimes \Ca) \in \E_0(\mathbb{C}\mhyphen \Cat) \) coincides with 
\begin{equation}
    \colim_{\Eaa'} (K\otimes \Ca)_{\emptyset_/} \cong \colim_{\Eaa'} (K_{\emptyset_/} \otimes {\Ca}_{\emptyset_/}) \cong (\colim_{\Eaa'} K_{\emptyset_/} ) \otimes {\Ca}_{\emptyset_/}  \in \mathbb{C}\mhyphen \Cat\ ,
 \end{equation}
where in the last step we use that the tensor product of \(\VCat\) is a left adjoint. The last expression coincides with \((\colim_{\Eaa'} K) \otimes \Ca \in \E_0(\mathbb{C}\mhyphen\Cat)\) thus making \(\E_0(\mathbb{C}\mhyphen\Cat)\) tensor cocomplete. Furthermore we can employ \cite[Corollary 2.7.2 (2)]{Dag3Lurie} which tells us that the forgetful functor \( \E_\infty (\E_0(\mathbb{C}\mhyphen \Cat)) \ra \E_0(\mathbb{C}\mhyphen \Cat)\) detects sifted colimits, so in particular it preserves them. Hence, we are in the situation of Lemma \ref{lemma:ColimitsInPullbackBicat}, and conclude that \(\Pois_0\mhyphen\Cat\) admits all sifted colimits. 

Finally, we need to ensure that the tensor product of \(\Pois_0\mhyphen\Cat\) commutes with sifted colimits. The above argument for 
\(\E_0(\mathbb{C}\mhyphen\Cat)\) being tensor cocomplete goes through in the exact same way for \(\E_0(\mathbb{C}_{\varepsilon}\mhyphen\Cat)\). Moreover, the functor \(\otimes_{\mathbb{C}_\varepsilon} \mathbb{C}\) is monoidal and cocontinuous (since it is a left adjoint), so the isomorphism witnessing the commutativity of tensoring and taking the colimit is sent to the analogous isomorphism in \(\E_0(\mathbb{C}\mhyphen\Cat)\). Since the functor \(\fgt \colon \E_\infty (\mathbb{C}\mhyphen\Cat) \ra \E_0(\mathbb{C} \mhyphen \Cat)\) detects sifted colimits, is monoidal and conservative we can lift the isomorphism witnessing sifted tensor cocompleteness in \(\E_0(\mathbb{C}\mhyphen \Cat)\) to an isomorphism in \(\E_\infty (\mathbb{C}\mhyphen\Cat)\). 
Using that all of the functors in the ``pullback" preserves both the monoidal structure and sifted colimits it is a straightforward check that this data assembles to an isomorphism in \(\Pois_0\mhyphen\Cat\) between \((\colim_{\Eaa} J) \otimes \Ca\) and \(\colim_{\Eaa} (J\otimes \Ca)\), which completes the proof. 
\end{proof}

\begin{remark}
    Even though the categories  $\Pois_0\mhyphen\Cat$ and $\BD_0\mhyphen\Cat$ are bicategories, we are only interested in the colimits needed to compute factorization homology. That is, specific colimits over the category of disks, which is an \((\infty, 1)\)-category. Hence, all colimits needed for factorization homology will factor through the (2,1)-subcategories of $\Pois_0\mhyphen\Cat$ and $\BD_0\mhyphen\Cat$.
\end{remark}

\begin{proposition}
The symmetric monoidal functor 
\begin{equation}
    -\otimes_{\C[[\hbar]]} \C_\varepsilon \colon \BD_0\mhyphen\Cat \to \Pois_0\mhyphen\Cat
\end{equation}
preserves sifted colimits. 
\end{proposition}
\begin{proof} 
Consider the following commuting diagram
\begin{equation}\label{diag:BD0aPois0}
   \begin{tikzcd}[column sep=10ex]
        \BD_0\mhyphen\Cat & \mathbb{C}[[\hbar]]\mhyphen\Cat
        \\ \Pois_0\mhyphen\Cat & \mathbb{C}_\varepsilon\mhyphen\Cat
        \arrow[from=1-1, to=1-2, "\fgt_{\BD_0}"]
        \arrow[from=1-1, to=2-1, swap, "-\otimes_{\C[[\hbar]]} \C_\varepsilon"]
        \arrow[from=1-2, to=2-2, "-\otimes_{\C[[\hbar]]} \C_\varepsilon"]
        \arrow[from=2-1, to=2-2, "\fgt_{\Pois_0}",swap]
    \end{tikzcd} \ . 
\end{equation}
First, we note that the vertical functor 
\begin{equation}
    -\otimes_{\C[[\hbar]]} \C_\varepsilon \colon \mathbb{C}[[\hbar]]\mhyphen\Cat \longrightarrow \mathbb{C}_\varepsilon\mhyphen\Cat
\end{equation}
is a left adjoint, so it preserves all colimits. Next, we claim that the two horizontal functors 
\begin{equation}
    \fgt_{\BD_0} \colon \BD_0\mhyphen\Cat \to \E_0(\C[[\hbar]]\mhyphen\Cat) \to \C[[\hbar]]\mhyphen\Cat \quad \text{and} \quad \fgt_{\Pois_0} \colon \Pois_0\mhyphen\Cat \to \E_0(\C_\varepsilon\mhyphen\Cat) \to \C_\varepsilon\mhyphen\Cat
\end{equation}
both preserve and reflect sifted colimits. We will show that this is true for $\fgt_{\BD_0}$, the argument for $\fgt_{\Pois_0}$ is analogous. 
Let $J \colon \Eaa \to \BD_0\mhyphen\Cat$ be a sifted diagram. From the proof of Lemma \ref{lemma:ColimitsInPullbackBicat} we recall that a colimit of shape $\Eaa$ in $\BD_0\mhyphen\Cat$ is a triple $\left( \colim_\Eaa J_1, \colim_\Eaa J_2, \Psi \right)$, where $J_1 \colon \Eaa \to \E_0(\C[[\hbar]]\mhyphen\Cat)$, $J_2 \colon \Eaa \to \E_\infty(\C\mhyphen\Cat)$ and $\Psi$ is an isomorphism between the images of their respective colimits under the functors $H_1 \colon \E_0(\C[[\hbar]]\mhyphen\Cat) \to \E_0(\C\mhyphen\Cat)$ and $H_2 \colon \E_\infty(\C\mhyphen\Cat) \to \E_0(\C\mhyphen\Cat)$. It is clear that the functor 
\begin{equation}
    \BD_0\mhyphen\Cat \longrightarrow \E_0(\C[[\hbar]]\mhyphen\Cat), \quad (X_1,X_2,\Psi_{X_1,X_2}) \longmapsto X_1
\end{equation}
preserves colimits. In order to show that it also reflects colimits, let $\pi \colon J \Rightarrow \Delta(X)$ be a cocone under $J$, such that that its image under $\BD_0\mhyphen\Cat \to \E_0(\C[[\hbar]]\mhyphen\Cat)$ is a colimit cocone. In other words, $X$ is a triple $( X_1, X_2, \Psi_{X_1,X_2} )$ together with maps $u_{X_1} \colon \colim_\Eaa J_1 \to X_1$ and $u_{X_2} \colon \colim_\Eaa J_2 \to X_2$ making the following diagram commute 
\begin{equation}
\begin{tikzcd}
    H_1(\colim_\Eaa J_1) \arrow[d,"H_1(u_{X_1})",swap] \arrow[r,"\Psi"] & H_2(\colim_\Eaa J_2) \arrow[d,"H_2(u_{X_2})"]\\
    H_1(X_1) \arrow[r,"\Psi_{X_1,X_2}"] & H_2(X_2)  \ \ .
\end{tikzcd}
\end{equation}
By assumption we have that $H_1(u_{X_1})$ is an isomorphism and by commutativity of the above diagram so is $H_2(u_{X_2})$. Since $H_2$ is conservative, the cocone $\pi$ is a colimit. Finally, it follows from \cite[Proposition 4.4.2.9]{HTT}, combined with \cite[Corollary 2.1.2.2]{HTT}, that the forgetful functor
\begin{equation}
    \E_0(\mathbb{C}[[\hbar]]\mhyphen\Cat) \longrightarrow \mathbb{C}[[\hbar]]\mhyphen\Cat 
\end{equation}
preserves and reflects sifted colimits, and we conclude that $\fgt_{\BD_0}$ indeed preserves and reflects sifted colimits. 

The observation that $\left( -\otimes_{\C[[\hbar]]} \C_\varepsilon \right) \circ \fgt_{\Pois_0}$ preserves sifted colimits and $\fgt_{\BD_0}$ reflects them, together with commutativity of Equation \eqref{diag:BD0aPois0}, then implies the assertion that $-\otimes_{\C[[\hbar]]} \C_\varepsilon \colon \BD_0\mhyphen\Cat \to \Pois_0\mhyphen\Cat$ preserves sifted colimits. \qedhere
\end{proof}

\newcommand{\exc}{\mathrm{exc}} 

\subsection{Global quantum observables via factorization homology}\label{sec:GlobalQObsFH}
We have seen in the previous section that the categories $\Pois_0\mhyphen\Cat$ and $\BD_0\mhyphen\Cat$ are tensor sifted cocomplete. Hence, we can compute factorization homology on 2-manifolds with coefficients in $\Pois_0\mhyphen\Cat$ and $\BD_0\mhyphen\Cat$. 
In this section we use this fact to outline an approach to quantization of categories via factorization homology.

We start with a classical system described by the categorified observables 
\begin{equation}
    \Obs^{\cl}\colon \Man^{\ori}_2 \to \E_\infty\left(\mathbb{C} \mhyphen \Cat\right) \ \ .
\end{equation}
We think of the value at a 2-dimensional \emph{oriented} surface $\Sigma$ as the symmetric monoidal category of perfect sheaves on the moduli stack $\mathcal{F}$ of solutions to the classical equations of motion (on $\Sigma\times \R$ for 3-dimensional Chern-Simons theory). The algebra of \emph{classical observables} or functions on $\mathcal{F}$ can be defined as $\mathcal{O}_\Sigma \coloneqq \operatorname{End}_{\Obs^{\cl}(\Sigma)}(1)$. The main example for us are sheaves on the moduli stack of flat principal $G$-bundles on $\Sigma$, for a reductive algebraic group $G$, which we will study in detail in Section~\ref{sec:CS}. We only consider classical observables which are local in the sense that they satisfy excision. Hence, there is an $\E_\infty$-algebra $\Obs^{\cl}_{\loc}$ of local observables, namely the value of $\Obs^{\cl}$ on disks, and the global observables can be reconstructed via factorization homology
\begin{align}
    \Obs^{\cl}(\Sigma) = \int_{\Sigma} \Obs^{\cl}_{\loc} \ \ . 
\end{align}
There is a slight subtlety here. Every $\E_\infty$-algebra has a canonical structure of a framed $\E_2$-algebra. However, the balanced structure used to compute the global observables might differ from this one. We denote the category of symmetric monoidal balanced $\C$-linear categories by $\E_\infty^ \mathsf{f}(\C\mhyphen\Cat)$.    

We are interested in quantizing almost Poisson-0 structures on $\Obs^{\cl}$ which are also local. Since we are interested in theories on oriented manifold the coefficients for factorization homology are framed $\E_2$-algebras. This leads us to introduce framed versions of almost Possion and almost BD-categories.
\begin{definition}
A \emph{framed $\BD_2$-category} is a $\BD_2$-category together with a balancing on the $\C[[\hbar]]$-linear braided monoidal category $\Ca_\hbar$ which is part of the $\BD_2$-category. A \emph{framed $\Pois_2$-category} is defined analogously. We denote by $\Pois_2^\mathsf{f}\mhyphen\Cat$ and $\BD_2^\mathsf{f}\mhyphen\Cat$ the bicategory of framed $\Pois_2$-categories and framed $\BD_2$-category, respectively.     
\end{definition}

\begin{notation}\label{not:functors-satisfying-excision}
    Let \(\Fun^\exc(\Man_2^{\ori},-)\) denote functors which satisfy excision. 
\end{notation}

With these definitions at our disposal we can describe Poisson structures on $\Obs^{\cl}$ in terms of the local observables $\Obs^{\cl}_{\loc}$: 

\begin{proposition}\label{Prop: aP-structure}
The following diagram commutes and the horizontal functors are equivalences
\begin{equation}
\begin{tikzcd}
    \Pois_2^\mathsf{f}\mhyphen\Cat \ar[r, "\sim"] \ar[d, "\varepsilon=0",swap] & \fE_2(\Pois_0\mhyphen\Cat) \ar[r, "\int"] & \Fun^{\exc}(\Man^{\ori}_2,\Pois_0\mhyphen\Cat) \ar[d, "\varepsilon=0"] \\ 
    \E_\infty^ \mathsf{f}(\C\mhyphen\Cat) \ar[rr, "\int",swap] & & \Fun^{\exc}(\Man^{{\ori}}_2,\C\mhyphen\Cat)
\end{tikzcd} \ . 
\end{equation} 
\end{proposition}
\begin{proof}
The only part which does not follow directly from the discussion so far is the compatibility with balancing which follows by inspecting the equivalence from Theorem~\ref{prop:additivity}.  
\end{proof}

The idea is now to locally quantize the classical observables $\Obs^{\cl}_{\loc} \in \fE_2(\Pois_0\mhyphen\Cat) \cong \Pois_2^\mathsf{f}\mhyphen\Cat$ to a framed $\BD_2$-category of quantum observables $\Obs^{\operatorname{q}}_{\loc} \in \BD_2^\mathsf{f}\mhyphen\Cat$. Define the quantum observables on an oriented surface $\Sigma$ via factorization homology:
\begin{equation}
    \Obs^{\operatorname{q}}(\Sigma) \coloneqq \int_{\Sigma} \Obs^{\operatorname{q}}_{\loc} \in \BD_0\mhyphen\Cat \ \ .
\end{equation}
That this provides a quantization of the almost Poisson structure on $\Obs^{\operatorname{cl}}$ is a consequence of the following theorem. 

\begin{theorem}\label{Thm: quant FH}
  The following diagram commutes and the horizontal functors are equivalences
\begin{equation}
\begin{tikzcd}
 \BD_2^\mathsf{f}\mhyphen\Cat \ar[r, "\sim"] \ar[d, "\lim_{h\rightarrow 0}",swap] & \fE_2(\BD_0\mhyphen\Cat) \ar[r, "\int"] \ar[d, "\lim_{h\rightarrow 0}"] & \Fun^{\exc}(\Man^{\ori}_2,\BD_0\mhyphen\Cat) \ar[d, "\lim_{h\rightarrow 0}"] \\ 
  \Pois_2^\mathsf{f}\mhyphen\Cat \ar[r, "\sim"]  & \fE_2(\Pois_0\mhyphen\Cat) \ar[r, "\int"] & \Fun^{\exc}(\Man^{\ori}_2,\Pois_0\mhyphen\Cat) 
\end{tikzcd} \ \ . 
\end{equation}
\end{theorem}
\begin{proof}
Follows from the same arguments as used in the proof of Proposition~\ref{Prop: aP-structure}, replacing $\C_\varepsilon$ with $\C_\hbar$.
\end{proof}
In summary we get that constructing a local quantization of the classical observables is equivalent to quantizing the framed $\Pois_2$-category of local observables into a framed $\BD_2$-category of local quantum observables. The global quantum observables can be computed via factorization homology. By Theorem~\ref{Thm: FH=Sk} we can compute these in terms of enriched skein categories.

\begin{remark}
    We define the \emph{algebra of quantum functions} as $\mathcal{O}_\Sigma^q\coloneqq \operatorname{End}_{\Obs^{q}(\Sigma)}(\star)$ where $\star$ is the pointing of the $\BD_0$-algebra $\Obs^{q}(\Sigma)$. This is a quantization of the algebra of functions $\mathcal{O}_\Sigma^{\cl }$ which can be explicitly described by the enriched skein algebra. 
\end{remark}

\newcommand{\Ann}{\mathbb{A}\mathrm{nn}} 

\subsection{Internal endomorphism algebra}\label{ssec:IntEndAlg}
In this section we extend the notion of the `internal skein algebra' from \cite{safronovQMM} to the $\cat{V}$-enriched setting. 
We also show how it can be used to recover equivariant Poisson structures on algebras from the category of global quantum observables on manifolds with boundary. 

We start with the general $\cat{V}$-enriched situation. Let $(\cat{A},\otimes,\dots )$ be a monoidal $\cat{V}$-category and $( \cat{M}, \triangleright,\dots )$ a left $\cat{A}$-module. For a fixed object $\cat{O}\in \cat{M}$ we can construct a lax monoidal $\cat{V}$-functor
\begin{align} \label{eq:IEA}
    \cat{A}^{\op} & \longrightarrow \cat{V} \\ 
    a &\longmapsto \cat{M}(a\triangleright \cat{O}, \cat{O})  \ . 
\end{align} 
The lax monoidal structure uses the identity on $\cat{O}$ as unit and 
\begin{equation}\label{eq:lax_mon_str_intendoalg}
\begin{split}
  \cat{M}(b\triangleright \cat{O}, \cat{O}) \monprod[\Va] \cat{M}(a\triangleright \cat{O}, \cat{O}) & \xrightarrow{\mathmakebox[5em]{\id {\monprod[\Va]} (b\triangleright -)}} \cat{M}(b\triangleright \cat{O}, \cat{O}) \monprod[\Va] \cat{M}(b\triangleright a\triangleright \cat{O}, b\triangleright \cat{O})
  \\ &
  \xrightarrow{\mathmakebox[5em]{\circ_{\cat{M}}}} \cat{M}(b \triangleright a \triangleright \cat{O}, \cat{O}) 
\end{split} 
\end{equation}
together with the isomorphism $\cat{M}(b \triangleright a \triangleright \cat{O}, \cat{O}) \cong \cat{M}((b \otimes a) \triangleright \cat{O}, \cat{O}) $ which is part of the left module structure. 
Recall from Definition \ref{defn:free-cocompletion} that \(\widehat{\Aa}\) denotes the functor \(\Va\)-category $[\cat{A}^{\op},\cat{V}]$. 
Moreover, from Lemma \ref{lemma:LaxMonVsAlgStructure} we know that providing an algebra in the free cocompletion \(\widehat{\Aa}\) is equivalent to providing a lax monoidal functor \(\Aa^{\op} \ra \Va\). 

\begin{definition} \label{defn:InternalEndAlg}
    We call the algebra object $\IEAlong{\cat{A}}{\cat{O}}$ in $\widehat{\cat{A}}$ corresponding to the lax monoidal functor constructed in Equation \eqref{eq:IEA} the \emph{internal endomorphism algebra} associated to $\cat{O}$, see also~\cite[Def.~3.5]{safronovQMM}. 
\end{definition}

We are mostly interested in the following situation arising for global quantum observables on manifolds with boundaries: Let $\cat{M}$ be a $\C$-linear symmetric monoidal category (e.g.\ the classical observables on a manifold with boundary) equipped with the structure of a module over a $\C$-linear symmetric monoidal category $\cat{A}$ (e.g.\ the local classical observables) together with a first order deformation $(\cat{A}_\varepsilon,\cat{M}_\varepsilon)$ and a compatible formal deformation $(\cat{A}_\hbar,\cat{M}_\hbar)$ to a module over a monoidal category (e.g.\ the quantum observables). 
We want to show that the algebras $\IEAlong{\cat{A}_\hbar\!}{\cat{O}}$, $\IEAlong{\cat{A}_\varepsilon\!}{\cat{O}}$ and $\IEAlong{\cat{A}}{\cat{O}}$ are deformations of each other. For this we first construct functors 
\begin{align}
    \widehat{\cat{A}_\hbar} \xrightarrow{\lim_{\hbar\to 0} } \widehat{\cat{A}_\varepsilon}  \xrightarrow{\varepsilon=0} \widehat{\cat{A}} \ \ .
\end{align}
The map $\lim_{\hbar\to 0}$ is constructed as a functor in $\widehat{\Aa}$ in two steps. First we form the functor 
$\cat{A}_\hbar^{\op} \to \C[[\hbar]]\lmod \xrightarrow{-\otimes_{\C[[h]]} \C[\varepsilon]} \C[\varepsilon]\lmod$. Next we observe that this induces a well-defined functor $\cat{A}_\varepsilon^{\op}\to   \C[\varepsilon]\lmod $ which is the image under $\lim_{\hbar\to 0}$. The functor setting $\varepsilon=0$ is constructed in the same way. This construction preserves lax monoidal structures on functors. 

\begin{proposition}\label{Prop: Defor skein alg}
With the notation introduced above we have 
\begin{align}
\lim_{\hbar\to 0} \IEAlong{\cat{A}_\hbar\!}{\cat{O}} & \cong \IEAlong{\cat{A}_\varepsilon\!}{\cat{O}} \\
 \IEAlong{\cat{A}_\varepsilon\!}{\cat{O}}_{(\varepsilon=0)} & \cong \IEAlong{\cat{A}}{\cat{O}}
\end{align} 
as lax monoidal functors.
\end{proposition}
\begin{proof}
The image of the $\Ch$-linear action functor
\begin{equation} 
    \triangleright_\hbar  \colon  \cat{A}_\hbar \catprod \cat{M}_\hbar \to \cat{M}_\hbar 
\end{equation}
under the 2-functor $-\otimes_{\C[[h]]}  \C[\varepsilon]$ is isomorphic to the $\Ce$-linear action functor 
\begin{equation}
\triangleright_\varepsilon  \colon  \cat{A}_\varepsilon \catprod \cat{M}_\varepsilon \to \cat{M}_\varepsilon 
\end{equation} 
since  $\cat{A}_\hbar$ is a quantization of $\cat{A}_\varepsilon$. 
Moreover, $\lim_{\hbar\to 0} \IEAlong{\cat{A}_\hbar\!}{\cat{O}}$ is the internal endomorphism algebra associated to the action functor $\triangleright_\hbar \otimes_{\C[[h]]} \C[\varepsilon] $. Finally, the isomorphism of action functors $\triangleright_\hbar \otimes_{\C[[h]]} \C[\varepsilon] \cong \triangleright_\varepsilon$ induces the isomorphism between the corresponding internal endomorphism algebras.

Similarly, the second isomorphism comes from the fact that the action functors $\triangleright_\varepsilon \otimes _\Ce \C$ and $\triangleright$ are isomorphic, since they are two components of an $\Pois_0$-functor.
\end{proof}

\begin{example} \label{ex:defnPoisson}
Let $M$ be a smooth affine algebraic variety equipped with an action of a reductive algebraic group $G$. The category of perfect sheaves on the quotient stack $M/ G$ is the category of finitely generated projective modules\footnote{The reader might rightfully wonder projective as elements of which (enriched) category. Here we mean projective as elements of the $\C$-linear category of $\cat{O}(M)$-modules inside of $\operatorname{Rep} (G)$.} over $\cat{O}(M)$ seen as an object of $\operatorname{Rep} (G)=U(\mathfrak{g})\lmod$. This is in a canonical way a module category over ${U(\mathfrak{g})\lmod^{\operatorname{f.d.}}}$. 
The internal endomorphism algebra $\IEAlong{U(\mathfrak{g})\lmod^{\operatorname{f.d.}}} 
{\cat{O}(M)} $ can be identified with the algebra $\cat{O}(M)$ as an object of 
$\widehat{U(\mathfrak{g})\lmod^{\operatorname{f.d.}}}= U(\mathfrak{g})\lmod$.

Now assume that we have a first order deformation $(\operatorname{Perf}(M/ G)_\varepsilon, U_\varepsilon(\mathfrak{g}))$, where we fix the deformation of $U(\mathfrak{g})\lmod^{\operatorname{f.d.}}$ to be the one coming from the classical $r$-matrix as explained in Example~\ref{Ex: P_2 from r}. The algebra $\IEAlong{U_\varepsilon(\mathfrak{g})\lmod^{\operatorname{f.d.}}} {\cat{O}(M)}$ is, according to Proposition~\ref{Prop: Defor skein alg}, a first order deformation of $\cat{O}(M)$ where the multiplication is simultaneously a morphism in $U_\varepsilon(\mathfrak{g})\lmod$. If this comes from a bracket $\{-,-\}\colon \mathcal{O}(M)\otimes \mathcal{O}(M) \to  \mathcal{O}(M)$ the condition on the multiplication
to be a morphism in $U_\varepsilon(\mathfrak{g})\lmod$ is equivalent to the action map $G\times M \to M$ to be Poisson. 
For the formal deformation constructed 
from $U_\hbar(\mathfrak{g})$, the internal endomorphism algebra gives rise to an equivariant deformation quantization of it. 

In the case where we fix the deformation of $U(\mathfrak{g})\lmod$ to be the Drinfeld category introduced in Section~\ref{Drinfeld category} we will find deformation quantization of quasi-Poisson structures.   
In the next section we will compute the Poisson structure on moduli spaces of flat principal bundles induced by deformations of their category of perfect sheaves. 
\end{example}

\begin{example}\label{ex:InternalSkeinAlgAnn}
    Let $\cat{A}$ be a ribbon $\cat{V}$-enriched category and let $\Sigma=\mathsf{Ann}$ be an annulus with a marked interval on its boundary. 
    We consider the enriched skein category $\mathbf{Sk}_\cat{A}(\Ann)$ as a module over $\mathbf{Sk}_\cat{A}(\mathbb{D}^2) \cong \cat{A}$ with the module structure induced by the embedding $\mathcal{P} \colon \mathbb{D}^2 \hookrightarrow \Ann$ of a disk along the marked interval. The distinguished object $\mathcal{O}$ is the image of the empty manifold under the embedding $\mathcal{P}$. From the excision-property of the enriched skein category we have
    \begin{equation}
        \mathbf{Sk}_\cat{A}(\Ann) \cong \cat{A} \coendtens{\cat{A} \catprod \cat{A}} \cat{A} \ \ , 
    \end{equation}
    where on the right we are using the coend tensor product from Section \ref{Sec:coend_tensor_product}. 
    Using this the internal endomorphism algebra takes the following form: 
    \begin{align*}
        \cat{A}^{\op} & \xrightarrow{\IEAlong{\Aa}{\Oa}} \Va \\ 
        a & \mapsto \mathbf{Sk}_\cat{A} (\Ann) (a \triangleright \Oa, \Oa) \\
        & \cong \int^{(b,b') \in \cat{A} \catprod \cat{A}} \cat{A}(a, 1 \triangleleft (b,b')) \monprod[\Va] \cat{A}((b,b') \triangleright 1,1) \\
        & \cong \int^{(b,b') \in \cat{A} \catprod \cat{A}} \cat{A}(a,b\otimes b') \monprod[\Va] \cat{A}(b',b^\vee) \\ 
        & \cong \int^{b \in \cat{A}} \cat{A}(a, b \otimes b^\vee) \ . 
    \end{align*}
For the second equivalence we used that $\cat{A}$ has duals and in the last line we applied the enriched co-Yoneda lemma \cite[Section 3]{KellyVCat}. The algebra $\IEAlong{\Aa}{\Oa} \cong \int^{b \in \cat{A}} \yo(b \otimes b^\vee)$ is called the (enriched) \emph{internal skein algebra}. 
See also e.g.\ \cite[Definition 2.18]{GJS} for the case \(\Va=\Vect\). 
The product \eqref{eq:lax_mon_str_intendoalg} on the internal skein algebra comes from ``vertical stacking''. 
This is illustrated in Figure \ref{fig:entire-figure-stacking} below.

\begin{figure}[H]
\centering
\begin{overpic}[scale=1,tics=10]{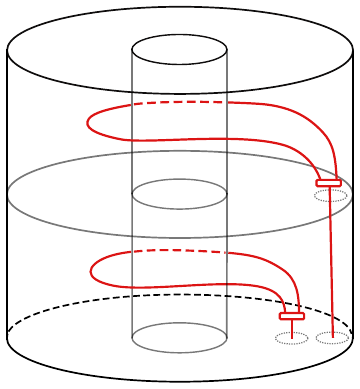}    
\put(84.5, 8.5){$b$}
\put(74, 8.5){$a$}
\end{overpic}
\caption{Composition for the internal skein algebra of the annulus.}
\label{fig:entire-figure-stacking}
\end{figure}
\end{example}

\section{Poisson algebras and quantization} 
Finally, we will apply the formalism of deformations of categories to recover Poisson structures on the internal endomorphism algebra. 
More concretely, in Section~\ref{sect:internal-endo-algebras-and-fusion} we establish how the internal endomorphism algebra of a surface (with boundary) behaves under fusion. In Section~\ref{sec:Cemod} we collect some results on \(\Ce\)-modules and \(\Ce\)-natural transformations needed in Section~\ref{sec: Fusion and Poisson} 
where we recover Poisson structures on classical internal endomorphism algebras. 
In Section~\ref{sect:poisson-brackets-and-skeins} we compute the deformation of the internal endomorphism algebra of a general marked surface using skein-theoretic tools.
Lastly, in Section~\ref{sec:CS} we give the main example of the above, coming from a deformation of the category of representations of a quantum group, namely the representation and character variety associated to a surface and a (reductive) algebraic group.

\newcommand{\intEndAlg}[2]{\IEAsymbol_{#1, #2}} 
\newcommand{\Afused}{\IEAsymbol_{\Sigma^\text{f}}}
\newcommand{\ASigma}{\IEAsymbol_{\Sigma}}
\newcommand{\OSigma}{\Oa_\Sigma}
\newcommand{\Ofused}{\Oa_{\Sigma^\text{f}}}
\newcommand{\OzPop}{\Oa_{\mathbb D_\text{poc}}} 
\newcommand{\diskpoc}{\mathbb D_{\mathrm{poc}}}
\newcommand{\SigmaFused}{\Sigma^\text{f}}
\newcommand{\zPop}{\mathrm{poc}} 
\newcommand{\ASigmaHat}{\widehat{\otimes} \ASigma}
\newcommand{\Daytimes}{\widehat{\otimes}}
\newcommand\drawBraiding[4]{
    \draw[thick] (#1,#2).. controls (#1, #2+0.4*#4) and (#1+#3, #2+0.6*#4)..(#1+#3, #2+#4);
    \draw[thick] (#1+#3, #2).. controls (#1+#3, #2+0.1*#4) and (#1+0.7*#3, #2+0.3*#4).. (#1+0.6*#3, #2+0.4*#4);
    \draw[thick] (#1, #2+#4).. controls (#1, #2+0.9*#4) and (#1+0.4*#3, #2+0.6*#4).. (#1+0.4*#3, #2+0.6*#4); }

\subsection{Internal endomorphism algebras for surfaces with boundary}\label{sect:internal-endo-algebras-and-fusion}
Let $\Aa$ be a $\Va$-enriched ribbon category, let $\Sigma$ be a surface and let $V\subset \partial \Sigma$ be a finite collection of \emph{marked points}. This choice of marked points defines an action\footnote{These marked points also specify intervals on the boundary along which one can perform excision, as we discuss in this section. The use of points as opposed to intervals is conventional (they carry the same combinatorial data) and is common in the Poisson geometry literature.} of the category $\Aa^{\catprod V}$ on $\SkCat{\Sigma}{\Aa}$, with $a_v \in \Aa$ acting by marking a new point near $v\in V$. From the discussion in Section~\ref{ssec:IntEndAlg} this data then gives rise to the internal endomorphism algebra $ \intEndAlg{\Sigma}{V}$ in $\widehat{\Aa}^{\boxtimes V}$, i.e.\ a lax monoidal functor $({\Aa^{\catprod V}})^{\op} \to \Va$. In this section we describe how the algebra \(\intEndAlg{\Sigma}{V}\) depends on $(\Sigma, V)$. In particular, we will describe how it interacts with fusion of surfaces.

Given $(\Sigma, V)$ as above, fusion at two distinct points $v_1, v_2 \in V$ is given by gluing in a disk along the boundary at $v_1, v_2$ and marking a new point \(v_3\) on the glued-in disk. The disk we are gluing in can be thought of as a flat pair of pants, hence we call it the \emph{pair of chaps} and denote it by \(\mathbb D_\zPop\). Figure \ref{fig:FusionLocalPic} below gives a local illustration of the fusion procedure.  

\begin{figure}[H]
     \centering
     \begin{subfigure}[b]{0.45\textwidth}
         \centering
           \begin{overpic}[width=\textwidth, tics=10]{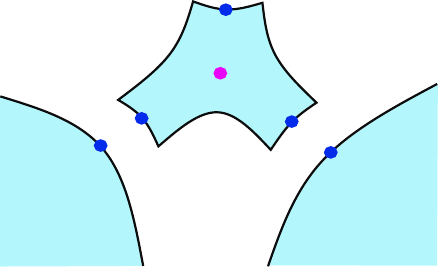}
           \put(25,27){\(v_1\)}
           \put(69,27){\(v_2\)}
           \put(53,45){\(b\)}
           \end{overpic}
         \caption{The surface \((\Sigma, V)\) and a pair of chaps \(\mathbb D_\zPop\).    }
         \label{fig:PreFusion}
     \end{subfigure}
     \hfill
     \begin{subfigure}[b]{0.45\textwidth}
         \centering
               \begin{overpic}[width=\textwidth, tics=10]{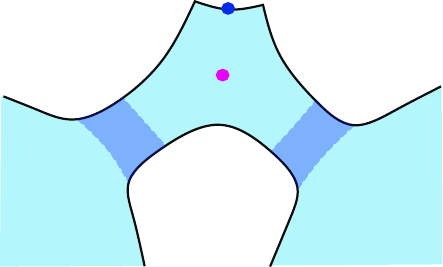}
                    \put(26,27){\(\mathbb D_1\)}
                    \put(67,27){\(\mathbb D_2\)}
                    \put(53,45){\(b\)}
                    \put(51,61){\(v_3\)}
           \end{overpic}
         \caption{The resulting fused surface \((\Sigma^\text{f}, V^\text{f})\).}
         \label{fig:FusedSurface}
     \end{subfigure}
     \caption{A local illustration of fusion of \((\Sigma, V)\) at the marked points \(v_1, v_2 \in V\). Each marked point determines a (isotopy class of) thickened embedding of an interval as in Theorem \ref{thm:Excision}.}
          \label{fig:FusionLocalPic}
\end{figure}

Before explaining how internal endomorphism algebras behave with respect to fusion we record some observations and facts that will be useful. 

\begin{observation}
     The fused surface can be written as \(\SigmaFused = \mathbb D_{\zPop} \sqcup_{\mathbb D_1\sqcup \mathbb D_2} \Sigma\), where \(\mathbb D_1\), \(\mathbb D_2\) are disks as illustrated in Figure \ref{fig:FusedSurface}.
     It follows that the distinguished object of the fused surface can be decomposed as \(\Ofused = \OzPop \otimes \OSigma\).
\end{observation}

Recall that \(\Aa \cong \SkCat{\mathbb{D}^2}{\Aa}\). To make it easier to distinguish the corresponding actions we use the notation \(\Aa_i \coloneqq \SkCat{\mathbb D_i}{\Aa}\), for \(i \in\{1,2\}\). The product \(\Aa_1 \catprod \Aa_2\) inherits a monoidal structure from \(\Aa_1\) and \(\Aa_2\), i.e.
\begin{align}\label{eq:MonProdC}
	\otimes_{\Aa_1 \catprod \Aa_2} \colon \Aa_1 \catprod \Aa_2 \catprod \Aa_1 \catprod \Aa_2 \xrightarrow{\id \catprod \beta \catprod \id} \Aa_1 \catprod \Aa_1 \catprod \Aa_2 \catprod \Aa_2 \xrightarrow{\otimes_{\Aa_1} \catprod \otimes_{\Aa_2}} \Aa_1 \catprod \Aa_2,
\end{align}
where \(\otimes_{\Aa_i}, i\in \{1,2\}\) denotes the monoidal products of \(\Aa_1\) and \(\Aa_2\). The product \(\Aa_1 \catprod \Aa_2\) acts on \(\Aa\) from the right by \(a \lhd (a_1, a_2) \coloneqq a \otimes a_1 \otimes a_2\) for \(a\in \Aa, a_1 \in \Aa_1, a_2 \in \Aa_2\). From the multiplication of \(\Aa_1 \catprod \Aa_2\) above we then have 
\begin{align}\label{eq:ActTwiceToMult}
    a \lhd (a_1, a_2) \lhd (a'_1, a'_2)  \xrightarrow{\id \catprod \big(\id \catprod \beta_{a_2, a'_1} \catprod \id\big)} a \lhd (a_1 \otimes a'_1, a_2 \otimes a'_2).
\end{align}
Finally, recall that the monoidal product \(\otimes \colon \Aa \catprod \Aa \ra \Aa\) gives rise to the functor \(\Daytimes \colon \widehat{\Aa}\boxtimes \widehat{\Aa} \ra \widehat{\Aa}\) as explained in Section \ref{sect:Completions}. 
We are now ready to state how internal endomorphism algebras behave with respect to fusion. 

\begin{proposition}\label{prop:MarkedFusionInternalEndAlgs}
    The algebra \(\intEndAlg{\SigmaFused}{V^{\text{f}}}\)
    is obtained from \(\intEndAlg{\Sigma}{V}\) by applying the tensor product functor $\CompProd \colon \widehat{\Aa}^{\boxtimes 2} \to \widehat{\Aa}$ on the two factors in $\widehat{\Aa}^{\boxtimes V}$ corresponding to $v_1, v_2 \in V$.  
\end{proposition}
\begin{proof}

For notational simplicity, consider \(V = \{v_1, v_2\}\) as in Figure \ref{fig:FusionLocalPic}, and \(V^\text{f} = \{v_3\}\). 
This justifies omitting the marked points from the notation, and we have \(\ASigma \coloneqq \IEAlong{\Aa_1\catprod \Aa_2}{\Oa_\Sigma} \in \widehat{\Aa}_1 \boxtimes \widehat{\Aa}_2 \) and \(\Afused \coloneqq \IEAlong{A}{\Oa_{\Sigma^\text{f}}} \in \widehat{\Aa} \). 
We first compare \(\Afused\) and \(\ASigmaHat\) as objects of \(\Comp{\Aa}\) before we spell out both algebra structures and compare them. 

\nparagraph{\emph{As objects of \(\Comp{\Aa}\).}} Let \(b\in \Aa\cong \SkCat{\diskpoc}{\Aa}\) be an arbitrary object as illustrated in Figure \ref{fig:FusionLocalPic} above. Then we have
\begin{align}
	\Afused(b) &= \SkCat{\SigmaFused}{\Aa}\left(b \rhd \Ofused, \Ofused \right) \stackrel{\ref{thm:Excision}}{\cong} \left(\Aa \coendtens{\Aa_1\catprod \Aa_2} \SkCat{\Sigma}{\Aa}\right) \left(b\otimes \OSigma, \OzPop\otimes \OSigma \right)
	\\ &\stackrel{\eqref{eq:CoendRelProd}}{=}\int^{(a_1, a_2) \in \Aa_1\catprod \Aa_2}  \Aa\left(b, \OzPop \lhd \left(a_1, a_2\right)\right) \monprod[\Va] \SkCat{\Sigma}{\Aa}\left(\left( a_1, a_2\right) \rhd \OSigma, \OSigma \right) \stackrel{\eqref{eq:CompletionFunctorEquation}}{=} \CompProd \ASigma(b) \ . 
\end{align}
The first step is exactly the definition of \(\Afused\), where \(\Ofused = \OzPop \otimes \OSigma\). The second equivalence follows from the excision-property of skein categories established in Theorem \ref{thm:Excision}. Then we use Equation \eqref{eq:CoendRelProd} to write out what the coend relative tensor product explicitly is. Finally, by using Equation \eqref{eq:CompletionFunctorEquation} we see that this is exactly \(\ASigmaHat\). Note that functoriality follows from that of the coend in both cases, so \(\Afused\) and \(\ASigmaHat\) indeed agree as functors. 

\nparagraph{\textit{Algebra structure of \(\Afused\).}} We start with unravelling the lax monoidal structure of \(\Afused\) as given in Equation \eqref{eq:lax_mon_str_intendoalg}. Since we want to compare it to that of \(\ASigmaHat\) we need to understand the lax monoidal structure at the level of the coend coming from the coend relative tensor product. For any two objects \(b, c \in \Aa\), the lax monoidal structure on $\Afused$ is given by
\begin{center}
\begin{align}\label{eq:LaxMonStructureAfused}
\begin{tikzcd}[row sep=3.0ex]
    \Afused(b) \monprod[\Va] \Afused(c) = \SkCat{\SigmaFused}{\Aa}(b\rhd \Ofused, \Ofused) \monprod[\Va] \SkCat{\SigmaFused}{\Aa}(c\rhd \Ofused, \Ofused) 
    \\ \\ \SkCat{\SigmaFused}{\Aa}(b\rhd \Ofused, \Ofused) \monprod[\Va] \SkCat{\SigmaFused}{\Aa}(b\rhd c\rhd \Ofused, b\rhd \Ofused)
    \\ \\ \left(\Aa \coendtens{\Aa_1\catprod \Aa_2} \SkCat{\Sigma}{\Aa}\right) \left(b\otimes \OSigma, \OzPop\otimes \OSigma \right) \monprod[\Va]     \left(\Aa \coendtens{\Aa_1\catprod \Aa_2} \SkCat{\Sigma}{\Aa}\right) \left((b\otimes c) \otimes \OSigma, b\otimes \OSigma \right) 
    \\ \\
    \begin{tabular}{c}
    $\displaystyle \int^{(a_1, a_2), (a'_1, a'_2) \in \Aa_1\catprod\Aa_2} \Aa(b, \OzPop \lhd (a_1, a_2)) \monprod[\Va] \Aa(b\otimes c, b \lhd (a'_1, a'_2)) $
    \\ $\monprod[\Va]  \SkCat{\Sigma}{\Aa}((a_1, a_2) \rhd \OSigma, \OSigma) \monprod[\Va] \SkCat{\Sigma}{\Aa} ((a'_1, a'_2) \rhd \OSigma, \OSigma) $
    \end{tabular}
    \\  \\
    \begin{tabular}{c}
    $\displaystyle \int^{(a_1, a_2), (a'_1, a'_2) \in \Aa_1\catprod\Aa_2} \Aa\left(b \lhd \left(a'_1, a'_2\right), \OzPop \lhd \left(a_1, a_2\right)\lhd\left(a'_1, a'_2\right)\right) \monprod[\Va] \Aa\left(b\otimes c, b \lhd (a'_1, a'_2)\right) $
    \\ $\monprod[\Va] \SkCat{\Sigma}{\Aa}\left((a_1, a_2) \rhd \OSigma, \OSigma\right) \monprod[\Va] \SkCat{\Sigma}{\Aa} \left((a_1, a_2) \rhd (a'_1, a'_2) \rhd \OSigma, (a_1, a_2)\rhd\OSigma\right)$
    \end{tabular}
    \\ \\ 
    \displaystyle \int^{(a_1, a_2), (a'_1, a'_2) \in \Aa_1\catprod\Aa_2} \Aa(b\otimes c, \OzPop \lhd (a_1, a_2) \lhd(a'_1, a'_2)) \monprod[\Va] \SkCat{\Sigma}{\Aa}((a_1, a_2)\rhd (a'_1, a'_2) \rhd \OSigma, \OSigma)
    \\ \\  \displaystyle \int^{(a_1, a_2), (a'_1, a'_2) \in \Aa_1\catprod\Aa_2} \Aa(b\otimes c, \OzPop \lhd (a_1\otimes a'_1, a_2\otimes a'_2)) \monprod[\Va] \SkCat{\Sigma}{\Aa}((a_1\otimes a'_1, a_2\otimes a'_2) \rhd \OSigma, \OSigma)
    \\ \\  \displaystyle \int^{(a , a') \in \Aa_1\catprod \Aa_2} \Aa(b\otimes c, \OzPop \lhd (a\otimes a')) \monprod[\Va] \SkCat{\Sigma}{\Aa}(a \otimes a' \rhd \OSigma, \OSigma) = \Afused(b\otimes c).
  \arrow["{\id \monprod[\Va] b\,\rhd - }" ', from=1-1, to=3-1]
  \arrow["\ref{thm:Excision}", "{\cong}" ', from=3-1, to=5-1]
  \arrow["{\eqref{eq:CoendRelProd}}" , from=5-1, to=7-1]
  \arrow["{-\lhd(a'_1, a'_2) \monprod[\Va] \id \monprod[\Va] \id \monprod[\Va] (a_1, a_2)\rhd -}"', from=7-1, to=9-1]
  \arrow["{\circ_\Aa \monprod[\Va] \circ_{\SkCat{\Sigma}{\Aa}}}"', from=9-1, to=11-1]
  \arrow["\eqref{eq:ActTwiceToMult}", "{\id \catprod \beta \catprod \id \monprod[\Va] \id}"', from=11-1, to=13-1]
  \arrow["{ a, a' \mapsto a \otimes a'}"', from=13-1, to=15-1]
\end{tikzcd}
\end{align}
\end{center}
The first step corresponds to the first step of Equation \eqref{eq:lax_mon_str_intendoalg}. Afterwards we need to compose, and for this we first use \( \SigmaFused = \diskpoc \amalg_{\mathbb D_1\amalg \mathbb D_2} \Sigma\) and Theorem \ref{thm:Excision} to write the Hom-objects of \(\SkCat{\SigmaFused}{\Aa}\) as the coend relative product. In the third step we use Equation \eqref{eq:CoendRelProd} and the fact that \(\monprod[\Va]\) is a left adjoint to commute it with the coend. The fourth and fifth step follow from Equation \eqref{eq:CompositionCoendRelProd}, i.e.\ from how the coend relative tensor product composes. In the sixth step we have two objects of \(\Aa_1 \catprod \Aa_2\) acting on \(\Aa\), and we use Equation \eqref{eq:ActTwiceToMult} to relate this to the action of their tensor product. Finally, in the last step we define a map out of the coend by saying what it does on components: namely send them to the components of the bottom coend corresponding to \(a\otimes a' = (a_1, a_2) \otimes (a_1', a_2')\). 

\nparagraph{\textit{Algebra structure of \(\ASigmaHat\).}} 
We now unravel the lax monoidal structure of \(\ASigmaHat\) in a few steps. 
More explicitly, we first examine the algebra structure of \(\ASigmaHat\) which comes from the algebra structure of \(\ASigma\). Then we use Lemma \ref{lemma:LaxMonVsAlgStructure} to obtain the corresponding lax monoidal structure. 

Since the category \(\Aa_1 \catprod \Aa_2\) acts on \(\SkCat{\Sigma}{\Aa}\) we know from Section \ref{ssec:IntEndAlg} that this gives rise to the internal endomorphism algebra \(\ASigma \in \widehat{\Aa}_1 \boxtimes \widehat{\Aa}_2\) for the distinguished object \(\OSigma\). That is, \(\ASigma\) is a lax monoidal functor in \(\widehat{\Aa}_1 \boxtimes \widehat{\Aa}_2\), where Equation \eqref{eq:lax_mon_str_intendoalg} gives that for any two objects \((a_1, a_2), (a'_1, a'_2) \in \Aa_1 \catprod \Aa_2\) the lax monoidal structure is
\begin{center}
\begin{align}\label{eq:ASigmaLaxMon}
\begin{tikzcd}
	\ASigma(a_1, a_2) \monprod[\Va] \ASigma(a'_1, a'_2) = \SkCat{\Sigma}{\Aa} ((a_1, a_2)\rhd \OSigma, \OSigma) \monprod[\Va] \SkCat{\Sigma}{\Aa}((a'_1, a'_2)\rhd \OSigma, \OSigma)
	\\ \SkCat{\Sigma}{\Aa} ((a_1, a_2)\rhd \OSigma, \OSigma) \monprod[\Va] \SkCat{\Sigma}{\Aa}((a_1, a_2)\rhd (a'_1, a'_2) \rhd \OSigma, (a_1, a_2)\rhd \OSigma)
	\\ \SkCat{\Sigma}{\Aa}((a_1, a_2) \rhd (a'_1, a'_2) \rhd \OSigma, \OSigma) \cong \ASigma((a_1, a_2)\otimes (a'_1, a'_2))
	\arrow[from=1-1, to=2-1, "{\id \monprod[\Va] (a_1, a_2)\,\rhd -}"']
	\arrow[from=2-1, to=3-1, "{\circ_{\SkCat{\Sigma}{\Aa}}}"']
\end{tikzcd}
\end{align}
\end{center}

By Lemma \ref{lemma:LaxMonVsAlgStructure} the corresponding algebra structure is a map \(\mu_{\ASigma} \colon \ASigma \widehat{\otimes}_{\Aa_1 \catprod \Aa_2} \ASigma \ra \ASigma\), where \(\widehat{\otimes}_{\Aa_1 \catprod \Aa_2}\) is the completion of the monoidal product of \(\Aa_1 \catprod \Aa_2\) from  Equation \eqref{eq:MonProdC}. The corresponding algebra structure on \(\ASigmaHat\) is given by 
\begin{align}\label{eq:AlgStructureASigmaHat}
	\left(\ASigmaHat \right) \widehat{\otimes}\ \left(\ASigmaHat \right) \xrightarrow{\widehat{\id \catprod \beta \catprod \id}} \widehat{\otimes} \left( \ASigma \widehat{\otimes}_{\Aa_1 \catprod \Aa_2} \ASigma \right) 
	\xrightarrow{\widehat{\otimes}(\mu_{\ASigma}) } \widehat{\otimes} \ASigma,
\end{align}
where \(\widehat{\otimes}=\widehat{\otimes}_{\Aa}\) is the completed tensor product of \(\Aa\). We now use Lemma \ref{lemma:LaxMonVsAlgStructure}, or more explicitly Equation \eqref{eq:AlgToLaxFunctor}, to get the corresponding lax monoidal structure on \(\ASigmaHat\). 
For any objects \(b, c \in \Aa\) consider the computation in Equation~\eqref{eq:lax-mon-structure-otimes-internal-end-alg} below. 

The first step corresponds to the first step of Equation \eqref{eq:AlgToLaxFunctor}. In the second step we have used the dinatural transformation of the coend to get an expression we can directly compare to the lax monoidal structure of \(\Afused\) later. The third step is simply an application of the co-Yoneda lemma. Note that the coend expression we now have is exactly that of the left hand-side of Equation \eqref{eq:AlgStructureASigmaHat}, so we now first apply the braiding (combined with some identities) to get the coend corresponding to \(\widehat{\otimes} \left( \ASigma \widehat{\otimes}_{\Aa_1 \catprod \Aa_2} \ASigma \right) \). Then we use the algebra structure of \(\ASigma\) corresponding to the lax monoidal structure as spelled out in Equation \eqref{eq:ASigmaLaxMon} in step 5 and 6. As before, in the final step we have the induced map out of the coend coming from sending components of the coend to the components corresponding to \(a\otimes a'= (a_1, a_2)\otimes (a_1', a_2')\). 

\begin{center}
\begin{align}\label{eq:lax-mon-structure-otimes-internal-end-alg}
\begin{tikzcd}[row sep=3.0ex]
	\ASigmaHat(b) \monprod[\Va] \ASigmaHat(c)
	\\ \\ 
        \Aa(b\otimes c, b\otimes c) \monprod[\Va] \ASigmaHat(b) \monprod[\Va] \ASigmaHat(c)
	\\ \\ 
        \begin{tabular}{c}
        $\displaystyle \int^{b,c\in \Aa} \int^{(a_1,a_2), (a_1', a_2') \in \Aa_1\catprod \Aa_2} \Aa(b\otimes c, b\otimes c) \monprod[\Va] \Aa(b, a_1\otimes a_2) \monprod[\Va] \Aa(c, a_1'\otimes a_2') $
	\\ $\monprod[\Va] \SkCat{\Sigma}{\Aa}\left( \left(a_1,a_2\right)\rhd \OSigma, \OSigma \right)  \monprod[\Va] \SkCat{\Sigma}{\Aa}\left( \left(a'_1, a'_2\right) \rhd \OSigma, \OSigma \right)$
        \end{tabular}
	\\ \\ 
        \displaystyle \int^{(a_1,a_2), (a'_1, a'_2) \in \Aa_1 \catprod \Aa_2}  \Aa(b\otimes c, a_1\otimes a_2 \otimes a'_1 \otimes a'_2) \monprod[\Va] \SkCat{\Sigma}{\Aa}\left( \left(a_1,a_2\right)\rhd \OSigma, \OSigma \right)  \monprod[\Va] \SkCat{\Sigma}{\Aa}\left( \left(a'_1, a'_2\right) \rhd \OSigma, \OSigma \right)
	\\ \\ 
        \displaystyle \int^{(a_1,a_2), (a'_1, a'_2) \in \Aa_1 \catprod \Aa_2} \Aa(b\otimes c, a_1\otimes a'_1 \otimes a_2 \otimes a'_2) \monprod[\Va] \SkCat{\Sigma}{\Aa}\left( \left(a_1,a_2\right)\rhd \OSigma, \OSigma \right)  \monprod[\Va] \SkCat{\Sigma}{\Aa}\left( \left(a'_1, a'_2\right) \rhd \OSigma, \OSigma \right)
	\\ \\ 
        \begin{tabular}{c}
        $\displaystyle \int^{(a_1,a_2), (a'_1, a'_2) \in \Aa_1 \catprod \Aa_2} \Aa(b\otimes c, a_1\otimes a'_1 \otimes a_2 \otimes a'_2) \monprod[\Va] \SkCat{\Sigma}{\Aa}\left( \left(a_1,a_2\right)\rhd \OSigma, \OSigma \right)  $
	\\ $\monprod[\Va] \SkCat{\Sigma}{\Aa}\left( (a_1,a_2)\rhd \left(a'_1, a'_2\right) \rhd \OSigma, (a_1,a_2) \rhd \OSigma \right)$
        \end{tabular}
	\\ \\ 
        \displaystyle \int^{(a_1,a_2), (a'_1, a'_2) \in \Aa_1 \catprod \Aa_2} \Aa(b\otimes c, a_1\otimes a'_1 \otimes a_2 \otimes a'_2) \monprod[\Va] \SkCat{\Sigma}{\Aa}\left( (a_1 \otimes a'_1,a_2 \otimes a'_2) \rhd \OSigma, \OSigma \right)
	\\ \\ 
        \displaystyle \int^{(a , a')\in \Aa_1 \catprod \Aa_2} \Aa(b\otimes c, a \otimes a') \monprod[\Va] \SkCat{\Sigma}{\Aa}\left( a\otimes a' \rhd \OSigma, \OSigma \right) =  \ASigmaHat(b\otimes c).
	\arrow["{1_\Va \monprod[\Va] \id \monprod[\Va] \id}"', from=1-1, to=3-1]
	\arrow["{\iota_{b\otimes c}}"', from=3-1, to=5-1]
	\arrow["{\textrm{co-Yoneda}}", "\cong"', from=5-1, to=7-1]
	\arrow["{\id \catprod \beta \catprod \id \monprod[\Va] \id \monprod[\Va] \id}"', from=7-1, to=9-1]
	\arrow["{\id \monprod[\Va] \id \monprod[\Va] (a_1,a_2)\rhd -}"', from=9-1, to=11-1]
	\arrow["{\id \monprod[\Va] \circ_{\SkCat{\Sigma}{\Aa}}}"', from=11-1, to=13-1]
	\arrow["{a, a' \mapsto a\otimes a'}"',  from=13-1, to=15-1]
\end{tikzcd}
\end{align}
\end{center}

Upon inspection we see that the lax monoidal structure of \(\ASigmaHat\) is equivalent to the lax monoidal structure of \(\Afused\) as given in Equation \eqref{eq:LaxMonStructureAfused}. In particular, since the braiding is only applied to the Hom-objects of \(\Aa\) and the rest involves manipulating the Hom-objects of the skein category of \(\Sigma\) it does not matter that the order of doing so is opposite to that of Equation \eqref{eq:LaxMonStructureAfused}. \qedhere

\end{proof}

\begin{remark} \label{rem:fusion-to-ap2-cat-teaser}
    When we apply the above result to an $\Pois_2$-category, we will recover a generalization of fusion of quasi-Poisson structures of \cite{AKSM}. 
    We do this in Section~\ref{sec: Fusion and Poisson}.
\end{remark}

\newcommand{\SigmaTilde}{\widetilde{\Sigma}}
\newcommand{\VTilde}{\widetilde{V}}

\subsection{Intermezzo: \texorpdfstring{$\Ce$}{Ce}-modules and natural transformations}\label{sec:Cemod}
In this section we gather results about \(\Ce\)-modules and \(\Ce\)-natural transformations needed to make Remark \ref{rem:fusion-to-ap2-cat-teaser} precise, i.e.\ to apply the fusion result from Proposition~\ref{prop:MarkedFusionInternalEndAlgs} to \(\Pois_2\)-categories. 

Let $M$ be a $\Ce$-module. Choosing a splitting (in the category of complex vector spaces) of the short exact sequence
\begin{equation}
    \label{eq:SESCeMod}0 \to \operatorname{Ker} \varepsilon \cdot \to M \xrightarrow{\varepsilon \cdot} \varepsilon M \to 0 
\end{equation}
gives an isomorphism $M \cong \varepsilon M \oplus \operatorname{Ker} \varepsilon $. The action of $\varepsilon$ is the inclusion $\varepsilon M \hookrightarrow \operatorname{Ker}\varepsilon \cdot$.

Let us denote $M_0 \coloneqq M \otimes_\Ce \C  \cong M/\varepsilon M$. A $\Ce$-module is \emph{free} (isomorphic to $M_0\otimes_\C \Ce$) iff
$\varepsilon M = \operatorname{Ker} \varepsilon \cdot$.\footnote{This is easy to see from $M \cong \varepsilon M \oplus \operatorname{Ker} \varepsilon \cdot$. Yet another equivalent condition to freeness is flatness: if $M$ is flat, then the inclusion $\C \xrightarrow{\varepsilon \cdot} \Ce$ gives an inclusion $M/\varepsilon M \cong M\otimes_\Ce \C \xrightarrow{\varepsilon \cdot} M$; this being an inclusion means that $\operatorname{Ker}\varepsilon\cdot = \varepsilon M$.}  
The canonical isomorphism $\varepsilon M \cong M/\operatorname{Ker}\varepsilon \cdot$ from \eqref{eq:SESCeMod} becomes $\varepsilon M \cong M/\varepsilon M = M_0$. We now introduce notation for these maps. 

\begin{definition}
    Let $M$ be a free $\Ce$-module. Denote by $[-]_0\colon M \to M_0 = M/\varepsilon M$ the canonical projection map and by $[-]_1 \colon \varepsilon M \to M_0$ the canonical isomorphism from \eqref{eq:SESCeMod}. In other words, for $m\in \varepsilon M$, we define $[m]_1 := [n]_0$ in $M_0$ for any $n\in M$ such that $m = \varepsilon n$.
\end{definition}

If we choose a splitting $M \cong M_0 \otimes_\C \Ce$, the map $[-]_1$ sends $\varepsilon m_0 $ to $ m_0$. This is well-defined since different choices of splittings $M \cong M_0 \otimes_\C \Ce$ differ by automorphisms of $M_0 \otimes_\C \Ce$  of the form $\id + \lambda_{M} \otimes \varepsilon\cdot$, where $\lambda_M \colon M_0 \to M_0$. 

\begin{observation}
    If we have a $\Ce$-linear operation $\circ \colon M\otimes_\Ce M' \to M''$ between free $\Ce$-modules it follows that $[a\circ a']_0= [a]_0 \circ [a']_0$, for $a\in M$ and $a'\in M'$. Moreover, if $[a\otimes a']_0 = [a]_0 \otimes [a']_0 = 0$ holds we have that $[a\circ a']_1 = [a]_0 \circ_0 [a']_1 + [a]_1\circ_0 [a']_0$ (if only $[a\circ a']_0 = 0$ holds, then the previous equation has an additional term $[a]_0\circ_1 [a']_0$).
\end{observation}

\begin{observation}
If $M$ and $N$ are both free modules, then $\Hom_{\Ce\lmod}(M, N)$ is free again, with $\Hom_{\Ce\lmod}(M, N)_0 \cong \Hom_{\C\lmod}(M_0, N_0)$. Thus, for a map $f \colon M \to N$, if $[f]_0=0$ we can define $[f]_1\colon M_0 \to N_0$.\footnote{More concretely, choosing splittings of $M$ and $N$ defines maps $[f]_0, [f]_1\colon M_0 \to N_0$ by
$f = [f]_0 + \varepsilon [f]_1$. Changing the splittings of $M$ and $N$ by $\lambda_M \colon M_0\to M_0$ and $\lambda_N \colon N_0 \to N_0$ keeps $[f]_0$ is unchanged, while $[f]_1$ changes to 
\[ [f]_1 \mapsto [f]_1 + \lambda_N \circ [f]_0 - [f]_0 \circ \lambda_M.\] Thus, if $[f]_0 = 0$, then $[f]_1$ is well defined. } 
If we have an additional map  $g \colon N \to O$ such that $[g\circ f]_0= [g]_0 \circ [f]_0=0$ it follows that $[g\circ f]_1 = [g]_0 \circ [f]_1 + [g]_1\circ [f]_0$.
\end{observation}

We now turn towards \(\Ce\)-linear categories, and start with a basic definition. 
\begin{definition}\label{defn:has-free-hom-spaces}
    Let \(\Ba_\varepsilon\) be a \(\Ce\)-linear category. 
    We say that $\Ba_\varepsilon$ \emph{has free Hom-spaces} if $\Hom_{\Ba_\varepsilon}(b_1, b_2)$ is a free $\Ce$-module for all objects $b_1, b_2\in \Ba_\varepsilon$.
\end{definition}
Let \(\Aa_\varepsilon\) be a \(\Ce\)-linear category and let \(\Ba_\varepsilon\) be a \(\Ce\)-linear category which has free Hom-spaces as in Definition~\ref{defn:has-free-hom-spaces}. 
Let $F_\varepsilon, G_\varepsilon \colon \Aa_\varepsilon \to \Ba_\varepsilon$ be two $\Ce$-linear functors. 
Our goal is to study the $\Ce$-module of natural transformations $\operatorname{Nat}(F_\varepsilon, G_\varepsilon)$. Denote by $\Aa_0$ the $\C$-linear category obtained from $\Aa_\varepsilon$ by applying $-\otimes_\Ce \C$ on Hom-spaces, and by $F_0$ the functor obtained by applying $-\otimes_\Ce \C$ on maps between Hom-spaces.

\begin{remark}
    When taking $\Aa_0$ one has to be careful e.g.\ when considering presheaves. For example, one might think that $(\Ce\lmod)_0$ is equivalent to $\C\lmod$. However, this is \emph{false}. Consider the $\Ce$-modules $\C$ and $\Ce$ as objects of $(\Ce\lmod)_0$. It is straightforward to see that $\Hom_{(\Ce\lmod)_0}(\C,\Ce)\cong \Hom_{(\Ce\lmod)_0}(\Ce,\C) \cong \C $. Nevertheless, the composition map 
    \begin{equation}
        \Hom_{(\Ce\lmod)_0}(\Ce,\C)\otimes \Hom_{(\Ce\lmod)_0}(\C,\Ce) \to \Hom_{(\Ce\lmod)_0}(\C,\C)
    \end{equation}
    is 0. There are no $\C$-vector spaces which behave in the same way and hence $(\Ce\lmod)_0$ can not be equivalent to $\C\lmod$.

    In fact, the functor $(-)\otimes_{\Ce} \C \colon \Ce\lmod\to \C\lmod$ factors through $(\Ce\lmod)_0$. It sends all morphism in $\Hom_{(\Ce\lmod)_0}(\C,\Ce)$ to zero, but is an isomorphism on $\Hom_{(\Ce\lmod)_0}(\Ce,\C)$. However, for the subcategory of free $\Ce$-modules the functor $(\Ce\lmod^\textnormal{free})_0\to \C\lmod$ is an equivalence.
\end{remark}

\begin{definition}\label{def:sqbr1Nat}
    Let $\alpha\colon F_\varepsilon \Rightarrow G_\varepsilon$ be a natural transformation. Then $[\alpha]_0\colon F_0 \Rightarrow G_0$ is defined on components by 
    \begin{equation}
        ([\alpha]_0)_a = [\alpha_a]_0, \quad \forall a \in \Aa_0 \ . 
    \end{equation}
    If $[\alpha]_0$ is zero and $\Ba_\varepsilon$ has free Hom-spaces, we define $[\alpha]_1 \colon F_0 \Rightarrow G_0$ by
    \begin{equation} 
        ([\alpha]_1)_a = [\alpha_a]_1, \quad \forall a \in \Aa_0 \ .
    \end{equation}
\end{definition}

\begin{proposition} \label{prop:sqbracket01}
    In the setting of Definition~\ref{def:sqbr1Nat} both $[\alpha]_0$ and $[\alpha]_1$ are well-defined natural transformations. 
\end{proposition}
\begin{proof}
    Both $[\alpha]_0$ and $[\alpha]_1$ are well defined since their components are well defined. The naturality of $\alpha$ is the equation
    \begin{equation} 
    G_\varepsilon(f) \circ \alpha_a = \alpha_{a'} \circ F_\varepsilon(f)
    \end{equation}
    for $f\colon a \to a'$.
    Taking $[-]_0$ of this equation gives the naturality of $[\alpha]_0$. If $[\alpha]_0$ vanishes (i.e.\ all of its components vanish), then taking $[-]_1$ of the naturality constraint gives
    \begin{equation} 
        [G_\varepsilon(f)]_0 \circ_0 [\alpha_a]_1= [\alpha_{a'}]_1 \circ_0 [F_\varepsilon(f)]_0 
    \end{equation}
    and we can use that, by definition, $[F_\varepsilon(f)]_0 = F_0([f]_0)$ to get the naturality of $[\alpha]_1$. Thanks to the vanishing of $[\alpha]_0$ the possible $\varepsilon$-linear deformations of $F_\varepsilon$, $G_\varepsilon$ and $\circ^\Ba_\varepsilon$ have no effect.
\end{proof}

\begin{remark}
Put differently, we have the following exact sequence
\begin{equation}
\begin{tikzcd}
    0 \arrow[r] & {\operatorname{Nat}(F_0, G_0)} \arrow[r, "\varepsilon \cdot"] & {\operatorname{Nat}(F_\varepsilon, G_\varepsilon)} \arrow[r, "{[-]_0}"] & {\operatorname{Nat}(F_0, G_0)}
\end{tikzcd}
\end{equation}
and $[-]_1$ is the induced isomorphism
\begin{equation} 
    \operatorname{Ker}[-]_0 \xrightarrow{\cong} \operatorname{Nat}(F_0, G_0) \ . 
\end{equation}
The image of the map $[-]_0\colon \operatorname{Nat}(F_\varepsilon, G_\varepsilon)_0 \to \operatorname{Nat}(F_0, G_0)$ is given by natural transformations  $\operatorname{Nat}(F_0, G_0)$ which can be extended to the first order. In other words, the injective map $\operatorname{Nat}(F_\varepsilon, G_\varepsilon)_0 \to {\operatorname{Nat}(F_0, G_0)}$ is \emph{not} surjective in general. 
Moreover, the $\Ce$-module $\operatorname{Nat}(F_\varepsilon, G_\varepsilon)$ is in general \emph{not} free. 
\footnote{A example where $\operatorname{Nat}(F_\varepsilon, G_\varepsilon)$ is \emph{not} free is as follows: Consider two functors $F_\varepsilon, G_\varepsilon \colon (\bullet, f) \to \Ce\lmod^\text{free}$, where the source category $(\bullet, f)$ has one object $\bullet$ and is freely generated by one non-identity morphism $f$. Let $F_\varepsilon(\bullet) = G_\varepsilon(\bullet) = \Ce$, $F_\varepsilon(f)=0$ and $G_\varepsilon(f)= \varepsilon r$ for $r \neq 0$. All natural transformations $\alpha \colon F_\varepsilon \Rightarrow G_\varepsilon$ are given by the component $\alpha_\bullet = b\varepsilon$, with $b \in \C$ arbitrary. In other words, $\operatorname{Nat}(F_\varepsilon, G_\varepsilon) \cong \C$ as $\Ce$-modules.}
Nevertheless, in Proposition \ref{prop:sqbracket01} we were able to define $[-]_1$ on components.
\end{remark}

We will apply Proposition \ref{prop:sqbracket01} in two, confusingly related, settings: with $\alpha$ a natural transformation between functors $\Aa_\varepsilon^{\catprod n} \to \Aa_\varepsilon$, and $\alpha$ a morphism in $\widehat{\Aa_\varepsilon}^\textnormal{free}$, i.e.\ a natural transformation between functors $\Aa_\varepsilon \to \Ce\lmod^\textnormal{free}$.

\begin{proposition}\label{prop:sqbr1_functors}
    Let $F_\varepsilon$ be a $\Ce$-linear functor $\Aa_\varepsilon \to \Ba_\varepsilon$ and let $\sigma_\varepsilon\colon A_\varepsilon \to A'_\varepsilon$ be a morphism in $\widehat{\Aa_\varepsilon}^\textnormal{free}$. Assume that $\widehat{F_\varepsilon}(A_\varepsilon)$ and $\widehat{F_\varepsilon}(A'_\varepsilon)$ are valued in free $\Ce$-modules. 
    Then $(\widehat{F_\varepsilon}(A_\varepsilon))_0 \cong \widehat{F_0}(A_0)$ and $[\widehat{F_\varepsilon}(\sigma_\varepsilon)]_0 \cong \widehat{F_0}([\sigma_\varepsilon]_0)$. Moreover, if $[\sigma_\varepsilon]_0 = 0$ then  
    \begin{equation} 
        [\widehat{F_\varepsilon}(\sigma_\varepsilon)]_1 \cong \widehat{F_0}([\sigma_\varepsilon]_1) \ . 
    \end{equation}
    Here, we are using the equivalence $(\Ce\lmod^\textnormal{free})_0\simeq\C\lmod$, i.e.\ defining  $A_0(a):= A_\varepsilon\otimes_{\Ce}\C$.
\end{proposition}
\begin{proof}
    For $b\in \Ba_\varepsilon$, we have a coend formula for $\widehat{F_\varepsilon}$ yielding
    \begin{equation}
        \widehat{F_\varepsilon}(A_\varepsilon)(b) = \int^{a\in \Aa_\varepsilon} \Hom_{\Ba_\varepsilon}(F_\varepsilon(b), a) \otimes_\Ce A_\varepsilon(a) \ . 
    \end{equation}
    Applying $[-]_0 = - \otimes_\Ce \C$ gives
     \begin{equation}
        (\widehat{F_\varepsilon}(A_\varepsilon))_0(b) = \int^{a\in \Aa_\varepsilon} \Hom_{\Ba_\varepsilon}(F_\varepsilon(b), a)_0 \otimes_\C A(a)_0
     = \int^{a\in \Aa_0} \Hom_{\Ba_0}(F_0(b), a) \otimes_\C A_0(a) = \widehat{F_0}(A_0) (b) \, 
     \end{equation}
     where the first equality follows by expressing the coend as a conical colimit using Equation~\eqref{Eq: coend = coeq}. Then, we can move $[-]_0$ inside the colimit as it is a left adjoint. We also use that $\Aa_\varepsilon$ and $\Aa_0$ have the same objects. The proof that the two presheaves are equal on morphisms is similar, since they are defined component-wise as maps between coends.

    The morphism
    \begin{equation}  
        \widehat{F_\varepsilon}(A_\varepsilon) \xrightarrow{\widehat{F_\varepsilon}(\sigma_\varepsilon)} \widehat{F_\varepsilon}(A'_\varepsilon) 
    \end{equation}
    is specified by its components
    \begin{equation} 
        \Hom_{\Ba_\varepsilon}(F_\varepsilon(b), a) \otimes_\Ce A_\varepsilon(a) \xrightarrow{\text{incl}_a}\widehat{F_\varepsilon}(A_\varepsilon)(b) \xrightarrow{\widehat{F_\varepsilon}(\sigma_\varepsilon)_b} \widehat{F_\varepsilon}(A'_\varepsilon)(b) \ ,
    \end{equation}
    which is equal to 
    \begin{equation} 
        \Hom_{\Ba_\varepsilon}(F_\varepsilon(b), a) \otimes_\Ce A_\varepsilon(a) \xrightarrow{\id \otimes (\sigma_\varepsilon)_a}
     \Hom_{\Ba_\varepsilon}(F_\varepsilon(b), a) \otimes_\Ce A'_\varepsilon(a) \xrightarrow{\text{incl}_a} \widehat{F_\varepsilon}(A'_\varepsilon)(b) \ .
    \end{equation}
    Taking $[-]_1$ of this composition, we get that 
    \begin{equation} 
        \Hom_{\Ba_0}(F_0(b), a) \otimes_\C A_0(a) \xrightarrow{[\text{incl}_a]_0}\widehat{F_0}(A_0)(b) \xrightarrow{[\widehat{F_\varepsilon}(\sigma_\varepsilon)_b]_1} \widehat{F_0}(A'_0)(b) \ . 
    \end{equation}
    is equal to
    \begin{equation} 
        \Hom_{\Ba_0}(F_0(b), a) \otimes_\C A_0(a) \xrightarrow{\id \otimes [(\sigma_\varepsilon)_a]_1}
     \Hom_{\Ba_0}(F_0(b), a) \otimes_\C A'_0(a) \xrightarrow{\text{incl}_a} \widehat{F_0}(A'_0)(b)\ . 
    \end{equation}
    These are components of $(\widehat{F_0}([\sigma_\varepsilon]_1))_b$, since $[\text{incl}_a]_0$ is the inclusion into the coend in $\Aa_0$. The proof for taking $[-]_0$ works analogously.
\end{proof}

\begin{proposition}\label{prop:sqbr1_nattr}
    Let $\Aa_\varepsilon$, $\Ba_\varepsilon$ be $\Ce$-linear categories such that $\Ba_\varepsilon$ has free Hom-spaces, let $F_\varepsilon, G_\varepsilon \colon \Aa_\varepsilon \to \Ba_\varepsilon$ be $\Ce$-linear functors and let $\tau_\varepsilon \colon F_\varepsilon \Rightarrow G_\varepsilon$ be a natural transformation. 
    Let $A_\varepsilon \in \widehat{\Aa_\varepsilon}$ such that $\widehat{F_\varepsilon}(A_\varepsilon)$ and $\widehat{G_\varepsilon}(A_\varepsilon)$ are valued in $\Ce\lmod^\textnormal{free}$. Then 
    \begin{equation} 
        [(\widehat{\tau_\varepsilon})_{A_\varepsilon}]_0 = (\widehat{[\tau]_0})_{A_0} 
    \end{equation}
    and if $[\tau]_0 = 0$, then
    \begin{equation} 
        [(\widehat{\tau_\varepsilon})_{A_\varepsilon}]_1 = (\widehat{[\tau]_1})_{A_0} \ . 
    \end{equation}
\end{proposition}
\begin{proof}
    The proof is analogous to that of Proposition~\ref{prop:sqbr1_functors} as the components of both sides can be computed by coends.  
\end{proof}

We now turn towards the main result in this section which uses the preliminary results on \(\Ce\)-linear categories and natural transformations in the setting of skein categories. 

\begin{proposition} \label{prop:freeskeinmodules}
    Let $\Sigma$ be a surface with non-empty boundary and $\Aa_\varepsilon$ be an $\Pois_2$-category such that
    \begin{itemize}
        \item as a category, $\Aa_\varepsilon \simeq \Aa_0 \otimes_\C \Ce$, and
        \item the monoidal product functor $\Aa_\varepsilon \catprod \Aa_\varepsilon \to \Aa_\varepsilon$ is a free $\Ce$-linear extension of the monoidal product on $\Aa_0$.
    \end{itemize}
    Then the Hom-spaces in the skein category $\SkCat{\Sigma}{\Aa_\varepsilon}$ are free $\Ce$-modules.
\end{proposition}

The above assumptions on $\Aa_\varepsilon$ are satisfied for the $\Pois_2$-categories constructed out of infinitesimally-braided monoidal categories, c.f.\ Example \ref{ex:iBMC}, as well as the infinitesimal version of the category of representations of the Drinfeld-Jimbo quantum group, c.f.\ Example \ref{Ex: P_2 from r}, by the Drinfeld-Kohno equivalence \cite[{Corollary~XIX.4.4}]{Kassel}. We collect a corollary of Proposition~\ref{prop:freeskeinmodules} before giving the proof.

\begin{corollary}\label{cor:IEAfree}
    Let $(\Sigma, V)$ be a marked surface and let $\Aa_\varepsilon$ be as in Proposition \ref{prop:freeskeinmodules}.
    Then the internal endomorphism algebra $\intEndAlg{\Sigma}{V}\in \widehat{\Aa_\varepsilon^{\catprod V}}$ is valued in $\Ce\lmod^\textnormal{free}$. 
    The same is true for $\widehat{\otimes} \intEndAlg{\Sigma}{V}$ if we apply the tensor product functor $\widehat{\otimes}\colon \widehat{\Aa_\varepsilon \catprod \Aa_\varepsilon} \to \widehat{\Aa_\varepsilon}$ on any pair of $\Aa_\varepsilon$-factors. 
    Finally, given two such internal endomorphism algebras $\intEndAlg{\Sigma_1}{V_1}$ and \(\intEndAlg{\Sigma_2}{V_2}\), the functor \(\intEndAlg{\Sigma_1}{V_1} \boxtimes \intEndAlg{\Sigma_2}{V_2}\in \widehat{\Aa_\varepsilon^{\catprod V_1 \sqcup V_2}}\) is also valued in $\Ce\lmod^\textnormal{free}$.
\end{corollary}
\begin{proof}
        The internal endomorphism algebra from Equation~\eqref{eq:IEA} evaluated on an object $a$ is given by a skein module $\SkCat{\Aa}{\Sigma}(a \triangleright \mathcal O, \mathcal O)$, which is free by Proposition \ref{prop:freeskeinmodules}.
        
        The claims for $\widehat{\otimes} \intEndAlg{\Sigma}{V}$ and $\intEndAlg{\Sigma_1}{V_1}\times \intEndAlg{\Sigma_2}{V_2} \in \widehat{\Aa_\varepsilon^{\catprod V_1 \sqcup V_2}}$ follow from the first one, as they correspond to internal endomorphism algebras of a fused and disjoint union of surfaces, respectively. See Proposition \ref{prop:MarkedFusionInternalEndAlgs} and Proposition \ref{prop:disksbuilidingblock}.\footnote{This proposition appears in the later section for organizational reasons, and does not depend on the freeness result we are proving here.} 
\end{proof}
\begin{proof}[{Proof of Proposition \ref{prop:freeskeinmodules}}]
    We know from Theorem \ref{thm:Excision} that the skein category for a surface can be computed using the relative tensor product. Since the boundary of our surface is non-empty, we can represent it as a gluing of a collection of $N$ disks over a collection of $M$ disks. Using the coend formula \eqref{eq:CoendRelProd} for Hom-spaces of the skein category, we see that the Hom-spaces are given by a colimit over a diagram in $\Ce\lmod$. The functor of which we compute the coend uses the action of $\Aa^{\catprod M}$ on $\Aa^{\catprod N}$, given by tensoring of objects and morphisms. However, it does not use the natural isomorphisms $a'\triangleright(a''\triangleright b) \cong (a'\otimes a'')\triangleright b$ which use the (potentially $\varepsilon$-deformed) braiding of $\Aa$. 
    
    The coend can be expressed as a colimit of the diagram \eqref{Eq: coend = coeq}. By assumption, the category $\Aa$ and the monoidal product are free extensions of a $\C$-linear monoidal category. Thus, we can see this diagram as an image of a diagram in $\C\lmod$ via the functor $\C \lmod \xrightarrow{\otimes \Ce} \Ce \lmod$. 
    This functor is left adjoint, therefore it preserves colimits and we can thus compute the Hom-space in $\C\lmod$ and apply $- \otimes \Ce$.
\end{proof}

\begin{remark}
    The proof of Proposition \ref{prop:freeskeinmodules} shows that surfaces with boundary are, heuristically, one-dimensional enough so that the skein Hom-spaces only depend on the underlying category and the monoidal product. By Equation~\eqref{eq:CompositionCoendRelProd} the braiding and the associator appears in the isomorphisms $a'\triangleright(a''\triangleright b) \cong (a'\otimes a'')\triangleright b$ defining the composition of the skein category; the composition is therefore, in general, deformed. 

    The annulus is special in this regard, as it can be written as a gluing of two disks over two disks (i.e.\ the fatgraph representing the annulus has no vertices of valence $\ge 3$), and hence there is no braiding involved in the action. Thus, the skein category of infinitesimally braided categories over the annulus has undeformed composition implying that the endomorphism algebra is undeformed.
    \footnote{This is just the fact that factorization homology of an $\E_2$ algebra over $S^1\times \mathbb R$ is the Hochschild homology of any of the two (isomorphic) $\E_1$-structures \cite[Theorem~3.51]{HLTetal}.} 
    The internal endomorphism algebra is deformed in general since the action of $\Aa_\varepsilon$ on $\SkCat{\Ann}{\Aa_\varepsilon}$ uses the braiding
    
    The proof of the above statement also shows why the Hom-spaces of skein categories of closed surfaces are not free $\Ce$-modules in general -- see Example \ref{Ex: torsion}. There we need to take a relative tensor product over the skein category of the annulus. This category will have, in general, $\Ce$-deformed monoidal structure and action, and thus the assumptions of Proposition \ref{prop:freeskeinmodules} is not satisfied. 
\end{remark}

\subsection{Fusion and Poisson structures}\label{sec: Fusion and Poisson}
We will now explain how to recover Poisson structures on classical internal endomorphism algebras from the formalism developed in this article. 

Let $(\Aa_\varepsilon, \Aa_0, \Psi_\Aa) \in \PCat{2}$ and let $\IEAsymbol_{\varepsilon} \coloneqq \IEAlong{\Aa_\varepsilon}{\mathcal O} \in 
(\widehat\Aa_\varepsilon)^{\boxtimes V}$ and $\IEAsymbol_0 \coloneqq \IEAlong{\Aa_0}{\mathcal O} \in (\widehat\Aa_0)^{\boxtimes V}$ denote the internal endomorphism algebras associated to $(\Sigma, V \subset \partial \Sigma)$ and the actions of $\Aa_\varepsilon$ and $\Aa_0$ on the corresponding factorization homologies over $\Sigma$ (as in Section~\ref{sect:internal-endo-algebras-and-fusion}).
Explicitly, $(\IEAsymbol_\varepsilon, \mu)$ is a lax monoidal functor $\Aa_\varepsilon ^{\catprod V} \to \Ce\lmod$, where $\mu_{X, Y}\colon \IEAsymbol_\varepsilon(X) \otimes \IEAsymbol_\varepsilon(Y) \to \IEAsymbol_\varepsilon(X\otimes Y)$ is the lax structure. 

From Proposition~\ref{Prop: Defor skein alg} and Corollary~\ref{cor:IEAfree} we know that $\IEAsymbol_\varepsilon$ lands in $\Ce\lmod^\textnormal{free}$ and agrees, after setting $\varepsilon=0$, with the symmetric lax monoidal functor $\IEAsymbol_0\colon \Aa_0^{\catprod V} \to \C\lmod$. We can therefore measure the failure of $\IEAsymbol_\varepsilon$ being \emph{braided} lax monoidal using the operation $[-]_1$ from Definition \ref{def:sqbr1Nat}. 
\begin{definition}[{Generalizing \cite[Section~2]{SeveraCenters}\footnote{This definition is a direct generalization of \v{S}evera's $\sigma$-tensor defined for the infinitesimally braided category $U(\mathfrak{g})\lmod$, with $\mathfrak{g}$ being a metric Lie algebra.}}] \label{def:of_p}
    Let $(\Sigma, V)$ be a marked surface and $(\Aa_\varepsilon, \Aa_0, \Psi_\Aa) \in \PCat{2}$. 
    Define $\sigma_{\Sigma, V}$ as the natural transformation $\IEAsymbol_0 \widehat{\otimes} \IEAsymbol_0 \to \IEAsymbol_0$ given by
    \begin{equation}\label{eq:defPoiss}
        \sigma_{\Sigma, V} :=  [\mu_{\Sigma, V} - (\mu_{\Sigma, V})^{\op_-}]_1 :=[\mu_{\Sigma, V} - \mu_{\Sigma, V} \circ \widehat{\beta}^{-1}_{\IEAsymbol_\varepsilon, \IEAsymbol_\varepsilon}]_1 \ , 
    \end{equation}
    i.e.\ its \emph{components}\footnote{In general, a natural transformation $\alpha\colon \widehat{\bigotimes} F_i \to G$ can equivalently be encoded by its so-called components $\alpha_{X_1, \dots, X_n}\colon \bigotimes F_i(X_i) \to G(\bigotimes X_i)$ defined analogously to \eqref{eq:AlgToLaxFunctor}.}
    are
    \begin{equation}\label{eq:defPoisscomponents}
        (\sigma_{\Sigma, V})_{X, Y} := [\mu_{X, Y} - \IEAsymbol_\varepsilon(\beta^{-1}_{X, Y}) \circ \mu_{Y, X} \circ \beta_{\IEAsymbol_\varepsilon(X), \IEAsymbol_\varepsilon(Y)}]_1\colon \IEAsymbol_0(X)\otimes \IEAsymbol_0(Y) \to \IEAsymbol_0(X\otimes Y)
    \end{equation}
    for all $X, Y \in \Aa_0^{\catprod V}$. Here, $\beta^-_{X, Y} = (\beta_{Y, X}^{-1})$ denotes the components of the opposite braiding (i.e.\ undercrossing) of $\Aa_\varepsilon^{\catprod V}$, while $\beta_{\IEAsymbol_\varepsilon(X), \IEAsymbol_\varepsilon(Y)}$ are components of the symmetry of the enriching category $\Ce\lmod$.
\end{definition}
We often abbreviate \(\sigma_{\Sigma, V}\) as above by \(\sigma\). 
The maps $\sigma$ are analogous to the (almost) Poisson bracket defined by a star product as in Equation \eqref{Eq: DQ}. We now show that $\sigma$ is a biderivation but not necessarily an antisymmetric one.
Thus, its antisymmetrization will give an almost Poisson bracket; see Remark \ref{rem:furtherpropertiessigma} for a description of the symmetric part and Section \ref{sec:CS} for examples.
\begin{proposition}
    The natural transformation  $\sigma_{\Sigma, V} \colon \IEAsymbol_0 \Daytimes \IEAsymbol_0 \to \IEAsymbol_0$ is a derivative, in each of its arguments, of the commutative algebra $\IEAsymbol_0\in \widehat{\Aa_0}^{\boxtimes V}$.
\end{proposition}
\begin{proof}
We prove this by working component-wise. 
Since $\IEAsymbol_\varepsilon$ is a lax monoidal functor, we have that 
\begin{equation}
    \underbrace{\IEAsymbol_\varepsilon(\phi_{X, Y, Z}) \circ \mu_{X\otimes Y, Z} \circ \mu_{X, Y}\otimes \id_Z}_{A_{X, Y, Z}:=} - \underbrace{\mu_{X, Y\otimes Z} \circ \id_X\otimes \mu_{Y, Z} \circ \phi_{\IEAsymbol_\varepsilon(X), \IEAsymbol_\varepsilon(Y), \IEAsymbol_\varepsilon(Z)}}_{B_{X, Y, Z}:=}= 0 \ . 
\end{equation}
The following linear combination thus also vanishes: 
\begin{align*}
    A_{X, Y, Z} - &\underline{\underline{B_{X, Y, Z}}} \\
    -\IEAsymbol_\varepsilon \left( \pics{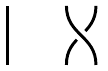}  \right) \circ (\underline{A_{X, Z, Y}} - & \underline{\underline{B_{X, Z, Y}}}) \circ  \pics{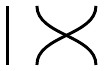}  \\
     +\IEAsymbol_\varepsilon \left( \pics{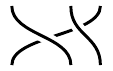}  \right) \circ ( \underline{A_{Z, X, Y}} - & B_{Z, X, Y}) \circ  \pics{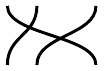} \ .  
\end{align*}
After taking $[-]_1$ of this equation\footnote{Of course, this is a version of the usual proof that the antisymmetrization of the $\hbar$-linear part of a star product is a derivative: Consider the vanishing expression 
\begin{align}
    (XY)Z - & \underline{\underline{X(YZ)}} \\
    -\underline{(XZ)Y} + & \underline{\underline{X(ZY)}}\\
    +\underline{(ZX)Y} - & Z(XY)
\end{align} and take the $\hbar$-linear part. The terms combine to give $\{X\cdot Y, Z\}-\{X, Z\}\cdot Y - X\cdot\{Y, Z\} = 0$.} 
the two underlined terms combine to $-(\mu_0)_{X\otimes Z, Y}\circ (\sigma_{\Sigma, V})_{X, Z} \otimes \id_{\IEAsymbol_0(Y)}$ up to a permutation and rebracketing which moves $Z$ next to $X$. Similarly, the double-underlined terms combine to $-(\mu_0)_{X, Y
\otimes Z}\circ\id_{\IEAsymbol_0(X)}\otimes (\sigma_{\Sigma, V})_{Y, Z}$ and the remaining two terms combine to $(\sigma_{\Sigma, V})_{X\otimes Y, Z} \circ (\mu_0)_{X, Y} \otimes \id_{\IEAsymbol_0(Z)}$. 
\end{proof}
An easy case, for which the first order deformation $\sigma$ vanishes, is the disk with two marked points on the boundary.

\begin{proposition}\label{prop:diskcomm}
    Let $\Sigma$ be a disk and $V=\{v_1, v_2\} \subset \partial \Sigma$ be two marked points on the boundary. Let $\Aa$ be a $\Va$-enriched ribbon category. Then the internal endomorphism algebra $\IEAsymbol \coloneqq \intEndAlg{\mathbb D}{V}$ seen as a functor
    \begin{equation}  
        (\Aa \catprod \Aa^{\textnormal{br.\,op}})^{\op} \to \Va 
    \end{equation}
    is braided lax monoidal. As a consequence, for $\Aa = \Aa_\varepsilon$ as in Definition \ref{def:of_p} we have that the natural transformation $\sigma_{\Sigma, V}$ vanishes.
\end{proposition}
\begin{proof}
    The lax structure $\mu$ on $\IEAsymbol\colon \Aa\catprod \Aa \to \Va$ is described in Figure \ref{fig:entire-figure-stacking} (for the annulus). It is braided if the following diagram commutes
    \begin{equation}
    \begin{tikzcd}[column sep = 1in]
	{\IEAsymbol(X_1, X_2)\otimes \IEAsymbol(Y_1, Y_2)} & {\IEAsymbol((X_1, X_2)\otimes (Y_1, Y_2))} \\
	{\IEAsymbol(Y_1, Y_2) \otimes \IEAsymbol(X_1, X_2)} & {\IEAsymbol((Y_1, Y_2)\otimes (X_1, X_2))}
	\arrow["{\mu_{(X_1, X_2), (Y_1, Y_2)}}", from=1-1, to=1-2]
	\arrow["{\IEAsymbol(\beta^{[\Aa\catprod \Aa^\textnormal{br.\,op}]^{\op}}_{(X_1, X_2), (Y_1, Y_2)})}", from=1-2, to=2-2]
	\arrow["{\beta_{\IEAsymbol(X_1, X_2), \IEAsymbol(Y_1, Y_2)}}"', from=1-1, to=2-1]
	\arrow["{\mu_{(Y_1, Y_2), (X_1, X_2)}}"', from=2-1, to=2-2]
\end{tikzcd} \ . 
\end{equation}
By convention, the braiding in the opposite category is given by overcrossings. Thus, 
on the level of skeins commutativity of this diagrams corresponds to the below equality which holds by isotopy invariance.
\begin{center}
  \begin{overpic}[scale=1.3,tics=10]{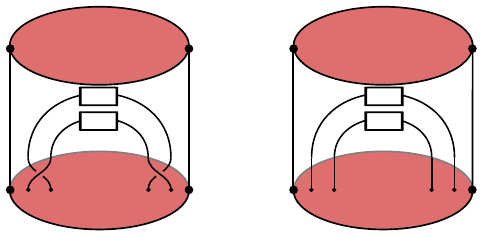}
  	\put(49,25){$=$}
     \put(4.8,6.4){$Y_1$}
     \put(9,6.4){$X_1$}
     \put(29,6.4){$X_2$}
     \put(33.8,6.4){$Y_2$}                                                        
     \put(19.3, 23.2){$g$}
     \put(19.3, 28){$f$}   
     \put(63,6.4){$Y_1$}
     \put(67.4,6.4){$X_1$}
     \put(87.4,6.4){$X_2$}
     \put(92.7,6.4){$Y_2$}                                                        
     \put(78.7, 23.0){$f$}
     \put(78.7, 28.3){$g$}
	\end{overpic} 
\end{center}
\end{proof}

The above proposition also appears in  \cite[Proposition~2.11]{safronovQMM} and \cite[Section~1.5]{GJS} phrased using the language of right adjoints to the tensor product. 
We will now show that the symmetric part of the braiding on $\Aa_\varepsilon$ induces an infinitesimal braiding on $\Aa_0$. 
\begin{definition}\label{def:braidingP2}
    Let $\Aa_\varepsilon$ be an $\Pois_2$-category such that the Hom-objects in $\Aa_\varepsilon$ are free $\Ce$-modules. Define the following set of maps 
    $t_{X, Y} \colon X\otimes Y \to X \otimes Y$ in $\Aa_0$ by
    \begin{equation}
    t_{X, Y} := [\beta^2_{X, Y} - \id_{X\otimes Y}]_1 \ . 
    \end{equation}
\end{definition}

\begin{proposition}
    The maps $t_{X,Y}$ are well-defined and define an infinitesimal braiding on $\Aa_0$. 
\end{proposition}
\begin{proof}
    The fact that $t_{X, Y}$ is well-defined and natural follows from the fact that $[\beta^2_{X, Y} - \id_{X\otimes Y}]_0$ vanishes. The symmetry of $t_{X,Y}$ follows from
    \[ \beta^{\Aa_0}_{Y, X} \circ t_{Y, X} \circ \beta^{\Aa_0}_{X, Y} = [ (\beta^{-1})^{\Aa_\varepsilon}_{Y, X}(\beta^2_{Y, X} - \id_{Y\otimes X}) \circ \beta^{\Aa_\varepsilon}_{X, Y} ]_1 = t_{X, Y}\]
    where, in the first equality, we could choose the powers of $\beta$ arbitrarily as different choices differ by terms which vanish modulo $\varepsilon$. Finally, that $t_{X\otimes Y, Z} \approx t_{X, Z} + t_{Y, Z}$ (see \cite[Eq.~(21)]{Cartier1993} for the precise equation) follows easily by using graphical calculus. 
\end{proof}
\begin{remark}\label{rem:furtherpropertiessigma}
    The infinitesimal braiding given by the maps \(t_{X, Y}\) can be used to express some useful properties of $\sigma \equiv \sigma_{\Sigma, V}$. Firstly, the symmetrization of $\sigma$ satisfies
    \begin{equation} 
        \sigma + \sigma\circ \beta^{\Aa_0}_{\IEAsymbol_0, \IEAsymbol_0} = \mu_0 \circ \widehat{t}_{\IEAsymbol_0, \IEAsymbol_0} \ , 
    \end{equation}
    or in components
    \begin{equation} 
        \sigma_{X, Y} + \IEAsymbol_0(\beta^{\Aa_0}_{X, Y})\circ\sigma_{Y, X} \circ \beta^{\C\lmod}_{\IEAsymbol_0(X), \IEAsymbol_0(Y)} = \IEAsymbol_0(t_{X, Y})\circ(\mu_0)_{X, Y} \ . 
    \end{equation}
    Similarly, one could define $\sigma^+$ by replacing the undercrossing in \eqref{eq:defPoiss} by an \emph{over}crossing. Then
    \begin{equation}
        \sigma^+ - \sigma = - \mu_0 \circ \widehat{t}_{\IEAsymbol_0, \IEAsymbol_0} \ . 
    \end{equation}
\end{remark}
Considering $\Aa \catprod \Aa$ instead of $\Aa \catprod \Aa^\textnormal{br.\,op}$ in Proposition \ref{prop:diskcomm}, we get the following easy observation.
\begin{corollary}
\label{prop:disksbuilidingblock}
    The internal endomorphism algebra
    $\IEAsymbol \colon{(\Aa\catprod \Aa)} \to  \Va$ assigned to a disk $\mathbb D$ with two marked points \(V=\{v_1, v_2\}\), has
    \begin{equation} \label{eq:sigmadisk}
        \sigma_{\mathbb D, \{v_1, v_2\}} = \mu_0 \circ (\, \widehat{t}_{24}) = \mu_0 \circ (\, \widehat{t}_{13}) = \frac 12 \mu_0 \circ (\, \widehat{t}_{13}) + \frac 12 \mu_0 \circ (\, \widehat{t}_{24})\ , 
    \end{equation}
    or in components
    \begin{equation}
        (\sigma_{\mathbb D, \{v_1, v_2\}})_{(X, X'), (Y, Y')} = \IEAsymbol_0(t_{X', Y'})\circ (\mu_0)_{(X, X'), (Y, Y')} \ . 
    \end{equation}
     For two marked surfaces $(\Sigma_1, V_1)$ and $(\Sigma_2, V_2)$, we have 
    \begin{equation}
        \intEndAlg{\Sigma_1\sqcup \Sigma_2}{V_1\sqcup V_2} = \intEndAlg{\Sigma_1}{V_1} \boxtimes \intEndAlg{\Sigma_2}{V_2} 
    \end{equation}
     as algebras.
\end{corollary}

Finally, let us turn to fusion. Explicitly, we will relate the maps $\sigma:= \sigma_{\Sigma, V = \{v_1, v_2\}}$ to the maps $\sigma_f := \sigma_{\Sigma^f, V=\{v\}}$, where $\Sigma^f$ is obtained by gluing $\Sigma$ along the small boundary intervals around $v_1$ and $v_2$ as illustrated in Figure~\ref{fig:FusionLocalPic}. 
Let us denote by $(\IEAsymbol, \mu) \in \operatorname{Alg}(\widehat{\Aa_\varepsilon \catprod \Aa_\varepsilon})$ the internal endomorphism algebra associated with the \emph{unfused surface} \(\Sigma\), and by $(\IEAsymbol_f, \mu_f)\in \operatorname{Alg}(\widehat{\Aa_\varepsilon})$ the internal endomorphism algebra assigned to the fused surface \(\SigmaFused\). 

We know from Proposition \ref{prop:MarkedFusionInternalEndAlgs} that $\IEAsymbol_f \cong \widehat{\otimes} \IEAsymbol$ and $\mu_f \cong \widehat{\otimes}(\mu) \circ J_{\IEAsymbol, \IEAsymbol}$, where $J$ is the natural transformation $[\widehat{\otimes}(\IEAsymbol)] \widehat{\otimes} [\widehat{\otimes}(\IEAsymbol)] \to \widehat{\otimes}(\IEAsymbol \widehat{\otimes}_{\widehat{\Aa \catprod \Aa}} \IEAsymbol)$ coming from the parenthesized tangle switching the middle two strands
\begin{equation}
    J := \vcenter{\hbox{\includegraphics[scale=0.7]{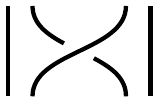}}} \ . 
\end{equation}
In other words, $J$ is a monoidal structure on the tensor product functor, and Proposition \ref{prop:MarkedFusionInternalEndAlgs} says that the algebra $(\IEAsymbol_f, \mu_f)$ is the image of the algebra $(\IEAsymbol, \mu)$ under the tensor product functor $(\widehat{\otimes}, J)$. 
It turns out that for the first-order deformations $\sigma$ and $\sigma_f$, such a relation is modified by a ``fusion'' term. 

\begin{theorem}\label{Thm: gluing Poisson}
    The first-order deformations $\sigma$ and $\sigma_f$ from Definition \ref{def:of_p} are related by
    \begin{equation}\label{eq:fusionIEA} 
        \sigma_f = \widehat{\otimes_0}(\sigma) \circ (J_0)_{\IEAsymbol_0, \IEAsymbol_0} + (\mu_f)_0 \circ \widehat{t}_{2,3} \ , 
    \end{equation}
    where $\widehat{t}_{2,3} = (\widehat{\id \otimes t \otimes \id})_{\IEAsymbol_0, \IEAsymbol_0}$ is the $\IEAsymbol_0\catprod \IEAsymbol_0$-component of the natural endomorphism of the 4-fold tensor product functor coming from the infinitesimal braiding of $\Aa_0$ acting on the second and the third component.
\end{theorem}
\begin{proof}
    Let $\widetilde{\sigma} \coloneqq \mu - \mu^{\op_-}$, i.e.\ $\sigma = [\widetilde{\sigma}]_1$, and similarly for $\widetilde{\sigma}_f \coloneqq \mu_f - \mu_f^{\op_-}$. Let us also denote the completed tensor product $\widehat{\otimes}$ by $T$ to lighten some of the notation; i.e.\ we have $T(\IEAsymbol) = \IEAsymbol_f$. Using Proposition~\ref{prop:MarkedFusionInternalEndAlgs} we can write
    \begin{equation}
    \label{eq:tsigmafused}
        \widetilde{\sigma}_f = T(\mu)\circ J_{\IEAsymbol,\IEAsymbol} - (T(\mu)\circ J_{\IEAsymbol, \IEAsymbol})^{\op_-} = \mu_f - (T(\mu)\circ J_{\IEAsymbol, \IEAsymbol})^{\op_-} \ , 
    \end{equation}
    where the $\op_-$ is computed in $\widehat{\Aa_\varepsilon}$. In analogue with the relation $\mu_f = T(\mu) \circ J_{\IEAsymbol,\IEAsymbol}$, let us compare Equation~\eqref{eq:tsigmafused} with 
    \begin{align}\label{eq:tsigmaunfused}
        T(\tilde{\sigma}) \circ J_{\IEAsymbol, \IEAsymbol} &= T(\mu - \mu^{\op_-}) \circ J_{\IEAsymbol, \IEAsymbol} \nonumber \\
        &= \mu_f - T(\mu^{\op_-})\circ J_{\IEAsymbol, \IEAsymbol} \ .
    \end{align}    
    We need to compare $\left(T(\mu)\circ J_{\IEAsymbol, \IEAsymbol}\right)^{\op_-}$ and $T(\mu^{\op_-})\circ J_{\IEAsymbol, \IEAsymbol}$, i.e.\ understand the compatibility of the monoidal functor $(T, J)$ with taking $\op_-$. The answer is the following interesting relation.\footnote{Since Equation~\eqref{eq:tensorandop} follows from an equality of braids, this is a general fact about $\E_2$-algebras in symmetric monoidal bicategories.} 
    Denote by $J^-$ the other canonical monoidal structure on $T$ given by undercrossing
    \begin{equation}
        J^- := \vcenter{\hbox{\includegraphics[scale=0.7]{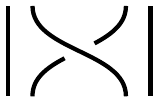}}} \ \  . 
    \end{equation}
    Then we have
    \begin{equation}\label{eq:tensorandop}
        \left(T(\mu)\circ J_{\IEAsymbol, \IEAsymbol}\right)^{\op_-} = T(\mu^{\op_-})\circ J^-_{\IEAsymbol, \IEAsymbol}
    \end{equation}
    which is proven simply by comparing the braids on both sides
    \begin{equation} 
        \vcenter{\hbox{\includegraphics[scale=0.7]{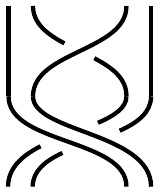}}} \quad = \quad \vcenter{\hbox{\includegraphics[scale=0.7]{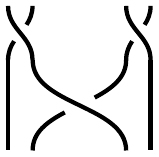}}} \ \  . 
    \end{equation}
    Therefore, the difference between Eq.\ \eqref{eq:tsigmafused} and Eq.\ \eqref{eq:tsigmaunfused} is only in the crossings of $J$ and $J^-$. 
    Concretely, we have
    \begin{align} \label{eq:fusedsigma}
        T(\tilde{\sigma})\circ J_{\IEAsymbol, \IEAsymbol} &= \mu_f - T(\mu^{\op_-})  \circ J_{\IEAsymbol,\IEAsymbol} \\
        &=\mu_f - T(\mu^{\op_-}) \circ (J^- + J - J^-)_{\IEAsymbol, \IEAsymbol} \\
        &=\mu_f - (T(\mu)\circ J_{\IEAsymbol, \IEAsymbol})^{\op_-} \;\;  - T(\mu^{\op_-}) \circ (J- J^-)_{\IEAsymbol, \IEAsymbol} \\
        &= \tilde{\sigma}_f - T(\mu^{\op_-}) \circ (J^-\circ (J^2- \id))_{\IEAsymbol, \IEAsymbol} \ , 
    \end{align}
    where we have used Equation~\eqref{eq:tensorandop} to get the third line and Equation~\eqref{eq:tsigmafused} to get the fourth line.

    We now take $[-]_1$ of both sides of \eqref{eq:fusedsigma} and use the results of Section \ref{sec:Cemod}. As all three terms of \eqref{eq:fusedsigma} vanish modulo $\varepsilon$, we get (using $[\widetilde{\sigma}_f]_1 = \sigma_f$) 
    \begin{equation} \label{eq:sigmafusionmiddlestep} 
        \sigma_f = [T(\widetilde{\sigma})\circ J_{\IEAsymbol, \IEAsymbol}]_1 + [T(\mu^{\op_-}) \circ (J^-\circ (J^2- \id))_{\IEAsymbol, \IEAsymbol} ]_1 \ .
    \end{equation}
    In $[T(\widetilde{\sigma})\circ J_{\IEAsymbol, \IEAsymbol}]_1$, the first term of the composition involving $\widetilde{\sigma}$ vanishes modulo $\varepsilon$. Using Proposition~\ref{prop:sqbr1_functors} we thus get that the first term is equal to $T_0(\sigma)\circ [J_{\IEAsymbol,\IEAsymbol}]_0$. Here, $T_0 = \widehat{\otimes_0}$ is the completion of the tensor product on $\Aa_0$. As the natural transformation $J$ is the extension of a natural transformation between uncompleted tensor products, we can use Proposition~\ref{prop:sqbr1_nattr} to get 
    \begin{equation}
        [T(\widetilde{\sigma})\circ J_{\IEAsymbol, \IEAsymbol}]_1 = T_0(\sigma)\circ (J_0)_{\IEAsymbol_0, \IEAsymbol_0} \ . 
    \end{equation}
    Similarly, in $[T(\mu^{\op_-}) \circ (J^-\circ (J^2- \id))_{\IEAsymbol, \IEAsymbol} ]_1$ (the second term of Equation~\eqref{eq:sigmafusionmiddlestep})
    the last term vanishes modulo $\varepsilon$. The first two terms give, using Proposition~\ref{prop:sqbr1_functors},
    $T_0(\mu_0) \circ (J_0)_{\IEAsymbol_0, \IEAsymbol_0}$, where we use the fact that $\mu_0$ is a symmetric monoidal functor. Finally, the last term satisfies $[(J^2-\id)_{\IEAsymbol, \IEAsymbol}]_1 = \widehat{(\id \otimes t \otimes 1)}_{\IEAsymbol_0, \IEAsymbol_0}$ by Proposition~\ref{prop:sqbr1_nattr} and Definition~\ref{def:braidingP2}.
    Combining these two calculations reduces Equation~\eqref{eq:sigmafusionmiddlestep} to the desired statement.
\end{proof}

\begin{remark}
    We can also write Equation~\eqref{eq:fusionIEA} component-wise. Since it is an equality of morphisms $\IEAsymbol_{0, f} \widehat{\otimes} \IEAsymbol_{0, f} \to \IEAsymbol_{0, f}$ and $\IEAsymbol_{0, f}$ is itself isomorphic to a tensor product $\IEAsymbol_{0, f}\cong \widehat{\otimes} \IEAsymbol_0$ an obvious extension of Lemma~\ref{lemma:LaxMonVsAlgStructure} tells us that both sides of Equation~\eqref{eq:fusionIEA} are characterized by a collection of natural morphisms indexed by quadruples of objects of $\Aa_0$
    \begin{equation}
        \IEAsymbol_0(X_1, X_2) \otimes \IEAsymbol_0(Y_1, Y_2) \to \IEAsymbol_{0, f}((X_1\otimes X_2)\otimes (Y_1\otimes Y_2)) \ . 
    \end{equation}
    The left hand side is given simply by $(\sigma_f)_{X_1\otimes X_2, Y_1\otimes Y_2}$, precomposed with the obvious morphisms $\iota\colon \IEAsymbol_0(X_1, X_2) \to \IEAsymbol_{0, f}(X_1 \otimes X_2)$ given on skeins by the embedding of the unfused surface into the fused surface.\footnote{These are just the components of the isomorphism $\widehat{\otimes}{\IEAsymbol_0} \cong \IEAsymbol_{0, f}$.} The right hand side uses the same morphisms, together with $\sigma_{(X_1, X_2), (Y_1, Y_2)}$, $\IEAsymbol_f((J_0)_{X_1, Y_1, X_2, Y_2})$ and $\IEAsymbol_{0,f}((t_{2,3})_{X_1, X_2, Y_1, Y_2})$; we leave the details to the reader as an exercise in coend calculus.
\end{remark}

\subsection{Poisson brackets and skeins}\label{sect:poisson-brackets-and-skeins}
In this section we compute the deformation $\sigma$ for a general marked surface $(\Sigma, V)$ in a skein-theoretic way. 
Looking at Equation~\eqref{eq:defPoiss} we can expect a difference $(\Gamma_X \text{ over } \Gamma_Y)  - (\Gamma_X \text{ under } \Gamma_Y)$ for two skeins $\Gamma_{X/Y}$. Let us record the following simple calculation of a difference of an overcrossing and an undercrossing:
\begin{equation}\label{eq:overminusunder}
    [\beta_{X,Y} - \beta_{Y, X}^{-1}]_1 = [\beta_{X, Y} \circ (\beta^2_{X, Y} - \id_{X\otimes Y})]_1 = [\beta_{X,Y}]_0 \circ [\beta^2_{X, Y} - \id_{X\otimes Y})]_1 = (\beta_{\Aa_0})_{X, Y} \circ t_{X, Y} \ , 
\end{equation}
where we used Proposition~\ref{prop:sqbr1_nattr}. 

\begin{proposition}\label{prop:goldman}
    Let $(\Sigma, V)$ be a marked surface and let $\sigma := \sigma_{\Sigma, V} \colon \IEAsymbol_0 \widehat{\otimes} \IEAsymbol_0 \to \IEAsymbol_0$ denote the morphism from Definition \ref{def:of_p}. 
    For \(X, Y \in \Aa_0^{\lvert V \rvert}\), where \(\Aa_0^{\lvert V\rvert}\) denotes \(\lvert V \rvert\) many copies of \(\Aa_0\), and $\Gamma_X \in \IEAsymbol_0(X)$, $\Gamma_Y \in \IEAsymbol_0(Y)$ we can compute $\sigma_{X, Y}(\Gamma_X, \Gamma_Y) \in \IEAsymbol_0(X\otimes Y)$ by the following steps: 
    \begin{enumerate}
        \item Since $\Aa_0$ is symmetric we can view the skeins $\Gamma_X$ and $\Gamma_Y$ as immersed graphs lying on the surface $\Sigma$ (as opposed to graphs embedded in $\Sigma \times [0,1]$).
        \item Choose representatives for the isotopy classes of $\Gamma_{X/Y}$ such that 
        \begin{enumerate}
            \item[(i)] $\Gamma_X$ starts to the left of $\Gamma_Y$ at each $v\in V$. 
                \begin{center}
 \begin{overpic}[scale=1,tics=10]{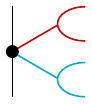}
  	\put(80,70){\color{contrastred}$\Gamma_X$}
        \put(80,15){\color{contrastblue}$\Gamma_Y$}
        \put(-10, 45){$v$}
    \end{overpic} 
                \end{center}
            \item[(ii)] $\Gamma_X$ and $\Gamma_Y$ only intersect with their strands (not coupons) at a finite number of transverse double intersection points.
        \end{enumerate} 
        \item Then the components of the natural transformation $\sigma$ are given by summing over the interior intersection points of $\Gamma_X$ and $\Gamma_Y$, inserting a coupon decorated with the infinitesimal braiding $t$:
         \begin{equation} \label{eq:sumintersections}
            \sigma_{X, Y}(\Gamma_X, \Gamma_Y) = \sum_{p \in (\Gamma_X \cap \Gamma_Y) \setminus V} \vcenter{\hbox{\begin{overpic}[scale=0.7, tics=10]{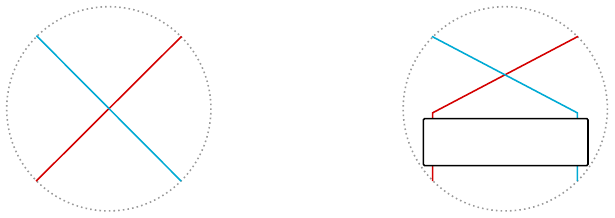}
            \put(47, 17){$\mapsto$}
            \put(75, 11.5){$t_{m_X^{\varepsilon}, m_Y^{\varepsilon'}}$}
            \put(2, 2){$m_X$}
            \put(28, 2){$m_Y$}
            \put(67.5, 2){$m_X$}
            \put(92, 2){$m_Y$}
            \put(12, 17){$p$}
            \put(16.4, 16.5){$\bullet$}
            \end{overpic}}}
        \end{equation}
        Here, $m_X$ and $m_Y$ are the objects of $\Aa_0$ labeling the ribbons of $\Gamma_X$ and $\Gamma_Y$ at the intersection $p$, while $\varepsilon$, $\varepsilon'$ denote a potential dual, depending on the orientation of the strands with respect to the coupon $t_{\bullet, \bullet}$.
    \end{enumerate}
\end{proposition}
\begin{proof}
    By Definition \ref{def:of_p} we know that $\sigma_{X, Y}(\Gamma_X, \Gamma_Y)$ is the $[-]_1$-part of an $\Aa_\varepsilon$-skein which near $v\in V$ looks as follows:
    \begin{center}
  \begin{overpic}[scale=1,tics=10]{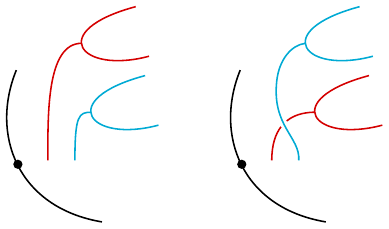}
  	\put(49,25){$-$}
  	\put(35,50){\color{contrastred}$\Gamma_X$}
   \put(35,32){\color{contrastblue}$\Gamma_Y$}
  	\put(92,32){\color{contrastred}$\Gamma_X$}
   \put(92,50){\color{contrastblue}$\Gamma_Y$}
  	\put(10,13){\color{contrastred}$X$}
   \put(17,13){\color{contrastblue}$Y$}
  	\put(67.5,13){\color{contrastred}$X$}
   \put(75,13){\color{contrastblue}$Y$}
   \put(2, 12){$v$}
   \put(59, 12){$v$}
	\end{overpic}  \ . 
\end{center}
Moving the two skeins closer to each other via isotopy, we get
    \begin{center}
  \begin{overpic}[scale=1,tics=10]{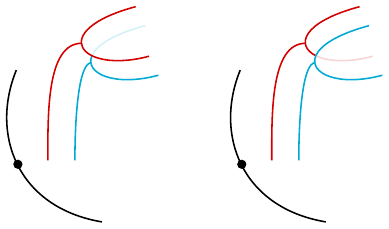}
  	\put(49,25){$-$}
  	\put(35,50){\color{contrastred}$\Gamma_X$}
   \put(35,32){\color{contrastblue}$\Gamma_Y$}
  	\put(94,55){\color{contrastred}$\Gamma_X$}
   \put(92,44){\color{contrastblue}$\Gamma_Y$}
  	\put(10,13){\color{contrastred}$X$}
   \put(17,13){\color{contrastblue}$Y$}
  	\put(67.5,13){\color{contrastred}$X$}
   \put(75,13){\color{contrastblue}$Y$}
   \put(2, 12){$v$}
   \put(59, 12){$v$}
	\end{overpic} \ . 
\end{center}
Identifying these skeins with an element of $\IEAsymbol_\varepsilon(X\otimes Y)$ by connecting the two strands labelled $X$ and $Y$, we get
    \begin{equation} \label{eq:goldmanproof}
  \begin{overpic}[scale=1,tics=10]{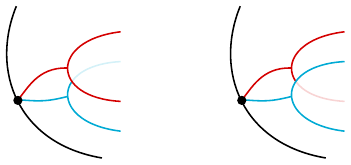}
  	\put(49,25){$-$}
   \put(30,30){\color{contrastred}$\Gamma_X$}
   \put(30,12){\color{contrastblue}$\Gamma_Y$}
   \put(98,32.5){\color{contrastred}$\Gamma_X$}
   \put(98,15){\color{contrastblue}$\Gamma_Y$}
  	\put(-15,19){$X\otimes Y$}
   \put(49,19){$X\otimes Y$}
   \put(2, 12){$v$}
   \put(66, 12){$v$}
   \end{overpic} \ . 
    \end{equation}
By assumption, $\Gamma_X$ and $\Gamma_Y$ intersect only at a finite number of transverse double points, which we (arbitrarily) enumerate as $p_1, \dots, p_N$. In the first term of \eqref{eq:goldmanproof} above, at all crossings $p_i$, $\Gamma_X$ is above $\Gamma_Y$, while in the second term $\Gamma_Y$ is above. We can thus write the difference \eqref{eq:goldmanproof} as the following telescopic sum
\begin{equation}
    \sum_{i = 1\dots N} \begin{pmatrix}\text{at crossings $p_1, \dots, p_{i-1}$, $\Gamma_X$ is under $\Gamma_Y$}\\\text{at crossings $p_{i}, \dots, p_{N}$, $\Gamma_X$ is over $\Gamma_Y$} \end{pmatrix} - \begin{pmatrix}\text{at crossings $p_1, \dots, p_{i}$, $\Gamma_X$ is under $\Gamma_Y$}\\\text{at crossings $p_{i+1}, \dots, p_{N}$, $\Gamma_X$ is over $\Gamma_Y$} \end{pmatrix}.
\end{equation}
Upon taking $[-]_1$ each of the summands contributes a term with a chord at the crossing $t$ according to Equation \eqref{eq:overminusunder}. In other words, we exactly get Equation \eqref{eq:sumintersections}. 
\end{proof}

\begin{remark}\label{rem:recover-goldman-et-al}
    Of course, for the (quasi-)Poisson bracket on the moduli spaces of flat connections (see Section \ref{sec:CS}), the formula \eqref{eq:sumintersections} is well known. It was first written down by  Goldman \cite{Goldman86} for classical Lie algebras. Generalizations for arbitrary metric Lie algebras and marked points on $\partial \Sigma$ later appeared in \cite{AMR1, RocheSzenes02, MassuyeauTuraev2012, LiBlandSevera, Nie2013, SeveraCenters}. 
    The formula \eqref{eq:sumintersections} usually contains a sign given by the orientation of the two strands at the intersection relative to the orientation of the surface. We recover this rule if we try to orient both strands in \eqref{eq:sumintersections} upwards, as each partial transposition of $t_{\bullet, \bullet}$ changes the sign.

    Choosing a deformation quantization of $\Aa_\varepsilon$, we obtain a deformation of the internal endomorphism algebra $(\IEAsymbol_0, \sigma)$. In the case of quantum groups, a quantization in terms of skeins appeared in \cite{AMR2} and \cite{RocheSzenes02}.
\end{remark}

\subsection{Moduli stacks of flat principal bundles}\label{sec:CS}
In this section we apply the ideas and concepts developed above to the study of deformation quantization of the moduli stack of flat principal bundles. 
Let $G$ be a reductive algebraic group with Lie algebra $\mathfrak{g}$. For an oriented surface $\Sigma$ the \emph{representation variety} $\operatorname{Rep}_G(\Sigma)$ is the variety of group homomorphisms $\pi_1(\Sigma)\to G$. The group $G$ acts on $\operatorname{Rep}_G(\Sigma)$ by conjugation. The \emph{character stack} is the quotient stack $\mathsf{Ch}_G(\Sigma) \coloneqq \operatorname{Rep}_G(\Sigma) / G$. Through the usual holonomy map it is equivalent to the moduli stack of flat $G$-bundles on $\Sigma$. The character variety is the affine GIT quotient $\operatorname{Rep}_G(\Sigma) / G$ and agrees with the variety of flat $G$-bundles modulo gauge transformations.

In the following we want to describe Poisson structures on the character stack $\mathsf{Ch}_G(\Sigma)$ and their deformation quantization, which are local in nature. 
Concretely, we will apply the formalism of this paper to deformations of categories of vector bundles or quasi-coherent sheaves on $\mathsf{Ch}_G(\Sigma)$. 
For an affine algebraic variety $X$ described by an algebra $\mathcal{O}(X)$ the category of quasi-coherent sheaves can be identified with the category $\QCoh(X) \coloneqq \mathcal{O}(X)\lmod$. 
This definition can be extended to more general stacks via right Kan-extension. 
For the character stack one finds locally that $\QCoh(\mathsf{Ch}_G(\mathbb{D}^2)) = \Rep(G)$ as a symmetric monoidal category. 
The category $\QCoh(\mathsf{Ch}_G(\Sigma))$ is local in the sense that it can be constructed from the value on disks via factorization homology (see \cite[Theorem~7.1]{BZBJIntegrating} for punctured surfaces and \cite[Sec.~2.3, Prop~4.5]{kinnear2024nonsemisimplecraneyettertheoryvarying} for a recent overview). Explicitly, we have
\begin{equation}
    \QCoh(\mathsf{Ch}_G(\Sigma)) = \int_\Sigma^{\Pres} \operatorname{Rep}(G) \,.
\end{equation}
Note that the factorization homology above is computed in $\Pr^L$. The category $\operatorname{Rep}(G)$ is the free cocompletion of the category $U(\mathfrak{g})\lmod^{\operatorname{f.d.}}$ of finite dimensional modules for the universal enveloping algebra $U(\mathfrak{g})$. 
This implies 
\begin{equation}\label{eq:completionFH}
    \int_\Sigma^{\Pres} \Rep(G) \cong \widehat{\int_\Sigma^{\Cat_\C } U(\mathfrak{g})\lmod^{\operatorname{f.d.}}}   \ \ . 
\end{equation}
On the right-hand side above factorization homology is computed in $\C$-linear categories, and hence can be computed via enriched skein categories. 
A direct consequence of what we discussed so far is that functions on the character variety are given by the (ordinary) skein algebra. 

The elements of the (enriched) skein category \(\SkCat{\Sigma}{U(\mathfrak{g})\lmod^{\operatorname{f.d.}}}\) define compact projective objects in $\QCoh(\mathsf{Ch}_G(\Sigma))$, i.e.\ vector bundles. Or, in more algebraic geometric terms, perfect sheaves. However, in general these are not all perfect sheaves. 
We think of the elements of factorization homology as a preferred generating set consisting of perfect sheaves which can be constructed via factorization homology. 
In the following we will study deformations of this subcategory which are closely related to Poisson structures on $\mathsf{Ch}_G(\Sigma)$.

For closed surfaces the character stack admits a symplectic structure going back to the work of Atiyah and Bott~\cite{AB}. In modern language it is the transgression of the natural 2-shifted symplectic structure on the classifying stack $BG$~\cite{SafronovPLasShifted}. In the case where $\Sigma$ has a non-empty boundary $\partial \Sigma$ the character stack still admits a Poisson structure which, in the setting of derived algebraic geometry, can be constructed by noting that the restriction $\mathsf{Ch}_G(\Sigma) \to \mathsf{Ch}_G(\partial \Sigma)$ is Lagrangian in the 1-shifted symplectic manifold $\mathsf{Ch}_G(\partial \Sigma)$ (the symplectic structure is again given by the AKSZ-construction~\cite{AKSZ}). Lagrangian submanifolds are in particular 1-shifted coisotropic, which in the derived setting contains a 0-shifted Poisson structure~\cite{MelaniSafronov1, MelaniSafronov2}. These are local in the sense that they are part of a fully extended topological field theory~\cite{CHS}. 

There are more traditional classical geometric constructions which are expected to describe this Poisson structure. For us two will be of particular interest: the quasi-Poisson structures constructed in~\cite{AKSM, LiBlandSevera} and those defined by Fock--Rosly~\cite{FR} which induce the same Poisson structure on the character variety.  

Both construction take advantage of the fact that for a choice of marked points $V\coloneqq\{v_1,\dots, v_k\} \subset \partial \Sigma$ the character stack can be written as a quotient of the framed moduli space $\mathsf{Ch}_G^{\operatorname{fr}}(\Sigma,V)$ of principal $G$-bundles on $\Sigma$, together with a trivialization at $V$ by the action of gauge transformations on the trivialization. The framed moduli space has the advantage of being an affine variety instead of a stack. The algebra of functions on\footnote{Here, $\chi(\Sigma)$ denotes the Euler characteristic.} $\mathsf{Ch}_G^{\operatorname{fr}}(\Sigma,V)\cong G^{k-\chi_\Sigma}$ together with its $G^k$-action can be identified with the internal endomorphism algebra $\intEndAlg{\Sigma}{V}\in \Rep G^{k}$. There are two natural first order deformations of $U(\mathfrak{g})\lmod^{\operatorname{f.d.}}$ relevant in this case:
\begin{itemize}
\item 
The deformation of $U(\mathfrak{g})\lmod^{\operatorname{f.d.}}$ provided by $U(\mathfrak{g})\lmod^{\operatorname{f.d.}}\otimes \Ce$ equipped with the infinitesimal braiding $t$ induced by the Killing form. This leads to a deformation of $\intEndAlg{\Sigma}{V}$ to a non-commutative algebra in $U(\mathfrak{g})\lmod^{\operatorname{f.d.}}\otimes \Ce$. We can directly conclude from Theorem~\ref{Thm: gluing Poisson} and Example~\ref{ex:defnPoisson} that this reproduces the $\sigma$-tensor of Ševera, see \cite[Eq.~(4)]{SeveraCenters}. The antisymmetrization of $\sigma$ then agrees with the quasi-Poisson structures constructed in~\cite{AKSM, LiBlandSevera}.   
\item The deformation $U_\varepsilon(\mathfrak{g})\lmod^{\operatorname{f.d.}}$, also leading to a first order deformation of $\intEndAlg{\Sigma}{V}$.  By Example~\ref{ex:defnPoisson} we can identify this deformation with an almost Poisson structure on $G^{k-\chi(\Sigma)}$ such that the action of $G^k$ is Poisson-Lie. The computations from~\cite{BZBJIntegrating} show that this Poisson structure agrees with the one introduced by Fock and Rosly in the case of one marked point. 

Using Theorem \ref{Thm: gluing Poisson}, we can also get a more direct comparison to the Fock-Rosly Poisson structure. If we decompose $r = r_\text{a} + t$ into antisymmetric and symmetric parts, the infinitesimal braiding on $U_\varepsilon(\mathfrak{g})\lmod$ is given by $2t$ (compare Example \ref{Ex: P_2 from r} and Definition \ref{def:braidingP2}). The forgetful functor $U_\varepsilon(\mathfrak{g})\lmod\to \Ce\lmod$ is monoidal but not braided, and the maps $\sigma$ differ by the action of the classical $r$-matrix. 
\begin{equation}\label{eq:sigmadifferenceforget}
    \sigma^{\Ce\lmod} = \sigma^{U_\varepsilon(\mathfrak{g})\lmod^{\operatorname{f.d.}}} - t + r_\text{a} \ . 
\end{equation}
On a general surface obtained by fusion at
vertices $V$, we thus have 
\begin{equation}
\sigma^{\Ce\lmod} = \sum_{v \in V} \quad 
\underbrace{\sum_{i \, @\, v} t_{i i'}}_{\eqref{eq:sigmadisk}} + \underbrace{\sum_{i < j \,@ \, v} 2 t_{j, i'}}_{\eqref{eq:fusionIEA}} + \underbrace{\sum_{i, j \,@ \, v} - t_{i, j'} + (r_\text{a})_{i, j'}}_{ \eqref{eq:sigmadifferenceforget}}
\end{equation}
where the three terms come from the contributions from disks, fusion and the forgetful functor $U_\varepsilon(\mathfrak{g})\lmod \to \Ce\lmod$. The inner sums run over disk ends (edge ends in the terminology of \cite{FR}) meeting the vertex $v$, and the unprimed and primed indices refer to the action on the two factors $\intEndAlg{\Sigma}{V}\widehat{\otimes} \intEndAlg{\Sigma}{V}$. In the above formula, the terms with $t$ combine to the second term of \cite[Eq.~(4.8)]{FR}, while the term involving $r_\text{a}$ is equal to the first term of \cite[Eq.~(4.8)]{FR}.
\end{itemize}
Note that a priori these are only almost Poisson structures. However, since both admit deformation quantizations we can already at this point conclude that they are actual (quasi-)Poisson structures. 

The factorization homology approach to the construction of these Poisson structures has the advantage of automatically being independent of the choice of combinatorial decompositions of the surface \(\Sigma\)
such as gluing patterns or pair of pants decompositions. 
The equivalence $U(\mathfrak{g})\lmod^{\operatorname{f.d.}}\otimes \Ce \cong U_\varepsilon(\mathfrak{g})\lmod^{\operatorname{f.d.}}$ of $\C[\varepsilon]$-linear ribbon categories (which is the identity modulo $\varepsilon$) relates the two first order deformations. 
This recovers the equivalence of categories between the category of Poisson $(G,r)$-spaces and equivariant Poisson maps and the category of quasi-Poisson spaces and equivariant maps respecting the quasi Poisson bivectors from~\cite{ManinPairsAKS, Mouquin}.

\subsubsection{Deformation quantization of moduli spaces of flat connections}
Recall that the two first order deformations used above have quantizations in terms of $U(\mathfrak{g})\lmod^\phi[[\hbar]]$ and $U_\hbar(\mathfrak{g})\lmod$, respectively. It follows from Proposition~\ref{prop:MarkedFusionInternalEndAlgs} and the computations in~\cite{BZBJIntegrating} that for one marked boundary point $U_\hbar(\mathfrak{g})\lmod$ recovers the equivariant deformation quantization constructed by Alekseev, Grosse, and Schomerus~\cite{AGS95}. 

On the other hand $U(\mathfrak{g})\lmod^\phi[[\hbar]]$ recovers, via Proposition~\ref{prop:MarkedFusionInternalEndAlgs}, the fusion procedure of Li-Bland and \v{S}evera \cite[Thm.~3~(2)]{LBSQ1}. To get that their star product is equal to the lax monoidal structure on the internal endomorphism algebra, it remains to compare the building block, i.e.\ the disk with two marked points. Under the identification \eqref{eq:completionFH}, this holds if correct convention is chosen for the isomorphisms $x \triangleright (y\triangleright a) \cong (x\otimes y) \triangleright a$, which enter the definition of the internal endomorphism algebra in \eqref{eq:lax_mon_str_intendoalg}. Concretely, we use an identification as in \eqref{eq:goldmanproof} at the ``left'' marked point, and braiding in the ``right'' marked point. This matches the star product of \cite{LBSQ1} for the disk with 2 points marked by $+$, which is described by a Kontsevich integral as in \cite[Corollary~3.2.2]{JanThesis}. Different conventions for the action at marked points lead to isomorphic deformation quantizations; with an added braiding, the isomorphism is given by $e^{\hbar C/4}$, the square root of the ribbon element of  $U(\mathfrak{g})\lmod^\phi[[\hbar]]$. The variant of the construction in \cite[Section~3]{LBSQ1} with an even associator, where their star product has no quantum corrections, again produces an isomorphic deformation quantization of functions on $G$: the isomorphism is given by the Kontsevich integral of the unknot.
The connection to factorization homology makes the independence of the choice of combinatorial structures manifest which, in the quantum case, is hard to verify otherwise~\cite{JanThesis}, \cite{JanWIP}.  
By functoriality of factorization homology the equivalence $U_\hbar(\mathfrak{g})\lmod \xrightarrow{\alpha^*} U(\mathfrak{g})\lmod^\Phi[[\hbar]]$ (see Theorem~\ref{Thm: equivalence}) of $\C[[h]]$-linear ribbon categories maps these deformation quantizations to each other. Establishing such a relation by hand seems hopelessly complicated.

\newpage
\appendix

\section{Background material on \texorpdfstring{$\VCat$}{V-Cat}} \label{app:Background-VCat}
In this appendix we will introduce some background material on \(\VCat\), as well as prove some basic results needed in the main text. 
First we recall the adjunction between \(\Va\)-enriched categories and \(\Va\)-graphs in Section \ref{appendix:VCatVGrph}, before briefly explaining how to impose relations between morphisms in the enriched setting in Section \ref{appendix:VCatGenRel}. 
In Section \ref{Sec:cocompleteness VCat} and \ref{sect:VCat-has-all-bicolimits} we  prove that \(\VCat\) has a bicolimits. 
Lastly, in Section \ref{app:reltens} we give the full list of relations imposed for the enriched Tambara relative tensor product, before proving that it is equivalent to the truncated bar construction.

Throughout this section we assume $\Va$ to be a cocomplete symmetric monoidal closed category. 

\subsection{\texorpdfstring{$\VCat$}{V-Cat} and \texorpdfstring{$\VGph$}{V-Gph}}\label{appendix:VCatVGrph}
We start by defining the notion of a \(\Va\)-graph before explaining how to obtain certain ``free'' \(\Va\)-categories from such \(\Va\)-graphs. 
The reference for this is \cite{Wolff}. 

\begin{definition}
A \emph{$\Va$-graph $\Omega$} consists of a set of objects $\Obj(\Omega)$ and for every pair $A,B \in \Obj(\Omega)$ an object $\Omega(A,B) \in \Va$.
Given two $\Va$-graphs $\Omega_1$ and $\Omega_2$, a \emph{$\Va$-graph morphism} $\Fa \colon \Omega_1 \to \Omega_2$ consists of a function $\Fa \colon \Obj(\Omega_1) \to \Obj(\Omega_2)$ together with maps 
$$
\Fa_{A,B} \colon \Omega_1(A,B) \longrightarrow \Omega_2(\Fa(A),\Fa(B)) 
$$
in $\Va$ for all $A,B \in \Obj(\Omega_1)$. We write $\VGph$ for the (1-)category of $\Va$-graphs and $\Va$-graph morphisms.
\end{definition}

There is a natural functor
$$
\VCat \xrightarrow{U} \VGph
$$
which forgets composition and identities. 

\begin{proposition}\cite[Proposition 2.2]{Wolff} \label{prop:FreeVCat}
The forgetful functor $U$ has a left adjoint $$
\Free \colon \VGph \longrightarrow \VCat
$$
sending a $\Va$-graph $\Omega$ to the free $\Va$-category $\Free(\Omega)$ defined as follows: 
\begin{itemize} 
    \item $\Obj(\Free(\Omega)) = \Obj(\Omega)$
    \item For $A,B \in \Obj(\Omega)$ let $\Delta_{A,B}$ be the set of all finite sequences $(A, X_1, \dots, X_n, B)$ with $X_i \in \Obj(\Omega)$ for $1 \leq i \leq n$. Define: 
    \begin{align} \label{eq:FreeVCatMorp}
    \Free(\Omega)(A,B) = \begin{cases}
     \coprod_{\Delta_{A,B}} \Omega(A,X_1) \otimes^\Va \Omega(X_1, X_2) \otimes^\Va \dots \otimes^\Va \Omega(X_n, B), \quad & \text{if }A \neq B \\
     \big(\coprod_{\Delta_{A,B}} \Omega(A,X_1) \otimes^\Va \Omega(X_1, X_2) \otimes^\Va \dots \otimes^\Va \Omega(X_n, B) \big) \coprod 1_\Va, \quad & \text{if } A=B
    \end{cases}
    \end{align}
\end{itemize}
\end{proposition}

Composition in \(\Free(\Omega)\) is defined in the following way. First, note that by assumption we have that $V \otimes^\Va -$ preserves colimits. Thus, we have: 
\begin{align*}
    \Free(\Omega)(A,B) & \otimes^\Va \Free(\Omega)(B,C) \\
    & \cong \coprod_{\Delta_{A,B} \times \Delta_{B,C}} \big( \Omega(A,X_1) \otimes^\Va \dots \otimes^\Va \Omega(X_n, B)\big) \otimes^\Va \big( \Omega(B,Y_1) \otimes^\Va \dots \otimes^\Va \Omega(Y_m, C) \big)
\end{align*}
when $A \neq B$ , $B \neq C$. If $A = B$ or $B = C$ we get a similar expression (including an extra copy of $1_\Va$). Denote by $\iota_{A,B}^\tau$ the canonical maps into the coproduct for each $\tau \in \Delta_{A,B}$. Composition $\circ_{A,B,C} \colon \Free(\Omega)(A,B) \otimes^\Va \Free(\Omega)(B,C) \to \Free(\Omega)(A,C)$ is defined on components by 
$$
(\circ_{A,B,C}) \circ (\iota^\tau_{A,B} \otimes \iota^{\tau'}_{B,C}) = \iota^{\tau\tau'}_{A,C} 
$$
where $\tau\tau' = (A, X_1, \dots, X_n, B, X'_1, \dots, X'_{n'}, C )$ is the concatenation of the sequences $\tau$ and $\tau'$. 

Define $j_A \colon 1_\Va \to \Free(\Omega)(A,A)$ to be $\iota^{1_\Va}_{A,A}$. In the case $A = B$, define 
$$
(\circ_{A,A,B}) \circ (\iota^{1_\Va}_{A,A} \otimes \iota^\tau_{A,B}) ) = \iota^\tau_{A,B} \circ l^\tau_{A,B}
$$
where $l^\tau_{A,B}$ is the left unitor 
$$
1_\Va \otimes \big(\Ga(A,X_1) \otimes \dots \otimes \Ga(X_n,B)\big) \xrightarrow{l^{\tau}_{A,B}} \big(\Ga(A,X_1) \otimes \dots \otimes \Ga(X_n,B)\big)
$$
in $\Va$. The case $B = C$ is analogous. 

\subsection{Generators and relations}\label{appendix:VCatGenRel}

Let $\Ca$ be a $\Va$-category and suppose we have two morphisms $f,g \colon 1_\Va \to \Ca(x,y)$ that we wish to identify in $\Ca$. We can do so by defining the following $\Va$-category 
$$
\Ca / f\sim g = \begin{cases} \Obj(\Ca/f \sim g) & =~~\Obj(C) \\
(\Ca / f\sim g) (a,b) & =~~\text{coequalizer in } \Va \text{ of the diagram } \eqref{diag:relationsVCat}
\end{cases} 
$$

\begin{equation}\label{diag:relationsVCat}
\begin{tikzcd}[column sep=huge]
\Ca(y,b) \otimes^\Va 1_\Va \otimes^\Va \Ca(a,x) \arrow[r,yshift=-1ex,"\id_\Va \otimes^\Va f \otimes^\Va \id_\Va",swap] \arrow[r,yshift=1ex,"\id_\Va \otimes^\Va g \otimes^\Va \id_\Va"] & \Ca(y,b) \otimes^\Va \Ca(x,y) \otimes^\Va \Ca(a,x) \arrow[r,yshift=-1ex,"-\circ-\circ-",swap] \arrow[r,yshift=1ex,"-\circ-\circ-"] &  \Ca(a,b)
\end{tikzcd}
\end{equation}

\subsection{Cocompleteness of \texorpdfstring{$\VCat$}{V-Cat}}\label{Sec:cocompleteness VCat}
We will first discuss the stricter notion of $\Cat$-weighted colimits in the sense of ordinary enriched category theory. Existence of strict $\Cat$-weighted colimits in $\VCat$ will then serve as an intermediary to prove existence of (weak) 2-colimits in $\VCat$.

\subsubsection{Strict 2-colimits}
We first recall the definition of a strict 2-colimit in \(\VCat\), before proving that we have all of these in \(\VCat\). 

\begin{definition}
A \emph{strict 2-colimit in $\VCat$} is a $\Cat$-weighted colimit where both the weight $W$ and the diagram $X$ are 2-functors indexed by a 2-category $\Da$ such that for any $\Cat$-enriched category $\Ea$ we have a natural \emph{isomorphism} of categories:
$$
\Hom_{\VCat}(\colim^W X, \Ea) \cong \Hom_{[\Da^{\op},\Cat]_{\operatorname{strict}}}(W, \Hom_{\VCat}(X(-), \Ea))
$$ 
where $[\Da^{\op},\Cat]_{\operatorname{strict}}$ is the 2-category of 2-functors, 2-natural transformations and modifications. 
\end{definition}

In order to show existence of all strict $\Cat$-weighted colimits in \(\VCat\), we first need two specific $\Cat$-(co)limits, namely the tensor and cotensor over $\Cat$. 
If $\Va$ has coproducts, the underlying set functor $\Va(1_\Va, - ) \colon \Va \to \mathsf{Set}$ has a left adjoint $(-) \cdot 1_\Va \colon \mathsf{Set} \to \Va$, where $X \cdot 1_\Va$ is the tensor of the set $X$ with the monoidal unit in $\Va$. This induces an adjunction
$$
(-)_\Va : \Cat \leftrightarrows \VCat : (-)_0 \ .
$$
Recall from Definition \ref{defn:underlying-ordinary-category} that for each \(\Ca\in \VCat\) the underlying (ordinary) category \(\Ca_0\) has the same objects and morphisms are given by $f \colon 1_\Va \to \Ca(c,c')$. For $C \in \Cat$, the free $\Va$-category $C_\Va$ also has the same objects as $C$ and the Hom-objects are $C_\Va(c,c') = C(c,c') \cdot 1_\Va = \coprod_{C(c,c')} 1_\Va$. \par 

Using the above, we define the tensoring over $\Cat$ by the functor
\begin{align*}
 \Cat \times \VCat & \xrightarrow{- \otimes -} \VCat \\
(C,\Da) & \mapsto \Ca \otimes \Da  = \Ca_\Va \catprod \Da
\end{align*}
where $\catprod$ on the right hand side is the tensor product in $\VCat$ as in Definition \ref{definition:tensor-product-of-Vcats}. 
This indeed turns $\VCat$ into a category tensored over $\Cat$ since we have natural isomorphisms 
\begin{align*}
    \Hom_{\VCat}(\Ca \otimes \Da, \Ea) & \cong \Hom_{\VCat}(\Ca_\Va, [\Da,\Ea]) \\
    & \cong \Hom_{\Cat}(\Ca, [\Da,\Ea]_0) \\
    & \cong \Hom_{\Cat}(\Ca, \Va\mhyphen\text{Fun}(\Da,\Ea))
\end{align*}
where in the first line we used that $\VCat$ is a symmetric monoidal closed category, in the second line we used the adjunction $(-)_\Va \dashv (-)_0$ and the last line follows from
$$
\Va(1_\Va, \Va\mhyphen\textbf{Nat}(F,G)) \cong \Va\mhyphen\text{Nat}(F,G)
$$
for each pair of functors $F,G \colon \Da \to \Ea$, where we used again that $(-)_0$ is right adjoint and thus preserves limits. \par  
Similarly, the cotensor is defined by the functor 
\begin{align*}
\Cat^{\op} \times \VCat & \xrightarrow{-\pitchfork-} \VCat \\
(\Ca,\Ea) & \mapsto \Ca \pitchfork \Da = [\Ca_\Va, \Da]  \ .
\end{align*}

\begin{proposition}\label{prp:strictcolimVCat}
$\VCat$ has strict $\Cat$-weighted colimits.
\end{proposition}

\begin{proof}
It is shown in \cite{Wolff} that $\VCat$ has all ordinary colimits. Since $\VCat$ is tensored and cotensored over $\Cat$ we conclude that it has all strict $\Cat$-weighted colimits, see for example \cite[Corollary 7.6.4.]{Riehl}.
\end{proof}

\newcommand{\TwoCatNew}{2\mathrm{Cat}} 
\newcommand{\BiCatNew}{\mathrm{BiCat}} 

\subsubsection{Bicolimits} \label{sect:VCat-has-all-bicolimits}
We now leverage the result that \(\VCat\) has all strict \(\Cat\)-weighted colimits to also prove that it has all bicolimits, a notion we first define. 
 
\begin{definition}
Let $[\Da^{\op},\Cat]_{\operatorname{bicat}}$ denote the 2-category of homomorphisms between bicategories. More precisely it is the 2-category of pseudo-functors $\Da^{\op} \to \Cat$, pseudo-natural transformations and modifications.
\end{definition}

\begin{definition} \label{defn:bicolimit-in-VCat}
Let \(\Da\) be a bicategory and $W \colon \Da \to \Cat$ and $X \colon \Da \to \VCat$ be pseudo-functors. 
The \emph{\(W\)-weighted 2-colimit of the diagram \(X\)} is an object $\colim^W X \in \VCat$ with the following universal property: 
Given any other $\Va$-category $\Ea$, there is an equivalence
\begin{align}
\Hom_{\VCat}(\colim^W X, \Ea) \cong \Hom_{[\Da^{\op},\Cat]_{\operatorname{bicat}}}(W, \Hom_{\VCat}(X(-), \Ea))
\end{align}
which is pseudo-natural in $\Ea$.
We for short call these colimits for \emph{bicolimits}.
\end{definition}

The inclusion of the (ordinary) category $\TwoCatNew$ of 2-categories and 2-functors into the (ordinary) category $\BiCatNew$ of bicategories and pseudo-functors has a left adjoint known as \emph{strictification}:
\begin{equation}
    \st : \BiCatNew \leftrightarrows \TwoCatNew : \iota \ .
\end{equation}
Moreover, the components of the unit map $\Aa \to \st(\Aa)$ are equivalences of bicategories \cite[Section 4.10]{CoherenceTriCats}. Higher categorical aspects of the strictification adjunction were studied in \cite{Campbell}, where in particular the following higher universal property of strictification is proven: For every bicategory $\Aa$ and 2-category $\Ca$ there is an isomorphism of 2-categories
\begin{equation}\label{Strictification}
[\st(\Aa), \Ca]_{\operatorname{pseudo}} \cong [\Aa, \Ca]_{\operatorname{bicat}}
\end{equation}
where $[-, -]_{\operatorname{pseudo}}$ is the 2-category of 2-functors, pseudo-natural transformations and modifications, see \cite[Corollary 3.6]{Campbell}. We denote by $\Fa'$ the image of the pseudo-functor $\Fa \colon \Aa \to \Ca$ under the above isomorphism. \par 

Notice that the 1-morphisms in $[-, -]_{\operatorname{pseudo}}$ are only pseudo-natural transformations. However, it is shown in \cite{Kelly2limits} that for any small 2-category $\Da$ and cocomplete 2-category $\Ca$ the inclusion 2-functor
\begin{equation}\label{AdjStrictPseudoNatTrans}
    [\Da, \Ca]_{\operatorname{strict}} \hookrightarrow [\Da, \Ca]_{\operatorname{pseudo}} 
\end{equation}
has a left adjoint that we will denote by $\widetilde{(-)}$.

\begin{proposition}\label{Prp:bicolimits VCat}
    $\VCat$ has all bicolimits.
\end{proposition}
\begin{proof}
Given any bicategory $\Da$ and pseudo-functors $W \colon \Da \to \Cat$ and $X \colon \Da \to \VCat$, we want to show that the $W$-weighted 2-colimit of the diagram $X$ exists. To that end, we use the strictification result \eqref{Strictification} and the adjunction in \eqref{AdjStrictPseudoNatTrans}:
\begin{align*}
\Hom_{[\Da^{\op},\Cat]_{\operatorname{bicat}}}(W, \Hom_{\VCat}(X(-),\Ea)) & \cong \Hom_{[\st(\Da)^{\op},\Cat]_{\operatorname{pseudo}}}(W', \Hom_{\VCat}(X(-), \Ea)') \\
& \cong \Hom_{[\st(\Da)^{\op},\Cat]_{\operatorname{strict}}}(\widetilde{W'}, \Hom_{\VCat}(X(-), \Ea)')
\end{align*}
We observe that in the diagram 
\begin{center}
\begin{tikzcd}[column sep = huge]
\st(\Da) \arrow[dr, "X'"] \arrow[drr, "\Hom_{\VCat}(X(-){,}\Ea)'", bend left = 15] & & \\
\Da \arrow[u,"\cong"] \arrow[r,"X",swap] & \VCat \arrow[r,"\Hom_{\VCat}(-{,}\Ea)",swap] & \Cat
\end{tikzcd}
\end{center}
the left triangle as well as the outer diagram commute by definition of the strictification of a pseudo-functor. 
It follows that also the right triangle commutes since the unit of the strictification adjunction is an equivalence. Finally, by Proposition \ref{prp:strictcolimVCat} we know that the $\widetilde{W'}$-weighted colimit of $X'$ exists. 
\end{proof}

Along the same line, we can also prove existence of bicolimits in the 2-category of locally presentable enriched categories. The proof is based on \cite{CFJ}, where 2-cocompleteness of $\mathsf{Pres}$ is proven for the $\Va = \mathsf{Set}$ case.

\begin{proposition}\label{prp:VPres_colimits}
$\VPres$ has all bicolimits.
\end{proposition}

\begin{proof}
Let $\VPres^\text{op}$ be the category whose objects are locally $\alpha$-presentable $\Va$-categories, 1-morphisms are right adjoints and 2-morphisms are enriched natural transformations. In \cite[Theorem 6.10]{Bird} it was shown that $\VPres^\text{op}$ has all products, inserters and equifiers and that they are computed in $\VCat$. It then follows from \cite[Proposition 5.2]{Kelly2limits} that $\VPres^\text{op}$ has all pseudo-limits. We conclude by the same arguments as in the proof of Proposition \ref{Prp:bicolimits VCat} that $\VPres^\text{op}$ has all bilimits and thus $\VPres$ has all bicolimits.
\end{proof}

\subsection{Relative tensor products}\label{app:reltens}
In this section we first complete the definition for the enriched Tambara relative tensor product from Definition \ref{df:TambaraTensProd} by spelling out the complete list of relations. This list of relations is tailored such that giving a functor out of the Tambara tensor product of two $\Va$-enriched categories is equivalent to giving a \emph{balanced functor} (see Definition \ref{df:BalancedFunctor}) out of the naive $\Va$-enriched tensor product. Finally, we prove that the Tambara relative tensor product agrees with the colimit of the truncated bar construction from Section \ref{sect:bar-construction}. 

\paragraph{Enriched Tamabara relative tensor product}

The following is the complete list of relations we impose in Definition \ref{df:TambaraTensProd}.

\begin{itemize}
\item Isomorphism:
\end{itemize}

\begin{center}
\begin{tikzcd}
1_\Va \cong 1_\Va \otimes^\Va 1_\Va \arrow[d,"(\iota_{m \triangleleft a, n}^{m ,a \triangleright n})^{-1} \otimes^\Va \iota_{m \triangleleft a, n}^{m ,a \triangleright n} "]  \\
\Free(\Omega)((m,a \triangleright n),(m \triangleleft a ,n)) \otimes^{\Va} \Free(\Omega)((m \triangleleft a,n),(m, a \triangleright n)) \arrow[d,"-\circ-"]  \\
\Free(\Omega)((m \triangleleft a,n),(m \triangleleft a,n))  
\end{tikzcd}        
\end{center}
\begin{equation}\label{reln:iso}
\sim 
\end{equation}
\begin{center}
\begin{tikzcd}
 1_\Va \arrow[d,"\id_{(m \triangleleft a,n)}"] \\
\Free(\Omega)((m \triangleleft a,n),(m \triangleleft a,n))
\end{tikzcd}  
\end{center}

Similarly, we impose the relation $\iota_{m \triangleleft a, n}^{m ,a \triangleright n} \circ (\iota_{m \triangleleft a, n}^{m ,a \triangleright n})^{-1} \sim \id_{(m, a \triangleright n)}$.

\begin{itemize}
\item Naturality:
\end{itemize}

\begin{center}
\begin{tikzcd}
1_\Va \cong 1_\Va \otimes^\Va 1_\Va \otimes^\Va  1_\Va \otimes^\Va 1_\Va \arrow[d,"f \otimes^\Va u \otimes^\Va g \otimes^\Va \iota_{m \triangleleft a, n}^{m, a \triangleright n}"] \\
\Ma(m,m') \otimes^\Va \Aa(a,a') \otimes^\Va \Na(n,n') \otimes^\Va \Free(\Omega)((m \triangleleft a,n),(m,a \triangleright n)) \arrow[d,"\id_\Va \otimes^\Va \triangleright_{(a,n),(a',n')} \otimes^\Va \id_\Va"] \\
\Free(\Omega)((m,a \triangleright n), (m',a' \triangleright n')) \otimes^\Va \Free(\Omega)((m \triangleleft a,n),(m,a \triangleright n)) \arrow[d,"-\circ-"] \\
\Free(\Omega)((m \triangleleft a,n),(m',a' \triangleright n'))
\end{tikzcd}    
\end{center}
\begin{equation}\label{reln:naturality}
\sim
\end{equation}

\begin{center}
\begin{tikzcd}
1_\Va \cong 1_\Va \otimes^\Va 1_\Va \otimes^\Va  1_\Va \otimes^\Va 1_\Va  \arrow[d,"\iota^{m',a'\triangleright n'}_{m'\triangleleft a',n'} \otimes^\Va f \otimes^\Va u \otimes^\Va g"] \\
\Free(\Omega)((m' \triangleleft a',n'),(m',a' \triangleright n')) \otimes^\Va \Ma(m,m') \otimes^\Va \Aa(a,a') \otimes^\Va \Na(n,n') \arrow[d,"\id_\Va \otimes^\Va \triangleleft_{(m,a),(m',a')} \otimes^\Va \id_\Va"] \\
\Free(\Omega)((m' \triangleleft a',n'),(m',a' \triangleright n')) \otimes^\Va \Free(\Omega)((m \triangleleft a,n),(m'\triangleleft a',n'))  \arrow[d,"-\circ-"]\\
\Free(\Omega)((m \triangleleft a,n),(m',a' \triangleright n'))
\end{tikzcd}
\end{center}

\begin{itemize}
\item Compatibility with the associator:
\end{itemize}

\begin{center}
\begin{tikzcd}
1_\Va \cong 1_\Va \otimes^\Va 1_\Va \otimes^\Va 1_\Va  \arrow[d,"(\id_m \otimes^\Va \beta^\Na_{a,b,n}) \otimes^\Va \iota_{m \triangleleft a,b \triangleright n}^{m,a \triangleright b \triangleright n} \otimes^\Va \iota_{m \triangleleft a \triangleleft b, n}^{m \triangleleft a, b \triangleright n}"] \\
\begin{tabular}{c}
$\Free(\Omega)((m,a \triangleright (b \triangleright n)),(m,(a\otimes b) \triangleright n)) \otimes^\Va \Free(\Omega)((m \triangleleft a, b \triangleright n),(m,a \triangleright(b\triangleright n)))$\\
$\otimes^\Va \Free(\Omega)((m \triangleleft a) \triangleleft b, n), (m \triangleleft a,b \triangleright n))$
\end{tabular} \arrow[d,"- \circ - \circ -"] \\
\Free(\Omega)(((m \triangleleft a) \triangleleft b,n),(m,(a \otimes b) \triangleright n))
\end{tikzcd}
\end{center}
\begin{equation}\label{reln:associatior}
\sim
\end{equation}
\begin{center}
\begin{tikzcd}
1_\Va \cong 1_\Va \otimes^\Va 1_\Va \arrow[d,"\iota_{m \triangleleft a \otimes b, n }^{m, a \otimes b \triangleright n} \otimes^\Va (\beta^\Ma_{m,a,b} \otimes^\Va \id_n)"] \\
\Free(\Omega)((m \triangleleft (a \otimes b),n),(m,(a \otimes b) \triangleright n)) \otimes^\Va \Free(\Omega)(((m \triangleleft a) \triangleleft b,n),(m\triangleleft(a \otimes b), n)) \arrow[d,"-\circ -"] \\
\Free(\Omega)(((m \triangleleft a) \triangleleft b,n),(m,(a \otimes b) \triangleright n))
\end{tikzcd}
\end{center}

\begin{itemize}
\item Compatibility with unitors:
\end{itemize}

\begin{center}
\begin{tikzcd}
1_\Va \cong 1_\Va \otimes^\Va 1_\Va \arrow[d,"(\id_m \otimes^\Va \eta^L_n) \otimes^\Va \iota_{m \triangleleft 1, n}^{m,1 \triangleright n}"] \\
\Free(\Omega)((m,1_\Aa \triangleright n),(m,n)) \otimes^\Va \Free(\Omega)((m \triangleleft 1_\Aa, n),(m , 1_\Aa \triangleright n))\arrow[d,"-\circ-"] \\
\Free(\Omega)((m \triangleleft 1_\Aa,n),(m,n))
\end{tikzcd}
\end{center}
\begin{equation}\label{reln:unitor}
\sim
\end{equation}
\begin{center}
\begin{tikzcd}
 1_\Va \arrow[d,"\eta^R_m \otimes^\Va \id_n"] \\
 \Free(\Omega)((m \triangleleft 1_\Aa, n),(m,n))
\end{tikzcd}
\end{center}

\paragraph{Balanced functors}
We here give the definition of balanced functors and natural transformations. 

\begin{definition}\label{df:BalancedFunctor}
A $\Va$-functor $\Fa \colon \Ma \catprod \Na \to \Da$ is called \emph{$\Aa$-balanced} if it is equipped with a natural isomorphism 
\begin{center}
    \begin{tikzcd}[sep = small]
    & \Ma \catprod \Na \arrow[dr,"\Fa"] \arrow[dd, Rightarrow, shorten >=10pt, shorten <=10pt,"\alpha"] \\
    \Ma \catprod \Aa \catprod \Na \arrow[ur,"\triangleleft \catprod \id_\Na"r] \arrow[dr,"\id_\Ma \catprod \triangleright",swap] & & \Da \\
    & \Ma \catprod \Na \arrow[ur, "\Fa",swap]
    \end{tikzcd}
\end{center}
such that 
\begin{center}
\begin{equation} 
\begin{tikzcd}[row sep=20pt, column sep=13pt] 
    && M \catprod A \catprod A \catprod N \arrow[d, swap, "\id_{M\catprod A} \catprod \triangleright ", ""{name=A}] \arrow[dr, bend left=20, "\id \catprod \otimes^\Aa \catprod \id", ""{name=U}]  \arrow[dll, bend right=15, swap, "\triangleleft \catprod \id_{A\catprod N}"] &
    \\ M\catprod A \catprod N \arrow[d, swap, "\triangleleft \catprod \id_N"] \arrow[dr, "\id_M \catprod \triangleright", ""{name=B}] &&M \catprod A \catprod N \arrow[Rightarrow, to=U, shorten >=10pt, shorten <=10pt, "\beta^N"] \arrow[d, swap, "\triangleleft \catprod \id_N"] \arrow[dr, "\id_M \catprod \triangleright", ""{name=E}] &M \catprod A \catprod N \arrow[d, "\id_M \catprod \triangleright"]
    \\ M \catprod N \arrow[drr, swap, bend right=30, "\Fa", ""{name=C}] & M\catprod N \arrow[dr, bend right=0, swap, "\Fa", ""{name=D}]& M\catprod N \arrow[d, "\Fa"]  \arrow[to=E, Rightarrow, shorten >=7pt, shorten <=7pt, "\alpha"] &M \catprod N \arrow[dl, bend left=15, "\Fa"]
    \\ &&\Da
    \arrow[from=B, to=A, Rightarrow, shorten >=35pt, shorten <=35pt, "\id"] \arrow[from=C, to=D, Rightarrow,  shorten >=10pt, shorten <=10pt, "\alpha"]
\end{tikzcd}
\end{equation}
\begin{equation}\label{diag:balancedfunctor} 
\  = \ 
\begin{tikzcd}
	&M\catprod A \catprod A \catprod N \arrow[dl, bend right=20, swap,  "\triangleleft \catprod \id_{A \catprod N}"] \arrow[d, "\id \catprod \otimes^\Aa \catprod \id"]
	\\ M \catprod A \catprod N \arrow[dr, swap, bend right=20, "\triangleleft \catprod \id_N"] \arrow[r, Rightarrow,  shorten >=8pt, shorten <=8pt, "\beta^M"]& M\catprod A \catprod N \arrow[d,  "\triangleleft \catprod \id_N"] \arrow[dr, bend left=20, "\id_M \catprod \triangleright"]
	\\ &M \catprod N  \arrow[d, swap, "\Fa"] \arrow[r, Rightarrow,  shorten >=10pt, shorten <=10pt, "\alpha"] & M\catprod N \arrow[dl, bend left=20, "\Fa" ]
	\\ & \Da
\end{tikzcd}
\end{equation}
\end{center}

An \emph{$\Aa$-balanced natural transformation} is a natural transformation $\gamma \colon \Fa \to \Ga$ between $\Aa$-balanced functors such that 

\begin{center}
\begin{tikzcd}[column sep=huge] 
	\Ma \catprod \Aa \catprod \Na \arrow[r, "\lhd \catprod \id_{\Na}"] \arrow[rd, bend right=20, swap, "\id \catprod \rhd"]  & \Ma \catprod \Na \arrow[r, "\Fa"] \arrow[d, Rightarrow, shorten >=2pt, shorten <=2pt, "\alpha_{\Fa}"] & \Da
	\\ &  \Ma \catprod \Na \arrow[ru, bend right=10, "\Fa"{name=F}] \arrow[ru, bend right=60, swap, "\Ga"{name=G}] \arrow[Rightarrow, from=F, to=G, shorten <=6pt, shorten >=6pt, "\gamma"]
\end{tikzcd}
$$
\sim
$$
\begin{tikzcd}[column sep=huge]
    \Ma \catprod \Aa \catprod \Na \arrow[r, "\lhd\catprod \id_{\Na}"] \arrow[rdd, bend right=20, swap, "\id\catprod \rhd"] & \Ma \catprod \Na \arrow[r, "\Fa"{name=F}] \arrow[dd, Rightarrow, shorten <=8pt, shorten >=8pt, "\alpha_{\Ga}"] \arrow[r, bend right=50, swap, pos=0.52354, "\Ga"{name=G}] \arrow[Rightarrow, from=F, to=G, shorten <=5pt, shorten >=5pt, "\gamma"] & \Da
    \\ \\ & \Ma \catprod \Na \arrow[ruu, bend right=20, swap, "\Ga"]
\end{tikzcd}
\end{center}
\end{definition}
The category of \(\Aa\)-balanced functors and natural transformations is denoted \(\BalFunA(\Ma\catprod \Na, \Da)\).

\subsubsection{Proof of equivalence of tensor products}\label{app:eqTambaraBar}
We now turn to the proof that the Tambara relative tensor product agrees with the colimit of the truncated bar construction. 

Before we give the proof of Theorem \ref{thm:eqTambaraBar}, recall from Definition \ref{defn:bicolimit-in-VCat} that given a diagram $X$ of shape $\Da$ in $\VCat$, its bicolimit is an object $\colim X$ such that for every $\Ea \in \VCat$ there is an equivalence
$$
\Hom_{\VCat}(\colim X, \Ea) \cong [\Da,\VCat](X, \Delta(\Ea))
$$
(pseudo)natural in $\Ea$, where $\Delta(\Ea)$ is the constant diagram. 
By \cite[Proposition 2.25]{Cooke19}, the category of cocones over $X$ with summit $\Ea$ has a concise description we recall now. 
By slight abuse of notation we also write $A,B,C$ for the image of $A,B,C \in \Da$ under $X$, and similarly for the morphisms. 
Objects $\sigma$ in $ [\Da,\VCat](X, \Delta(\Ea))$ are given by the data
\begin{center}
\begin{tikzcd}[sep = large]
A \arrow[r,yshift=1ex] \arrow[r] \arrow[r,yshift=-1ex] \arrow[drr,bend right=30,"\sigma_A"{name=A},swap]& B \arrow[dr,bend right=30,"\sigma_B"{name=B},swap] \arrow[r,yshift=0.5ex] \arrow[r,yshift=-0.5ex]\arrow[Rightarrow, shorten >=10pt, shorten <=10pt, from=A,"\sigma_{G_i}"]  & C \arrow[d,"\sigma_C"]\arrow[Rightarrow, shorten >=10pt, shorten <=10pt, from=B,"\sigma_{F_i}"] \\
&& \Ea
\end{tikzcd}
\end{center}
where 
\begin{equation}
\sigma_{G_i} \colon \sigma_A \Rightarrow \sigma_B \circ G_i~~(i = 0,1,2) \quad \mathrm{and} \quad \sigma_{F_i} \colon \sigma_B \Rightarrow \sigma_C \circ F_i~~(i = 0,1)
\end{equation}
satisfy the following relations
\begin{align}
\begin{split}
\sigma_C \kappa_1 & = \Delta_{02}\Delta_{10}^{-1} \\
\sigma_C \kappa_2 & = \Delta_{01}\Delta_{00}^{-1} \\
\sigma_C \kappa_3 & = \Delta_{11}\Delta_{12}^{-1}
\end{split}
\label{eq:Kappas}
\end{align}
with $\Delta_{ij} = (\sigma_{F_i} G_j)\sigma_{G_j}$. The morphisms in $[\Da,\VCat](X, \Delta(\Ea))$ are natural isomorphisms 
\begin{center}
\begin{tikzcd}[sep = large]
C \arrow[r, bend left, "\sigma_C"{name=A}]
   \arrow[r, swap, bend right, "\eta_C"{name=B}]
 & \Ea \arrow[Rightarrow, shorten >=2pt, shorten <=2pt, from=A, to=B, "\Xi_C"]
\end{tikzcd}
\end{center}
satisfying the relations 
\begin{align}
\eta^{-1}_{F_0} (\Xi F_0) \sigma_{F_0} & = \eta^{-1}_{F_1} (\Xi F_1) \sigma_{F_1} \label{eq:Xi1}\\
(\Delta^\eta_{ij})^{-1}(\Xi_C F_i G_j)\Delta^\sigma_{ij} & =  (\Delta^\eta_{kl})^{-1} (\Xi_C F_k G_l) \Delta^\sigma_{kl} \label{eq:Xi2}
\end{align}
for $i,k \in \{0,1\}$ and $j,l \in \{0, 1, 2\}$.

\begin{proof}[Proof of Theorem \ref{thm:eqTambaraBar}]
Let $X$ denote the truncated bar construction. We abbreviate 
$$
A = \Ma \catprod \Aa \catprod \Aa \catprod \Na, \quad B = \Ma \catprod \Aa \catprod \Na, \quad C = \Ma \catprod \Na \ \ .
$$
Following the proof of \cite[Theorem 2.27]{Cooke19}, we will show that there is an equivalence 
$$
[\Da,\VCat](X, \Delta(\Ea)) \xrightarrow{\Ia_\Ea} \BalFunA(\Ma \catprod \Na,\Ea)
$$
between cocones over $X$ with summit $\Ea$ and $\Aa$-balanced functors $\Ma \catprod \Na \to \Ea$. 
We define the functor $\Ia_\Ea$ on objects $\sigma \in [\Da,\VCat](X,\Delta(\Ea))$ by 
$$
\Ia_\Ea(\sigma) = \sigma_C \text{ with balancing } \alpha = \sigma_{\id_\Ma \catprod \triangleright} \sigma_{\triangleleft \catprod \id_\Na}^{-1} \colon \sigma_C (\triangleleft \catprod \id_\Na) \Rightarrow \sigma_C (\id_\Ma \catprod \triangleright) \ \  . 
$$
On morphisms we define $\Ia_\Ea(\Xi) = \Xi$.

\paragraph{$\Ia_\Ea$ is well-defined.}
The pair \(\left(\Ia_\Ea(\sigma), \alpha \right)\) is a balanced functor due to equations \eqref{eq:Kappas} in a similar fashion to the proof of \cite[Theorem 2.27]{Cooke19}. \(\Ia_\Ea (\Xi)\) is a natural transformation of balanced functors due to \eqref{eq:Xi1}.

\paragraph{$\Ia_\Ea$ is surjective on objects.}

Let $H \colon \Ma \catprod \Na \to \Ea$ be an $\Aa$-balanced $\Va$-functor with balancing $\alpha$. We define an element $\sigma \in [\Da,\VCat](X,\Delta(\Ea))$ by the following diagram
\begin{center}
\begin{tikzcd}[sep = large]
A \arrow[r,yshift=2ex,"\triangleleft"] \arrow[r, "\otimes"] \arrow[r,yshift=-2ex, "\triangleright"] \arrow[drr,bend right=40,"H \circ \triangleleft \circ \otimes"{name=A},swap]& B \arrow[dr,bend right=30,"H \circ \triangleleft"{name=B},swap] \arrow[r,yshift=1ex,"\triangleleft"] \arrow[r,yshift=-1ex,"\triangleright"]\arrow[Rightarrow, shorten >=10pt, shorten <=10pt, from=A,"\sigma_{G_i}"]  & C \arrow[d,"H"]\arrow[Rightarrow, shorten >=10pt, shorten <=10pt, from=B,"\sigma_{F_i}"] \\
&& \Ea
\end{tikzcd}
\end{center}
Note that for the sake of lighter notation we abbreviated for example $\triangleleft \catprod \id_{\Na} = \triangleleft$. The 2-morphisms for the above diagram are defined by $\sigma_{F_0} = \id$, $\sigma_{F_1} = \alpha$, $\sigma_{G_0} = \eqref{eq:sigmaG0}$, $\sigma_{G_1} = \id$, and $\sigma_{G_2} = \eqref{eq:sigmaG2}$: 

\begin{equation}
\begin{tikzcd}
& B \arrow[dr,"\triangleleft", bend left = 30] && \\
 A \arrow[ur, "\triangleleft", bend left = 30] \arrow[r,"\otimes",swap] & B \arrow[r,"\triangleleft",swap] \arrow[Rightarrow,u, shorten >=2.5pt, shorten <=2.5pt, "\beta_\Ma^{-1}",swap] & C \arrow[r,"H"] & \Ea  
\end{tikzcd} 
\label{eq:sigmaG0} 
\end{equation}

\begin{equation}
\begin{tikzcd}[column sep = large, row sep = tiny]
&& C \arrow[ddr,"H", bend left = 20] & \\
& B \arrow[ur,"\triangleleft",bend left =20] \arrow[dr,"\triangleright",bend left = 10,swap] && \\
A \arrow[ur, "\triangleright", bend left = 10,swap] \arrow[dr, "\otimes",bend right = 10] && C \arrow[uu,Rightarrow, shorten >=5pt, shorten <=5pt,"\alpha^{-1}",swap] \arrow[r,"H"] & \Ea \\
& B \arrow[ur,"\triangleright",bend right = 10] \arrow[uu, Rightarrow, shorten >=2.5pt, shorten <=2.5pt, "\beta_\Na^{-1}",swap] \arrow[dr,"\triangleleft",bend right = 10,swap] && \\
&& C \arrow[uu,Rightarrow,shorten >=5pt, shorten <=5pt,"\alpha",swap] \arrow[uur,bend right = 20,"H",swap] &
\label{eq:sigmaG2} 
\end{tikzcd}
\end{equation}
It is immediate that $\Ia_\Ea (\sigma) = (H, \alpha)$. However, we need to check that $\sigma$ is indeed an element in $[\Da,\VCat](X,\Delta(\Ea))$, i.e.~that the Equations \eqref{eq:Kappas} are satisfied. The equation for $\kappa_1$ is satisfied due to \eqref{diag:balancedfunctor} for the balancing. The equations for $\kappa_2$ and $\kappa_3$ follows directly from the definitions of the $\sigma_{F_i}$ and $\sigma_{G_i}$.

\paragraph{$\Ia_\Ea$ is fully faithful.}

Let $\gamma \colon H \to K$ be a map of $\Aa$-balanced functors. Let $\sigma$ and $\eta$ be cocones defined as in the previous paragraph such that $\Ia_\Ea(\sigma) = (H,\alpha_H)$ and $\Ia_\Ea(\eta) = (K,\alpha_K)$. Define the natural isomorphism 
\begin{center}
\begin{tikzcd}[sep = large]
C \arrow[r, bend left, "\sigma_C = H"{name=A}]
   \arrow[r, swap, bend right, "\eta_C = K"{name=B}]
 & \Ea \arrow[Rightarrow, shorten >=2.5pt, shorten <=2.5pt,from=A, to=B, "\Xi_C"]
\end{tikzcd}
\end{center}
by $\Xi_C = \gamma$. We have to show that $\gamma$ is well-defined. Equation \eqref{eq:Xi1} reads 
\begin{center}
\begin{tikzcd}[column sep = large, row sep = normal]
B \arrow[r,"\triangleleft"] & A \arrow[r,"H"{name=1}] \arrow[r,"K"{name=2},bend right = 80,swap] & \Ea & = & B \arrow[ddr,"\triangleleft",swap,bend right = 40] \arrow[dr,"\triangleright",bend right = 20] \arrow[r,"\triangleleft"] & A \arrow[d,Rightarrow,shorten >=2.5pt, shorten <=2.5pt,"\alpha_H",swap]\arrow[r,"H"] & \Ea \\
&&&&& A \arrow[d,Rightarrow,shorten >=2.5pt, shorten <=2.5pt,"\alpha_K^{-1}"] \arrow[ur,"H"{name=3},near start,bend left = 20] \arrow[ur,"K"{name=4},bend right = 30,swap] & \\
&&&&& A  \arrow[uur,"K",bend right = 40,swap]
\arrow[Rightarrow,shorten >=5pt, shorten <=5pt,from=1,to=2,"\gamma"]
\arrow[Rightarrow,shorten >=5pt, shorten <=5pt,from=3,to=4,"\gamma"]
\end{tikzcd}
\end{center}
which holds by definition of a $\Aa$-balanced natural transformation. For Equations \eqref{eq:Xi2} one needs to show that $(\Delta^\eta_{ij})^{-1}(\gamma F_i G_j) \Delta^\sigma_{ij}$ does not depend on the values of $i$ and $j$. Independence of $i$ is again a direct consequence of $\eta$ being a $\Aa$-balanced natural transformation. For a detailed proof that the expression is also independent of $j$, we refer to the proof given in \cite[Theorem 2.27]{Cooke19}, which directly carries over to the enriched setting.
\end{proof}

\section{Pullbacks of bicategories} \label{appendix:pullback}
In Section \ref{sec:DefQuantPullback} we defined a pullback of symmetric monoidal bicategories ${\E_\infty(\C\mhyphen\mathsf{Cat})}$ and ${\E_i(\Ce\mhyphen\mathsf{Cat})}$ over ${\E_i(\C\mhyphen\mathsf{Cat})}$. In this appendix we consider slightly more general symmetric monoidal bicategories and prove some useful results about their pullbacks.
\par

Let us first specify what kind of symmetric monoidal bicategories (SMBCs) are relevant to us. Our motivating example is the bicategory $R\mhyphen\mathsf{Cat}$ of $R$-linear categories, for $R$ a commutative ring. As a bicategory, it is strict. Its ``monoidal product'' homomorphism $\catprod$ is strict, i.e. the components $\phi^\times$ in \cite[Def.~2.3]{SchommerPriesThesis} are identities. Still using the notation of loc.cit., the transformations $\alpha, \ell, r, \beta$ are isomorphisms, with identities as 2-cells, and the modifications $\pi, \mu, \lambda, \rho$ are identities as well. In short, $R\mhyphen\mathsf{Cat}$ has no non-trivial coherence 2-cells in its definition. 
\begin{notation}
    We call symmetric monoidal bicategories as above \emph{2-strict}. 
    Further, we only consider such 2-strict SMBCs here, and henceforth suppress 2-strict from the notation. 
\end{notation}
Note that this differs from a notion of (symmetric) Gray monoid \cite[Def.~2.26]{SchommerPriesThesis}, as the transformations $\alpha, r, \ell$ are not identities. One can moreover check that for \(\Caa\) a 2-strict SMBC, $\E_1(\Caa)$ is again a 2-strict SMBC. In this case the 2-cells giving $\E_1$-structure on the 1-cells $\alpha, \ell, r, \beta$ can be chosen as identities.
\par

Similarly, our prototype symmetric monoidal homomorphism is $-\otimes_R S \colon R\mhyphen \mathsf{Cat} \to S\mhyphen \mathsf{Cat}$ coming from a ring homomorphism $R\to S$. Using the notation of \cite[Def.~2.5]{SchommerPriesThesis}, the transformations $\chi$ are isomorphisms, the 2-cells $\chi_{fg}$ are identities and the modifications $\omega, \gamma, \delta, u$ are identities as well. Again, we will assume that our symmetric monoidal homomorphisms are 2-strict in this sense. If $H\colon \Caa \to \Daa$ is 2-strict, then so is the $\E_1$-extension $\E_1(H)\colon \E_1(\Caa) \to \E_1(\Daa)$. 

\begin{definition}\label{def:pullback}
Consider a diagram of of 2-strict symmetric monoidal bicategories and 2-strict symmetric monoidal homomorphisms  
\[  \Caa_1  \xrightarrow{H_1} \mathsf{\Daa} \xleftarrow{H_2} \Caa_2. \]
Then the \emph{pullback $\Caa_1 \times_{\Daa} \Caa_2$} is defined by cells
    \begin{align*}
        \text{a $0$-cell $\Ca$: } &(\Ca_1 \in \Caa_1, \Ca_2 \in \Caa_2, \psi_\Ca \colon H_1\Ca_1 \xrightarrow{\cong} H_2 \Ca_2 \in \Daa) \\
        \text{a $1$-cell $F\colon \Ca \to \Da$: } &(F_1\colon \Ca_1\to \Da_1 \in \Caa_1, F_2\colon \Ca_2\to\Da_2 \in \Caa_2, \varepsilon_F \colon \psi_\Da \circ H_1F_1 \Rightarrow H_2F_2 \circ \psi_\Ca \in \Daa)\\
        \text{a $2$-cell $\alpha\colon F \Rightarrow G$: } &(\alpha_1\colon F_1 \Rightarrow G_1 \in \Caa_1, \alpha_2\colon F_2\Rightarrow G_2 \in \Caa_2)
    \end{align*}
    where $\alpha_i$ intertwine $\varepsilon_F$ and $\varepsilon_G$ as in \eqref{eq:pullback2cell}, i.e.~
\begin{equation}\label{eq:pullback2cellgen}\begin{tikzcd}[sep=large]
	{H_1\Ca_1} & {H_2\Ca_2} \\
	{H_1\Da_1} & {H_2\Da_2}
	\arrow["{\psi_\Ca}", from=1-1, to=1-2]
	\arrow["{\psi_\Da}"', from=2-1, to=2-2]
	\arrow[""{name=0, anchor=center, inner sep=0}, "{H_1G_1}"{pos=0.4}, curve={height=-18pt}, from=1-1, to=2-1]
	\arrow["{H_2G_2}", curve={height=-18pt}, from=1-2, to=2-2]
	\arrow[""{name=1, anchor=center, inner sep=0}, "{H_1F_1}"'{pos=0.4}, curve={height=18pt}, from=1-1, to=2-1]
	\arrow["{\varepsilon_G}"{description}, shift right=3, shorten <=7pt, Rightarrow, from=2-1, to=1-2]
	\arrow["{H_1\alpha_1}", shorten <=7pt, shorten >=7pt, Rightarrow, from=1, to=0]
\end{tikzcd} \quad =
\quad
\begin{tikzcd}[sep=large]
	{H_1\Ca_1} & {H_2\Ca_2} \\
	{H_1\Da_1} & {H_2\Da_2}
	\arrow["{\psi_\Ca}", from=1-1, to=1-2]
	\arrow["{\psi_\Da}"', from=2-1, to=2-2]
	\arrow[""{name=0, anchor=center, inner sep=0}, "{H_2G_2}"{pos=0.6}, curve={height=-18pt}, from=1-2, to=2-2]
	\arrow["{H_1F_1}"', curve={height=18pt}, from=1-1, to=2-1]
	\arrow[""{name=1, anchor=center, inner sep=0}, "{H_2F_2}"'{pos=0.6}, curve={height=18pt}, from=1-2, to=2-2]
	\arrow["{\varepsilon_F}"{description, pos=0.4}, shift left=3, shorten >=7pt, Rightarrow, from=2-1, to=1-2]
	\arrow["{H_2\alpha_2}", shorten <=7pt, shorten >=7pt, Rightarrow, from=1, to=0]
\end{tikzcd}.
\end{equation}
The composition of 1-cells $G\circ F$, with $F\colon \Ca \to \Da$ and $G \colon \Da \to \Ea$, is defined by
\begin{equation}
    G\circ F = \left(G_1\circ F_1, G_2 \circ F_2, \begin{tikzcd}
	{H_1 \Ca_1} & {H_1 \Da_1} & {H_1 \Ea_1} \\
	{H_2 \Ca_2} & {H_2 \Da_2} & {H_2 \Ea_2}
	\arrow["{H_1 F_1}", from=1-1, to=1-2]
	\arrow["{H_1 G_1}", from=1-2, to=1-3]
	\arrow["{\psi_\Ca}"', from=1-1, to=2-1]
	\arrow["{H_2 F_2}"', from=2-1, to=2-2]
	\arrow["{H_2 G_2}"', from=2-2, to=2-3]
	\arrow["{\psi_\Da}"{description}, from=1-2, to=2-2]
	\arrow["{\psi_\Ea}", from=1-3, to=2-3]
	\arrow["{\varepsilon_F}"', shorten <=4pt, shorten >=4pt, Rightarrow, from=1-2, to=2-1]
	\arrow["{\varepsilon_G}"', shorten <=4pt, shorten >=4pt, Rightarrow, from=1-3, to=2-2]
\end{tikzcd}\right)
\end{equation}
For 2-cells, the vertical and horizontal composition is defined by the respective compositions in $\Caa_i$. 
\end{definition}
\begin{notation}
    We will call the cells $\Ca_i$, $F_i$ and $\alpha_i$ of $\Caa_i$ the \emph{components} of the cells $\Ca$, $F$ and $\alpha$, respectively. 
\end{notation}

\begin{definition}[Symmetric monoidal structure on the pullback]
    Let $\Ca, \Da$ be two 0-cells of $\Caa_1 \times_{\Daa} \Caa_2$. Their product is defined by
    \[ \Ca\catprod \Da = (\Ca_1 \catprod \Da_1, \Ca_2 \catprod \Da_2, H_1(\Ca_1\catprod \Da_1)\cong H_1 \Ca_1 \catprod H_1\Da_1 \xrightarrow{\psi_\Ca \catprod \psi_\Da} H_2\Ca_2\catprod H_2\Da_2 \cong H_2(\Ca_2\catprod \Da_2)  ). \]
    For two 1-cells $F\colon \Ca \to \Ca'$ and $G\colon \Da \to \Da'$, the product $F\catprod G$ is defined by having components $F_i \catprod G_i$ and the $\varepsilon_{F\catprod G}$ 2-cell is given by 
    \[\begin{tikzcd}
	{H_1(\Ca_1 \catprod \Da_1)} &&[20pt]& {H_1(\Ca_1'\catprod \Da_1')} \\
	& {H_1\Ca_1 \catprod H_1\Da_1} & {H_1\Ca_1' \catprod H_1\Da_1'} \\[10pt]
	& {H_2\Ca_2 \catprod H_2\Da_2} & {H_2\Ca_2' \catprod H_2\Da_2'} \\
	{H_2(\Ca_2 \catprod \Da_2)} &&& {H_2(\Ca_2' \catprod \Da_2')}
	\arrow["{H_1(F_1 \catprod G_1)}", from=1-1, to=1-4]
	\arrow[from=1-1, to=4-1]
	\arrow[from=2-2, to=2-3]
	\arrow["{H_2F_2\catprod H_2 G_2}"', from=3-2, to=3-3]
	\arrow["{\psi_\Ca\catprod \psi_\Da}"{description}, from=2-2, to=3-2]
	\arrow["{H_2(F_2\catprod G_2)}"', from=4-1, to=4-4]
	\arrow["{\psi_\Ca'\catprod \psi_\Da'}"{description}, from=2-3, to=3-3]
	\arrow["{H_1F_1\catprod H_1 G_1}", from=2-2, to=2-3]
	\arrow["\cong"', from=1-1, to=2-2]
	\arrow["\cong"', from=4-1, to=3-2]
	\arrow["\cong"', from=1-4, to=2-3]
	\arrow["\cong"', from=4-4, to=3-3]
	\arrow[from=1-4, to=4-4]
	\arrow["{\varepsilon_F\catprod \varepsilon_G}"'{pos=0.4}, shorten <=8pt, shorten >=8pt, Rightarrow, from=2-3, to=3-2]
\end{tikzcd}\]
The product of 2-cells is given by the corresponding product of their components.
\end{definition}
\begin{proposition}\label{prop:SMBCpullback}
    The above structure endows the pullback $\Caa_1 \times_{\Daa} \Caa_2$ with the structure of a 2-strict symmetric monoidal bicategory.
\end{proposition}
\begin{proof}
    Using the notation of \cite[Def.~2.3]{SchommerPriesThesis} and our assumption on the 2-strictness of our bicategories and homomorphisms, it is straightforward to check that the pullback $\Caa_1 \times_{\Daa} \Caa_2$ is again 2-strict. Moreover, the $2$-cells $\varepsilon$ associated to the $1$-cells $\alpha, \ell, r, \beta$ of $\Caa_1 \times_{\Daa} \Caa_2$ are identities. The axioms for a symmetric monoidal bicategory are then trivially satisfied, as all the relevant 2-cells are identities.
\end{proof}
\begin{proposition}\label{prop:SMBCpullbackmap}
    Consider the following diagram of 2-strict SMBCs, symmetric monoidal homomorphisms and symmetric monoidal transformations (i.e symmetric monoidal transformations $X_i$ with all 2-cell components equal to identity)
\[\begin{tikzcd}
	{\Caa_1} & \Daa & {\Caa_2} \\
	{\tilde{\Caa}_1} & \tilde\Daa & {\tilde{\Caa}_2}
	\arrow["{H_1}", from=1-1, to=1-2]
	\arrow["{H_2}"', from=1-3, to=1-2]
	\arrow["{\varphi_1}"', from=1-1, to=2-1]
	\arrow["{\tilde{H}_1}"', from=2-1, to=2-2]
	\arrow["{\varphi_2}", from=1-3, to=2-3]
	\arrow["{\tilde{H}_2}", from=2-3, to=2-2]
	\arrow["{\varphi_D}"{description}, from=1-2, to=2-2]
	\arrow["{X_1}"{description}, draw=none, from=2-1, to=1-2]
	\arrow["{X_2}"{description}, draw=none, from=2-2, to=1-3]
\end{tikzcd}\]
Then there is a 2-strict symmetric monoidal homomorphism $\Phi :=\varphi_1\times_{\varphi_D}\varphi_2 \colon \Caa_1\times_\Daa \Caa_2 \to \tilde{\Caa}_1 \times_{\tilde{\Daa}}\tilde{\Caa}_2$   given by  sending a 0-cell $(\Ca_1, \Ca_2, \psi\colon H_1\Ca_1 \to H_2 \Ca_2)$ to
\[    (\varphi_1(\Ca_1), \varphi_2(\Ca_2), \tilde H_1 \varphi_1 \Ca_1 \xrightarrow{(X_1)_{\Ca_1}} \varphi_\Daa H_1 \Ca_1 \xrightarrow{\varphi_\Daa(\psi)} \varphi_\Daa H_2 \Ca_2 \xrightarrow{(X_2)_{\Da_2}^{-1}} \tilde{H}_2 \phi_2 \Ca_2),\]
sending a 2-cell $(F_1, F_2, \varepsilon_F \colon \psi_\Da H_1 F_1 \Rightarrow H_2 F_2 \psi_\Ca)$ to $(\varphi_1 F_1, \varphi_2 F_2, \varepsilon_{\varphi F})$ with $\varepsilon_{\varphi F}$ given by
\[\begin{tikzcd}
	{\tilde{H}_1 \varphi_1 \Ca_1} &&& {\tilde{H}_2\varphi_2 \Ca_2} \\
	& {\varphi_\Daa H_1 \Ca_1} & {\varphi_\Daa H_2 \Ca_2} \\
	& {\varphi_\Daa H_1 \Da_1} & {\varphi_\Daa H_2 \Da_2} \\
	{\tilde{H}_1 \varphi_1 \Da_1} &&& {\tilde{H}_2\varphi_2 \Da_2}
	\arrow["{(X_1)_{\Ca_1}}"{description}, from=1-1, to=2-2]
	\arrow["{\varphi_\Daa \psi_\Ca}", from=2-2, to=2-3]
	\arrow["{\varphi_\Daa H_2 F_2}", from=2-3, to=3-3]
	\arrow["{\varphi_\Daa H_1 F_1}"', from=2-2, to=3-2]
	\arrow["{\varphi_\Daa \psi_\Da}"', from=3-2, to=3-3]
	\arrow["{\tilde{H}_1 \varphi_1 F_1}"{description}, from=1-1, to=4-1]
	\arrow["{\psi_{\varphi\Da}}"{description}, from=4-1, to=4-4]
	\arrow["{\psi_0{\varphi\Ca}}"{description}, from=1-1, to=1-4]
	\arrow["{\tilde{H}_2\varphi_2 F_2}"{description}, from=1-4, to=4-4]
	\arrow["{(X_1)_{\Da_1}}"{description}, from=4-1, to=3-2]
	\arrow["{(X_2)_{\Da_2}}"{description}, from=4-4, to=3-3]
	\arrow["{(X_2)_{\Ca_2}}"{description}, from=1-4, to=2-3]
	\arrow["{\varphi_\Daa \varepsilon^F}"{description}, draw=none, from=3-2, to=2-3]
\end{tikzcd}\]
and on 2-cells it is given by applying $\varphi_i$ to its components.
\end{proposition}
\begin{proof}
    It is straightforward to check that the resulting functor $\Phi:= \varphi_1 \times_{\varphi_\Daa} \varphi_2$ is 2-strict, with the comparison 1-cell
    $\chi_{\Ca, \Da}\colon \Phi(\Ca)\catprod \Phi(\Da) \to \Phi(\Ca\catprod \Da)$ given by
\[\begin{tikzcd}
	{(\varphi_1 \Ca_1\catprod \varphi_1 \Da_1,} & {\varphi_2 \Ca_2\catprod \varphi_2 \Da_2,} & {\psi_{\Phi\Ca \catprod \Phi\Da})} \\
	{(\varphi_1 (\Ca_1\catprod\Da_1),} & {\varphi_2 (\Ca_2\catprod\Da_2),} & {\psi_{\Phi(\Ca\catprod \Da)})}
	\arrow["{\chi^{\varphi_1}_{\Ca_1, \Da_1}}"{description}, from=1-1, to=2-1]
	\arrow["{\chi^{\varphi_2}_{\Ca_2, \Da_2}}"{description}, from=1-2, to=2-2]
	\arrow["{\varepsilon = \text{id}}"{description}, no head, from=1-3, to=2-3]
\end{tikzcd}\]
From the 2-strictness, the axioms of a symmetric monoidal homomorphism follow.
\end{proof}

\begin{proposition}\label{prop:Eiandpullback} There is a strict symmetric monoidal 2-equivalence  
    \[ \E_n(\Caa_1 \times_{\Daa} \Caa_2) \cong \E_n\Caa_1 \times_{\E_n\mathsf \Daa} \E_n\Caa_2. \]
\end{proposition}
\begin{proof}
    From the case $i=1$, the cases $i>1$ follow by Dunn's additivity, while the case $i=0$ easily follows from the computations for $i=1$. Let us therefore sketch the $i=1$ case, and moreover consider only strict SM (2, 1)-categories similar to $\C \mhyphen\Cat$ (e.g. the associator adjoint equivalence is an isomorphism). 
    
    An $\E_1$-algebra in $\mathsf{C}_{1}\times_{\mathsf N} \mathsf{C}_2$ is thus given by a tuple 
    \[(\Ca_i, \psi_\Ca \colon H_1 \Ca_1 \to H_2 \Ca_2,  p_{\Ca_i}\colon I_{\mathsf C_i}\to \Ca_i, \varepsilon_{p_\Ca}, m_{\Ca_i}\colon \Ca_i^2\to \Ca_i, \varepsilon_{m_\Ca}, \phi_{\Ca_i}, \rho_{\Ca_i},\lambda_{\Ca_i}),\] 
    where the $2$-cells $\phi, \rho, \lambda$ are compatible with the $2$-cells $\varepsilon$ as above in \eqref{eq:pullback2cellgen}.

    On the other hand, a 0-cell in the pullback  $\E_1\mathsf{C}_{1}\times_{\E_1\mathsf N} \E_1\mathsf{C}_2$ is given by 
    \[(\Ca_i, p_{\Ca_i}\colon I_{\mathsf C_i}\to \Ca_i,  m_{\Ca_i}\colon \Ca_i^2\to \Ca_i, \phi_{\Ca_i}, \rho_{\Ca_i},\lambda_{\Ca_i}, \psi_{\Ca}\colon H_1 \Ca_1 \to H_2 \Ca_2, \pi_{\psi_\Ca}, \mu_{\psi_{\Ca}}),\] 
    where $\pi_{\psi_{\Ca}}, \mu_{\psi_{\Ca}}$ are two-cells making the 1-cell $\psi_{\Ca'}$ a 1-cell in $\E_1\Daa$ (i.e. the pointing and the monoidal constraint on $\psi_\Ca$). 
    
    We can identify these 0-cells by declaring $\pi_{\psi_\Ca} =\varepsilon_{p_\Ca}$ and $\mu_{\psi_\Ca} = \varepsilon_{m_\Ca} $, using canonical isomorphisms such as $H_1(\Ca_1^2)\cong (H_1\Ca_1)^2$. Then, one can check that \eqref{eq:pullback2cellgen} for 2-cells $(\phi_{\Ca_i}, \rho_{\Ca_i},\lambda_{\Ca_i})$ are equivalent to the conditions saying that $\psi_{\Ca'}, \pi_{\psi_{\Ca'}}, \mu_{\psi_{\Ca'}}$ is an $\E_1$ (i.e. monoidal) 1-cell. This gives a bijection between 0-cells of  $\E_1(\mathsf{C}_{1}\times_{\mathsf N} \mathsf{C}_2)$ and  $\E_1\mathsf{C}_{1}\times_{\E_1\mathsf N} \E_1\mathsf{C}_2$.

    For a 1-cell $((F, \varepsilon_{F}), \pi_{F}, \mu_{F}) \in \E_1(\mathsf{C}_{1}\times_{\mathsf N} \mathsf{C}_2)$, the condition \eqref{eq:pullback2cellgen} is equivalent to the condition for $((F, \pi_F, \mu_F), \varepsilon_F) \in \E_1\mathsf{C}_{1}\times_{\E_1\mathsf N} \E_1\mathsf{C}_2$ saying that the 2-cell $\varepsilon_F$ is an $\E_1$ (monoidal) 2-cell.  Finally, 2-cells in both categories are pairs of 2-cells $\alpha_i\colon F_i \Rightarrow G_i$, and their compatibilities with other structures are equivalent. 

    This way, we obtain a bijection of cells between these pullback 2-categories. It is easy to check that this bijection is strictly compatible with composition. Finally, one checks that it respects the unit 0-cells, products and symmetry of these 2-categories in a strict way (\cite[Def.~2.5]{SchommerPriesThesis}).
    \end{proof}

\newpage
\addcontentsline{toc}{section}{References}
\printbibliography

@misc{stacks-project,
  author       = {The {Stacks project authors}},
  title        = {The Stacks project},
  howpublished = {\url{https://stacks.math.columbia.edu}},
  year         = {2021},
}

@article{Po20,
    author = {Leonid Positselski},
    title = {Contraadjusted modules, contramodules and reduced cotorsion modules},
    year = {2020},
}

@article{BQR,
    author = {Borceux, Francis and Quinteiro, Carmen and Rosický, Jiří},
    title = {A theory of enriched sketches},
    year = {1998},
    journal={Theory and Applications of Categories},
    volume = "4",
    number = "3",
    pages = "47-72",
}

@article{gepner2020integral,
  title={Integral representation theorems for DQ-modules},
  author={Gepner, David and Petit, Francois},
  journal={arXiv preprint arXiv:2004.10176},
  year={2020}
}

@book{CoendCalc,
   title={(Co)end Calculus},
   ISBN={9781108746120},
   url={http://dx.doi.org/10.1017/9781108778657},
   DOI={10.1017/9781108778657},
   publisher={Cambridge University Press},
   author={Loregian, Fosco},
   year={2021},
   month=jun,
    eprint={1501.02503},
    archivePrefix={arXiv},
    primaryClass={math.CT}
}

@article{AQFTCat,
    author = "Benini, Marco and Perin, Marco and Schenkel, Alexander and Woike, Lukas",
    title = "{Categorification of algebraic quantum field theories}",
    eprint = "2003.13713",
    archivePrefix = "arXiv",
    primaryClass = "math-ph",
    reportNumber = "ZMP-HH/20-8, Hamburger Beitraege zur Mathematik Nr. 830",
    doi = "10.1007/s11005-021-01371-8",
    journal = "Lett. Math. Phys.",
    volume = "111",
    number = "2",
    pages = "35",
    year = "2021"
}

@article {Wolff,
    AUTHOR = {Wolff, Harvey},
     TITLE = {{$V$}-cat and {$V$}-graph},
   JOURNAL = {J. Pure Appl. Algebra},
  FJOURNAL = {Journal of Pure and Applied Algebra},
    VOLUME = {4},
      YEAR = {1974},
     PAGES = {123--135},
      ISSN = {0022-4049,1873-1376},
   MRCLASS = {18D10},
  MRNUMBER = {346029},
MRREVIEWER = {F.\ E. J. Linton},
       DOI = {10.1016/0022-4049(74)90018-8},
}

@book{Riehl,
  title = {Categorical Homotopy Theory},
  ISBN = {9781107261457},
  DOI = {10.1017/cbo9781107261457},
  publisher = {Cambridge University Press},
  author = {Riehl,  Emily},
  year = {2014},
  month = may 
}

@article{AFfh,
author = {Ayala, David and Francis, John},
title = {Factorization homology of topological manifolds},
journal = {Journal of Topology},
volume={8},
number={4},
pages = {1045--1084},
year = {2015}
}

@article {AFT-stratified,
    AUTHOR = {Ayala, David and Francis, John and Tanaka, Hiro Lee},
     TITLE = {Factorization homology of stratified spaces},
   JOURNAL = {Selecta Math. (N.S.)},
  FJOURNAL = {Selecta Mathematica. New Series},
    VOLUME = {23},
      YEAR = {2017},
    NUMBER = {1},
     PAGES = {293--362},
      ISSN = {1022-1824,1420-9020},
   MRCLASS = {57P05 (55N40 57N80 57R40)},
  MRNUMBER = {3595895},
MRREVIEWER = {Jason\ Stuart\ Hanson},
       DOI = {10.1007/s00029-016-0242-1},
}

@book{HLTetal,
   title={Lectures on Factorization Homology, $\infty$-Categories, and Topological Field Theories},
   ISBN={9783030611637},
   ISSN={2197-1765},
   DOI={10.1007/978-3-030-61163-7},
   journal={SpringerBriefs in Mathematical Physics},
   publisher={Springer International Publishing},
   year={2020},
    author={Araminta Amabel and Artem Kalmykov and Lukas Müller and Hiro Lee Tanaka},
    year={2019},
    eprint={1907.00066},
    archivePrefix={arXiv},
    primaryClass={math.AT}}

@article{ABSV,
  title = {Infinitesimal braidings and pre-Cartier bialgebras},
  ISSN = {1793-6683},
  DOI = {10.1142/s0219199724500299},
  journal = {Communications in Contemporary Mathematics},
  publisher = {World Scientific Pub Co Pte Ltd},
  author = {Ardizzoni,  Alessandro and Bottegoni,  Lucrezia and Sciandra,  Andrea and Weber,  Thomas},
  year = {2024},
  month = aug 
}

@article{HV,
  title = {Bosonization of curved Lie bialgebras},
  volume = {30},
  ISSN = {1370-1444},
  DOI = {10.36045/j.bbms.221202},
  number = {5},
  journal = {Bulletin of the Belgian Mathematical Society - Simon Stevin},
  publisher = {The Belgian Mathematical Society},
  author = {Heckenberger,  Istvan and Vendramin,  Leandro},
  year = {2023},
  month = dec,
    eprint={2209.02115},
    archivePrefix={arXiv},
    primaryClass={math.QA},
    eprint={2306.00558},
    archivePrefix={arXiv},
    primaryClass={math.QA}
}

@book{CoherenceTriCats,
author = {Gordon,  Robert and Power,  Anthony J. and Street, Ross},
  title     = "Coherence for Tricategories",
  publisher = "American Mathematical Society",
  series    = "Memoirs of the American Mathematical Society",
  month     =  sep,
  year      =  1995,
  address   = "Providence, RI"
}

@article{Campbell,
  title = {How strict is strictification?},
  volume = {223},
  ISSN = {0022-4049},
  DOI = {10.1016/j.jpaa.2018.10.004},
  number = {7},
  journal = {Journal of Pure and Applied Algebra},
  publisher = {Elsevier BV},
  author = {Campbell,  Alexander},
  year = {2019},
  month = jul,
  pages = {2948–2976},
    eprint={1802.07538},
    archivePrefix={arXiv},
    primaryClass={math.CT}
}

@article{Kelly2limits,
author = {Kelly, Max },
title = {Elementary observations on 2-categorical limits},
year = {1989}
}

@article{KellyVCat,
author = {Kelly, Max},
title = {Basic concepts of enriched category theory},
journal = {Reprints in Theory and Applications of Categories},
number = {10},
year = {2005}
}

@Book{MSSOperads,
  author    = {Markl, Martin and Shnider, Steve and Stasheff, James D.},
  publisher = {American Mathematical Society},
  title     = {Operads in Algebra, Topology and Physics},
  year      = {2002},
  isbn      = {9780821843628},
  series    = {Mathematical surveys and monographs},
  lccn      = {2002016342},
}

@Article{DayStreet1997,
  author   = {Day, Brian and Street, Ross},
  journal  = {Advances in Mathematics},
  title    = {Monoidal {Bicategories} and {Hopf} {Algebroids}},
  year     = {1997},
  issn     = {0001-8708},
  month    = jul,
  number   = {1},
  pages    = {99--157},
  volume   = {129},
  doi      = {10.1006/aima.1997.1649},
}

@Article{JoyalStreet1993,
  author  = {Joyal, André and Street, Ross},
  journal = {Advances in Mathematics},
  title   = {Braided {Tensor} {Categories}},
  year    = {1993},
  issn    = {0001-8708},
  month   = nov,
  number  = {1},
  pages   = {20--78},
  volume  = {102},
  doi     = {10.1006/aima.1993.1055},
}

@article{BZBJIntegrating,
  title = {Integrating quantum groups over surfaces},
  volume = {11},
  ISSN = {1753-8424},
  DOI = {10.1112/topo.12072},
  number = {4},
  journal = {Journal of Topology},
  publisher = {Wiley},
  author = {Ben‐Zvi,  David and Brochier,  Adrien and Jordan,  David},
  year = {2018},
  month = aug,
  pages = {874–917},
    eprint={1501.04652},
    archivePrefix={arXiv},
    primaryClass={math.QA}
}

@book {CruttwellThesis,
    AUTHOR = {Cruttwell, Geoff  S. H.},
     TITLE = {Normed spaces and the change of base for enriched categories},
      NOTE = {Thesis (Ph.D.)--Dalhousie University (Canada)},
 PUBLISHER = {ProQuest LLC, Ann Arbor, MI},
      YEAR = {2009},
     PAGES = {143},
      ISBN = {978-0494-50056-9},
   MRCLASS = {99-05},
  MRNUMBER = {2713532},
       URL = {https://www.reluctantm.com/gcruttw/publications/thesis4.pdf},
}

@article{Cooke19,
  title = {Excision of skein categories and factorisation homology},
  volume = {414},
  ISSN = {0001-8708},
  DOI = {10.1016/j.aim.2022.108848},
  journal = {Advances in Mathematics},
  publisher = {Elsevier BV},
  author = {Cooke,  Juliet},
  year = {2023},
  month = feb,
  pages = {108848},
    eprint={1910.02630},
    archivePrefix={arXiv},
    primaryClass={math.QA}
}

@Article{KasselTuraev1998,
  author    = {Christian Kassel and Vladimir Turaev},
  journal   = {Duke Mathematical Journal},
  title     = {{Chord diagram invariants of tangles and graphs}},
  year      = {1998},
  number    = {3},
  pages     = {497 -- 552},
  volume    = {92},
  doi       = {10.1215/S0012-7094-98-09215-8},
  publisher = {Duke University Press},
}

@Article{Cartier1993,
  author   = {Cartier, Pierre},
  journal  = {Les rencontres physiciens-mathématiciens de Strasbourg -RCP25},
  title    = {Construction combinatoire des invariants de {Vassiliev}-{Kontsevich} des nœuds},
  year     = {1993},
  pages    = {1--10},
  volume   = {45},
  language = {fr},
  url      = {http://www.numdam.org/item/RCP25_1993__45__1_0/},
}

@article{Kontsevich2003,
  title = {Deformation Quantization of Poisson Manifolds},
  volume = {66},
  ISSN = {0377-9017},
  DOI = {10.1023/b:math.0000027508.00421.bf},
  number = {3},
  journal = {Letters in Mathematical Physics},
  publisher = {Springer Science and Business Media LLC},
  author = {Kontsevich,  Maxim},
  year = {2003},
  month = dec,
  pages = {157–216},
    eprint={q-alg/9709040},
    archivePrefix={arXiv},
    primaryClass={q-alg}
}

@Article{DrinfeldGal1990,
  author   = {Drinfeld, Vladimir G.},
  journal  = {Algebra i Analiz},
  title    = {On quasitriangular quasi-{Hopf} algebras and on a group that is closely connected with {$\mathrm{Gal}(\overline{\mathbb {Q}}/\mathbb{Q})$}},
  year     = {1990},
  issn     = {0234-0852},
  note     = {(\textit{translation in} Leningrad Mathematical Journal, 1991, 2:4, 829–860)},
  number   = {4},
  pages    = {149--181},
  volume   = {2},
  mrnumber = {1080203},
}

@article{DrinQG,
author = {Drinfeld, Vladimir G.},
title = {Quantum groups},
journal = {Journal of Soviet Mathematics},
volume = {41},
pages = {898--915},
year = {1988}
}

@Article{DrinfeldQuasiHopf,
  author   = {Drinfeld, Vladimir G.},
  journal  = {Algebra i Analiz},
  title    = {Quasi-{Hopf} algebras},
  year     = {1989},
  issn     = {0234-0852},
  note     = {(\textit{translation in} Leningrad Mathematical Journal, 1990, 1:6, 1419–1457},
  number   = {6},
  pages    = {114--148},
  volume   = {1},
  mrnumber = {1047964},
  url      = {https://mathscinet.ams.org/mathscinet-getitem?mr=1047964},
}

@incollection{DrinKZ,
author = {Drinfeld, Vladimir G.},
title = {Quasi-{H}opf algebras and {K}nizhnik-{Z}amolodchikov equations},
book = {Problems of Modern Quantum Field Theory},
series = {Research Reports in Physics},
publisher = {Springer},
pages = {1--13},
year = {1989}
}

@book{Kassel,
author = {Kassel, Christian},
title = {Quantum groups},
publisher = {Springer},
address = {New York, N.Y.},
year = {1995}
}

@Article{LeMurakamiParallel,
  author  = {Le, Thang T. Q. and Murakami, Jun},
  journal = {Journal of Pure and Applied Algebra},
  title   = {Parallel version of the universal {Vassiliev}-{Kontsevich} invariant},
  year    = {1997},
  issn    = {0022-4049},
  month   = oct,
  number  = {3},
  pages   = {271--291},
  volume  = {121},
  doi     = {10.1016/S0022-4049(96)00054-0},
}

@book{daSilvaWeinstein,
  title={Geometric models for noncommutative algebras},
  author={Da Silva, Ana Cannas and Weinstein, Alan},
  volume={10},
  year={1999},
  publisher={American Mathematical Soc.},
  isbn = {978-0-8218-0952-5},
  series = { Berkeley Mathematics Lecture Notes}
}

@article{Knudson,
 ISSN = {00029947},
 author = {Knudson, David W. },
 journal = {Transactions of the American Mathematical Society},
 pages = {55--70},
 publisher = {American Mathematical Society},
 title = {On the Deformation of Commutative Algebras},
 volume = {140},
 year = {1969},
 doi = {10.2307/1995122}
}

@Article{HKR,
  author   = {Hochschild, Gerhard and Kostant, Bertram and Rosenberg, Alex},
  journal  = {Transactions of the American Mathematical Society},
  title    = {Differential forms on regular affine algebras},
  year     = {1962},
  issn     = {0002-9947, 1088-6850},
  number   = {3},
  pages    = {383--408},
  volume   = {102},
  doi      = {10.1090/S0002-9947-1962-0142598-8},
}

@article{GuttRawnsley,
title = {Equivalence of star products on a symplectic manifold; an introduction to Deligne's Čech cohomology classes},
journal = {Journal of Geometry and Physics},
volume = {29},
number = {4},
pages = {347-392},
year = {1999},
issn = {0393-0440},
doi = {https://doi.org/10.1016/S0393-0440(98)00045-X},
author = {Gutt, Simone  and Rawnsley, John},
}

@article{Deligne95,
  title={D{\'e}formations de l'alg{\`e}bre des fonctions d'une vari{\'e}t{\'e} symplectique: comparaison entre {Fedosov} et {De Wilde}, {Lecomte}},
  author={Deligne, Pierre},
  journal={Selecta Mathematica},
  volume={1},
  number={4},
  pages={667--697},
  year={1995},
  publisher={Birkh{\"a}user-Verlag}
}

@misc{LimitNlab, 
    title={(\(\infty\),1)-limit}, 
    author =   {{nLab authors}},
    howpublished={\url{https://ncatlab.org/nlab/show/\%28\%E2\%88\%9E\%2C1\%29-limit\#InOvercategories}}, 
    shorthand = {{nLab}}
    }

@article{ANPS,
  title = {Batalin-Vilkovisky structures on moduli spaces of flat connections},
  volume = {443},
  ISSN = {0001-8708},
  DOI = {10.1016/j.aim.2024.109580},
  journal = {Advances in Mathematics},
  publisher = {Elsevier BV},
  author = {Alekseev,  Anton and Naef,  Florian and Pulmann,  Ján and Ševera,  Pavol},
  year = {2024},
  month = may,
  pages = {109580},
    eprint={2210.08944},
    archivePrefix={arXiv},
    primaryClass={math.QA}
}

@article{CW2015,
  title = {Triviality of the higher formality theorem},
  volume = {143},
  ISSN = {1088-6826},
  DOI = {10.1090/proc/12670},
  number = {12},
  journal = {Proceedings of the American Mathematical Society},
  publisher = {American Mathematical Society (AMS)},
  author = {Calaque,  Damien and Willwacher,  Thomas},
  year = {2015},
  month = apr,
  pages = {5181–5193},
    eprint={1310.4605},
    archivePrefix={arXiv},
    primaryClass={math.QA}
}

@phdthesis{CorinaThesis,
  TITLE = {{Generalized character varieties and quantization via factorization homology}},
  AUTHOR = {Keller, Corina},
  URL = {https://theses.hal.science/tel-04397332},
  NUMBER = {2023UMONS021},
  SCHOOL = {{Universit{\'e} de Montpellier}},
  YEAR = {2023},
  MONTH = Feb,
  TYPE = {Theses},
  HAL_ID = {tel-04397332},
  HAL_VERSION = {v1},
}

@misc{Dag3Lurie,
    title={Derived Algebraic Geometry III: Commutative Algebra},
    author={Lurie, Jacob},
    year={2007},
    eprint={math/0703204},
    archivePrefix={arXiv},
    primaryClass={math.CT}
}

@misc{DAGXLurie,
  author = {Lurie, Jacob},
  title = {Derived Algebraic Geometry X: Formal Moduli Problems},
  url = {https://people.math.harvard.edu/~lurie/papers/DAG-X.pdf},
  year = {2011},
}

@article{CPTVV,
   title={Shifted Poisson structures and deformation quantization},
   volume={10},
   ISSN={1753-8424},
   DOI={10.1112/topo.12012},
   number={2},
   journal={Journal of Topology},
   publisher={Wiley},
   author={Calaque, Damien and Pantev, Tony and Toën, Bertrand and Vaquié, Michel and Vezzosi, Gabriele},
   year={2017},
   month=apr, pages={483–584},
    eprint={1506.03699},
    archivePrefix={arXiv},
    primaryClass={math.AG}}

@misc{SchommerPriesThesis,
    title={The Classification of Two-Dimensional Extended Topological Field Theories},
    author={Christopher J. Schommer-Pries},
    year={2011},
    eprint={1112.1000},
    archivePrefix={arXiv},
    primaryClass={math.AT}
}

@Article{LeMurakamiCategory,
  author    = {Le, Thang T. Q. and Murakami, Jun},
  journal   = {Communications in Mathematical Physics},
  title     = {Representation of the category of tangles by {Kontsevich's} iterated integral},
  year      = {1995},
  number    = {3},
  pages     = {535--562},
  volume    = {168},
  comment   = {Komentar},
  publisher = {Springer, Berlin/Heidelberg},
  zbl       = {0839.57008},
}

@Book{Weibel,
  author    = {Weibel, Charles A.},
  publisher = {Cambridge University Press},
  title     = {An {Introduction} to {Homological} {Algebra}},
  year      = {1994},
  address   = {Cambridge},
  isbn      = {9780521559874},
  series    = {Cambridge {Studies} in {Advanced} {Mathematics}},
  doi       = {10.1017/CBO9781139644136},
}

@misc{HTT,
 author =  {Lurie, Jacob},
 title = {Higher topos theory},
 howpublished = {\url{https://www.math.ias.edu/~lurie/papers/HTT.pdf}},
 year = {2017},
}

@misc{GJS,
  title = {The finiteness conjecture for skein modules},
  volume = {232},
  ISSN = {1432-1297},
  url = {http://dx.doi.org/10.1007/s00222-022-01167-0},
  DOI = {10.1007/s00222-022-01167-0},
  number = {1},
  journal = {Inventiones mathematicae},
  publisher = {Springer Science and Business Media LLC},
  author = {Gunningham,  Sam and Jordan,  David and Safronov,  Pavel},
  year = {2022},
  month = dec,
  pages = {301–363},
      eprint={1908.05233},
      archivePrefix={arXiv},
      primaryClass={math.QA}
}

@misc{kinnear2024nonsemisimplecraneyettertheoryvarying,
      title={Non-semisimple Crane-Yetter theory varying over the character stack}, 
      author={Patrick Kinnear},
      year={2024},
      eprint={2404.19667},
      archivePrefix={arXiv},
      primaryClass={math.QA},
  doi = {10.48550/ARXIV.2404.19667},
}

@incollection{FR,
author = {Fock, Vladimir~V. and Rosly, Alexei~A.}, 
title= {Poisson structure on moduli of flat connections on {R}iemann surfaces and the $r$-matrix}, 
booktitle = {Moscow Seminar in Mathematical Physics},
publisher = {American Mathematical Society},
address = {Providence, RI},
volume = {191},
series = {Amer. Math. Soc. Transl. Ser. 2},
pages = {67--86},
year = {1999}
}

@article{AB,
author = {Atiyah, Michael~F. and Bott, Raoul},
title = {The {Y}ang--{M}ills equations over {R}iemann surfaces},
journal = {Philosophical Transactions of the Royal Society of London. Series A, Mathematical and Physical Sciences},
volume = {308},
number = {1505},
pages = {523--615},
year = {1983}
}

@article{LiBlandSevera,
author = {Li-Bland, David and {\v{S}}evera, Pavol},
   sortname = {LiBland},
title = {Moduli spaces for quilted surfaces and {P}oisson structures},
journal = {Documenta Mathematica},
volume = {20},
pages = {1071--1135},
year = {2015}
}

@article{SeveraCenters,
   title={Left and right centers in quasi-Poisson geometry of moduli spaces},
   volume={279},
   ISSN={0001-8708},
   DOI={10.1016/j.aim.2015.01.023},
   journal={Advances in Mathematics},
   publisher={Elsevier BV},
   author={Ševera, Pavol},
   year={2015},
   month=jul, pages={263–290},
    eprint={1402.2322},
    archivePrefix={arXiv},
    primaryClass={math.SG}}

@article{AKSM,
author = {Alekseev, Anton and Kosmann-Schwarzbach, Yvette and Meinrenken, Eckhard},
title = {Quasi-{P}oisson manifolds},
journal = {Canadian Journal of Mathematics},
volume =  {54}, 
number = {1}, 
pages = {3--29},
year = {2000}
}

@article{Mouquin,
author = {Mouquin, Victor},
title = {The {F}ock–{R}osly {P}oisson structure as defined by a quasi-triangular r-matrix},
journal = {Symmetry, Integrability and Geometry: Methods and Applications},
volume = {13},
year = {2017}
}

@article{ManinPairsAKS,
author = {Alekseev, Anton and Kosmann-Schwarzbach, Yvette},
title = {Manin pairs and moment maps},
journal = {Journal of Differential Geometry},
volume = {56},
number = {1},
year = {2000}
}

@Book{Turaev,
  author    = {Turaev, Vladimir G.},
  publisher = {De Gruyter},
  title     = {Quantum {Invariants} of {Knots} and 3-{Manifolds}},
  year      = {2016},
  isbn      = {9783110435221},
  month     = jul,
  doi       = {10.1515/9783110435221},
  journal   = {Quantum Invariants of Knots and 3-Manifolds},
}

@Article{RT,
  author  = {Reshetikhin, Nikolai Y. and Turaev, Vladimir G.},
  journal = {Communications in Mathematical Physics},
  title   = {Ribbon graphs and their invariants derived from quantum groups},
  year    = {1990},
  issn    = {1432-0916},
  month   = jan,
  number  = {1},
  pages   = {1--26},
  volume  = {127},
  doi     = {10.1007/BF02096491},
}

@article{CHR2014HomotopyPullback,
  title={Bicategorical homotopy pullbacks},
  author={Cegarra,  Antonio M. and Heredia, Benjamín A. and Remedios, Josué},
  journal={Theory and Applications of Categories},
  volume={30},
  number={6},
  pages={147--205},
  year={2015},
      eprint={1404.2715},
    archivePrefix={arXiv},
}

@article {EK,
    AUTHOR = {Edwards, Robert D. and Kirby, Robion C.},
     TITLE = {Deformations of spaces of imbeddings},
   JOURNAL = {Ann. of Math. (2)},
  FJOURNAL = {Annals of Mathematics. Second Series},
    VOLUME = {93},
      YEAR = {1971},
     PAGES = {63--88},
      ISSN = {0003-486X},
   MRCLASS = {57.01},
  MRNUMBER = {283802},
MRREVIEWER = {K.\ Lamotke},
       DOI = {10.2307/1970753},
}

@article{safronovQMM,
  title={A categorical approach to quantum moment maps},
  author={Safronov, Pavel},
  journal={Theory \& Applications of Categories},
  volume={37},
  year={2021},
    eprint={1901.09031},
    archivePrefix={arXiv},
    primaryClass={math.QA}
}

@misc{PSVQv2,
  author        = {Pulmann, Ján and \v{S}evera, Pavol},
  title = {Quantization of Poisson Hopf algebras (version 2)},
   archiveprefix = {arXiv},
  eprint        = {1906.10616v2},
  primaryclass  = {math.QA},
}

@article{FGS,
  title = {Davydov–Yetter cohomology and relative homological algebra},
  volume = {30},
  ISSN = {1420-9020},
  DOI = {10.1007/s00029-024-00917-7},
  number = {2},
  journal = {Selecta Mathematica},
  publisher = {Springer Science and Business Media LLC},
  author = {Faitg,  Matthieu and Gainutdinov,  Azat M. and Schweigert,  Christoph},
  year = {2024},
  month = feb,
    eprint={2202.12287},
    archivePrefix={arXiv},
    primaryClass={math.QA}
}

@article{CY,
     author = {Crane, Louis and Yetter, David N.},
     title = {Deformations of (bi)tensor categories},
     journal = {Cahiers de Topologie et G\'eom\'etrie Diff\'erentielle Cat\'egoriques},
     pages = {163--180},
     publisher = {Dunod \'editeur, publi\'e avec le concours du CNRS},
     volume = {39},
     number = {3},
     year = {1998},
     mrnumber = {1641842},
     zbl = {0916.18005},
     language = {en},
    eprint={q-alg/9612011},
    archivePrefix={arXiv},
    primaryClass={q-alg}
}

@article{LVdB2,
  title = {Hochschild cohomology of abelian categories and ringed spaces},
  volume = {198},
  ISSN = {0001-8708},
  DOI = {10.1016/j.aim.2004.11.010},
  number = {1},
  journal = {Advances in Mathematics},
  publisher = {Elsevier BV},
  author = {Lowen,  Wendy and Van den Bergh,  Michel},
  year = {2005},
  month = dec,
  pages = {172–221},
    eprint={math/0405227},
    archivePrefix={arXiv},
    primaryClass={math.KT}
}

@unpublished{JanWIP,
    title = {Triangulations and Drinfeld associators in deformation quantization of representation varieties},
    author = {Pulmann, Ján},
    note = {in preparation}
}

@misc{CHS,
      title={The AKSZ Construction in Derived Algebraic Geometry as an Extended Topological Field Theory}, 
      author={Calaque, Damien and Haugseng, Rune and Scheimbauer, Claudia},
      year={2022},
      eprint={2108.02473},
      archivePrefix={arXiv},
      primaryClass={math.CT},
}

@article{MelaniSafronov1,
  title = {Derived coisotropic structures I: affine case},
  volume = {24},
  ISSN = {1420-9020},
  DOI = {10.1007/s00029-018-0406-2},
  number = {4},
  journal = {Selecta Mathematica},
  publisher = {Springer Science and Business Media LLC},
  author = {Melani,  Valerio and Safronov,  Pavel},
  year = {2018},
  month = mar,
  pages = {3061–3118},
    eprint={1608.01482},
    archivePrefix={arXiv},
    primaryClass={math.AG}
}

@inbook{Theo,
  title = {Heisenberg-Picture Quantum Field Theory},
  ISBN = {9783030781484},
  ISSN = {2296-505X},
  DOI = {10.1007/978-3-030-78148-4_13},
  booktitle = {Representation Theory,  Mathematical Physics,  and Integrable Systems},
  publisher = {Springer International Publishing},
  author = {Johnson-Freyd,  Theo},
  year = {2021},
  pages = {371–409},
      eprint={1508.05908},
      archivePrefix={arXiv},
      primaryClass={math-ph}, 
}

@article{SafronovPLasShifted,
  title = {Poisson-Lie structures as shifted Poisson structures},
  volume = {381},
  ISSN = {0001-8708},
  DOI = {10.1016/j.aim.2021.107633},
  journal = {Advances in Mathematics},
  publisher = {Elsevier BV},
  author = {Safronov,  Pavel},
  year = {2021},
  month = apr,
  pages = {107633},
    eprint={1706.02623},
    archivePrefix={arXiv},
    primaryClass={math.AG}
}

@article{MelaniSafronov2,
  title = {Derived coisotropic structures II: stacks and quantization},
  volume = {24},
  ISSN = {1420-9020},
  DOI = {10.1007/s00029-018-0407-1},
  number = {4},
  journal = {Selecta Mathematica},
  publisher = {Springer Science and Business Media LLC},
  author = {Melani,  Valerio and Safronov,  Pavel},
  year = {2018},
  month = mar,
  pages = {3119–3173},
    eprint={1704.03201},
    archivePrefix={arXiv},
    primaryClass={math.AG}
}

@article{AKSZ,
  title = {The Geometry of the Master Equation and Topological Quantum Field Theory},
  volume = {12},
  ISSN = {1793-656X},
  DOI = {10.1142/s0217751x97001031},
  number = {07},
  journal = {International Journal of Modern Physics A},
  publisher = {World Scientific Pub Co Pte Lt},
  author = {Alexandrov,  Mikhail and Schwarz,  Albert and Zaboronsky,  Oleg and Kontsevich,  Maxim},
  year = {1997},
  month = mar,
  pages = {1405–1429},
    eprint={hep-th/9502010},
    archivePrefix={arXiv},
    primaryClass={hep-th}
}

@misc{Davidov,
    title={Twisting of monoidal structures},
    author={Davydov, Alexei A. },
    year={1997},
    eprint={q-alg/9703001},
    archivePrefix={arXiv},
    primaryClass={q-alg}
}

@article{DijkgraafWitten,
  title = {Topological gauge theories and group cohomology},
  volume = {129},
  ISSN = {1432-0916},
  DOI = {10.1007/bf02096988},
  number = {2},
  journal = {Communications in Mathematical Physics},
  publisher = {Springer Science and Business Media LLC},
  author = {Dijkgraaf,  Robbert and Witten,  Edward},
  year = {1990},
  month = apr,
  pages = {393–429}
}

@phdthesis{JacobsThesis,
    author = {Jacobs, Andrew D.},
    title = {Nonstandard quantum groups: twisting constructions and noncommutative differential geometry},
    school = {University of St Andrews},
    url = {https://research-repository.st-andrews.ac.uk/handle/10023/13693},
    year = 1998
}

@Article{Goldman86,
  author   = {Goldman, William M.},
  journal  = {Inventiones mathematicae},
  title    = {Invariant functions on {Lie} groups and {Hamiltonian} flows of surface group representations},
  year     = {1986},
  issn     = {1432-1297},
  month    = jun,
  number   = {2},
  pages    = {263--302},
  volume   = {85},
  doi      = {10.1007/BF01389091},
}

@article{MassuyeauTuraev2012,
  title = {Quasi-Poisson Structures on Representation Spaces of Surfaces},
  volume = {2014},
  ISSN = {1073-7928},
  DOI = {10.1093/imrn/rns215},
  number = {1},
  journal = {International Mathematics Research Notices},
  publisher = {Oxford University Press (OUP)},
  author = {Massuyeau,  Gwénaël and Turaev,  Vladimir},
  year = {2012},
  month = oct,
  pages = {1–64},
    eprint={1205.4898},
    archivePrefix={arXiv},
    primaryClass={math.GT}
}

@article{Nie2013,
   title={The quasi-Poisson Goldman formula},
   volume={74},
   ISSN={0393-0440},
   DOI={10.1016/j.geomphys.2013.06.010},
   journal={Journal of Geometry and Physics},
   publisher={Elsevier BV},
   author={Nie, Xin},
   year={2013},
   month=dec, pages={1–17},
    eprint={1301.5231},
    archivePrefix={arXiv},
    primaryClass={math.DG}}

@article{BFFLS,
  title = {Deformation theory and quantization. I. Deformations of symplectic structures},
  volume = {111},
  ISSN = {0003-4916},
  DOI = {10.1016/0003-4916(78)90224-5},
  number = {1},
  journal = {Annals of Physics},
  publisher = {Elsevier BV},
  author = {Bayen,  François and Flato,  Moshé and Fronsdal,  Christian and Lichnerowicz,  André and Sternheimer,  Daniel},
  year = {1978},
  month = mar,
  pages = {61–110}
}

@Article{RocheSzenes02,
  author    = {Roche, Philippe and Szenes, András},
  journal   = {Advances in Mathematics},
  title     = {Trace Functionals on Noncommutative Deformations of Moduli Spaces of Flat Connections},
  year      = {2002},
  issn      = {0001-8708},
  month     = jun,
  number    = {2},
  pages     = {133–192},
  volume    = {168},
  doi       = {10.1006/aima.2001.2045},
  publisher = {Elsevier BV},
    eprint={math/0008149},
    archivePrefix={arXiv},
    primaryClass={math.QA}
}

@Article{AMR1,
  author   = {Andersen, Jørgen E. and Mattes, Josef and Reshetikhin, Nicolai},
  journal  = {Topology},
  title    = {The poisson structure on the moduli space of flat connections and chord diagrams},
  year     = {1996},
  issn     = {0040-9383},
  month    = oct,
  number   = {4},
  pages    = {1069--1083},
  volume   = {35},
  doi      = {10.1016/0040-9383(95)00059-3},
}

@Article{AMR2,
  author    = {Andersen, Jørgen E. and Mattes, Josef and Reshetikhin, Nicolai},
  journal   = {Mathematical Proceedings of the Cambridge Philosophical Society},
  title     = {Quantization of the algebra of chord diagrams},
  year      = {1998},
  issn      = {1469-8064, 0305-0041},
  month     = nov,
  number    = {3},
  pages     = {451--467},
  volume    = {124},
  doi       = {10.1017/S0305004198002813},
  publisher = {Cambridge University Press},
}

@article {Tambara,
    AUTHOR = {Tambara, Daisuke},
     TITLE = {A duality for modules over monoidal categories of
              representations of semisimple {H}opf algebras},
   JOURNAL = {J. Algebra},
  FJOURNAL = {Journal of Algebra},
    VOLUME = {241},
      YEAR = {2001},
    NUMBER = {2},
     PAGES = {515--547},
      ISSN = {0021-8693,1090-266X},
   MRCLASS = {16W30 (16D90 16S40 18D10)},
  MRNUMBER = {1843311},
MRREVIEWER = {Bodo\ Pareigis},
       DOI = {10.1006/jabr.2001.8771},
}

@book {CG1,
    AUTHOR = {Costello, Kevin and Gwilliam, Owen},
     TITLE = {Factorization algebras in quantum field theory. {V}ol. 1},
    SERIES = {New Mathematical Monographs},
    VOLUME = {31},
 PUBLISHER = {Cambridge University Press, Cambridge},
      YEAR = {2017},
     PAGES = {ix+387},
      ISBN = {978-1-107-16310-2},
   MRCLASS = {81-01 (17B69 18D50 53D55 81R05 81R10)},
  MRNUMBER = {3586504},
MRREVIEWER = {Domenico\ Fiorenza},
       DOI = {10.1017/9781316678626},
}

@book {CG2,
    AUTHOR = {Costello, Kevin and Gwilliam, Owen},
     TITLE = {Factorization algebras in quantum field theory. {V}ol. 2},
    SERIES = {New Mathematical Monographs},
    VOLUME = {41},
 PUBLISHER = {Cambridge University Press, Cambridge},
      YEAR = {2021},
     PAGES = {xiii+402},
      ISBN = {978-1-107-16315-7},
   MRCLASS = {81-02 (18M99 70S10 81R10 81T70)},
  MRNUMBER = {4300181},
MRREVIEWER = {Domenico\ Fiorenza},
       DOI = {10.1017/9781316678664},
}

@incollection {primer,
    AUTHOR = {Ayala, David and Francis, John},
     TITLE = {A factorization homology primer},
 BOOKTITLE = {Handbook of homotopy theory},
    SERIES = {CRC Press/Chapman Hall Handb. Math. Ser.},
     PAGES = {39--101},
 PUBLISHER = {CRC Press, Boca Raton, FL},
      YEAR = {2020},
      ISBN = {978-0-815-36970-7},
   MRCLASS = {55N40 (57N35 57R56)},
  MRNUMBER = {4197982},
}

@article{SafronovBracesPoissonAdditivity,
   title={Braces and Poisson additivity},
   volume={154},
   ISSN={1570-5846},
   DOI={10.1112/s0010437x18007212},
   number={8},
   journal={Compositio Mathematica},
   publisher={Wiley},
   author={Safronov, Pavel},
   year={2018},
   month=jul, pages={1698–1745},
    eprint={1611.09668},
    archivePrefix={arXiv},
    primaryClass={math.AG}
}

@article{GwilliamRejzner2017,
   title={Relating Nets and Factorization Algebras of Observables: Free Field Theories},
   volume={373},
   ISSN={1432-0916},
   DOI={10.1007/s00220-019-03652-9},
   number={1},
   journal={Communications in Mathematical Physics},
   publisher={Springer Science and Business Media LLC},
   author={Gwilliam, Owen and Rejzner, Kasia},
   year={2020},
   month=jan, pages={107–174},
    eprint={1711.06674},
    archivePrefix={arXiv},
    primaryClass={math-ph}}

@article{BZBJ2,
   title={Quantum character varieties and braided module categories},
   volume={24},
   ISSN={1420-9020},
   DOI={10.1007/s00029-018-0426-y},
   number={5},
   journal={Selecta Mathematica},
   publisher={Springer Science and Business Media LLC},
   author={Ben-Zvi, David and Brochier, Adrien and Jordan, David},
   year={2018},
   month=jul, pages={4711–4748},
    eprint={1606.04769},
    archivePrefix={arXiv},
    primaryClass={math.QA}
}

@Article{LBSQ1,
  author        = {Li-Bland, David and \v{S}evera, Pavol},
    sortname = {LiBland},
  journal       = {International Mathematics Research Notices},
  title         = {On Deformation Quantization of Poisson–Lie Groups and Moduli Spaces of Flat Connections: Fig. 1.},
  year          = {2014},
  issn          = {1687-0247},
  month         = sep,
  number        = {15},
  pages         = {6734 - 6751},
  volume        = {2015},
  archiveprefix = {arXiv},
  doi           = {10.1093/imrn/rnu154},
  eprint        = {1307.2047},
  primaryclass  = {math.QA},
  publisher     = {Oxford University Press},
}

@unpublished{Walker,
    author = {Walker, Kevin},
    title = {TQFTs (early incomplete draft)},
    note = {unpublished} , 
    url = {https://canyon23.net/math/tc.pdf}  , 
}

@incollection {Yetter92,
    AUTHOR = {Yetter, David N.},
     TITLE = {Tangles in prisms, tangles in cobordisms},
 BOOKTITLE = {Topology '90 ({C}olumbus, {OH}, 1990)},
    SERIES = {Ohio State Univ. Math. Res. Inst. Publ.},
    VOLUME = {1},
     PAGES = {399--443},
 PUBLISHER = {de Gruyter, Berlin},
      YEAR = {1992},
      ISBN = {3-11-012598-6},
   MRCLASS = {57M25 (57N10)},
  MRNUMBER = {1184424},
}

@article {MW11,
    AUTHOR = {Morrison, Scott and Walker, Kevin},
     TITLE = {Higher categories, colimits, and the blob complex},
   JOURNAL = {Proc. Natl. Acad. Sci. USA},
  FJOURNAL = {Proceedings of the National Academy of Sciences of the United
              States of America},
    VOLUME = {108},
      YEAR = {2011},
    NUMBER = {20},
     PAGES = {8139--8145},
      ISSN = {0027-8424,1091-6490},
   MRCLASS = {18D20 (18D05 57R56)},
  MRNUMBER = {2806651},
MRREVIEWER = {Theo\ Johnson-Freyd},
       DOI = {10.1073/pnas.1018168108},
}

@book {Gabriel-Ulmer,
    AUTHOR = {Gabriel, Peter and Ulmer, Friedrich},
     TITLE = {Lokal pr\"asentierbare {K}ategorien},
    SERIES = {Lecture Notes in Mathematics},
    VOLUME = {Vol. 221},
 PUBLISHER = {Springer-Verlag, Berlin-New York},
      YEAR = {1971},
     PAGES = {v+200},
   MRCLASS = {18AXX},
  MRNUMBER = {327863},
MRREVIEWER = {J.\ R.\ Isbell},
}

@article {KellyLFP,
    AUTHOR = {Kelly, Max },
     TITLE = {Structures defined by finite limits in the enriched context.
              {I}},
      NOTE = {Third Colloquium on Categories, Part VI (Amiens, 1980)},
   JOURNAL = {Cahiers Topologie G\'eom. Diff\'erentielle},
  FJOURNAL = {Cahiers de Topologie et G\'eom\'etrie Diff\'erentielle},
    VOLUME = {23},
      YEAR = {1982},
    NUMBER = {1},
     PAGES = {3--42},
      ISSN = {0008-0004},
   MRCLASS = {18D15 (18D20)},
  MRNUMBER = {648793},
MRREVIEWER = {Walter\ Tholen},
}

@article {CFJ,
    AUTHOR = {Chirvasitu, Alex and Johnson-Freyd, Theo},
     TITLE = {The fundamental pro-groupoid of an affine 2-scheme},
   JOURNAL = {Appl. Categ. Structures},
  FJOURNAL = {Applied Categorical Structures. A Journal Devoted to
              Applications of Categorical Methods in Algebra, Analysis,
              Order, Topology and Computer Science},
    VOLUME = {21},
      YEAR = {2013},
    NUMBER = {5},
     PAGES = {469--522},
      ISSN = {0927-2852,1572-9095},
   MRCLASS = {18D05 (14F35 18D10)},
  MRNUMBER = {3097055},
MRREVIEWER = {Leovigildo\ M.\ Alonso Tarrio},
       DOI = {10.1007/s10485-011-9275-y},
}

@phdthesis{Bird,
  title        = {Limits in 2-categories of locally-presented categories},
  author       = {Bird, Gregory J.},
  year         = {1984},
  note         = {Available at \url{http://science.mq.edu.au/~street/BirdPhD.pdf}},
  school       = {University of Sydney},
  type         = {PhD thesis}
}

@phdthesis{JanThesis,
  doi = {10.13097/ARCHIVE-OUVERTE/UNIGE:154218},
  author = {Pulmann,  Ján},
  language = {en},
  title = {On Quantization of Moduli Spaces and Poisson-Hopf Algebras},
  publisher = {Université de Genève},
  year = {2021},
  copyright = {Free access}
}

@Article{AGS95,
  author    = {Alekseev, Anton and Grosse, Harald and Schomerus, Volker},
  journal   = {Communications in Mathematical Physics},
  title     = {Combinatorial quantization of the Hamiltonian Chern-Simons theory I},
  year      = {1995},
  number    = {2},
  pages     = {317--358},
  volume    = {172},
  publisher = {Springer},
}

@misc{AraujoGuuHudson2025,
      title={Skein Construction of Balanced Tensor Products}, 
      author={Manuel Araújo and Jin-Cheng Guu and Skyler Hudson},
      year={2025},
      eprint={2501.05747},
      archivePrefix={arXiv},
      primaryClass={math-ph},
}

\end{document}